%% file: dissertation.tex
\documentclass[letterpaper,12pt]{report}	

\input{structural/preamble}  

\input{structural/about}

\degree{Doctor of Philosophy}
\degreeabbr{PhD} 

\dissertation  

\begin{document}

\input{structural/body}

\end{document}

%% file: structural/preamble.tex
\usepackage[authoryear, round]{natbib}
\usepackage{styles/utformat}  		

\usepackage{pdflscape}
\usepackage{graphicx}
\usepackage{amsmath,amsthm,amsfonts,amscd}
\usepackage{eucal} 	 	
\usepackage{verbatim}      	
\usepackage{styles/citesort}         	%
\usepackage{algorithm}
\usepackage{bibentry}
\usepackage{xcolor}
\usepackage{markdown}
\usepackage{booktabs}
\usepackage{url}

\usepackage{multirow}
\usepackage{makecell}
\usepackage{tabularx}
\usepackage{bbm}
\usepackage{mathtools}
\usepackage{hyperref}
\usepackage{tabularx}
\hypersetup{colorlinks=false}

\oneandonehalfspacing

\newtheorem{definition}{\bf Definition}
\newtheorem{problem}{\bf Problem}
\newtheorem{assumption}{\bf Assumption}
\newtheorem{lemma}{\bf Lemma}
\newtheorem{remark}{\bf Remark}
\newtheorem{theorem}{\bf Theorem}

\newtheorem{example}{\bf Example}

\newcommand{\argmin}{\mathop{\rm argmin}\limits}

\usepackage{bm}
\usepackage{dsfont}
\usepackage{enumerate}
\usepackage{enumitem}
\usepackage{algorithm}
\usepackage{algpseudocode}
\usepackage{amsmath}

\newcommand{\latexe}{{\LaTeX\kern.125em2%
                      \lower.5ex\hbox{$\varepsilon$}}}

\chardef\bslash=`\\	

\makeatletter		
\def\square{\RIfM@\bgroup\else$\bgroup\aftergroup$\fi
  \vcenter{\hrule\hbox{\vrule\@height.6em\kern.6em\vrule}%
                                              \hrule}\egroup}
\makeatother		

\nobibliography*

%% file: structural/about.tex
\author{Apurva Patil}  	

\address{Your Email or Physical Address (for Vita)}  

\title{Advancing Frontiers of Path Integral \\Theory for Stochastic Optimal Control} 

\supervisor[Takashi Tanaka]
	{Dongmei ``Maggie" Chen}

\committeemembers
	[Luis Sentis]
	[Mitch Pryor]
	{Ann Majewicz Fey}

%% file: structural/body.tex

\copyrightpage  

\commcertpage   

\titlepage      


\topskip0pt
\vspace*{\fill}
\hspace{20mm} Dedicated to my parents and my brother.
\vspace*{\fill}




%
\utabstract
\input{chapters/abstract}    

\tableofcontents   

\listoftables      
\listoffigures     

%
%

\include{chapters/introduction}
\include{chapters/CC}
\include{chapters/game}
\include{chapters/hierarchy}
\include{chapters/deception}
\include{chapters/continuous_time_deception}
\include{chapters/LQR}
\include{chapters/publications}
\include{chapters/summary}

%
\appendices

\include{chapters/appendix}


\bibliographystyle{plainnat}  
\bibliography{references}        




%% file: chapters/abstract.tex
Stochastic Optimal Control (SOC) problems arise in systems where the dynamics are influenced by random disturbances and uncertainties, such as robotic systems navigating unpredictable environments or financial systems subject to asset price fluctuations. These problems are modeled using stochastic differential equations (SDEs), where the goal is to design control inputs that minimize a cost function while accounting for the probabilistic nature of the system's evolution. Traditional methods like dynamic programming, though powerful, are computationally expensive, particularly for high-dimensional and non-linear systems due to the \textit{curse of dimensionality}. The path integral control approach offers a promising alternative for solving SOC problems, especially in complex, high-dimensional systems. Based on the Feynman-Kac theorem, the path integral method transforms the SOC problem into an expectation over noisy system trajectories, enabling efficient policy computation through Monte Carlo sampling. As the number of Monte Carlo samples approaches infinity, the policy generated by the path integral method converges to an optimal policy. Due to its purely simulator-driven nature, path integral control is applicable to high-dimensional, nonlinear, and SOC problems, where conventional model-based policy synthesis methods often face challenges. The ability to parallelize the Monte Carlo simulations on Graphics Processing Units (GPUs) makes the path integral approach scalable and well-suited for real-time applications. Additionally, it operates directly on the system’s non-linear dynamics, allowing it to handle a wide range of practical systems where traditional linear approximations fall short. The path integral approach is inherently robust to uncertainties, making it particularly relevant for safety-critical applications such as autonomous vehicle navigation, where accounting for sensor noise and unpredictable conditions is essential for safe operation. Its real-time policy generation capability makes it adaptable to dynamic environments, further enhancing its utility in autonomous systems and robotics. This dissertation develops and applies the path integral control theory to six classes of SOC problems: Chance-Constrained SOC, Two-Player Zero-Sum Stochastic Differential Games, Deceptive Control, Task Hierarchical Control, Risk Mitigation of Stealthy Attack and Discrete-Time Stochastic Linear Quadratic Regulator (LQR). Additionally, it presents a sample complexity analysis of the path integral controller for discrete-time stochastic LQR problems. Through these contributions, this research work lays foundational groundwork for simulator-driven autonomy in solving complex SOC problems, particularly those involving non-linear dynamics and high-dimensional state spaces.

%% file: chapters/introduction.tex
\chapter{Introduction}
Making the best possible decisions within the available time and computational resource constraints is a necessary skill for truly autonomous systems. In many engineering scenarios, optimal decision-making requires solving Stochastic Optimal Control (SOC) problems. SOC refers to a class of control problems that involves making optimal decisions for systems whose dynamics are influenced by random disturbances or uncertainties. These uncertainties can arise from various factors such as environmental noise, model inaccuracies, or unobserved phenomena. Examples include robotic systems navigating uncertain environments, financial systems where asset prices evolve stochastically, and energy systems subject to demand fluctuations. Unlike deterministic control problems, where the system's behavior can be predicted exactly, SOC problems account for the randomness in the system. The general SOC problem involves a dynamical system described by stochastic differential equations (SDEs). These equations model the evolution of the system's state over time, driven by both control inputs and stochastic processes, typically represented as \textit{Wiener} or \textit{Brownian motion} \cite{durrett2019probability}. The challenge lies in designing control inputs that minimize a given cost function, such as energy consumption or deviation from a target state, while considering the probabilistic evolution of the system. This cost function is often an expectation over all possible system trajectories, making the optimization inherently complex.\par

At the core of SOC is the Bellman principle of optimality, which states that the optimal policy at any given time depends only on the current state of the system, not on the prior history of states. This leads to the formulation of the Hamilton-Jacobi-Bellman (HJB) equation, a partial differential equation (PDE) that characterizes the value function of the SOC problem. Solving the HJB PDE yields the optimal control policy, but for high-dimensional systems and non-linear dynamics, solving the HJB equation analytically is often intractable. One of the commonly used approach to solve HJB PDE is \textit{dynamic programming} which breaks down the SOC problem into smaller subproblems by recursively solving the HJB PDE. While powerful, this method suffers from the \textit{curse of dimensionality}, where the computational cost grows exponentially with the number of state variables. This makes the traditional dynamic programming approaches impractical for real-time control.  \par

The path integral approach, rooted in the principles of statistical physics, offers an alternative method for solving nonlinear SOC problems online ~\cite{kappen2005path, theodorou2010generalized, williams2016aggressive}. The fundamental insight behind path integral control lies in the Feynman-Kac theorem \cite{oksendal2003stochastic}, a result from stochastic calculus that connects PDEs with stochastic processes. This theorem allows the transformation of the HJB equation into an expectation of an exponential cost function over noisy system trajectories which allows for efficient computation of optimal policies through Monte Carlo sampling methods. According to the strong law of large numbers \cite{durrett2019probability}, as the number of Monte Carlo samples approaches infinity, the policy generated by the path integral method converges to an optimal policy. The path integral method can be flexibly applied to high-dimensional control problems that involve nonlinear, stochastic, and hybrid dynamics -- areas where conventional model-based approaches often face challenges. Moreover, it relies on Monte Carlo simulations that can be massively parallelized on Graphics Processing Units (GPUs). Millions of trajectories can be simulated simultaneously significantly reducing computation time making the path integral approach less susceptible to the curse of dimensionality \cite{williams2017model}. Moreover, in certain cases, it can be formally shown that the required number of samples of Monte Carlo simulations grows only logarithmically as a function of the dimension of the control input \cite{patil2024discrete}, in contrast to exact dynamic programming, whose computational cost grows exponentially. \par 

The path integral approach is inherently robust to uncertainties. By accounting for the stochastic nature of the system, it ensures that control policies are not only optimized for expected outcomes but also for handling variability in system behavior. This makes the path inetgral method highly relevant for safety-critical applications, where uncertainties can lead to catastrophic failures if not properly managed. For instance, in autonomous vehicle navigation, the ability to optimize trajectories under uncertain road conditions or sensor noise is crucial for ensuring safe operation. Another significant benefit of the path integral control approach is its ability to handle non-linear systems. Many real-world systems, such as autonomous robots or drones, have non-linear dynamics that make traditional control methods, which rely on linear approximations, inadequate. The path integral method, however, operates directly on the system's non-linear dynamics, making it well-suited for controlling systems where non-linearities are inherent and crucial to their behavior. Moreover, the path integral approach is particularly effective for online computation of control policies which is important for applications such as autonomous driving or robotic navigation, where decisions must be made in real time. The path integral method can generate control inputs dynamically as new information becomes available, making it adaptable to rapidly changing environments. Finally, the path integral approach determines control inputs solely through simulators, bypassing the need for analytical models of the plant. This is particularly beneficial in scenarios where a simulator engine exists, but creating a mathematical (equation-based) model of the system poses challenges. As applications in robotics, autonomous systems, and other fields continue to grow, the path integral control approach is likely to play a critical role in advancing the state of the art in control theory and practice.

\section{History of Path Integral Control}\label{Sec: history}
The Path Integral Control is a control algorithm inspired by the Path Integral formulation of quantum mechanics. The origin of path integral control can be traced back to the stochastic variational treatment of quantum mechanics \cite{nelson1966derivation, nelson2020dynamical,rosenbrock1985variational,rosenbrock1995stochastic,guerra1983quantization,YASUE1981327}.  Using Nelson's diffusion model \cite{nelson1966derivation}, references \cite{guerra1983quantization,YASUE1981327} derived an SOC problem whose associated HJB equation coincides with the linear Schr\"{o}dinger equation. 
This connection between SOC and quantum mechanics was exploited by Itami \cite{2001193,2003637}, who proposed a Monte Carlo (Metropolis) algorithm to evaluate the path integral representation of the wave function. Noticing the semi-classical limit of the   Schr\"{o}dinger equation approximates the HJB equation, \cite{2003637} proposed a Monte Carlo-based solution approach to deterministic nonlinear optimal control problems. This is the first appearance of what we now know as path integral control to the best of the our knowledge. However, it should be noted that mathematically relevant problems, such as risk-sensitive control \cite[Ch.~VI]{fleming2006controlled},\cite{whittle1981risk, jacobson1973optimal}, distributionally robust control against relative entropy ambiguity set \cite{petersen2000minimax}, and linearly solvable Markov-Decision Processes (MDP) \cite{todorov2006linearly,todorov2009efficient} (also known as the KL control problems) have been investigated in separate research threads.The line of work \cite{todorov2006linearly,todorov2009efficient} has been broadly accepted by the neuroscience community, where linearly solvable MDPs are now widely used to reason about brain's functionality in spatial navigation (e.g., \cite{piray2021linear}).

Kappen \cite{kappen2005path} showed that a certain class of \emph{stochastic} optimal control problems can be solved by the path integral method without semi-classical approximation. For such SOC problems, it was shown that the associated stochastic HJB equation can be linearized by the Cole-Hopf transformation (a typical change-of-variables technique to relate the HJB equation with Burgers' equation and the heat equation \cite{evans2022partial}). The Feynman-Kac lemma \cite{oksendal2003stochastic} is then invoked to show that the solution to the linearized PDE admits a path integral representation, which can be numerically evaluated by Monte Carlo simulations. The authors of \cite{theodorou2010generalized, theodorou2010reinforcement} demonstrated the path integral for policy improvements (PI$^2$), who pioneered significant developments of path integral control in robot learning \cite{pastor2011skill,sugimoto2011phase,rombokas2012tendon,okada2018acceleration,abraham2020model,bhardwaj2022storm}. 
A receding horizon implementation of path integral control, called Model Predictive Path Integral (MPPI) control \cite{williams2017information,williams2017model}, demonstrated aggressive driving \cite{williams2016aggressive}, where real-time Monte Carlo simulations were parallelized on GPUs. The path integral control method has also been applied to constrained systems and systems with non-differentiable dynamics \cite{satoh2020nonlinear,carius2022constrained}.

An alternative derivation of path integral control bypassing the Feynman-Kac lemma was presented in the work \cite{theodorou2012relative}. In \cite{theodorou2012relative}, the authors used the variational lower bound and the Girsanov theorem to establish a connection between the linearly solvable MDP (the KL control problems) and the problem studied by Kappen \cite{kappen2005path}.
Yet another derivation of path integral control is possible by formulating an SOC as a statistical inference problem, and then applying the standard graphical model inference algorithms \cite{kappen2012optimal, levine2018reinforcement}.

It is worth noting that, in all derivations, the path integral control method faces the same limitation. It is optimal only for the class of SOC problems whose dynamic programming equations (e.g., HJB PDEs) are linearizable. Removal of this restriction has been studied in the literature. In \cite{satoh2016iterative}, the authors proposed an iterative applications of the Feynman-Kac lemma. The MPPI framework \cite{williams2017information} utilizes a heuristic approach based on the KL-divergence minimization. However, this limitation has not been satisfactorily removed as of today.

\section{Path Integral Control in Real-World Application}
Path Integral Control or its model-predictive variant known as Model Predictive Path Integral (MPPI), has gained significant traction in real-world applications. Its foundation in stochastic optimal control, coupled with its ability to handle complex dynamics and non-convex cost landscapes, makes it highly suitable for a range of domains requiring adaptive, real-time decision-making. Unlike deterministic planners, path integral control naturally incorporates uncertainty into its decision-making process. Its reliance on sampling-based computation allows effective parallelization on modern hardware like GPUs. Moreover, path integral controller does not require pre-computed policies or extensive training, making it highly adaptive to changing environments. While methods like Model Predictive Controller (MPC), Reinforcement Learning (RL), and Proportional-Integral-Derivative (PID) controller have their respective niches, path integral controller offers a unique blend of robustness, computational efficiency, and adaptability that makes it a compelling choice in scenarios where uncertainty and real-time performance are critical. This section discusses prominent application areas and provides a comparative analysis where path integral control outperforms alternative methods.

\subsubsection{Autonomous Driving}
One of the most visible applications of path integral control is in autonomous driving. The ability to compute control trajectories in real time under uncertainty has made path integral control an attractive choice for vehicle control tasks such as trajectory planning, obstacle avoidance, and dynamic lane following. In \cite{williams2018information}, the authors demonstrated the efficacy of MPPI in controlling a small-scale autonomous vehicle. Their implementation leveraged GPU acceleration to achieve real-time performance, showcasing MPPI’s scalability to hardware capabilities. Unlike conventional optimization-based planners (e.g., Sequential Quadratic Programming or RRT*), path integral control inherently accounts for stochasticity in its formulation. This stochastic nature provides robustness in environments with high sensor noise or dynamic obstacles. Moreover, compared to reinforcement learning (RL)-based methods, MPPI does not require extensive offline training and can adapt online to unforeseen changes in the environment.

\subsubsection{Robotic Manipulation}
In robotics, manipulation tasks often involve high degrees of freedom and require precise control in the presence of uncertainty. Path integral controller's sampling-based approach allows it to handle such complexities effectively. In manufacturing and logistics, robotic arms can employ path integral controller for tasks such as pick-and-place, assembly, and welding. \cite{bhardwaj2022storm} demonstrated that path integral controller is a promising tool for manipulators with complex, non-smooth dynamics, and cost functions. They showed that path integral controller can tightly integrate perception into the control problem by utilizing learned cost functions from raw sensor data. While methods such as CHOMP (Covariant Hamiltonian Optimization for Motion Planning) and STOMP (Stochastic Trajectory Optimization for Motion Planning) are popular for offline or semi-online trajectory optimization in robotic manipulation, there are several compelling reasons why MPPI often serves as a superior choice for real-time control of manipulators:
\begin{itemize}
    \item CHOMP/STOMP are primarily focused on optimizing a geometric or kinematic path, with dynamics or actuation constraints added in a secondary manner. They typically produce a reference trajectory to be tracked by a lower-level controller. On the other hand, MPPI directly solves an optimal control problem that accounts for system dynamics at each time step. Instead of splitting planning and control, MPPI unifies them in a receding-horizon framework, adjusting control inputs as the state evolves. This is particularly powerful if the manipulator must react to disturbances, changing goals, or time-varying constraints in real-time.
    \item CHOMP/STOMP usually perform iterative trajectory updates in an offline or batch manner. They can handle moderate uncertainty by re-optimizing, but the update frequency is often not high enough for truly dynamic tasks (e.g., tracking a moving object). MPPI constantly re-samples control sequences at a high rate. This means the manipulator can handle abrupt changes (like collisions or shifting targets) within milliseconds, which is crucial for safe, robust manipulation.
    \item CHOMP relies on gradient-based updates of a functional that measures collision cost and smoothness. While effective for certain geometric settings, it may face difficulties when contact or friction models become discontinuous. STOMP is stochastic, like MPPI, but it modifies trajectories in configuration space, often with partial reliance on gradient-like terms derived from the cost. MPPI is purely sampling-based. Non-smooth dynamics (e.g., contact) do not invalidate the update step, because no gradient needs to be calculated—trajectories that produce collisions or large friction forces simply get higher cost and lower exponential weighting in the rollouts.
    \item CHOMP/STOMP can be parallelized to some degree, but they are usually iterative solvers that sequentially refine trajectories. Real-time or near real-time performance for high-dimensional arms can be challenging. MPPI is explicitly designed to exploit parallel hardware, simulating many rollouts simultaneously. Modern GPUs have made real-time MPPI feasible even for complex manipulators.
\end{itemize}



\subsubsection{Legged Locomotion}

The authors in \cite{alvarez2024real} implement MPPI control on a real Unitree Go1 robot, showing that a sampling‐based approach can be run online at 100 Hz. This represents a departure from purely simulation‐oriented studies and demonstrates that MPPI, previously considered too computationally expensive for higher‐dimensional platforms, can be deployed successfully on physical hardware. By leveraging the full robot model (body plus legs, rather than just foot‐terrain contacts), the control framework plans for non‐traditional support contacts (e.g., using the robot’s torso or shoulders) during challenging tasks. This results in behaviors like:
\begin{itemize}
    \item Traversing uneven or tall obstacles (a box nearly the robot’s height) by hopping, re‐positioning its feet, or briefly using its body for support. 
    \item Pushing and manipulating a box to user‐specified goal locations, sometimes nudging it with the legs, shoulders, or torso—without the robot ever being explicitly told how to make contact.
\end{itemize}
A key enabler is the efficient parallelization of MuJoCo: the authors evaluate around 30–50 simulations in parallel at each policy update. Even though MuJoCo runs on a CPU (rather than GPU/TPU), it is sufficiently fast for in‐the‐loop control at real‐time rates. This shows that ``sampling‐based, whole‐body" control can now handle high‐dimensional dynamics in practice. When applied to real hardware, the MPPI controller displays resilience to terrain uncertainties and external pushes. By sampling over possible joint angles—parameterized through cubic splines—the method uncovers emergent strategies (such as using the torso as an additional contact point or slightly ``kicking" the box) without manual scripting or offline training. Unlike reinforcement learning methods that often require days of simulation to learn a policy, path integral control is purely online. The robot does not need an offline data‐collection or policy‐optimization phase; it instead samples state‐action trajectories in real time and updates its control inputs based on those rollouts.

\subsubsection{UAV Navigation and Exploration}
Unmanned Aerial Vehicles (UAVs) frequently operate in environments with limited visibility and high uncertainty. Path integral controller has been applied to UAV navigation tasks to enhance robustness and agility. \cite{minarik2024model} utilized MPPI for a quadrotor navigating through dense forests, outperforming conventional MPC in terms of both safety and computational efficiency. By leveraging stochastic sampling, MPPI could better anticipate and react to potential collisions. While MPC is a strong contender in UAV navigation, its reliance on linearized models often limits its applicability in highly nonlinear environments. Path integral control, on the other hand, directly simulates the nonlinear dynamics, offering a more accurate representation of the system.

\subsubsection{Multi-Robot Systems}
In \cite{streichenberg2023multi}, multi-agent path integral control is applied to interaction-aware motion planning in urban canals, where multiple autonomous vessels must navigate constrained waterways in real time while avoiding collisions and complying with traffic rules. Multi-Agent MPPI surpasses numerous conventional multi-agent planning and control methods by offering sampling-based flexibility, interaction-aware cost and coupling, real-time adaptability, robustness to model uncertainty, and unified global and local exploration. Its rollout mechanism naturally handles discrete changes and discontinuities without requiring derivatives, making it more effective than traditional MPC solutions that rely on gradient-based solvers. Methods like potential fields or graph search may struggle to capture the continuous nature of multi-robot interactions or may involve extensive discretization, but MPPI sidesteps these issues by directly integrating inter-robot interactions into the cost function, thereby facilitating cooperative or socially compliant maneuvers in dense traffic. In the urban canal setting, \cite{streichenberg2023multi} continuously sample and weight potential trajectories, enabling rapid re-planning as traffic density shifts, unexpected maneuvers occur, or environmental disturbances emerge. This parallel-friendly re-planning ensures robust collision avoidance, goal-reaching, and adherence to navigation constraints, which can be especially challenging for traditional deterministic planners prone to faltering under uncertainty or model mismatch. By generating a wide variety of trajectories at each timestep, MPPI also mitigates the risk of local optima that often hinder multi-agent coordination, ultimately delivering safer and more efficient motion plans.

\subsubsection{Roll-to-Roll Manufacturing}
Roll-to-roll (R2R) manufacturing is a continuous processing technology essential for scalable production of thin-film materials and printed electronics, but precise control remains challenging due to subsystem interactions, nonlinearities, and disturbances. Existing Model Predictive Control (MPC) methods address many of these issues but require solving computationally demanding quadratic programs, limiting their applicability for large-scale R2R systems. In our recent work \cite{martin2025model}, we propose a Model Predictive Path Integral (MPPI) control formulation for R2R systems, leveraging a GPU-based Monte-Carlo sampling approach to efficiently approximate optimal controls online. Crucially, MPPI easily handles non-differentiable cost functions, enabling the incorporation of complex performance criteria relevant to advanced manufacturing processes.

\section{Organization of the Dissertation}
In this work, we develop path integral theory to solve six classes of stochastic optimal control problems as depicted in Figure \ref{Fig. Dissertation Structure}.
\begin{figure}[h]
    \centering
    \!\!\!\!\!\!\!\!\!\!\!\!\includegraphics[scale=0.4]{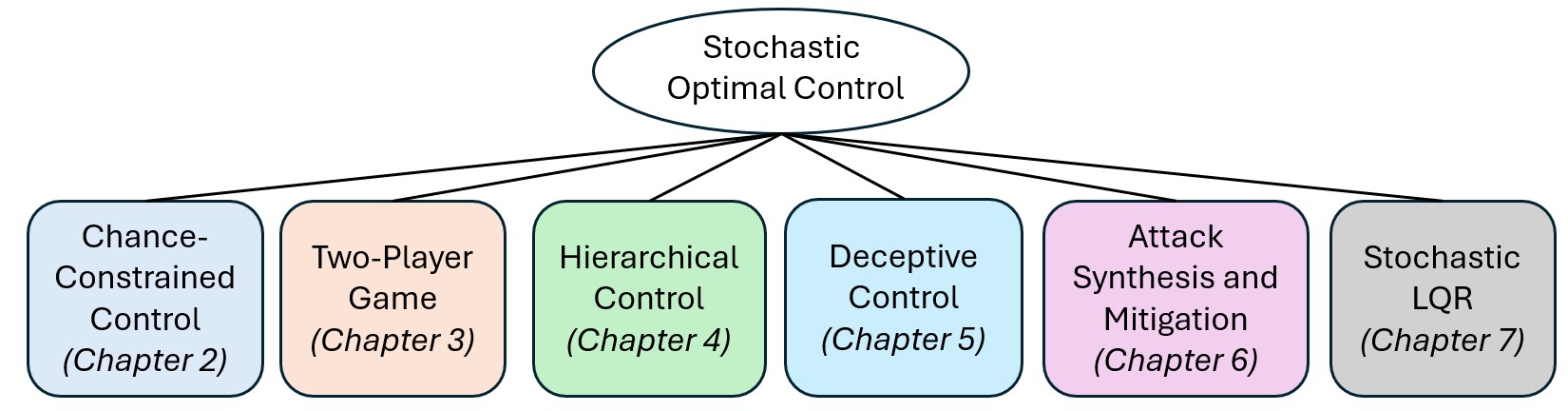} 
      
        \caption{Dissertation Structure} 
        \label{Fig. Dissertation Structure}
\end{figure}
This dissertation is organized as follows: we develop the path integral control theory for six classes of stochastic optimal control problems: Chance-Constrained SOC (Chapter \ref{Sec: Continous-Time Chance-Constrained Stochastic Optimal Control}), Two-Player Zero-Sum Stochastic Differential Games (Chapter \ref{Sec: Two-Player Zero-Sum Stochastic Differential Game}), Task Hierarchical Control (Chapter \ref{Sec: task hierarchy}), Deceptive Control (Chapter \ref{Sec: deception}), Deceptive Attack Synthesis and Its Mitigation (Chapter \ref{Sec: Deceptive Attack Synthesis and Its Mitigation for Nonlinear Cyber-Physical Systems: Path Integral Approach}) and Stochastic LQR Problem (Chapter \ref{Sec: sample complexity of LQR}). Chapter \ref{Sec: sample complexity of LQR} also derives the sample complexity of the path integral controller applied to a discrete-time stochastic LQR problem. Finally, chapter \ref{Sec: publications} provides a list of publications resulting from this work.

%% file: chapters/CC.tex
\chapter[Chance-Constrained Stochastic Optimal Control]{Chance-Constrained Stochastic Optimal Control}
\label{Sec: Continous-Time Chance-Constrained Stochastic Optimal Control}
\section{Motivation}
In safety-critical missions, quantitative characterization of system uncertainties is of critical importance as it impacts the overall safety of the system operation. System uncertainties arise due to unmodeled dynamics, unknown system parameters, and external disturbances. Subsequently, the control policies should be designed to accommodate such uncertainties in order to achieve user-defined safety requirements. Robust control is a popular paradigm to guarantee safety against set-valued uncertainties \cite{kothare1996robust, kuwata2007robust}. Generally, a robust control approach aims to synthesize a policy that optimizes the worst-case performance, commonly known as the minimax policy. While such a strategy is suitable for applications where safety is an absolute requirement, the computation of the exact minimax policy is often intractable, which necessitates a sequential outer-approximation of the uncertain sets. This often results in overly conservative solutions \cite{calafiore2006scenario}. Also, robust control is difficult to apply when the uncertainty is modeled probabilistically using random variables with unbounded support (e.g., Gaussian distributions).\par

\textit{Chance-constrained} stochastic optimal control (SOC) is an alternative paradigm for policy synthesis under uncertainty \cite{blackmore2011chance, oguri2019convex}. Unlike robust control, this approach aims to optimize the system performance by accepting a user-specified threshold for the probability of failure (e.g., collisions with obstacles). Notably, the acceptance of the possibility of failure is often effective to reduce the conservatism of the controller even if the introduced probability of failure is practically negligible \cite{vidyasagar2001randomized}. Consequently, chance-constrained SOC is also widely used as a framework for policy synthesis for systems in safety-critical missions.
\section{Literature Review}
In this work, we consider a continuous-time continuous-space chance-constrained SOC problem. A central challenge in this problem stems from the fact that the continuous-time end-to-end probability of failure is generally challenging to evaluate and optimize against. A common course of action found in the literature is to look for a tractable approximation of the chance-constrained SOC problem in order to use existing tools from optimization.  In \cite{blackmore2011chance} a discrete-time chance-constrained SOC problem is converted to a disjunctive convex program by approximating the chance constraint using Boole’s inequality. The disjunctive convex program is then solved using branch-and-bound techniques. The work presented in \cite{ono2015chance} solves the discrete-time chance-constrained problem by formulating its dual. However, in this work, the joint chance constraint is conservatively approximated using Boole's inequality. Hence, the duality gap is nonzero and the obtained solution is suboptimal. In \cite{wang2020non}, authors use statistical moments of the distributions in concentration inequalities to upper-bound the chance constraint for discrete-time trajectory planning under non-Gaussian uncertainties, and solve the problem using nonlinear program solvers. A scenario-based optimization method that translates the chance constraints into deterministic ones is provided in \cite{de2021scenario}. Taking ideas from distributionally robust optimization, deterministic convex reformulation of the chance-constrained stochastic optimal control problems is proposed in \cite {li2021distributionally} for discrete-time linear system dynamics. Different from the above works in our work, we consider continuous-time, control-affine system dynamics. For the continuous-time chance-constrained planning problem the authors in \cite{jasour2021convex} use risk contours to transform the original problem into a deterministic planning problem and use convex methods based on sum-of-squares optimization to obtain continuous-time trajectories. In \cite{nakka2019trajectory} a continuous-time chance-constrained SOC problem is converted to a deterministic control problem with convex constraints using generalized polynomial chaos expansion and Chebyshev inequality. The deterministic optimal control problem is then solved using sequential convex programming.  Other approaches that consider approximations of the chance constraints include approaches based on concentration of measure inequalities \cite{hokayem2013chance}, Bernstein approximation \cite{nemirovski2007convex}, and moment based surrogate \cite{paulson2019efficient}. A common limitation of the above approaches is the conservatism introduced by the approximation of the chance constraints, leading to overly cautious policies that compromise performance or cause artificial infeasibilities \cite{nemirovski2007convex, frey2020collision, patil2023upper, patil2022upper}. 
Deep reinforcement learning algorithms such as soft actor-critic also have been used to solve the chance-constrained SOC problems \cite{huang2021risk}. 

\section{Contributions}
In our work, we utilize the path integral approach to numerically solve the posed chance-constrained SOC problem using open-loop samples of system trajectories. 
The contributions of this work are as follows: 
\begin{enumerate}
    \item  We leverage the notion of exit time from the continuous-time stochastic calculus \cite{chern2021safe, shah2011probability} to formulate a chance-constrained SOC problem. The dual of this chance-constrained SOC problem is constructed by incorporating the chance constraint into the cost function. The chance constraint is then transformed into an expectation of an indicator function, which enjoys the same additive structure as the primal cost function. No approximation nor time discretization is introduced in this transformation.
    \item Given a fixed dual variable, we evaluate the dual objective function by numerically solving the Hamilton-Jacobi-Bellman (HJB) partial differential equation (PDE). 
    \item It is shown that under a certain assumption on the system dynamics and cost function, a strong duality holds between the primal problem and the dual problem (the duality gap is zero).
    \item We propose a novel path-integral-based dual ascent algorithm to numerically solve the dual problem. This allows us to solve the original chance-constrained problem online via open-loop samples of system trajectories. Finally, we present simulation studies on chance-constrained motion planning for spatial navigation of mobile robots. The solution obtained using the path integral approach is compared with that of the finite difference method.
    \item Finally, we provide an open-source library to solve continuous-time chance-constrained stochastic optimal control problem using the path integral controller: \url{https://github.com/patil-apurva/CC_SOC}
\end{enumerate}

\section*{Notations}
Let $\mathbb{R}$, $\mathbb{R}^n$, and $\mathbb{R}^{m\times n}$ be the set of real numbers, the set of $n$-dimensional real vectors, and the set of $m\times n$ real matrices. If a stochastic process $x(t)$ starts from $x$ at time $t$, then let $P_{x,t}\left(\mathcal{E}\right)=P\big(\mathcal{E}({x})\;|\;{x}(t)=x\big)$ denote the probability of event $\mathcal{E}({x})$ involving the stochastic process ${x}(t)$ conditioned on ${x}(t)=x$, and let $\mathbb{E}_{x, t}\left[F\left({x}\right)\right]=\mathbb{E}\left[F\left({x}\right)\;|\;{x}\left(t\right)=x\right]$ denote the expectation of $F\left({x}\right)$ (a functional of ${x}\left(t\right)$) also conditioned on ${x}\left(t\right)=x$. Let $\mathds{1}_{\mathcal{E}}$ be an indicator function, which returns $1$ when the condition $\mathcal{E}$ holds and 0 otherwise. The $\bigvee$ symbol represents a logical OR implying existence of a satisfying event among a collection. $\text{Tr}(A)$ denotes the trace of a matrix $A$. $\partial_x$ and $\partial^2_x$ are used to define, respectively, the first and second-order partial derivatives w.r.t. $x$. If $x$ is a vector then $\partial_x$ returns a column vector and $\partial^2_x$ returns a matrix. The short forms a.a. and a.s. denote almost all and almost surely, respectively. Table \ref{tab:notation CC} represents the mathematical notations frequently used in this chapter. 

\begin{table}
\begin{center}
\begin{tabular}{||c | c || c | c||} 
 \hline
 \textbf{Notation} & \textbf{Description} & \textbf{Notation} & \textbf{Description} \\ [0.5ex] 
 \hline\hline
 $\mathcal{X}_s$ & safe region & $\partial\mathcal{X}_s$ & boundary of the safe region \\ 
 \hline
 $x(t)$ & controlled process & $\hat{x}(t)$ & uncontrolled process \\
 \hline
 $u(x(t),t)$ & control input & $w(t)$ & Wiener process \\
 \hline
$P_\mathrm{fail}$ & probability of failure & $t_f$ & exit time \\
 \hline
 $ C\left(x_0, t_0, {u}(\cdot)\right)$ & cost function & $\psi\left({x}({t}_{f})\right)$ & terminal cost \\ 
 \hline
 $ V\left({x}(t), t\right)$ & running state cost & $\Delta$ & risk tolerance \\ 
 \hline
 $ \eta$ & Lagrange multiplier & $\lambda$ & PDE linearizing constant \\ 
 \hline
 $ J(x,t; \eta)$ & value function & $\xi\left(x,t;\eta\right)$ & transformed value function \\ 
 \hline
 $ \gamma$ & step size of gradient ascent & $S$ & cost-to-go \\ 
 [1ex] 
 \hline
\end{tabular}
\caption{Table of frequently used mathematical notation in Chapter \ref{Sec: Continous-Time Chance-Constrained Stochastic Optimal Control}}
\label{tab:notation CC}
\end{center}
\end{table}

\section{Preliminaries}
Let $\mathcal{X}_{s}\subseteq\mathbb{R}^n$ be a bounded open set representing a safe region, $\partial\mathcal{X}_{s}$ be its boundary, and closure $\overline{\mathcal{X}_{s}}=\mathcal{X}_{s}\cup\partial\mathcal{X}_{s}$.
\subsection{Controlled and Uncontrolled Processes}
Consider a controlled process ${x}(t)\in\mathbb{R}^n$ driven by following control-affine It\^{o} stochastic differential equation (SDE):
\begin{equation}\label{SDE}
\begin{aligned}
   d{x}(t)=&{f}\left({x}(t), t\right)dt+{G}\left({x}(t), t\right){u}({x}(t), t)dt+{\Sigma}\left({x}(t), t\right)d{w}(t),
\end{aligned}
\end{equation}
where ${u}\left({x}(t), t\right)\in\mathbb{R}^m$ is a control input, ${w}(t)\in\mathbb{R}^k$ is a $k$-dimensional standard Wiener process on a suitable probability space $\left(\Omega, \mathcal{F}, P\right)$, ${f}\left({x}(t), t\right)\in\mathbb{R}^n$, ${G}\left({x}(t), t\right)\in\mathbb{R}^{n\times m}$ and ${\Sigma}\left({x}(t), t\right)\in\mathbb{R}^{n\times k}$. Let $\hat{{x}}(t)\in\mathbb{R}^n$ be an uncontrolled process driven by the following SDE:
\begin{equation}\label{uncontrolled SDE}
  d\hat{{x}}(t)\!=\!\!{f}\!\left(\hat{{x}}(t),\! t\right)\!dt\!+\!{\Sigma}\!\left(\hat{{x}}(t),\! t\right)\!d{w}(t). 
\end{equation}
At the initial time $t_0$, ${x}(t_0)=\hat{{x}}(t_0)=x_0\in\overline{\mathcal{X}_s}$. Throughout this work, we assume sufficient regularity in the coefficients of (\ref{SDE}) and (\ref{uncontrolled SDE}) so that unique strong solutions exist \cite[Chapter 1]{oksendal2013stochastic}. In the rest of the dissertation, for notational compactness, the functional dependencies on $x$ and $t$ are dropped whenever it is unambiguous.

\subsection{Probability of Failure}
 For a given finite time horizon $t\in[t_0, T]$, $t_0<T$, if the system (\ref{SDE}) leaves the safe region $\mathcal{X}_{s}$ at any time $t\in(t_0, T]$, then we say that it fails. 
\begin{definition}[Probability of failure]
 The probability of failure $P_\mathrm{fail}\left(x_0,t_0,u(\cdot)\right)$ of system (\ref{SDE}) starting at $(x_0, t_0)$ and operating under the policy $u(\cdot)$ is defined as   
\begin{equation}\label{pfail}
    P_\mathrm{fail}\left(x_0,t_0,u(\cdot)\right)=P_{x_0,t_0}\left(\bigvee_{t\in(t_0, T]} {x}(t)\notin \mathcal{X}_{s}\right).
\end{equation}
\end{definition}
\subsection{Exit Times}
Let $\mathcal{Q}=\mathcal{X}_s\times[t_0, T)$ be a bounded set with the boundary  $\partial\mathcal{Q}=\left(\partial\mathcal{X}_s\times[t_0,T]\right)\cup\left(\mathcal{X}_s\times\{T\}\right)$, and closure $\overline{\mathcal{Q}}=\mathcal{Q}\cup\partial\mathcal{Q}=\overline{\mathcal{X}_s}\times[t_0, T]$. We define exit time ${t}_{f}$ for process ${x}(t)$ as
\begin{equation*}\label{tf with Q}
 {t}_{f} \coloneqq \text{inf}\{t> t_0: ({x}(t), t)\notin \mathcal{Q}\}.   
\end{equation*}
Alternatively, ${t}_f$ can be defined as 
\begin{equation}\label{tf}
{t}_{f} \coloneqq 
\begin{cases}
T, & \!\!\!\!\!\!\!\!\!\!\!\!\!\!\!\!\!\!\!\!\!\!\!\!\!\!\!\!\!\!\!\!\!\!\!\!\!\!\text{if}\;\; {x}(t)\in\mathcal{X}_{s}, \forall t\in(t_0, T),\\
\text{inf}\;\{t\in(t_0, T) : {x}(t)\notin\mathcal{X}_{s}\}, & \text{otherwise}.
 \end{cases}
\end{equation}
For one-dimensional state space, the domains $\mathcal{X}_s$, $\mathcal{Q}$ and their boundaries $\partial \mathcal{X}_s$, $\partial\mathcal{Q}$ are shown in Figure \ref{Fig. computational domain}. The figure also depicts the exit times for two realizations of trajectories $\{x(t), t\in[t_0, {t}_f]\}$. 
\begin{figure}[h]
    \centering
    \!\!\!\!\!\!\!\!\!\!\!\!\includegraphics[scale=0.5]{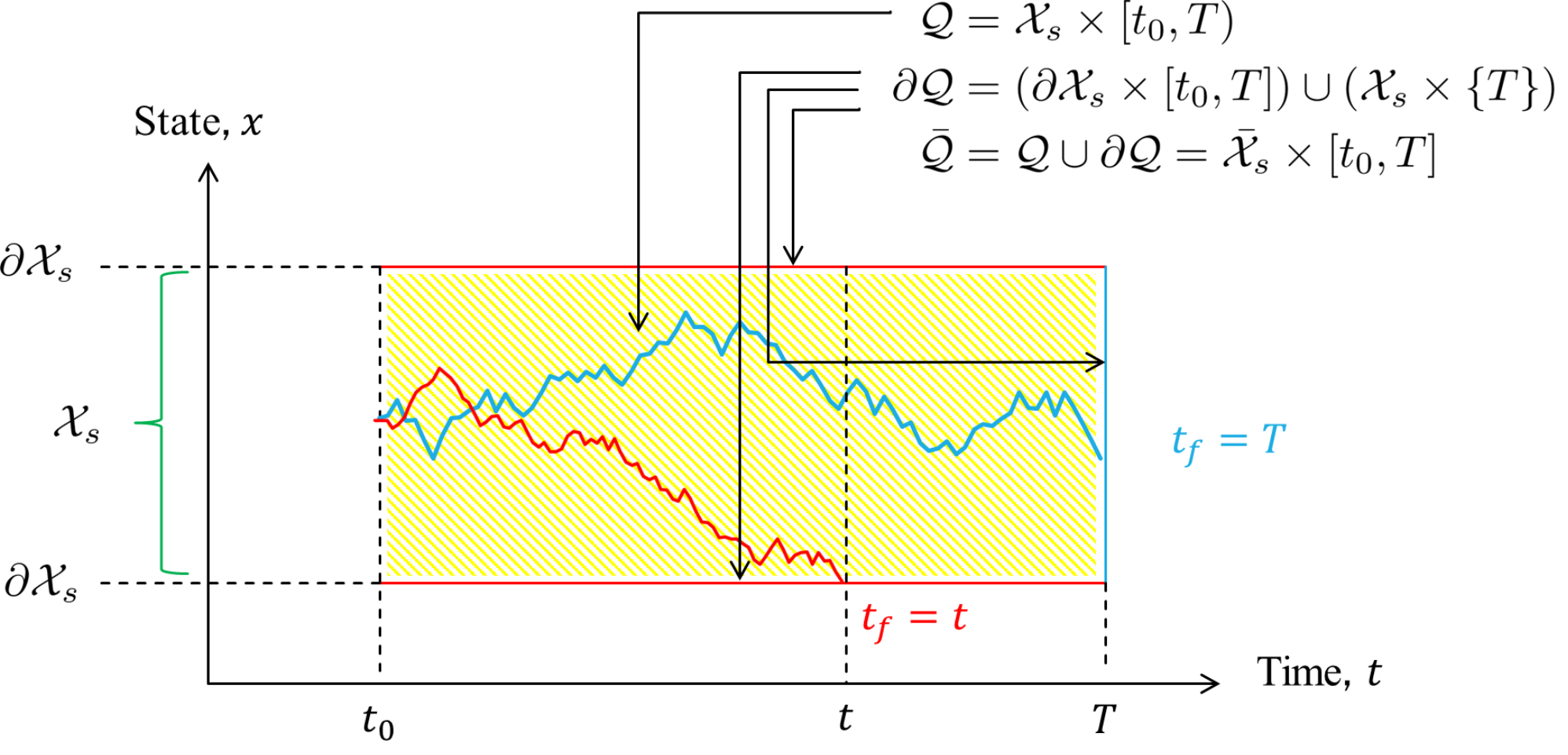} 
      
        \caption{Computational domains and exit times $t_f$} 
        \label{Fig. computational domain}
\end{figure}

Similarly, exit time $\hat{{t}}_f$ for process $\hat{{x}}(t)$ is defined as

\begin{equation}\label{t_hat_f with Q}
 \hat{{t}}_{f} \coloneqq \text{inf}\{t> t_0: (\hat{{x}}(t), t)\notin \mathcal{Q}\}.   
\end{equation}

Note that by the above definitions, $\left({x}({t}_{f}),{t}_{f}\right)\in\partial{\mathcal{Q}}$ and $\left(\hat{{x}}(\hat{{t}}_{f}),\hat{{t}}_{f}\right)\in\partial{\mathcal{Q}}$.

\section{Problem Formulation}
We first formalize a chance-constrained SOC problem in Section \ref{Sec: CC-SOC}. In Section \ref{Sec: dual}, we formulate its dual.

\subsection{Chance-Constrained SOC Problem}\label{Sec: CC-SOC}
Consider a cost function that is quadratic in the control input and has the form:
\begin{equation*}\label{C}
\begin{split}
 C\left(x_0, t_0, {u}(\cdot)\right)&\coloneqq\mathbb{E}_{x_0,t_0}\Bigg[\psi\left({x}({t}_{f})\right)\cdot\mathds{1} _{{x}({t}_{f})\in \mathcal{X}_{s}}+\int_{t_0}^{{t}_{f}}\left(\frac{1}{2}{u}^\top{R}\left({x}(t), t\right){u}+ V\left({x}(t), t\right)\right)dt\Bigg]
\end{split}
\end{equation*}
where $\psi\left({x}({t}_{f})\right)$ denotes a terminal cost, $V\left({x}(t), t\right)$ a state dependent running cost, and $R\left({x}(t), t\right)\in\mathbb{R}^{m\times m}$ a given positive definite matrix (for all values of ${x}(t)$ and $t$). $\mathds{1} _{{x}({t}_{f})\in \mathcal{X}_{s}}$ returns $1$ if the state of the system at the exit time ${t}_f$ is inside the safe region and $0$ otherwise i.e., the terminal cost is active only when the state of the system at the exit time ${t}_f$ is safe. Note that our cost function is defined over time horizon $[t_0, {t}_{f}]$ instead of $[t_0, T]$, because we do not consider the cost incurred after the system fails. We wish to find an optimal control policy for system (\ref{SDE}) such that $C\left(x_0, t_0, u(\cdot)\right)$ is minimal and the probability of failure (\ref{pfail}) is below a specified threshold. This problem can be formulated as a chance-constrained SOC problem as follows:
\begin{problem}[Chance-constrained SOC problem]\label{Problem: Risk-constrained SOC problem}
    \begin{equation}\label{CC-SOC}
    \begin{aligned}
    \min_{{u}(\cdot)}\;& \mathbb{E}_{x_0, t_0}\!\!\left[\psi\left({x}({t}_{f})\right)\!\cdot\!\mathds{1} _{{x}({t}_{f})\in \mathcal{X}_{s}}\!+\!\!\int_{t_0}^{{t}_{f}}\!\!\!\left(\frac{1}{2}{u}^\top\!{R}{u}+\!V\!\right)\!dt\right]\\
    {s.t.} \;\; & d{x}={f}dt+{G}{u}dt+{\Sigma}d{w},\quad {x}(t_0)=x_0,\\
      &P_{x_0, t_0}\left(\bigvee_{t\in(t_0, T]} {x}(t)\notin \mathcal{X}_{s}\right)\leq\Delta,
    \end{aligned}    
    \end{equation}
    where $\Delta\in(0,1)$ represents a given risk tolerance over the horizon $[t_0, T]$, and the admissible policy $u(\cdot)$ is measurable with respect to the $\sigma$-algebra generated by ${x}(s), 0\leq s\leq t$.
\end{problem}

\subsection{Dual SOC Problem}\label{Sec: dual}
We define the Lagrangian associated with Problem \ref{Problem: Risk-constrained SOC problem} as

\begin{equation}\label{eq: lagrangian}
\mathcal{L}\left(x_0, t_0, {u}(\cdot); \eta\right) +=  C\left(x_0, t_0, {u}(\cdot)\right)  + \eta P_{x_0, t_0}\left(\bigvee_{t\in(t_0, T]} {x}(t)\notin \mathcal{X}_{s}\right)-\eta\Delta
\end{equation}
where $\eta\geq 0$ is the \textit{Lagrange multiplier}.
Using a standard equality in probability theory \cite{durrett2019probability}, $P_{\mathrm{fail}}$ in the chance constraint of (\ref{CC-SOC}) can be transformed into an expectation of an indicator function as
\begin{equation}\label{pfail2}
    P_{x_0, t_0}\left(\bigvee_{t\in(t_0, T]} {x}(t)\notin \mathcal{X}_{s}\right)=\mathbb{E}_{x_0, t_0}\left[\mathds{1} _{{x}({t}_{f})\in \partial\mathcal{X}_{s}}\right].
\end{equation}
Here, $\mathds{1} _{{x}({t}_{f})\in \partial\mathcal{X}_{s}}$ returns $1$ if the state of the system \eqref{SDE} at the exit time is on the boundary of the safe set $\partial \mathcal{X}_s$ (i.e. the state of the system at the exit time escapes the the safe region) and $0$, otherwise. Using \eqref{pfail2}, the Lagrangian \eqref{eq: lagrangian} can be reformulated as
    \begin{equation}\label{Lagrangian2}
\!\!\!\!\mathcal{L}\!\left(x_0, t_0, {u}(\cdot); \eta\right) \!= \!C\!\left(x_0, t_0, {u}(\cdot)\right) \!+\! \eta\! \left[\mathbb{E}_{x_0, t_0}\!\!\left[\mathds{1} _{{x}({t}_{f})\in \partial\mathcal{X}_{s}}\right]\!-\!\Delta\right]
\end{equation}
 Observe that for any $\eta\geq0$ if we define $\phi:\overline{\mathcal{X}_s}\to\mathbb{R}$ as\footnote{In what follows, function $\phi(x; \eta)$ sets a boundary condition for a PDE and we often need technical assumptions on the regularity of $\phi(x; \eta)$ (e.g., continuity on $\overline{\mathcal{X}_s}$) to guarantee the existence of a solution of the PDE. When such requirements are needed, we approximate  (\ref{phi(x)}) as $\phi(x; \eta) \approx \psi(x)B(x) + \eta\left(1-B(x)\right)$, where $B(x)$ is a smooth bump function on $\mathcal{X}_s$.}: 
\begin{equation}\label{phi(x)}
    \phi\left({x}; \eta\right)\coloneqq\psi\left({x}\right)\cdot \mathds{1} _{{x}\in \mathcal{X}_{s}}+\eta\cdot\mathds{1} _{{x}\in \partial\mathcal{X}_{s}}-\eta\Delta,
\end{equation}
then, the second term in (\ref{Lagrangian2}) can be absorbed in a new terminal cost function $\phi$ as follows:
\begin{equation}\label{Chat}
    \mathcal{L}\!\left(x_0,\!t_0,\! {u}(\cdot); \eta\right)=\mathbb{E}_{x_0, t_0}\!\!\left[\phi\!\left({x}({t}_{f}); \eta\right)\!+\!\!\!\int_{t_0}^{{t}_{f}}\!\!\!\left(\!\frac{1}{2}{u}^\top\!R{u}+V\!\right)\!dt\!\right].
\end{equation}
Now we formulate the dual problem as follows:
\begin{problem}[Dual SOC problem]\label{prob: dual problem}
\begin{align}\label{eq: dual}
     \max_{\eta} \min_{u(\cdot)}&\;\;\mathcal{L}\left(x_0, t_0, {u}(\cdot);\eta\right)\\
     {s.t.}&\;\; \eta\geq 0\nonumber\\
     &\;\; d{x}={f}dt+{G}{u}dt+{\Sigma}d{w},\quad {x}(t_0)=x_0,\nonumber
\end{align}
\end{problem}
In order to solve Problem \ref{prob: dual problem} we first solve the subproblem 
\begin{equation}\label{eq: g of eta}
g(\eta)\coloneqq \min_{u(\cdot)}\mathcal{L}\left(x_0, t_0, {u}(\cdot); \eta\right).
\end{equation}
Note that, $\mathcal{L}$ possesses the time-additive Bellman structure i.e., for any $\eta\geq0$ and $t_0\leq t\leq{t}_f$, we can write 
\begin{align*}
    \mathcal{L}\left(x_0,t_0, {u}(\cdot); \eta\right)=&\mathbb{E}_{x_0, t_0}\left[\int_{t_0}^{t}\!\!\left(\frac{1}{2}{u}^\top R{u}+V\right)dt\right]   + \mathbb{E}_{x, t}\left[\phi\left({x}({t}_{f}); \eta\right)+\int_{t}^{{t}_{f}}\!\!\left(\frac{1}{2}{u}^\top R{u}+V\right)dt\right].
\end{align*}
Therefore problem \eqref{eq: g of eta} can be solved by utilizing dynamic programming without having to introduce any conservative approximation of the failure probability $P_{\mathrm{fail}}$. After solving the subproblem \eqref{eq: g of eta} we solve the dual problem
\begin{equation}\label{eq: dual problem}
   \max_{\eta \geq0} g(\eta).
\end{equation}
 Since the dual function $g(\eta)$ is the pointwise infimum of a family of affine functions of $\eta$, it is concave even when the primal problem is not convex. Moreover, since affine functions are upper semicontinuous, $g(\eta)$ is also upper semicontinuous.

\section{Synthesis of Optimal Control Policies}
In this section, we first express the dual function in terms of the solution to an HJB PDE parametrized by the dual variable $\eta$. Next, we show that the strong duality holds between the primal problem \eqref{CC-SOC} and the dual problem \eqref{eq: dual} (the duality gap is zero) under a certain assumption on the system dynamics and the cost function. Finally, we propose a novel Monte-Carlo-based dual ascent algorithm to numerically solve the dual problem \eqref{eq: dual problem}. This implies that the optimal control input for the original chance-constrained problem (primal problem) can be computed online by real-time Monte-Carlo simulations.

\subsection{Computation of the Dual Function}\label{Sec: Computation of the Dual Function}
In this section, we compute the dual function by solving problem \eqref{eq: g of eta} using dynamic programming. For each $(x,t)\in\overline{\mathcal{Q}}$, $\eta\geq0$ and an admissible policy $u(\cdot)$ over $[t, T)$, we define the cost-to-go function:
\begin{equation*}\label{value function}
    \mathcal{L}\!\left(x, t, {u}(\cdot); \eta\right)\!=\!\mathbb{E}_{x, t}\!\!\left[\phi\left({x}({t}_{f});\eta\right)\!+\!\!\int_{t}^{{t}_{f}}\!\!\!\left(\frac{1}{2}{u}^\top\!R{u}+\!V\!\right)\!dt\right]\!\!.
\end{equation*}
Now, we state the following theorem.
\begin{theorem}\label{theorem: solution to risk-minimizing soc}
Suppose for a given $\eta\geq0$, there exists a function $J:\overline{\mathcal{Q}}\rightarrow \mathbb{R}$ such that 
\begin{enumerate}[label=(\alph*)]
    \item $J(x,t; \eta)$ is continuously differentiable in $t$ and twice continuously differentiable in $x$ in the domain $\mathcal{Q}$;
    \item $J(x,t;\eta)$ solves the following dynamic programming PDE (HJB PDE):
    \begin{equation}\label{HJB PDE}
  \!\!\!\!\!\!\!\!\!\!\!\!\begin{cases}
         -\partial_tJ\!=\!-\frac{1}{2}\!\left(\partial_xJ\right)^\top\!\!GR^{-1}G^\top\!\partial_xJ\!+\!V+\!\!f^\top\!\partial_xJ+\frac{1}{2}\text{Tr}\left(\Sigma\Sigma^\top\partial^2_xJ\right),
 &  \forall(x,t)\in\mathcal{Q}, \\
    \!\!\underset{\substack{(x,t)\to(y,s) \\ (x,t)\in\mathcal{Q}}}{\lim}J(x,t;\eta)=\phi(y;\eta), & \!\!\!\forall(y,s)\in\partial\mathcal{Q}.
  \end{cases}
    \end{equation}
\end{enumerate}
Then, the following statements hold:
\begin{enumerate}
\item $J(x,t;\eta)$ is the \textit{value function} for Problem \eqref{eq: g of eta}. That is,
\begin{equation*}\label{J as value function}
    J\left(x, t; \eta\right) = \min_{{u}(\cdot)} \mathcal{L}\left(x, t, {u}(\cdot);\eta\right),\quad \forall\;(x,t)\in\overline{\mathcal{Q}}.
\end{equation*}

\item The solution to Problem \eqref{eq: g of eta} is given by 
\begin{equation}\label{optimal policy}
    u^*(x,t; \eta)=-R^{-1}\left(x, t\right){G}^\top\left(x, t\right)\partial_xJ\left(x, t; \eta\right).
\end{equation}
\end{enumerate}
\hfill$\blacksquare$
\end{theorem}

\begin{proof}
Let $J$ be the function satisfying (a) and (b). By Dynkin's formula \cite[Theorem 7.4.1]{oksendal2013stochastic}, for each $(x,t)\in\overline{\mathcal{Q}}$ and $\eta\geq 0$, we have
\begin{equation}\label{ito with dJ}
 \mathbb{E}_{x,t}\left[J\left({x}({t}_{f}), {t}_{f};\eta\right)\right]\!=\!J(x,t; \eta) +  \mathbb{E}_{x,t}\left[\int_{t}^{{t}_{f}}\!\!\!\!\!\!dJ\left({x}(s), s; \eta\right)\right]\!,
\end{equation}
where
\begin{equation}\label{dJ}
    \begin{aligned}
\!\!\!dJ\left({x}(s), s; \eta\right)=&(\partial_tJ)ds + (d{x})^\top\partial_xJ+\frac{1}{2}(d{x})^\top\!\!\left(\partial_x^2J\right)\!(d{x})\\
=&(\partial_tJ)ds + (f\!+Gu)^\top\!(\partial_xJ)ds+\left(\Sigma d{w}\right)^\top(\partial_xJ)+\!\frac{1}{2}\text{Tr}\left(\Sigma\Sigma^\top\!\partial_x^2J\right)ds.
    \end{aligned}
\end{equation}
Notice that the term $\int_{t}^{{t}_{f}}\left(\Sigma d{w}\right)^\top(\partial_xJ)$ is an It\^o integral. Using the property of It\^o integral \cite[Chapter 3]{oksendal2013stochastic}, we get 
\begin{equation}\label{ito integral}
    \mathbb{E}_{x,t}\int_{t}^{{t}_{f}}\left(\Sigma d{w}\right)^\top(\partial_xJ)=0
\end{equation}
Substituting (\ref{dJ}) into (\ref{ito with dJ}) and using \eqref{ito integral}, we have  
\begin{equation}\label{ito}
    \begin{aligned}
       \mathbb{E}_{x,t}\left[J\left({x}({t}_{f}), {t}_{f};\eta\right)\right]&=J(x,t;\eta)+\mathbb{E}_{x,t}\left[\int_{t}^{{t}_{f}}\!\!\!\left(\!\partial_tJ + (f+Gu)^{\!\top}\!(\partial_xJ)+\frac{1}{2}\text{Tr}\left(\Sigma\Sigma^{\!\top}\!\partial_x^2J\right)\!\right)\!ds\right]\!\!.
    \end{aligned}
\end{equation}
By the boundary condition of the PDE (\ref{HJB PDE}),\newline $J\left({x}({t}_{f}), {t}_{f}; \eta\right) =\phi\left({x}({t}_{f});\eta\right)$. Hence, from (\ref{ito}), we obtain
\begin{equation}\label{ito after plugging BC}
    \begin{aligned}
       J(x,t;\eta)&=\mathbb{E}_{x,t}\left[\phi\left({x}({t}_{f});\eta\right)\right]-\mathbb{E}_{x,t}\left[\int_{t}^{{t}_{f}}\!\!\!\left(\!\partial_tJ + (f+Gu)^{\!\top}\!(\partial_xJ)+\frac{1}{2}\text{Tr}\left(\Sigma\Sigma^{\!\top}\!\partial_x^2J\right)\!\right)\!ds\right]\!\!.
    \end{aligned}
\end{equation}
Now, notice that the right hand side of the PDE in (\ref{HJB PDE}) can be expressed as the minimum value of a quadratic form in $u$ as follows:
\begin{equation*}
    \begin{aligned}
         -\partial_tJ&=-\frac{1}{2}\left(\partial_xJ\right)^\top GR^{-1}G^\top\partial_xJ+V+f^\top\partial_xJ+\frac{1}{2}\text{Tr}\left(\Sigma\Sigma^\top\partial^2_xJ\right)\\
         &=\min_{{u}}\left[\frac{1}{2}u^\top Ru+V+\left(f+Gu\right)^\top\partial_xJ
         +\frac{1}{2}\text{Tr}\left(\Sigma\Sigma^\top\partial^2_xJ\right)\right]\!\!.
          \end{aligned} 
\end{equation*}
Therefore, for an arbitrary $u$, we have 
\begin{equation}\label{pde ineqaulity}
    \begin{aligned}
         \!\!\!\!-\partial_tJ\!\leq\!\frac{1}{2}u^\top\!Ru\!+\!V\!+\!\left(f\!+\!Gu\right)^\top\!\!\partial_xJ\!
         +\!\frac{1}{2}\text{Tr}\!\left(\Sigma\Sigma^\top\partial^2_xJ\right)
          \end{aligned} 
\end{equation}
where the equality holds iff
\begin{equation}\label{optimal policy2}
    u = -R^{-1}G^\top\partial_xJ.
\end{equation}
Combining (\ref{ito after plugging BC}) and (\ref{pde ineqaulity}), we obtain
\begin{equation}\label{J leq C_hat}
  \begin{aligned}
       \!J(x,t; \eta)\!&\leq\!\mathbb{E}_{x,t}\!\left[\phi\!\left({x}({t}_{f}); \eta\right)\right]\!+\!\mathbb{E}_{x,t}\!\left[\!\int_{t}^{{t}_{f}}\!\!\!\left(\!\frac{1}{2}{u}^\top\!R{u}\!+\!V\!\!\right)\!ds\right]\\
       &= \mathcal{L}\left(x,t,{u}(\cdot);\eta\right).
    \end{aligned}   
\end{equation}
This proves the statement (i). Since (\ref{J leq C_hat}) holds with equality iff (\ref{optimal policy2}) is satisfied, the statement (ii) also follows. 
\end{proof}
Theorem \ref{theorem: solution to risk-minimizing soc} implies that the solution of problem \eqref{eq: g of eta} can be expressed in terms of the HJB PDE \eqref{HJB PDE} parameterized by the dual variable $\eta$. Notice that \eqref{HJB PDE} is a nonlinear PDE, which is in general, difficult to solve. In what follows, we linearize the PDE \eqref{HJB PDE} whose solution can be obtained by utilizing the Feyman-Kac lemma \cite{williams2017model}. First, we make the following assumption which is essential to linearize the PDE \eqref{HJB PDE}.
\begin{assumption}\label{Assum: linearity}
For all $(x,t)\in\overline{\mathcal{Q}}$, there exists a positive constant $\lambda$ satisfying the following equation: 
\begin{equation}\label{lambda}
  \Sigma(x, t)\Sigma^\top(x, t) = \lambda G(x, t)R^{-1}(x,t) G^\top(x, t). 
\end{equation}
\end{assumption}
The above assumption implies that the control input in the direction with higher noise variance is cheaper than that in the direction with lower noise variance. See \cite{kappen2005path} for further discussion on this condition. 
Suppose Assumption 1 holds. Using the constant $\lambda$ that satisfies (\ref{lambda}), we introduce the following transformed value function $\xi(x,t;\eta)$:
\begin{equation}\label{exp transformation}
 J(x,t; \eta) = -\lambda\,\text{log}\left(\xi\left(x,t;\eta\right)\right).
\end{equation}
Transformation (\ref{exp transformation}) allows us to write the PDE (\ref{HJB PDE}) in terms of $\xi\left(x,t;\eta\right)$ as:

\begin{equation}\label{transformed HJB PDE}
  \!\!\!\!\!\!\begin{cases}
     \begin{aligned}
         \!\partial_t\xi\!=&\frac{V\xi}{\lambda}-\!\frac{1}{2}\text{Tr}\!\left(\Sigma\Sigma^\top\!\partial^2_x\xi\right)\!+\!\frac{1}{2\xi}\!\left(\partial_x\xi\right)^{\!T}\!\!\Sigma\Sigma^\top\!\partial_x\xi-\!\frac{\lambda}{2\xi}\!\left(\partial_x\xi\right)^\top\!\!\left(GR^{-1}G^\top\right)\!\partial_x\xi\!-\!f^\top\partial_x\xi,
          \end{aligned} & \forall(x,t)\in\mathcal{Q}, \\
    \!\!\underset{\substack{(x,t)\to(y,s) \\ (x,t)\in\mathcal{Q}}}{\lim}\xi(x,t;\eta)\!=\!\text{exp}{\left(-\frac{\phi(y;\eta)}{\lambda}\right)}, &\forall(y,s)\in\partial\mathcal{Q}.
  \end{cases}
    \end{equation}
Using Assumption 1 in the equation \eqref{transformed HJB PDE}, we rewrite PDE (\ref{HJB PDE}) as a linear PDE in terms of $\xi\left(x,t;\eta\right)$ as:

\begin{equation}\label{linearized risk-minimizing HJB}
 \!\!\begin{cases}
     \!\partial_t\xi\!=\!\frac{V\xi}{\lambda}\!-\!f^\top\partial_x\xi-\frac{1}{2}\text{Tr}\left(\Sigma\Sigma^\top\partial^2_x\xi\right),       & \forall(x,t)\!\in\!\mathcal{Q}, \\
    \!\!\underset{\substack{(x,t)\to(y,s) \\ (x,t)\in\mathcal{Q}}}{\lim}\xi(x,t; \eta)\!=\!\text{exp}{\left(-\frac{\phi(y; \eta)}{\lambda}\right)}, & \!\!\!\forall(y,s)\!\in\!\partial\mathcal{Q}.\\  
  \end{cases}
\end{equation}
Now we find the solution of the linearized PDE \eqref{linearized risk-minimizing HJB} using the Feynman-Kac lemma.
\begin{lemma}[Feynman-Kac lemma]\label{theorem: E and !}
    At any $\eta\geq0$, the solution to the linear PDE \eqref{linearized risk-minimizing HJB} exists. Moreover, the solution is unique in the sense that $\xi$ solving \eqref{linearized risk-minimizing HJB} is given by
    \begin{equation}\label{xi}
       \begin{aligned}
  \!\!\!\!\!\!\xi\!\left(x,t;\eta\right) & \!= \!\mathbb{E}_{x,t}\!\! \left[\text{exp}\!\left(\!\!-\frac{\phi\left(\hat{{x}}({\hat{{t}}_{f}});\eta\right)}{\lambda}\!-\!\frac{1}{\lambda}\!\int_{t}^{\hat{{t}}_{f}}\!\!\!\!V\!\!\left(\hat{{x}}(r), r\right)\!dr \!\!\right)\!\right]\!.\\
  \end{aligned}
    \end{equation}
    where $\hat{{x}}(t)$ evolves according to the dynamics \eqref{uncontrolled SDE} starting at $(x,t)$ and the exit time $\hat{{t}}_f$ is defined according to \eqref{t_hat_f with Q}.
    
    \begin{proof}
    The proof follows from \cite[Theorem 9.1.1]{oksendal2013stochastic}.
    \end{proof}

\end{lemma}
Now we state the following theorem which proves that under Assumption \ref{Assum: linearity}, for a given value of $\eta$, the optimal policy of problem \eqref{eq: g of eta} exists and is unique.
\begin{theorem}\label{theorem: E and ! of policy}
    If there exists a positive constant $\lambda$ that satisfies Assumption 1, then for any $\eta\geq0$ a \textit{unique} value function $J(x,t; \eta)$ of the Problem \eqref{eq: g of eta} exists and is given by $J(x,t;\eta) = -\lambda\,\text{log}\left(\xi\left(x,t;\eta\right)\right)$ where $\xi(x,t;\eta)$ is given by \eqref{xi}. Consequently, there exists a \textit{unique} optimal policy given by \eqref{optimal policy}. 
\end{theorem}
\begin{proof}
According to Lemma \ref{theorem: E and !}, the solution of the linear PDE \eqref{linearized risk-minimizing HJB} exists and is unique in the sense that $\xi$ solving \eqref{linearized risk-minimizing HJB} is given by \eqref{xi}. Therefore $J(x,t;\eta)$ is unique and is given by \eqref{exp transformation}. Consequently, from Theorem \ref{theorem: solution to risk-minimizing soc}, a unique optimal policy exists and is given by \eqref{optimal policy}.
\end{proof} 

\subsection{Strong Duality}
Due to the weak duality \cite{boyd2004convex}, the value of the dual problem \eqref{eq: dual problem} is always a lower bound for the primal problem \eqref{CC-SOC}. The difference between the values of \eqref{eq: dual problem} and \eqref{CC-SOC} is called the duality gap. When the duality gap is zero, we say that the strong duality holds. In our current study on chance-constrained SOC, strong duality has a practical significance as it implies that an optimal solution to the ``hard-constrained" problem \eqref{CC-SOC} can be obtained by solving the ``soft-constrained" problem \eqref{eq: g of eta}, provided that the dual variable $\eta$ is properly chosen.\par
Since \eqref{CC-SOC} is a non-convex optimization problem in general, establishing the strong duality between \eqref{CC-SOC} and \eqref{eq: dual problem} is not trivial. In Appendix \ref{sec: Nonconvex Optimization and Strong Duality}, we outline general conditions under which a non-convex optimization problem admits a zero duality gap. In the sequel, we apply the result in Appendix \ref{sec: Nonconvex Optimization and Strong Duality} to \eqref{CC-SOC} to delineate the conditions under which the chance-constrained SOC admits the strong duality.\par 
The following assumption is reminiscent of the Slater's condition, which is often a natural premise for strong duality:
\begin{assumption}\label{assum: strict feasibility}
There exists a policy $\widetilde{u}(\cdot)$ such that $P_\text{fail}(x_0, t_0, \widetilde{u}(\cdot))-\Delta<0$, i.e., Problem \ref{Problem: Risk-constrained SOC problem} is strictly feasible.
\end{assumption}
Using Lemma \ref{lem:existence_dual_sol} in the Appendix \ref{sec: Nonconvex Optimization and Strong Duality}, we can show that under Assumptions \ref{Assum: linearity} and \ref{assum: strict feasibility}, there exists a dual optimal solution $\eta^*$ such that $0\leq\eta^*<\infty$ such that $g(\eta^*)=\underset{{\eta\geq 0}}{\max}\;g(\eta)$.
To proceed further, we need the following assumption. We conjecture that this assumption is valid under mild conditions; a formal analysis is postponed as future work. 
\begin{assumption}\label{assum: continuity of Pfail}
    Suppose Assumption 1 holds so that for each $\eta\geq 0$ a unique optimal policy $u^*(x,t;\eta)$ of Problem \eqref{eq: g of eta} exists (c.f., Theorem \ref{theorem: E and ! of policy}). We further assume that the function  $\eta\mapsto P_{\mathrm{fail}}(x_0,t_0, u^*(x,t;\eta)):[0, \infty)\rightarrow[0,1]$ is continuous.
\end{assumption}
 Now notice that under Assumption  \ref{Assum: linearity}, the solution $u^*(\cdot;\eta)$ of the problem ${\argmin}_{u(\cdot)}\; \mathcal{L}\left(x_0, t_0, {u}(\cdot); \eta\right)$ exists and is unique for each $\eta\geq0$. Therefore, statements (i) and (ii) of the Assumption \ref{asmp:continuity} in the Appendix \ref{sec: Nonconvex Optimization and Strong Duality} hold true. Moreover, by the Assumption \ref{assum: continuity of Pfail}, statement (iii) of Assumption \ref{asmp:continuity} in the Appendix \ref{sec: Nonconvex Optimization and Strong Duality} is satisfied. Now under Assumption \ref{asmp:continuity} using Lemma \ref{lem:comp_slackness}, we can prove the following complementary slackness statements:
 \begin{enumerate}[label=(\alph*)]
        \item If $\eta^*=0$, then $P_{\mathrm{fail}}(x_0,t_0, u^*(\cdot\;;{\eta^*})-\Delta\leq0$
        \item If $\eta^*>0$, then $P_{\mathrm{fail}}(x_0,t_0, u^*(\cdot\;;{\eta^*}))-\Delta=0$
    \end{enumerate}

    The following theorem is the main result of this section.
\begin{theorem}\label{theorem: strong duality}
Consider problem \ref{Problem: Risk-constrained SOC problem} and suppose Assumptions \ref{Assum: linearity}, \ref{assum: strict feasibility} and \ref{assum: continuity of Pfail} hold. Then there exists a dual optimal solution $0\leq\eta^*<\infty$ that maximizes $g(\eta)$ and a unique optimal policy $u^*(\cdot\;;\eta^*)$ of problem \eqref{eq: g of eta} is an optimal policy of Problem \ref{Problem: Risk-constrained SOC problem} (primal problem) such that $C\left(x_0, t_0, {u}^*(\cdot\;;\eta^*)\right)=g(\eta^*)$ i.e., the duality gap is zero.
\end{theorem}
\begin{proof}
Refer to the proof of Theorem \ref{theorem: strong duality2} in the Appendix \ref{sec: Nonconvex Optimization and Strong Duality}.
\end{proof}
 Similar to our approach, the work presented in \cite{ono2015chance} solves the chance-constrained problem by formulating its dual. However, in this work, the joint chance constraint is conservatively approximated using Boole's inequality. Hence, the duality gap is nonzero and the obtained solution is suboptimal. Unlike this method, in our work, we prove that the strong duality exists between the primal chance-constrained problem \eqref{CC-SOC} and its dual \eqref{eq: dual} under certain assumptions on the system dynamics and cost function. Consequently, the chance-constrained problem can be solved by evaluating the dual objective function (c.f. Section \ref{Sec: Computation of the Dual Function}) and solving the dual problem (c.f. Section \ref{Sec: Solution of the Dual Problem}). \par
 In order to solve the dual problem, it is natural to use the gradient ascent algorithm $\eta \leftarrow \eta+\gamma(P_{\mathrm{fail}}(x_0,t_0,u^*(\cdot;\eta)))$ to iteratively update the dual variable $\eta$. Here, $\gamma$ is the step size, $u^*(.;\eta)$ is the optimal policy solving \eqref{eq: g of eta}, and $P_{\mathrm{fail}}(x_0,t_0,u^*(\cdot;\eta))$ is the probability of failure under the policy $u^*(.;\eta)$. For the dual ascent algorithm to be practical, we need to be able to evaluate the gradient $P_{\mathrm{fail}}-\Delta$ efficiently in each iteration. Therefore, next, we study how to compute $P_{\mathrm{fail}}$ for the given value of $\eta$. 
  
\subsection{Risk Estimation}
We present two approaches to compute $P_{\mathrm{fail}}(x_0,t_0,u^*(.;\eta))$ for a given value of $\eta$. The first approach is PDE-based in which we find the optimal policy $u^*(.;\eta)$ first, and then evaluate its risk $P_{\mathrm{fail}}(x_0,t_0,u^*(.;\eta))$ by solving a PDE. Despite conceptual simplicity, this approach is difficult to implement computationally unless there exists an analytical expression of $u^*(.;\eta)$. To circumvent this difficulty, we also present an importance-sampling-based approach. This approach allows us to numerically compute  $P_{\mathrm{fail}}(x_0,t_0,u^*(.;\eta))$ without ever constructing $u^*(.;\eta)$ and is more computationally amenable.
\subsubsection{PDE-Based Approach}
We state the following theorem:
\begin{theorem}\label{Theorem: risk estimation}
Suppose there exists a function $J:\overline{\mathcal{Q}}\rightarrow \mathbb{R}$ such that 
\begin{enumerate}[label=(\alph*)]
    \item $J(x,t)$ is continuously differentiable in $t$ and twice continuously differentiable in $x$ in the domain $\mathcal{Q}$;
    \item $J(x,t)$ solves the following PDE for the given optimal control policy $u^*(\cdot)$:
    \begin{equation}\label{risk PDE}
      \!\!\!\!\!\!\!\!\!\!\!\!\!\begin{cases}
 \begin{aligned}
    \!\!-\partial_{t}J\!=\!\left(\!f\!+\!Gu^*\right)^\top\!\!\partial_xJ\!+\!\frac{1}{2}\text{Tr}\!\left(\!\Sigma\Sigma^\top\!\partial^2_xJ\right)\!,
    \end{aligned}& \!\!\forall(x,t)\!\in\!\mathcal{Q}, \\
    \!\!\underset{\substack{(x,t)\to(y,s) \\ (x,t)\in\mathcal{Q}}}{\lim}J(x,t)= \phi'(y), & \!\!\!\!\!\forall(y,s)\!\in\!\partial\mathcal{Q}.
    \end{cases}
          \end{equation}
\end{enumerate}
where ${\phi}'(x) \!=\! \mathds{1} _{{x}\in \partial\mathcal{X}_{s}}$. Then, $P_{\mathrm{fail}}$ is given by
\begin{equation*}
    P_\mathrm{fail}\left(x_0,t_0,u^*(\cdot)\right) = J(x_0,t_0).
\end{equation*}
\end{theorem}

\begin{proof}
    Let $J(x,t)$ be the function satisfying (a) and (b). By Dynkin's formula, for each $(x,t)\in\overline{\mathcal{Q}}$ we have
\begin{equation*}
    \begin{aligned}
       \mathbb{E}_{x,t}\left[J\left({x}({t}_{f}), {t}_{f}\right)\right]&=J(x,t)+\!\mathbb{E}_{x,t}\!\!\left[\!\int_{t}^{{t}_{f}}\!\!\!\left(\!\!\partial_tJ \!+\! (f\!+\!Gu^*)^\top\!(\partial_xJ)\!+\!\frac{1}{2}\text{Tr}\!\left(\!\Sigma\Sigma^\top\!\partial_x^2J\right)\!\!\right)\!ds\!\right]\!.
    \end{aligned}
\end{equation*}
The second term on the right side contains, in parentheses, the PDE in (\ref{risk PDE}) and is therefore zero. Hence, we obtain
\begin{equation}\label{flag1}
       J(x,t)=\mathbb{E}_{x,t}\left[J\left({x}({t}_{f}), {t}_{f}\right)\right].
\end{equation}
From the boundary condition of the PDE (\ref{risk PDE})
\begin{equation}\label{flag2}
     \!\!\mathbb{E}_{x,t}\!\left[J\!\left({x}({t}_{f}), {t}_{f}\right)\right] \!=\! \mathbb{E}_{x,t}\left[\phi'\left({x}({t}_{f})\right)\right] \!=\mathbb{E}_{x, t}\left[\mathds{1} _{{x}({t}_{f})\in \partial\mathcal{X}_{s}}\right].
 \end{equation}
Combining (\ref{flag1}), (\ref{flag2}), and  we obtain
\begin{equation*}
   J(x_0,t_0) = P_\mathrm{fail}\left(x_0,t_0,u^*(\cdot)\right). 
\end{equation*}
\end{proof}
\begin{remark}
 We do not prove here the existence of the solution of the PDE \eqref{risk PDE}, rather we suppose that a solution exists. Note that PDE (\ref{risk PDE}) is a special case of the Cauchy-Dirichlet problem. For proof of the existence of a solution to the Cauchy-Dirichlet problem, we refer the readers to \cite[Chapter 6]{friedman1975stochastic}.
\end{remark} 
Theorem \ref{Theorem: risk estimation} implies that if we have the solution for the optimal policy $u^*(\cdot)$, $P_\mathrm{fail}\left(x_0,t_0,u^*(\cdot)\right)$ can be computed by solving the PDE \eqref{risk PDE}. Next, we present an importance-sampling-based approach to find $P_{\mathrm{fail}}$ without constructing $u^*(\cdot)$.
\subsubsection{Importance-Sampling-Based Approach}
Let $\mathcal{T}$ be the space of trajectories $x\coloneqq\{x(t), t\in[t_0, {t}_f]\}$. Let $Q^*(x)$ be the probability distribution of the trajectories defined by system \eqref{SDE} under the optimal policy $u^*(\cdot)$. Using \eqref{pfail2}, we can write
\begin{align}\label{pfail under Q*}
P_{\mathrm{fail}}(x_0, t_0, u^*(\cdot)) &= \int_{\mathcal{T}}\mathds{1} _{{x}({t}_{f})\in \partial\mathcal{X}_{s}}Q^*(dx).
\end{align}
Suppose we generate an ensemble of large number of $N$ trajectories $ \{x^{(i)}\}_{i=1}^N$ under the distribution $Q^*$. Then according to the strong law of large numbers, as $N\rightarrow\infty$, 
\begin{align*}
\frac{1}{N} \sum_{i=1}^{N} \mathds{1} _{{x}^{(i)}({t}_{f})\in \partial\mathcal{X}_{s}} \overset{a.s.}{\rightarrow} P_{\mathrm{fail}}(x_0, t_0, u^*(\cdot)) \quad x^{(i)}\sim Q^*(x). 
\end{align*}
This implies that a Monte Carlo algorithm is applicable to numerically evaluate $P_{\mathrm{fail}}$. However, such a Monte Carlo algorithm is impractical since it is difficult to sample trajectories from $Q^*$ as the optimal policy $u^*$ is unknown. Fortunately, an importance sampling scheme is available which allows us to numerically evaluate $P_{\mathrm{fail}}$ using a trajectory ensemble sampled from the distribution $P(x)$ of the uncontrolled system (2). The next theorem provides the details:  
\begin{theorem}\label{Thm: PI fast}
Suppose we generate an ensemble of a large number of  $N$ trajectories $ \{x^{(i)}\}_{i=1}^N$ under the distribution $P$. For each $i$, let $r^{(i)}$ be the path reward of the sample path $i$ given by 
\begin{align}\label{r(i)}
    r^{(i)} &= \text{exp}\!\left(\!\!-\frac{\phi\left({{x}^{(i)}}({{{t}}_{f}});\eta\right)}{\lambda}\!-\!\frac{1}{\lambda}\!\int_{t}^{{{t}}_{f}}\!\!\!\!V\!\!\left({{x}^{(i)}}(r), r\right)\!dr \!\!\right)
\end{align}
Let us define $r \coloneqq \sum_{i=1}^{N}r^{(i)}$. Then as $N\rightarrow\infty$, 
\begin{equation*}
    \sum_{i=1}^{N}\frac{r^{(i)}}{r}\mathds{1}_{{x}^{(i)}({t}_{f})\in \partial\mathcal{X}_{s}} \overset{a.s.}{\rightarrow} P_{\mathrm{fail}}(x_0, t_0, u^*(\cdot)) \quad x^{(i)}\sim P(x).
\end{equation*}
\end{theorem}
\begin{proof}
    Notice that $P_{\mathrm{fail}}$ in \eqref{pfail under Q*} can be equivalently written as
\begin{equation}\label{pfail under P}
    P_{\mathrm{fail}}(x_0, t_0, u^*(\cdot))  = \int_{\mathcal{T}}\mathds{1} _{{x}({t}_{f})\in \partial\mathcal{X}_{s}}\frac{dQ^*}{dP}(x)P(dx)
\end{equation}
where the Radon-Nikodym derivative $\frac{dQ^*}{dP}(x)$ represents the likelihood ratio of observing a sample path $x$ under distributions $Q^*$ and $P$.  Now, according to Theorem \ref{Thm: likelihood ratio} in the Appendix \ref{sec: The Likelihood Ratio}, for a given ensemble of $N$ trajectories $\{x^{(i)}\}_{i=1}^N$ sampled under the distribution $P$, the likelihood ratio $\frac{dQ^*}{dP}$ of observing a sample path $x^{(i)}$ is given by $\frac{r^{(i)}/r}{1/N}$. Therefore, using \eqref{pfail under P}, by the strong law of large numbers, as $N\rightarrow\infty$ we get
\begin{align*}
    \frac{1}{N} &\sum_{i=1}^{N}\frac{r^{(i)}/r}{1/N}\mathds{1} _{{x}^{(i)}({t}_{f})\in \partial\mathcal{X}_{s}} \overset{a.s.}{\rightarrow} P_{\mathrm{fail}}(x_0, t_0, u^*(\cdot))\\
    {x}^{(i)} &\sim P(x)
\end{align*}
which completes the proof.
\end{proof} 
Theorem \ref{Thm: PI fast} provides us a sampling-based approach to numerically compute the probability of failure $P_{\mathrm{fail}}(x_0, t_0, u^*(\cdot))$. According to Theorem \ref{Thm: PI fast}, if we sample $N$ trajectories under distribution $P$, then we can approximate the probability of failure as 
\begin{equation}\label{approx pfail under P}
   \!\!\!\!\!\! P_{\mathrm{fail}}(x_0, t_0, u^*(\cdot))\! \approx\! \sum_{i=1}^{N}\frac{r^{(i)}}{r}\mathds{1}_{{x}^{(i)}({t}_{f})\in \partial\mathcal{X}_{s}} \quad  {x}^{(i)} \sim P(x).
\end{equation}
where $r^{(i)}$ is defined by \eqref{r(i)} and $r = \sum_{i=1}^{N} r^{(i)}$. Note that sampling under distribution $P$ is easy since we only need to simulate the uncontrolled system dynamics \eqref{uncontrolled SDE}. 
\subsection{Generalized Risk-Estimation PDE}
In select applications, requiring that the solution to (\ref{risk PDE}) satisfies the boundary condition everywhere on $\partial\mathcal{Q}$ can be overly restrictive. In this section, we consider a generalized version of PDE (\ref{risk PDE}) by reducing the requirement in the boundary condition to hold only for a subset of points on the boundary $(x,t)\in\partial\mathcal{Q}$ called the \textit{regular points}. We show that if a solution of such a PDE exists, then we can still interpret it as the probability of failure. Before we state the generalized version of Theorem \ref{Theorem: risk estimation} we need the following definitions and an auxiliary lemma. 
\begin{definition}[Regular and irregular points]\label{Def: reg point}
A point $(x,t)\in\partial\mathcal{Q}$ is called regular for $\mathcal{Q}$ with respect to ${x}(t)$ if $P_{x,t}\left({t}_{f}=t_0\right)=1$; i.e., a.a. paths ${x}(t)$ starting from $x$ at time $t$ leave $\mathcal{Q}$ immediately. Otherwise the point $(x,t)$ is called irregular.
\end{definition}

\begin{example}\label{example: punctured ball}
{Consider a punctured ball $\mathcal{B}=\{x\in\mathbb{R}^d: 0<\|x\|<1\}$ and its boundary $\partial \mathcal{B}=\{x\in\mathbb{R}^d: \|x\|=1\} \cup \{\mathbf{0}\}$ i.e, the boundary $\partial \mathcal{B}$ contains the surface of the ball $\mathcal{B}$ and the origin. If a Brownian motion ${w}(t)$ starts at the origin, it returns to $\mathcal{B}$ immediately. On the other hand, if it starts anywhere on the surface of $\mathcal{B}$, it leaves $\mathcal{B}$ immediately (c.f., the 0-1 law \cite[Corollary 9.2.7]{oksendal2013stochastic}). Hence, the origin is irregular for $\mathcal{B}$ with respect to ${w}(t)$, whereas all points on the surface of $\mathcal{B}$ are regular.}   
\end{example}
For more interesting examples of regular and irregular points, see \cite[Chapter 4]{karatzas2012brownian}.

\begin{definition}[Hunt's condition]
A process ${x}(t)$ is said to satisfy Hunt's condition if every semipolar set for process ${x}(t)$ is also polar for ${x}(t)$. A semipolar set is a countable union of thin sets and a measurable set $E$ is called thin for process ${x}(t)$ if for all starting points, ${x}(t)$ does not hit $E$ immediately, a.s. A measurable set $F$ is called polar for process ${x}(t)$ if for all starting points, ${x}(t)$ never hits $F$, a.s.
\end{definition}
Hunt's condition holds for Brownian motion \cite{blumenthal2007markov}. See \cite{oksendal2013stochastic} for a discussion on the requirements on the coefficients of the SDE (\ref{SDE}) in order for it to satisfies Hunt's condition.

\begin{lemma}\label{Lemma: for generalized thm}
Let $\mathcal{Q}^I$ denote the set of irregular points of $\mathcal{Q}$ with respect to process ${x}(t)$. Suppose ${x}(t)$ satisfies the Hunt's condition. Then $({x}({t}_{f}), {t}_{f})\notin\mathcal{Q}^I$ a.s.
\end{lemma}
\begin{proof}
From the definition of irregular points, the set $\mathcal{Q}^I$ is semipolar for ${x}(t)$. Furthermore, since ${x}(t)$ satisfies Hunt's condition, the set $\mathcal{Q}^I$ is polar for ${x}(t)$, and therefore $({x}({t}_{f}), {t}_{f})\notin\mathcal{Q}^I$ a.s.
\end{proof}

We can now state the generalized version of Theorem \ref{Theorem: risk estimation} as follows:

\begin{theorem}\label{theorem: Generalized Risk-Minimizing HJB PDE}
Suppose ${x}(t)$ satisfies Hunt's condition and there exists a function $J:\overline{\mathcal{Q}}\rightarrow \mathbb{R}$ such that
\begin{enumerate}[label=(\alph*)]
    \item $J(x,t)$ is continuously differentiable in $t$ and twice continuously differentiable in $x$ in domain $\mathcal{Q}$;
    \item $J(x,t)$ solves the following PDE for a given admissible control policy $u(\cdot)$:
    \begin{equation}\label{generalized risk PDE}
      \!\!\!\!\!\!\!\!\!\!\!\!\begin{cases}
 \begin{aligned}
    \!\!-\partial_{t}J\!=\!\left(\!f\!+\!Gu\right)^T\!\!\partial_xJ\!+\!\frac{1}{2}\text{Tr}\!\left(\!\Sigma\Sigma^T\!\partial^2_xJ\right)\!,
    \end{aligned}& \!\!\forall(x,t)\!\!\in\!\!\mathcal{Q}, \\
    \!\!\underset{\substack{(x,t)\to(y,s) \\ (x,t)\in\mathcal{Q}}}{\lim}J(x,t)= \phi'(y), & \!\!\!\!\!\!\!\!\!\!\!\!\!\!\!\!\!\!\!\!\!\!\!\!\!\!\forall\;\text{regular }(y,s)\!\!\in\!\!\partial\mathcal{Q}.
    \end{cases}
          \end{equation}
\end{enumerate}
Then, $P_{\mathrm{fail}}$ is given by
\begin{equation*}
    P_{\mathrm{fail}} = J(x_0,t_0).
\end{equation*}
\end{theorem}

\begin{proof}
Let $J(x,t)$ be the function satisfying (a) and (b). By Dynkin's formula, for each $(x,t)\in\overline{\mathcal{Q}}$ we have
\begin{equation*}
    \begin{aligned}
       \mathbb{E}_{x,t}\left[J\left({x}({t}_{f}), {t}_{f}\right)\right]&=J(x,t)+\!\mathbb{E}_{x,t}\!\!\left[\!\int_{t}^{{t}_{f}}\!\!\!\left(\!\!\partial_tJ \!+\! (f\!+\!Gu)^T\!(\partial_xJ)\!+\!\frac{1}{2}\text{Tr}\!\left(\!\Sigma\Sigma^T\!\partial_x^2J\right)\!\!\right)\!ds\!\right]\!.
    \end{aligned}
\end{equation*}
The second term on the right side contains, in parentheses, the PDE in (\ref{generalized risk PDE}) and is therefore zero. Hence, we obtain
\begin{equation}\label{flag_1}
       J(x,t)=\mathbb{E}_{x,t}\left[J\left({x}({t}_{f}), {t}_{f}\right)\right].
\end{equation}
 Since ${x}(t)$ satisfies Hunt's condition, from Lemma \ref{Lemma: for generalized thm} we know that $({x}({t}_{f}), {t}_{f})$ does not belong to the irregular set of $\mathcal{Q}$ a.s. Hence from the boundary condition of PDE (\ref{generalized risk PDE}),
 \begin{equation}\label{flag_2}
     \!\!\mathbb{E}_{x,t}\!\left[J\!\left({x}({t}_{f}), {t}_{f}\right)\right] \!=\! \mathbb{E}_{x,t}\left[\phi'\left({x}({t}_{f})\right)\right] .
 \end{equation}
Combining (\ref{flag_1}) and (\ref{flag_2}) we obtain
\begin{equation*}
   J(x_0,t_0) =  P_{\mathrm{fail}}. 
\end{equation*}
\end{proof}
Note that Theorem \ref{theorem: Generalized Risk-Minimizing HJB PDE} is more useful than Theorem \ref{Theorem: risk estimation} for the risk estimation in Example \ref{example: punctured ball} because the boundary condition at the origin is relaxed. 
\begin{remark}
{PDE (\ref{generalized risk PDE}) can be considered as a special case of the generalized Dirichlet-Poisson problem, the existence of whose solution is proved in \cite[Chapter 9]{oksendal2013stochastic}.}    
\end{remark}

\subsection{Solution of the Dual Problem \eqref{eq: dual problem}}\label{Sec: Solution of the Dual Problem}
We present two numerical approaches to solve the dual problem \eqref{eq: dual problem}. The first approach is the finite difference method (FDM) which is one of the most popular approaches to numerically solve the partial differential equations. Despite its popularity, this approach suffers from the \textit{curse of dimensionality} and the computational cost grows exponentially as the state-space dimension increases. To circumvent this difficulty, we also present the path integral approach. This approach is more computationally amenable. It allows us to numerically solve the dual problem \eqref{eq: dual problem} online via open-loop samples of system trajectories.
\subsubsection{Finite Difference Method}
When the geometry of the computational domain is simple, it is straightforward to form discrete approximations to spatial differential operators with a high order of accuracy via Taylor series expansions \cite{Grossmann2007}. A finite number of grid points are placed at the interior and on the boundary of the computational domain, and the solution to the PDE is sought at these finite set of locations. Once the grid is determined, finite difference operators are derived to approximate spatial derivatives in the PDE. In this work, centered formulas are used to approximate spatial operators with up to eight-order accuracy at the interior grid points. For grid points near the boundary, asymmetric formulas that maintain the order of accuracy are used in conjunction with Dirichlet boundary condition. The use of finite difference operators yields a system of ordinary differential equations (ODEs) that is integrated in time with the desired method. The MATLAB suite of ODE solvers \cite{Shampine_1997_matlab} provides a number of efficient ODE integrators with high-temporal order, error control, and variable time-stepping that advance the solution in time.\par 

Note that FDM numerically solves the HJB PDE (\ref{linearized risk-minimizing HJB}) backward in time. The solution needs to be computed offline and the optimal control policy needs to be stored in a look-up table. For a given state and time, the optimal input for real-time control is obtained from the stored look-up table. Similar to the PDE \eqref{linearized risk-minimizing HJB}, FDM can be utilized to compute the probability of failure by numerically solving the PDE \eqref{risk PDE} backward in time. It is well known that the computational complexity and memory requirements of FDM increase exponentially as the state-space dimension increases. Therefore, FDM is intractable for systems with more than a few state variables. Moreover, this method is inconvenient for real-time implementation since the HJB PDE needs to be solved backward in time. Also, FDM computes a global solution over the entire domain $\mathcal{Q}$ even if the majority of the state-time pairs $(x,t)$ will never be visited by the actual system. 
To overcome these difficulties, we present the path integral approach which numerically solves the dual problem \eqref{eq: dual problem} in real-time. 
\subsubsection{Path Integral Approach}
Assumption \ref{Assum: linearity} implies that the stochastic noise has to enter the system dynamics via the control channels. 
Therefore, in what follows, we assume that system \eqref{SDE} can be partitioned into subsystems that are directly and non-directly driven by the noise as:
\begin{equation*}
    \begin{aligned}
  \begin{bmatrix}
   d{x}^{(1)} \\  d{x}^{(2)}
  \end{bmatrix}= \begin{bmatrix}
    {f}^{(1)}({x}, t) \\ {f}^{(2)}({x}, t)
  \end{bmatrix}\!dt + \begin{bmatrix}
    \mathbf{0}\\{G}^{(2)}\!\left({x}, t\right)
  \end{bmatrix}\!{u}({x}, t)dt+\begin{bmatrix}
    \mathbf{0}\\{\Sigma}^{(2)}\!\left({x}, t\right)
  \end{bmatrix}\!d{w}
\end{aligned}
\end{equation*}
where $\mathbf{0}$ denotes a zero matrix of appropriate dimensions. The path integral control framework utilizes the fact that the solution to PDE (\ref{linearized risk-minimizing HJB}) admits the Feynman-Kac representation (\ref{xi}).
The optimal control input $u^*(x,t;\eta)$ (\ref{optimal policy}) can be obtained by taking the gradient of (\ref{xi}) with respect to $x$ \cite{williams2017model, theodorou2010generalized}. We obtain
\begin{equation}\label{path integral control}
 u^*(x,t;\eta)dt=\mathcal{G}\left(x,t\right)\frac{\mathbb{E}_{x,t}\left[\text{exp}{\left(-\frac{1}{\lambda}S\right)}\Sigma^{(2)}\left(x,t\right)d{w}(t)\right]}{\mathbb{E}_{x,t}\left[\text{exp}{\left(-\frac{1}{\lambda}S\right)}\right]}, 
\end{equation}
where matrix $\mathcal{G}\left(x,t\right)$ is defined as
\begin{equation*}
\mathcal{G}\!\left(x,t\right)\!=\!R^{-1}\!(x,t){G^{(2)}}^{\!\!\top}\!\!(x,t)\!\left(\!G^{(2)}(x,t)R^{-1}\!(x,t){G^{(2)}}^{\!\!\top}\!\!(x,t)\right)^{\!\!\!-1}. 
\end{equation*}
and $S$ represents the cost-to-go of a trajectory of system $\hat{{x}}(t)$ starting at $(x,t)$ , i.e., \[S=\phi\left(\hat{{x}}(\hat{{t}}_{f});\eta\right)+\int_{t}^{\hat{{t}}_{f}} V\left(\hat{{x}}(t), t\right)dt.\] To evaluate expectations in (\ref{xi}) and (\ref{path integral control}) numerically, we can discretize the uncontrolled dynamics (\ref{uncontrolled SDE}) and use Monte Carlo sampling \cite{williams2017model}. Unlike FDM, the path integral framework solves PDE (\ref{linearized risk-minimizing HJB}) in the forward direction. It evaluates a solution locally without requiring knowledge of the solution nearby so that there is no need for a (global) discretization of the computational domain. For real-time control, Monte Carlo simulations can be performed in real-time in order to evaluate (\ref{path integral control}) for the current $(x,t)$. Similar to (\ref{xi}) and (\ref{path integral control}), the failure probability \eqref{approx pfail under P} can be numerically computed by Monte Carlo sampling using the importance sampling method. The Monte Carlo simulations can be parallelized by using the Graphic Processing Units (GPUs) and thus the path integral approach is less susceptible
to the curse of dimensionality.\par
Now, we present the path-integral-based dual ascent algorithm to numerically solve the dual problem \eqref{eq: dual problem}. The algorithm for that is given in Algorithm 1. 
\begin{algorithm}
\caption{Dual ascent via path integral approach}\label{alg:cap}
\begin{algorithmic}[1]
\Require Error tolerance $\epsilon>0$, learning rate $\gamma>0$
\State Set $\eta=0$ 
\State Find $u^*(\cdot;\eta=0)$ using \eqref{path integral control}\label{sol of linear HJB pde}
\State Compute the failure probability $P_\mathrm{fail}\left(x_0,t_0,u^*(\cdot;\eta=0)\right)$ using \eqref{approx pfail under P}. \label{sol of risk pde}
\If {$P_\mathrm{fail}\left(x_0,t_0,u^*(\cdot;\eta=0)\right)\leq\Delta$} 
\State return $u^*(\cdot;\eta=0)$ 
\EndIf

\State Choose initial $\eta>0$.
\While{True} 
\State Find $u^*(\cdot;\eta)$ using \eqref{path integral control} \label{sol of linear HJB pde 2}
\State Compute the failure probability $P_\mathrm{fail}\left(x_0,t_0,u^*(\cdot;\eta)\right)$ using \eqref{approx pfail under P} \label{sol of risk pde 2}
\If {$|P_\mathrm{fail}\left(x_0,t_0,u^*(\cdot;\eta)\right)-\Delta|<\epsilon$}
\State return $u^*(\cdot;\eta)$

\EndIf
\State $\eta \leftarrow \eta + 
\gamma(P_\mathrm{fail}\left(x_0,t_0,u^*(\cdot;\eta)\right)-\Delta)$
\EndWhile 

\end{algorithmic}
\end{algorithm}
The algorithm starts by dealing with a special case $\eta=0$. We find the optimal policy \eqref{path integral control}. Then, we compute $P_\mathrm{fail}\left(x_0,t_0,u^*(\cdot;\eta=0)\right)$ by importance-sampling approach. If $P_\mathrm{fail}\left(x_0,t_0,u^*(\cdot;\eta=0)\right)\leq \Delta$, we return $u^*(\cdot;\eta=0)$. Otherwise, the algorithm chooses some initial $\eta$ and iteratively updates the dual variable $\eta \leftarrow \eta+\gamma(P_{\mathrm{fail}}(x_0,t_0,u^*(\cdot;\eta))-\Delta)$ with a learning rate of $\gamma$. Once $|P_\mathrm{fail}\left(x_0,t_0,u^*(\cdot;\eta)\right)-\Delta|$ is less than the error tolerance $\epsilon$, we return the policy $u^*(\cdot;\eta)$. Thus, using the Algorithm \ref{alg:cap}, we can numerically solve the original chance-constrained problem \eqref{CC-SOC} online via real-time Monte-Carlo simulations.
\section{Simulation Results}
In this section, we demonstrate the effectiveness of the proposed control synthesis framework in addressing the chance-constrained problem defined in \eqref{CC-SOC}. In Section \ref{Sec: sim_2D}, we employ a 2D state-space velocity input model to solve a mobile robot navigation problem using Algorithm \ref{alg:cap}. The solution derived from the path-integral method is compared with that of FDM. In Section \ref{Sec: sim_4D}, we extend the analysis to a unicycle model in a 4D state-space, and in Section \ref{Sec: sim_5D}, to a car model with a 5D state-space. In both the 4D and 5D cases, we solve the corresponding chance-constrained problems using Algorithm \ref{alg:cap}. Due to the high computational cost of applying FDM to models with 4D and 5D state spaces, we rely exclusively on the path-integral method for these examples.

\subsection{Input Velocity Model}\label{Sec: sim_2D}
The problem is illustrated in Figure \ref{Fig. sample trajs 2D} where a particle robot wants to navigate in a 2D space from a given start position (shown by a yellow star) to the origin (shown by a magenta star), by avoiding the collisions with the red obstacle and the outer boundary.\par
Let the states of the system be ${p}_x$ and ${p}_y$, the positions along $x$ and $y$, respectively. The system dynamics are given by the following SDEs:
\begin{equation}\label{velocity input model}
\begin{split}
    d{p}_x={v}_xdt+\sigma d{w}_x,\quad d{p}_y={v}_ydt+\sigma d{w}_y,
\end{split}
\end{equation}
where ${v}_x$ and ${v}_y$ are velocities along $x$ and $y$ directions, respectively. Assume 
\begin{equation*}
 v_x = \overline{v}_x + \widetilde{v}_x, \qquad  v_y = \overline{v}_y + \widetilde{v}_y, 
\end{equation*}
where $\overline{v}_x$ and $\overline{v}_y$ are nominal velocities given by $\overline{v}_x = -k_xp_x $ and $\overline{v}_y = -k_yp_y$ for some constants $k_x$ and $k_y$. Hence, (\ref{velocity input model}) can be rewritten as
\begin{equation}\label{input velocity model}
\begin{split}
    d{p}_x=-k_x{p}_xdt+\widetilde{{v}}_xdt+\sigma d{w}_x,\quad d{p}_y=-k_y{p}_ydt+\widetilde{{v}}_ydt+\sigma d{w}_y.
\end{split}
\end{equation}
Now, the goal is to design an optimal control policy $u^* = [\widetilde{v}^*_x\;\widetilde{v}^*_y]^\top$ for the chance-constrained problem (\ref{CC-SOC}). The initial state is $x_0 = [ -0.3 \; 0.3 ]^\top$. We set $k_x=k_y=0.5$, $V\left({x}(t)\right) = {p}_x^2(t) + {p}_y^2(t)$, $\psi\left({x}(T)\right) = {p}_x^2(T) + {p}_y^2(T)$, $R=\begin{bmatrix}
    1 & 0\\
    0& 1
\end{bmatrix}$, $t_0=0$, $T=2$, and $\sigma^2=0.01$. 
In Algorithm \ref{alg:cap}, we set the error tolerance $\epsilon = 0.01$ and learning rate $\gamma=0.1$. For FDM, the computational domain is discretized by a grid of $96\times 96$ points and the solver ode45 is used with a relative error tolerance equal to $10^{-3}$. In the path integral simulation, for Monte Carlo sampling, $10^5$ trajectories and a step size equal to $0.01$ are used. 

\begin{figure}[h]
    \centering
      \begin{tabular}{c c}
     \!\!\!\!\!\!\!\!\includegraphics[scale=0.5]{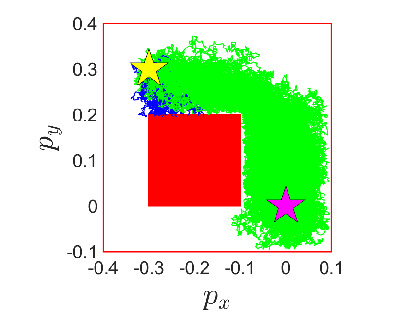} &\!\!\!\!\!\!\!\!\!\!\!\!\!\!\!\!\!\!\includegraphics[scale=0.5]{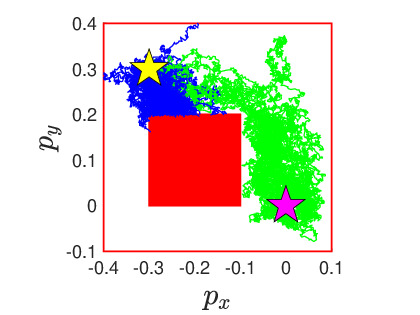} \\
      \!\!\!\!\!\!\!\!(a) $\Delta = 0.25$ & \!\!\!\!\!\!\!\!\!\!\!\!\!\!\!\!\!\! (b) $\Delta = 0.9$ \\
      \end{tabular}
        \caption{Robot navigation problem for the input velocity model. The start position is shown by a yellow star and the target position (the origin) by a magenta star. $100$ sample trajectories generated using optimal control policies for two values of $\Delta$ are shown. The trajectories are color-coded; blue paths collide with the obstacle or the outer boundary, while the green paths converge in the neighborhood of the magenta star.} 
        \label{Fig. sample trajs 2D}
\end{figure}

In Figure \ref{Fig. sample trajs 2D}, we plot $100$ sample trajectories generated using synthesized optimal policies for two values of $\Delta$. The trajectories are color-coded; the blue paths collide with the obstacle or the outer boundary, while the green paths converge in the neighborhood of the origin (the goal position). For a lower value of $\Delta$ the weight of blue paths is less and that of the green paths is more as compared to the higher value of $\Delta$. In other words, the failure probability for the lower value of $\Delta$ is less as compared to that of the higher value of $\Delta$.
\begin{figure}[h]
    \centering
      \begin{tabular}{c c}
     \!\!\!\!\!\!\!\!\!\includegraphics[scale=0.17]{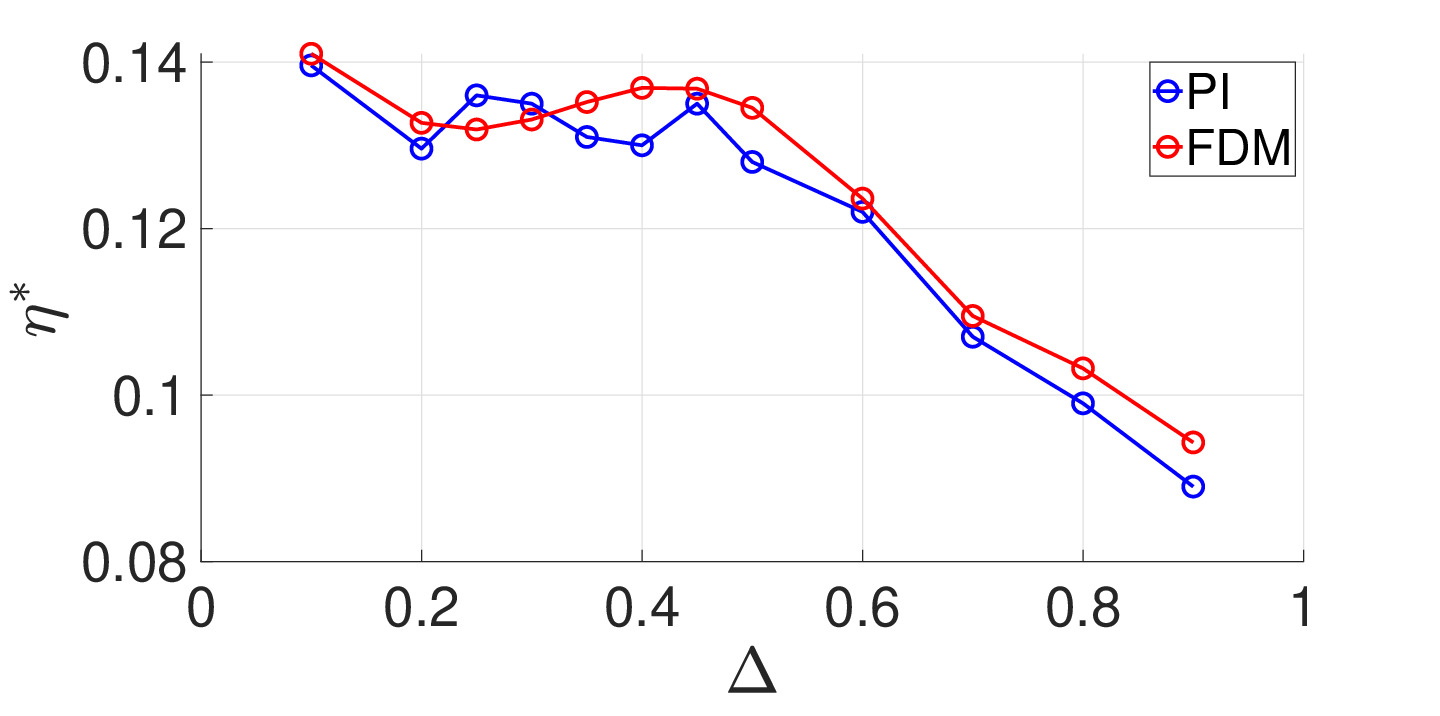} &\!\!\!\!\!\!\!\!\!\!\!\!\!\includegraphics[scale=0.17]{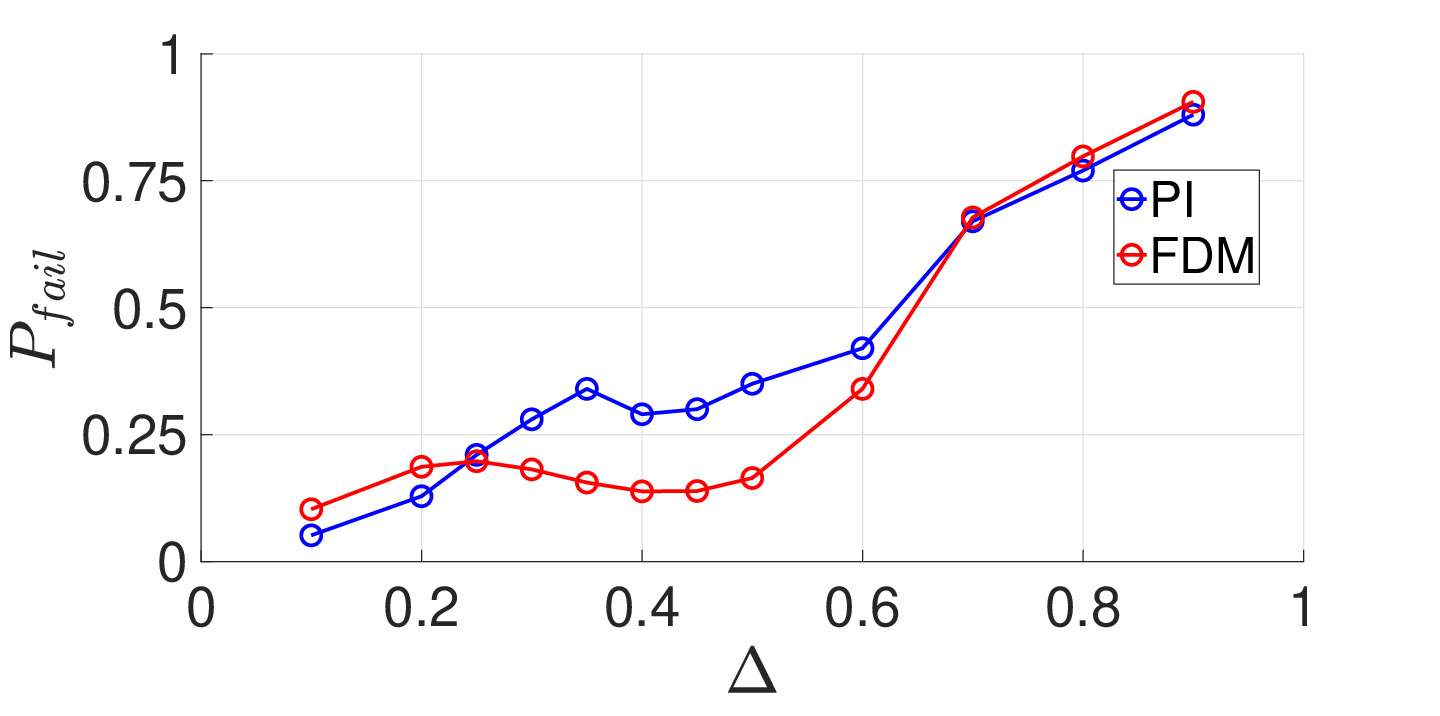} \\
      (a) $\eta^*$ vs $\Delta$ & \!\!\!\!\!\!\!\!\!\!\!\!\!(b) $P_{\mathrm{fail}}(x_0,t_0, u^*(\cdot\;;{\eta^*}))$ vs $\Delta$ \\
      \end{tabular}
        \caption{$\eta^*$ and $P_{\mathrm{fail}}(x_0,t_0, u^*(\cdot\;;{\eta^*}))$ vs $\Delta$ for input velocity model using path integral control and FDM.} 
        \label{Fig.eta pfail vs Delta(2D)}
\end{figure}

Figure \ref{Fig.eta pfail vs Delta(2D)} represents how the value of $\eta^*$ and $P_{\mathrm{fail}}(x_0,t_0, u^*(\cdot\;;{\eta^*}))$ change with respect to $\Delta$. The values obtained using path integral control and FDM are compared. We expect that the value of $\eta^*$ reduces and that of $P_{\mathrm{fail}}$ increases as $\Delta$ increases. Note that the graphs are not monotonic. This is because we used numerical methods to solve the PDE \eqref{linearized risk-minimizing HJB}, which caused some errors in the computation of $\eta^*$ and $P_{\mathrm{fail}}(x_0,t_0, u^*(\cdot\;;{\eta^*}))$.

\begin{figure*}[h]
    \centering
      \begin{tabular}{c c c c}
      \!\!\!\!\!\!\!\!\!\!\!\!\!\!\!\!\!\!\!\includegraphics[scale=0.27]{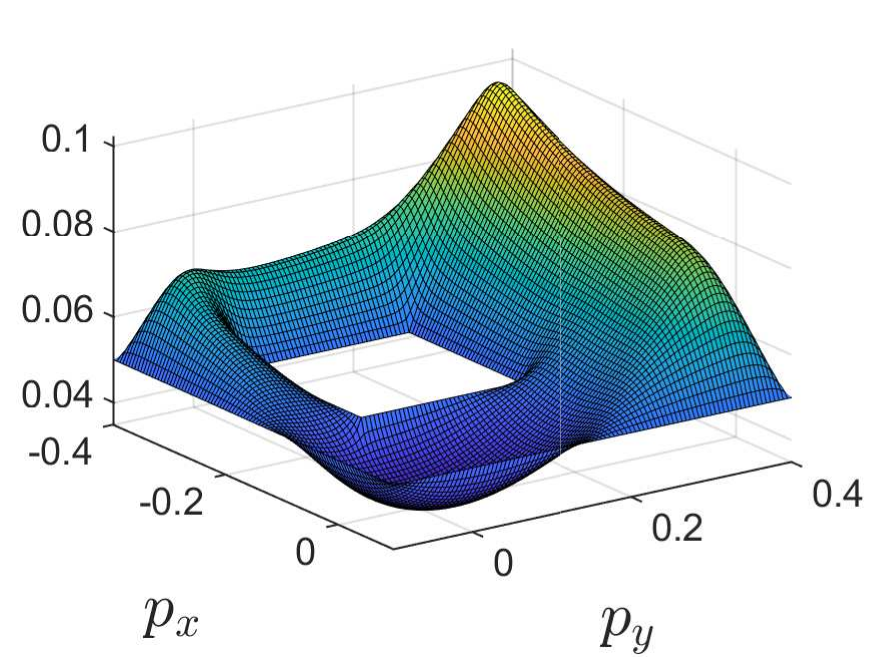} &\includegraphics[scale=0.27]{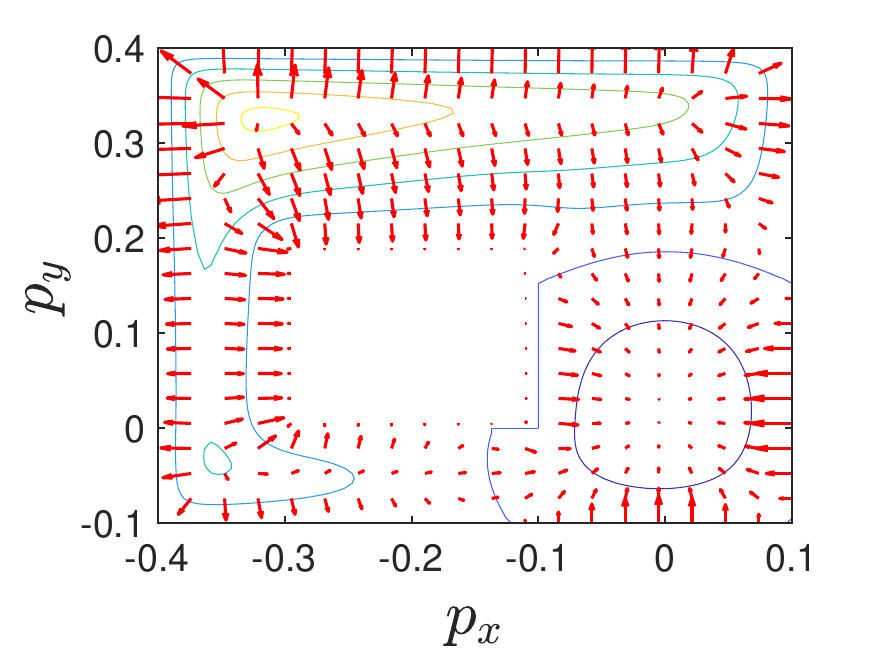} &\includegraphics[scale=0.27]{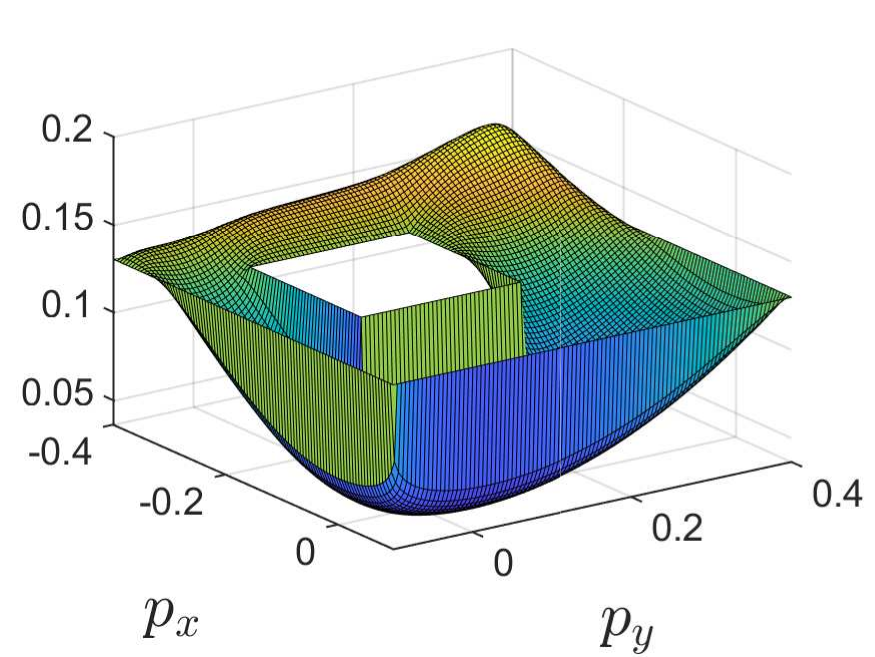} &\includegraphics[scale=0.27]{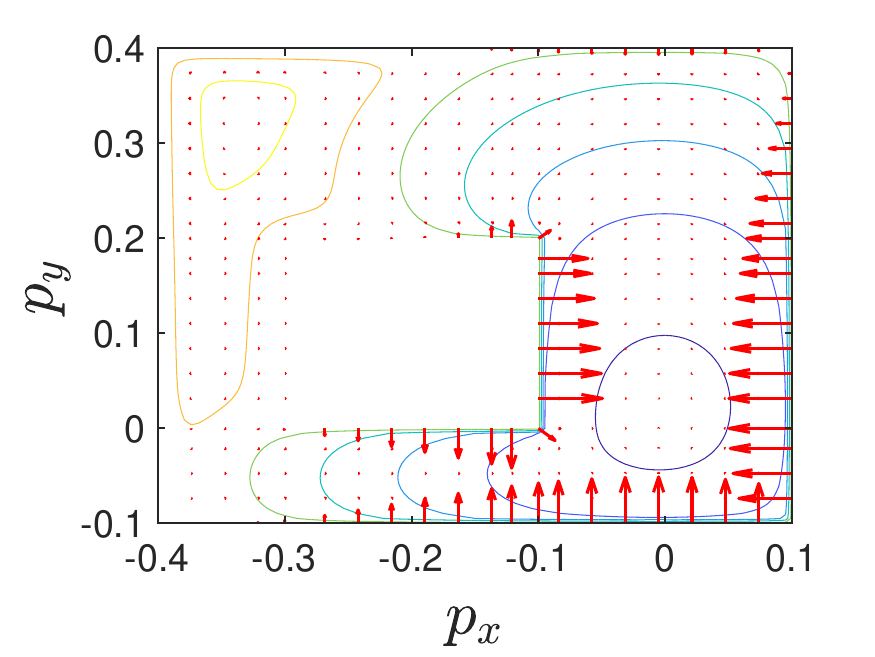}\\
      \!\!\!\!\!\!\!\!\!\!\!\!\!\!\!\!\!\!\!\small(a) $J(x,t_0)$ for $\eta = 0.05$ & \small(b) $u^*(x,t_0)$ for $\eta = 0.05$ &  \small (c) $J(x,t_0)$ for $\eta = 0.13$ & \small (d) $u^*(x,t_0)$ for $\eta = 0.13$\\
      
        \!\!\!\!\!\!\!\!\!\!\!\!\!\!\!\!\!\!\!\includegraphics[scale=0.27]{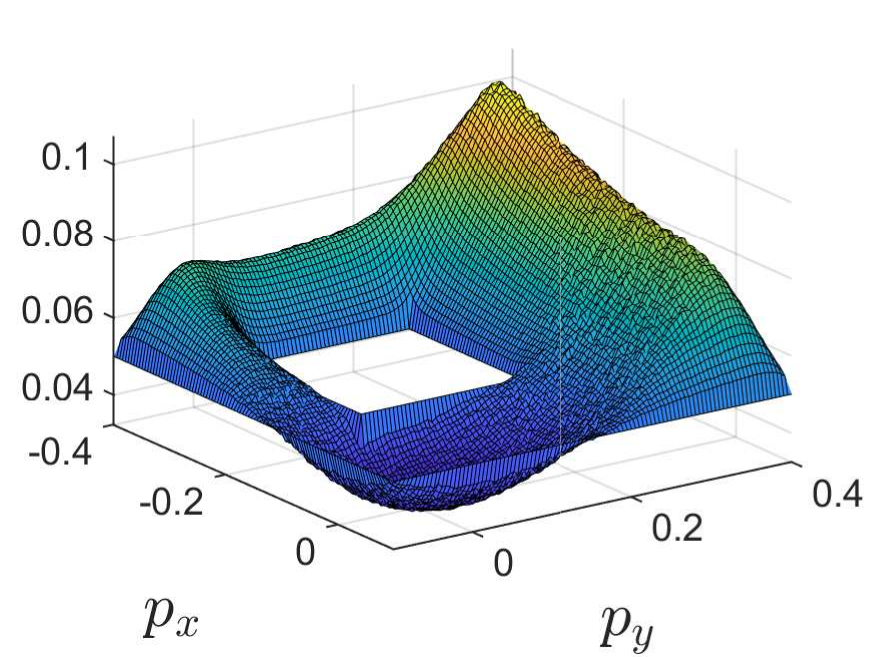} &\includegraphics[scale=0.27]{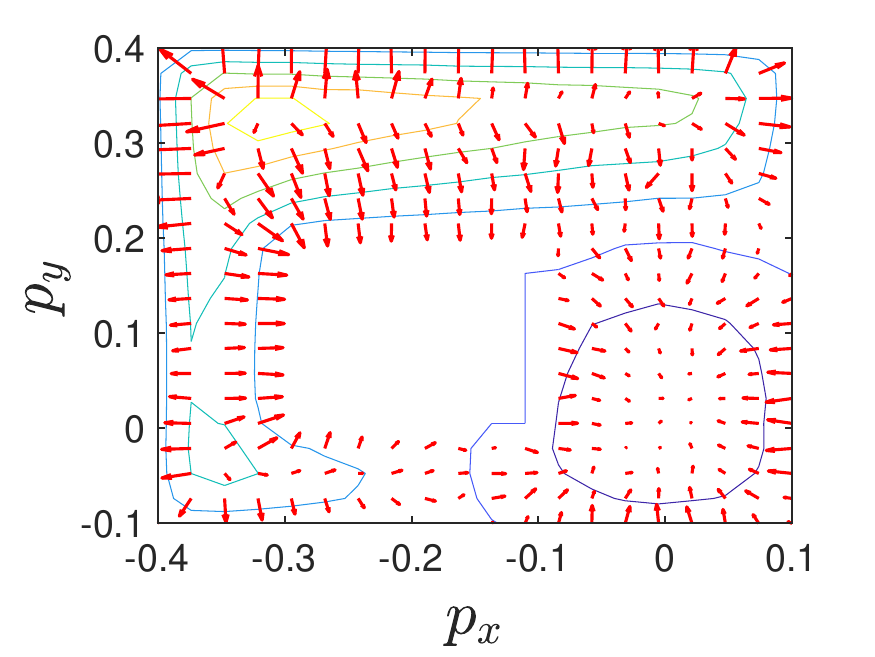}
        &\includegraphics[scale=0.27]{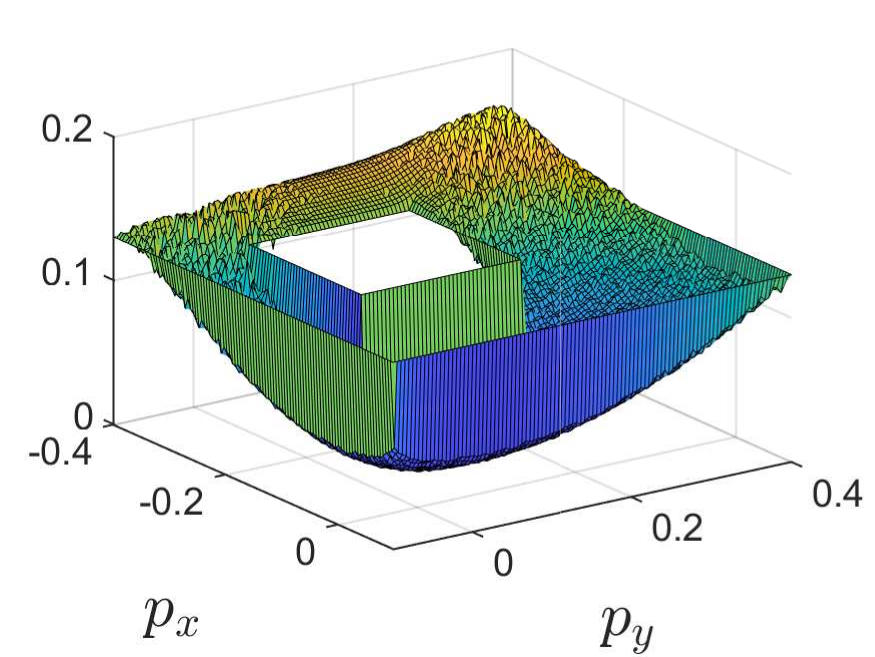} &\includegraphics[scale=0.27]{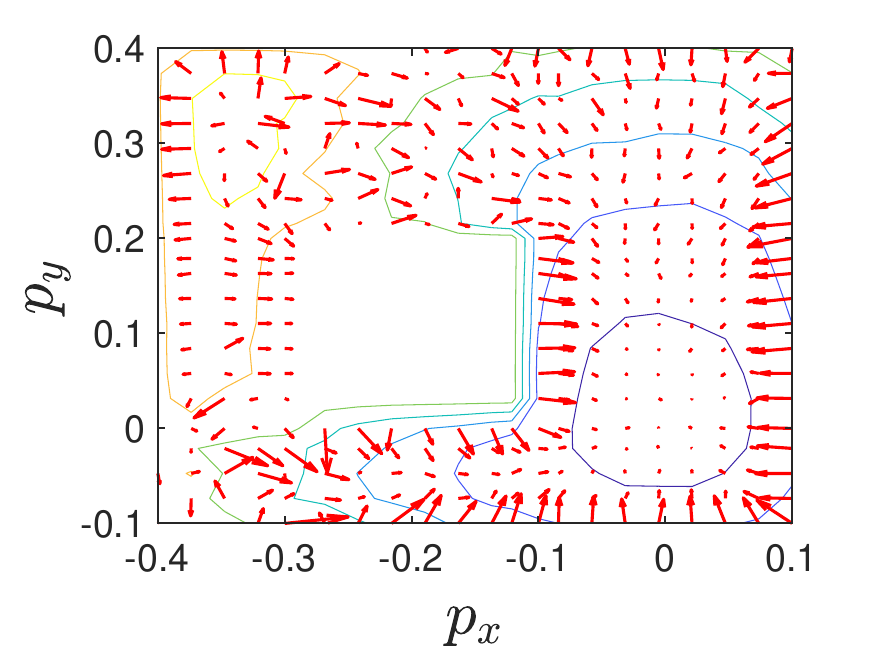}\\
        \!\!\!\!\!\!\!\!\!\!\!\!\!\!\!\!\!\!\!\small(e) $J(x,t_0)$ for $\eta = 0.05$ & \small(f) $u^*(x,t_0)$ for $\eta = 0.05$ & \small(g) $J(x,t_0)$ for $\eta = 0.13$ & \small(h) $u^*(x,t_0)$ for $\eta = 0.13$\\
      \end{tabular}
        \caption{Input velocity model: comparison of solutions $J(x,t_0)$ and $u^*(x,t_0)$ obtained from FDM (a-d) and path integral (e-h) for $\eta=0.05$ and $\eta=0.13$. The optimal control inputs $u^*(x,t_0)$ in (b, d, f, h) are plotted together with contours of $J(x, t_0)$.} 
        \label{Fig. surf plots FDM}
\end{figure*}
Figure \ref{Fig. surf plots FDM} shows the comparison of solutions $J(x,t_0)$ and $u^*(x,t_0)$ of the PDE \eqref{HJB PDE} obtained from FDM (top row) and path integral (bottom row) for $\eta=0.05$ and $\eta=0.13$. Since the path integral is a sampling-based approach its solutions are noisier compared to FDM, as expected. Notice that for $\eta=0.05$, due to a lower boundary value, the control inputs $u^*(x,t_0)$ push the robot towards the obstacle or the outer boundary from the most part of the safe region $\mathcal{X}_{s}$, except near the origin (the target position). Whereas, for $\eta=0.13$, aggressive inputs $u^*(x,t_0)$ are applied near the boundary to force the robot to go towards the goal position.

\begin{figure}[h]
    \centering
      \begin{tabular}{c c}
    \includegraphics[scale=0.37]{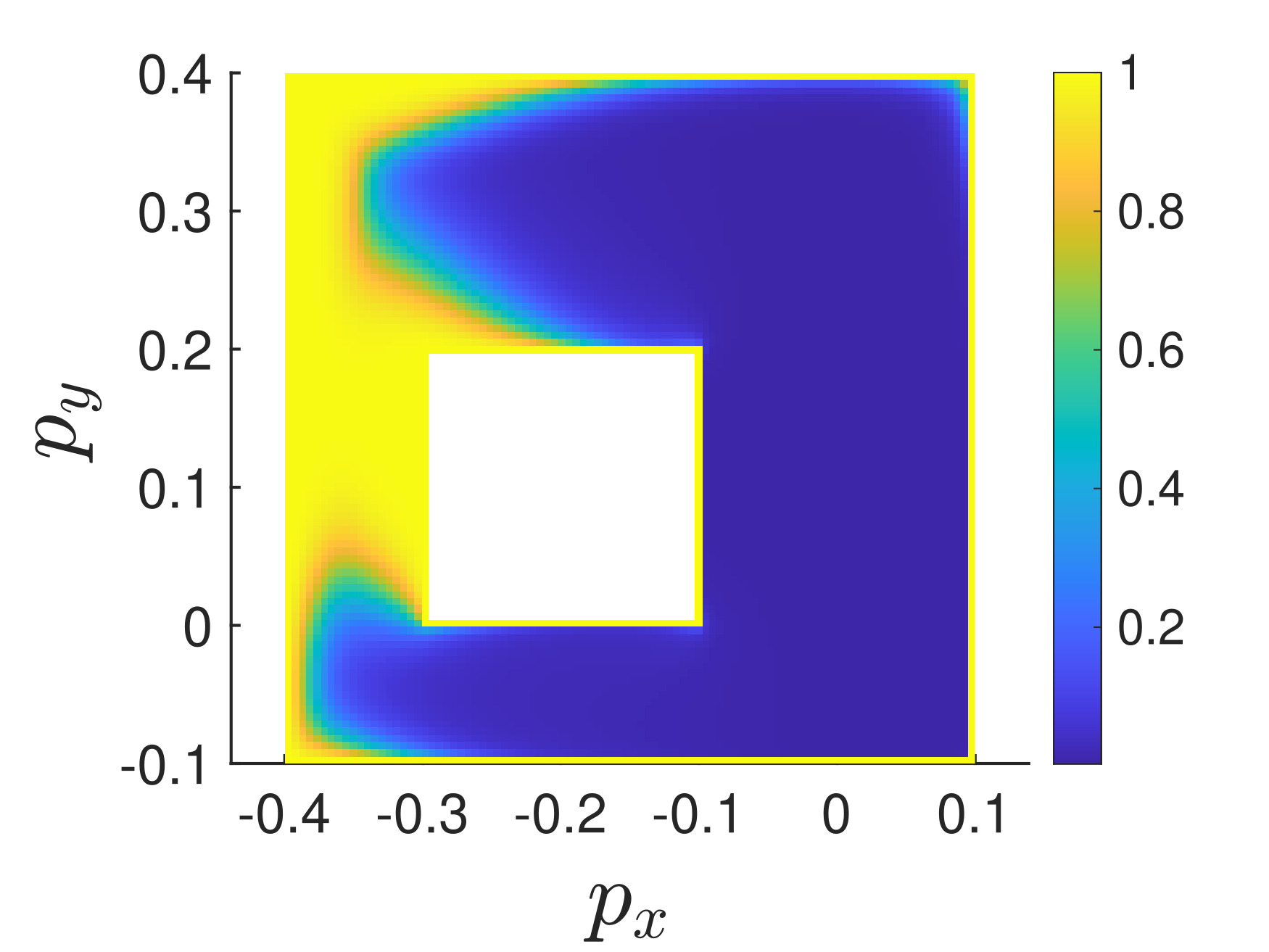} &\includegraphics[scale=0.37]{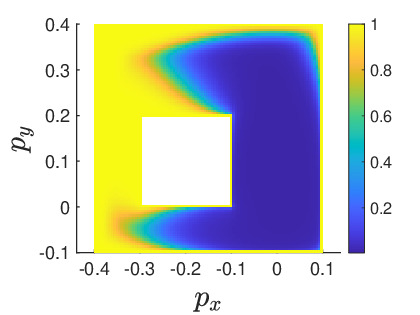} \\
      (a) $\Delta = 0.25$ & (b) $\Delta = 0.9$
      \end{tabular}
        \caption{Colormaps of the failure probabilities of the optimal policies synthesized for the input velocity model as functions of initial position.} 
        \label{Fig. risks}
\end{figure}

Figure \ref{Fig. risks} shows the colormaps of the failure probabilities of the synthesized optimal policies for $\Delta=0.25$ and $\Delta=0.9$ as functions of initial position. Notice that the region of higher failure probabilities is more for $\Delta = 0.9$ than that for the lower value of $\Delta = 0.25$.\par

The factors that affect the computation speed of FDM include the grid size, the choice of an ODE solver and its error tolerances. Whereas the computation speed of path integral depends on the number of Monte Carlo samples and the number of time steps. The computation time is evaluated on a machine with an Intel Core i7-9750H CPU clocked at 2.6 GHz. Both FDM and path integral approaches are implemented in MATLAB. Path integral simulations are run parallelly using MATLAB's parallel processing toolbox and the GPU Nvidia GeForce GTX 1650. The average number of iterations required by Algorithm \ref{alg:cap} to converge for both FDM and path integral is $3$. We choose the initial $\eta$ to be $0.05$ and the learning rate $\gamma$ to be $0.1$. For FDM, the average computation time for each iteration is approximately $20$ sec. For the path integral framework, the average computation time required for each iteration is $3$ sec. Note that no special effort was made to optimize the algorithm for speed. We plan to reduce the computation time of the algorithm in future work.

\subsection{Unicycle Model}\label{Sec: sim_4D}
The problem is illustrated in Figure \ref{Fig. sample trajs 4D}. Similar to the problem in Section \ref{Sec: sim_2D}, a particle robot wants to navigate in a 2D space from a given start position (shown by a yellow star) to the origin (shown by a magenta star), by avoiding the collisions with the red obstacles and the outer boundary. The states of the unicycle model ${x}=[{p}_x \; {p}_y \; {s}\;\; \theta ]^\top$ consist of its $x-y$ position $[{p}_x \; {p}_y]^\top$, speed ${s}$ and heading angle $\theta\in[0, 2\pi]$. The system dynamics are given by the following SDE:
    \begin{equation} \label{unicycle model}
\begin{aligned}
    \begin{bmatrix}
    d{p}_x\\d{p}_y\\d{s}\\d{\theta}
    \end{bmatrix}\!=\!
    -k
    \begin{bmatrix}
    {p}_x\\
    {p}_y\\
    {s}\\
    {\theta}
    \end{bmatrix}dt+
    \begin{bmatrix}
    {s}\cos{{\theta}}\\{s}\sin{{\theta}}\\0\\0
    \end{bmatrix}\!dt+\begin{bmatrix}
    0 & 0\\0 & 0\\1 & 0\\0 & 1
    \end{bmatrix} \!
    \left(\begin{bmatrix}
    a\\
    \omega
    \end{bmatrix}\!dt \!+\!
    \begin{bmatrix}
    \sigma & 0\\
    0 & \nu 
    \end{bmatrix}\!d{w}
    \right).
\end{aligned}
\end{equation}
The control input $u\coloneqq\begin{bmatrix} a & \omega \end{bmatrix}^\top$ consists of acceleration $a$ and angular speed $\omega$. $\sigma$ and $\nu$ are the noise level parameters. In the simulation, we set $\sigma=\nu=0.1$, $k=0.2$, $t_0=0$, $T=10$, $x_0=\begin{bmatrix}
-0.4 &-0.4 & 0 &0
\end{bmatrix}^\top$, $V({x}) = {p}_x^2 + {p}_y^2$ and $\psi\left({x}(T)\right) = {p}_x^2(T) + {p}_y^2(T)$. We solve the chance-constrained problem \eqref{CC-SOC} via path integral approach using Algorithm \ref{alg:cap}. Since the state-space of the system is of high order we do not use FDM to solve this problem. For Monte Carlo sampling, $10^4$ trajectories and a step size equal to $0.01$ are used.

\begin{figure}[h]
    \centering
      \begin{tabular}{c c}
     \includegraphics[scale=0.45]{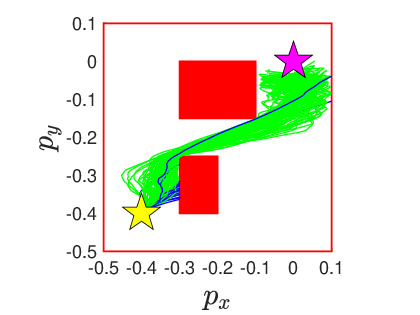} &\includegraphics[scale=0.45]{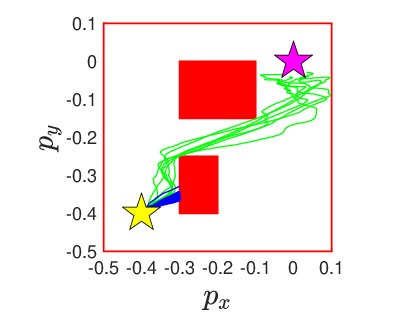} \\
       (a) $\Delta = 0.2$ &  (b) $\Delta = 0.9$\\
      \end{tabular}
        \caption{Robot navigation problem for the unicycle model. The start position is shown by a yellow star and the target position (the origin) by a magenta star. $100$ sample trajectories generated using optimal control policies for two values of $\Delta$ are shown. The trajectories are color-coded; blue paths collide with the obstacle or the outer boundary, while the green paths converge in the neighborhood of the magenta star.} 
        \label{Fig. sample trajs 4D}
\end{figure}

In Figure \ref{Fig. sample trajs 4D}, we plot $100$ sample trajectories generated using synthesized optimal policies for two values of $\Delta$. The trajectories are color-coded similar to the problem in Section \ref{Sec: sim_2D}. For a lower value of $\Delta$ the weight of blue paths is less and that of the green paths is more as compared to the higher value of $\Delta$. In other words, the failure probability for the lower value of $\Delta$ is less as compared to that of the higher value of $\Delta$.  

\begin{figure}[h]
    \centering
      \begin{tabular}{c c}
     \includegraphics[scale=0.15]{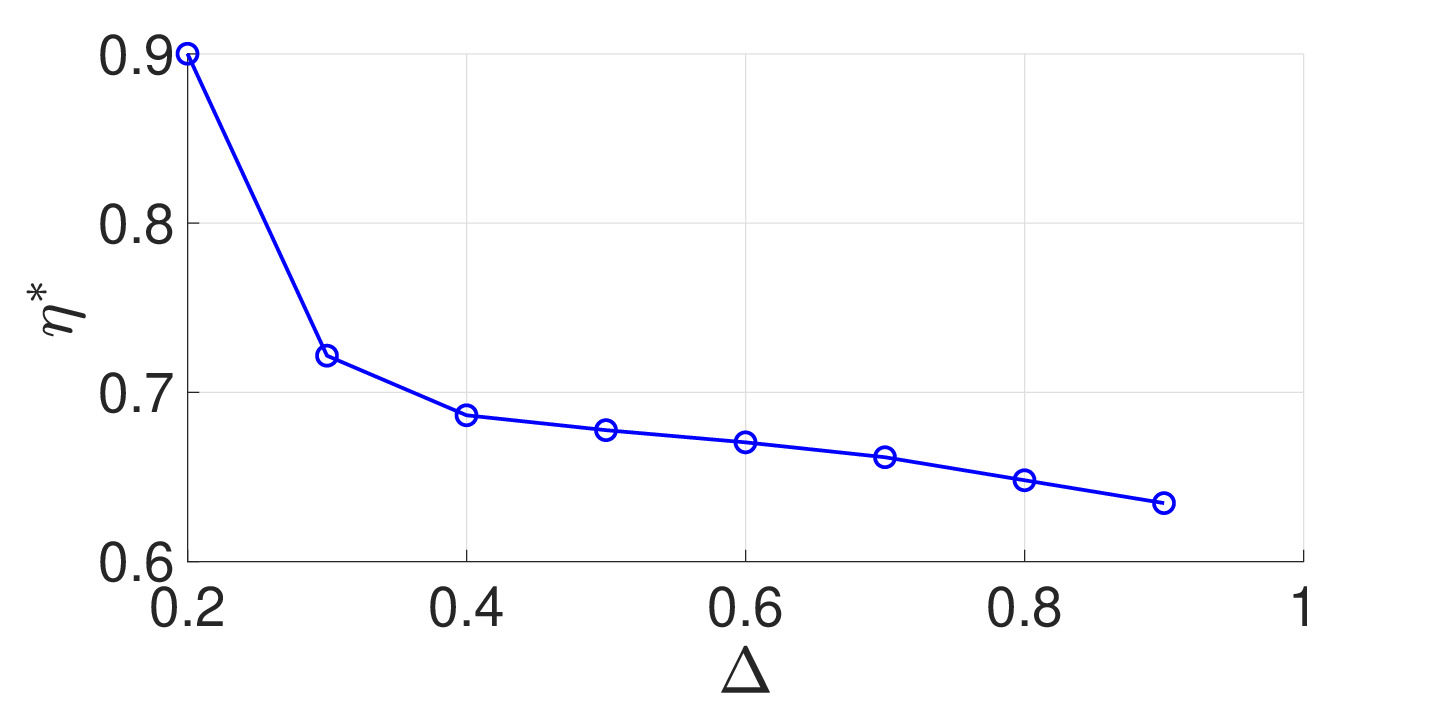} &\includegraphics[scale=0.15]{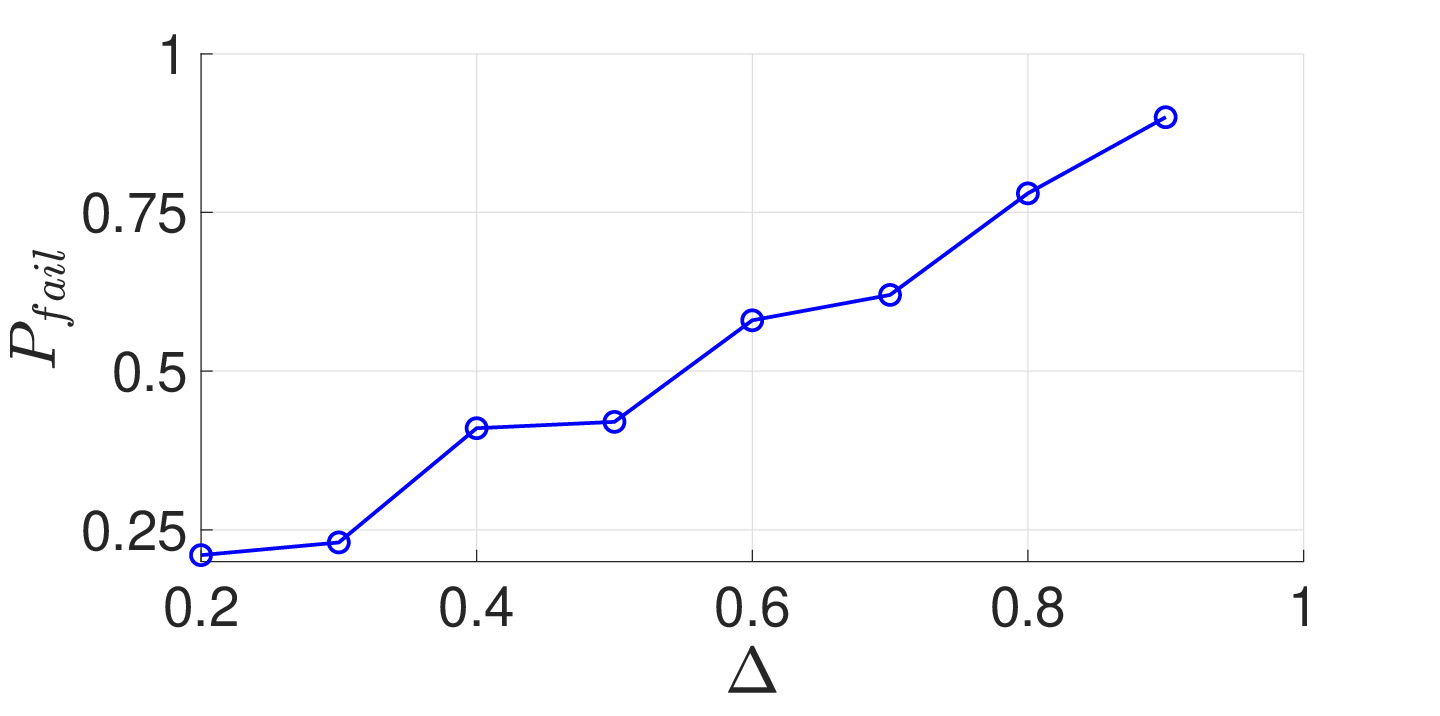} \\
      (a) $\eta^*$ vs $\Delta$ & (b) $P_{\mathrm{fail}}(x_0,t_0, u^*(\cdot\;;{\eta^*}))$ vs $\Delta$ \\
      \end{tabular}
        \caption{$\eta^*$ and $P_{\mathrm{fail}}(x_0,t_0, u^*(\cdot\;;{\eta^*}))$ vs $\Delta$ for unicycle model using path integral control.} 
        \label{Fig.eta pfail vs Delta(4D)}
\end{figure}

Figure \ref{Fig.eta pfail vs Delta(4D)} represents how the value of $\eta^*$ and $P_{\mathrm{fail}}(x_0,t_0, u^*(\cdot\;;{\eta^*}))$ change with respect to $\Delta$. As expected the value of $\eta^*$ reduces and that of $P_{\mathrm{fail}}$ increases as $\Delta$ increases.\par

The algorithm is implemented on the same machine mentioned in Section \ref{Sec: sim_2D}. It is run in MATLAB, and the  Monte Carlo simulations are run parallelly using MATLAB’s parallel processing toolbox. The average number of iterations required by the proposed algorithm to converge is $4$. We choose the initial $\eta$ to be $0.6$ and the learning rate $\gamma$ to be $0.08$. The average computation time required per iteration is $2.9$ sec. 
Note that no special effort was made to optimize the algorithm for speed. We plan to reduce the computation time of the algorithm in future work.

\subsection{Car Model}\label{Sec: sim_5D}
The problem is illustrated in Figure \ref{Fig. sample trajs 5D}. A car wants to navigate in a 2D space from a given start position (shown by a yellow star) to the origin (shown by a magenta star), by avoiding the collisions with the red obstacles and the outer boundary. The states of the car model $\begin{bmatrix}
    {p}_x&
    {p}_y&
    {s}&
    {\theta}&
    {\phi}
    \end{bmatrix}^\top$ consists of its $x-y$ position $\begin{bmatrix}
    {p}_x&
    {p}_y
    \end{bmatrix}^\top$, the speed $s$, heading angle $\theta$ and the front wheel angle $\phi$. $L$ is the inter-axle distance. The system dynamics are given by the following SDE:
    
    \begin{equation} \label{car model}
\begin{aligned}
    \begin{bmatrix}
    d{p}_x\\d{p}_y\\d{s}\\d{\theta}\\d\phi
    \end{bmatrix}\!=\!
    -k
    \begin{bmatrix}
    {p}_x\\
    {p}_y\\
    {s}\\
    {0}\\
    {0}
    \end{bmatrix}dt+
    \begin{bmatrix}
    {s}\cos{{\theta}}\\{s}\sin{{\theta}}\\0\\\frac{s\tan{\phi}}{L}
    \\0
    \end{bmatrix}\!dt+\begin{bmatrix}
    0 & 0\\0 & 0\\1 & 0\\0 & 0\\0 & 1
    \end{bmatrix} \!
    \left(\begin{bmatrix}
    a\\
    \zeta
    \end{bmatrix}\!dt \!+\!
    \begin{bmatrix}
    \sigma & 0\\
    0 & \nu 
    \end{bmatrix}\!d{w}
    \right).
\end{aligned}
\end{equation}
    The control input $u\coloneqq\begin{bmatrix} a & \zeta \end{bmatrix}^\top$ consists of acceleration $a$ and the front wheel angular rate $\zeta$. $\sigma$ and $\nu$ are the noise level parameters. In the simulation, we set $\sigma=\nu=0.07$, $k=0.2$, $t_0=0$, $T=10$, $L=0.05$, $x_0=\begin{bmatrix}
-0.4 &-0.4 & 0 &0& 0
\end{bmatrix}^\top$, $V({x}) = {p}_x^2 + {p}_y^2$ and $\psi\left({x}(T)\right) = {p}_x^2(T) + {p}_y^2(T)$. We solve the chance-constrained problem \eqref{CC-SOC} via path integral approach using Algorithm \ref{alg:cap}. Since the state-space of the system is of high order we do not use FDM to solve this problem. For Monte Carlo sampling, $10^4$ trajectories and a step size equal to $0.1$ are used. \par
In Figure \ref{Fig. sample trajs 5D}, we plot $100$ sample trajectories generated using synthesized optimal policies for two values of $\Delta$. The trajectories are color-coded similar to the previous two examples. We can observe that the failure probability for the lower value of $\Delta$ is less as compared to that of the higher value of $\Delta$.
Figure \ref{Fig.eta pfail vs Delta(5D)} represents how the value of $\eta^*$ and $P_{\mathrm{fail}}(x_0,t_0, u^*(\cdot\;;{\eta^*}))$ change with respect to $\Delta$. As expected the value of $\eta^*$ reduces and that of $P_{\mathrm{fail}}$ increases as $\Delta$ increases.\par

The algorithm is implemented on the same machine mentioned in the previous two examples. It is run in MATLAB, and the  Monte Carlo simulations are run parallelly using MATLAB’s parallel processing toolbox. The average number of iterations required by the proposed algorithm to converge is $11$. We choose the initial $\eta$ to be $0.77$ and the learning rate $\gamma$ to be $0.1$. The average computation time required per iteration is $12$ sec. Note that no special effort was made to optimize the algorithm for speed. We plan to reduce the computation time of the algorithm in future work.

\begin{figure}[h]
    \centering
      \begin{tabular}{c c}
     \includegraphics[scale=0.45]{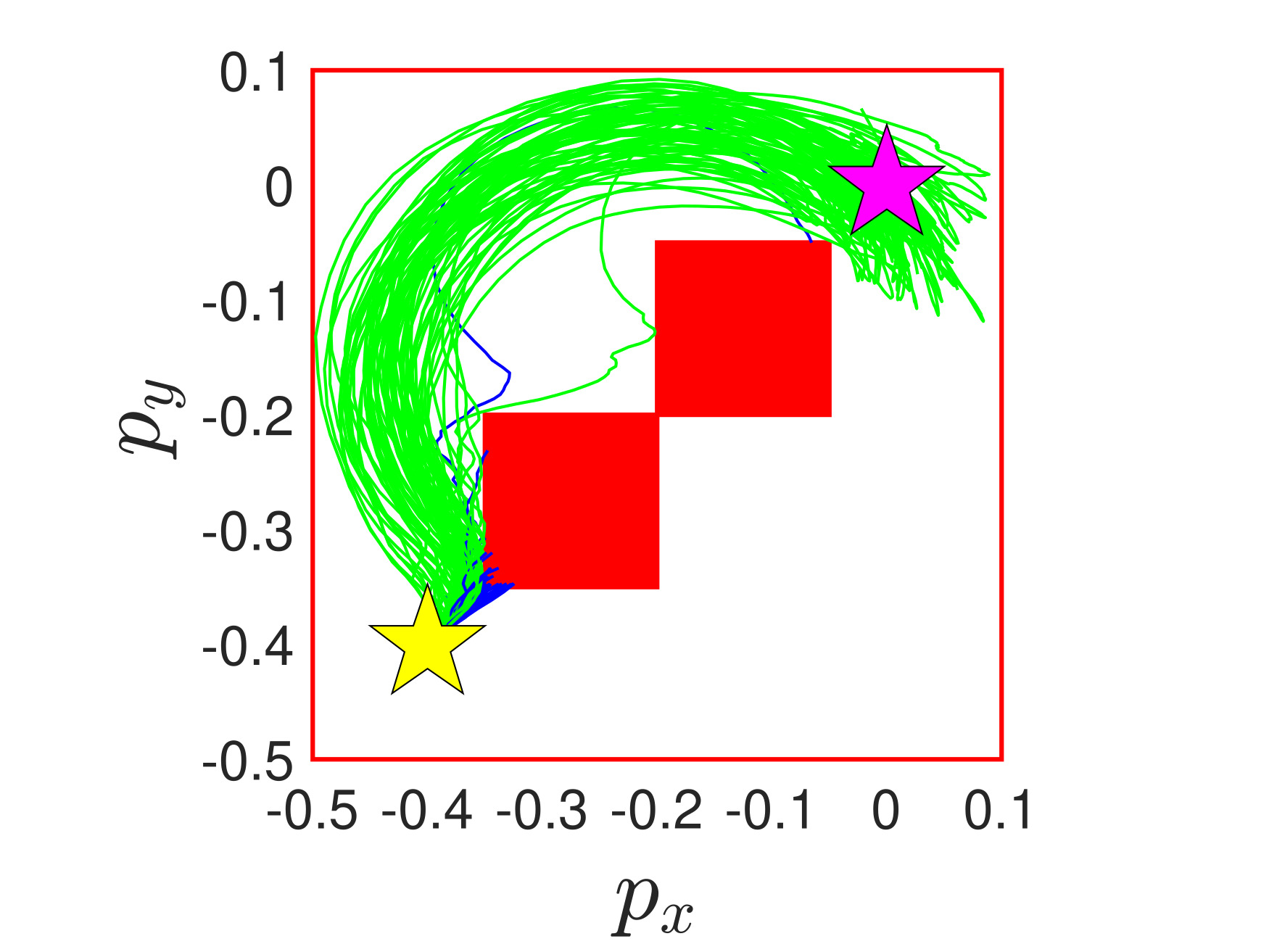} &\includegraphics[scale=0.45]{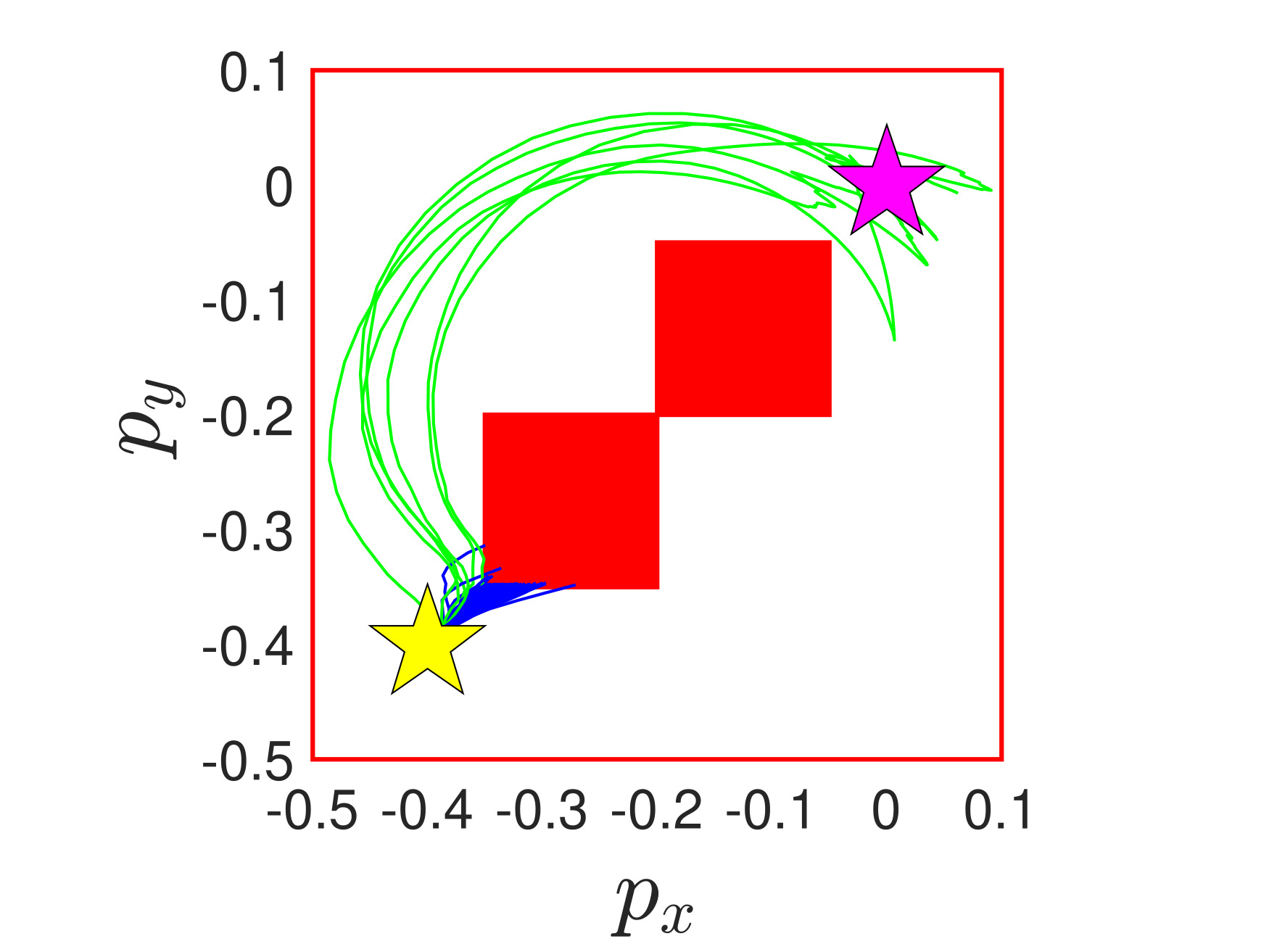} \\
     (a) $\Delta = 0.2$ &  (b) $\Delta = 0.9$\\
      \end{tabular}
        \caption{Robot navigation problem for a car model. The start position is shown by a yellow star and the target position (the origin) by a magenta star. $100$ sample trajectories generated using optimal control policies for two values of $\Delta$ are shown. The trajectories are color-coded; blue paths collide with the obstacle or the outer boundary, while the green paths converge in the neighborhood of the magenta star.} 
        \label{Fig. sample trajs 5D}
\end{figure}

\begin{figure}[h]
    \centering
      \begin{tabular}{c c}
    \!\!\!\!\!\!\!\!\!\!\!\!\!\!\!\!\!\!\!\includegraphics[scale=0.17]{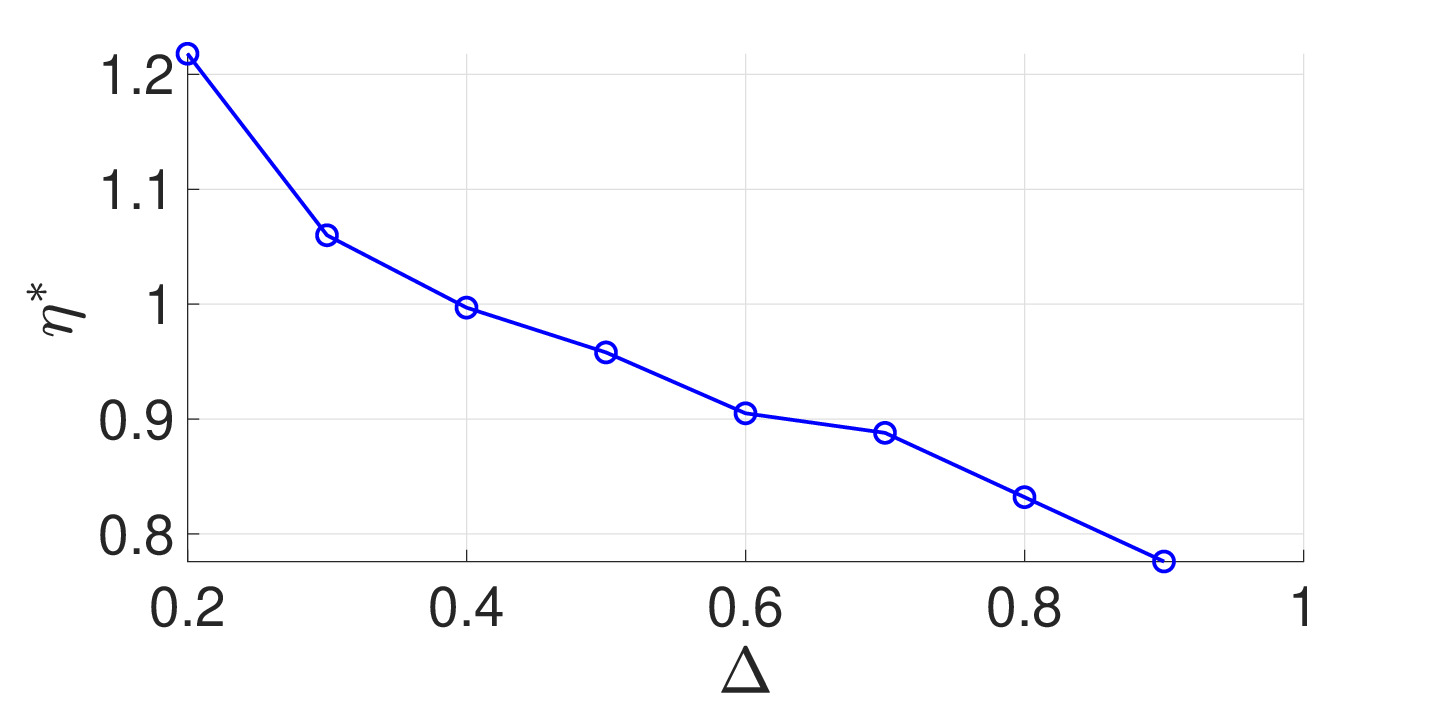} &\includegraphics[scale=0.17]{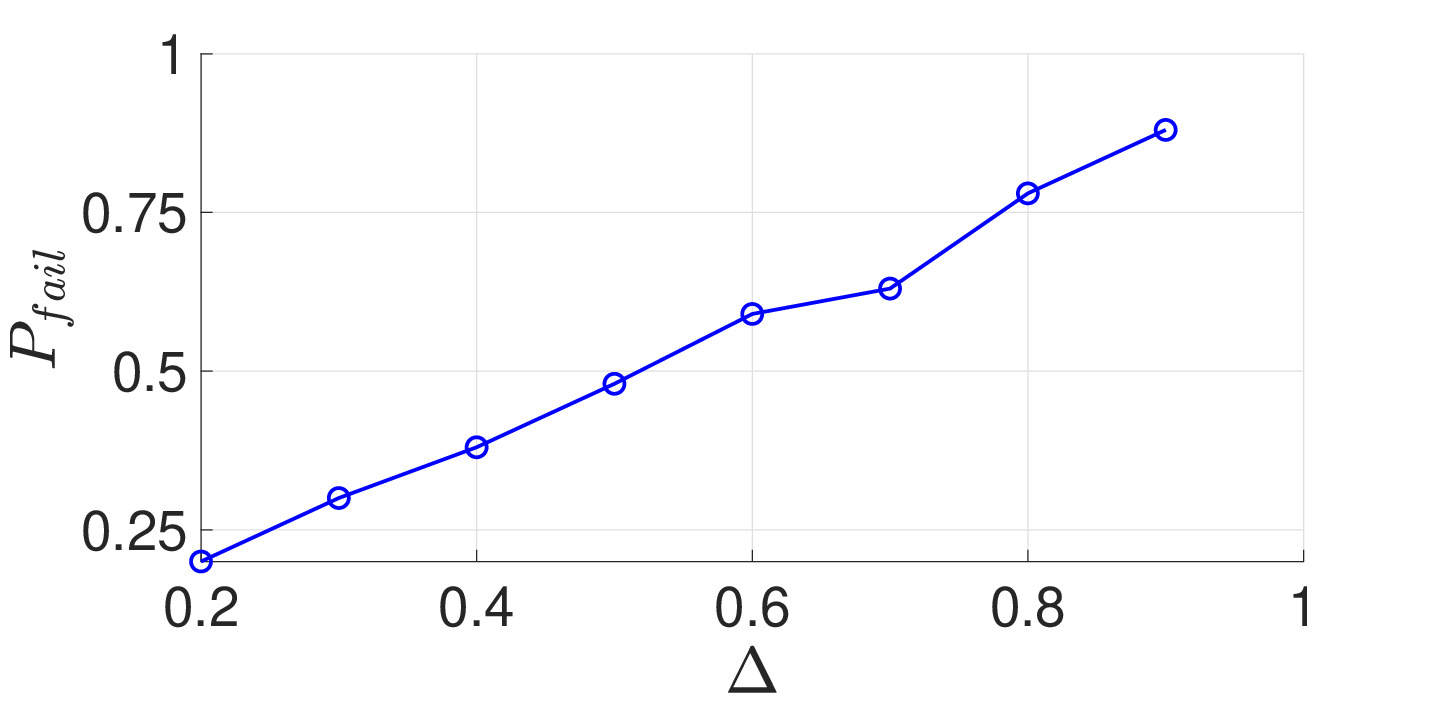} \\
      \!\!\!\!\!\!\!\!\!\!\!\!\!\!\!\!\!\!\!(a) $\eta^*$ vs $\Delta$ & (b) $P_{\mathrm{fail}}(x_0,t_0, u^*(\cdot\;;{\eta^*}))$ vs $\Delta$ \\
      \end{tabular}
        \caption{$\eta^*$ and $P_{\mathrm{fail}}(x_0,t_0, u^*(\cdot\;;{\eta^*}))$ vs $\Delta$ for a car model using path integral control.} 
        \label{Fig.eta pfail vs Delta(5D)}
\end{figure}

\section{Publications}
\begin{itemize}
     \item \textbf{A. Patil}, A. Duarte, A. Smith, F. Bisetti, T. Tanaka, ``Chance-Constrained Stochastic Optimal Control via Path Integral and Finite Difference Methods," \textit{2022 IEEE Conference on Decision and Control (CDC)}
    \item \textbf{A. Patil}, A. Duarte, F. Bisetti, T. Tanaka, ``Strong Duality and Dual Ascent Approach to Continuous-Time Chance-Constrained Stochastic Optimal Control," \textit{under review in Transactions on Automatic Control (TAC)}
\end{itemize}
\section{Future Work}

In order to prove the strong duality between the primal chance-constrained problem \eqref{CC-SOC} and its dual \eqref{eq: dual problem}, we required Assumption \ref{assum: continuity of Pfail}, where we assumed continuity of the function $\eta\mapsto P_{\mathrm{fail}}(x_0,t_0, u^*(x,t;\eta)):[0, \infty)\rightarrow[0,1]$. We will conduct a formal analysis which will prove that the Assumption \ref{assum: continuity of Pfail} holds true under mild conditions. Along with that, we plan to reduce the computation time of the proposed algorithm. We also plan to conduct the sample complexity analysis of the path integral control in order to investigate how the accuracy of Monte Carlo sampling affects the solution of chance-constrained SOC problems. Some preliminary results on this topic are presented in \cite{patil2024discrete, yoon2022sampling}. In order to achieve a strong duality between the primal chance-constrained problem and its dual, in this work, we require a certain assumption on the system dynamics and the cost function (Assumption \ref{Assum: linearity}). This assumption restricts the class of applicable system models and cost functions. In future work, we plan to find alternatives in order to get rid of this restrictive assumption (one such solution is provided in \cite{satoh2016iterative}). Another topic of future investigation is chance-constrained stochastic games.

%% file: chapters/game.tex
\chapter[Two-Player Zero-Sum Stochastic
Differential Game]{Two-Player Zero-Sum Stochastic
Differential Game}
\label{Sec: Two-Player Zero-Sum Stochastic Differential Game}
\section{Motivation and Literature Review}
Interactions among multiple agents are prevalent in many fields such as economics, politics, and engineering. Game theory studies the collective decision-making process of multiple interacting agents \cite{bacsar1998dynamic}. Two-person zero-sum games involve two players with conflicting interests, and one player's gain is the other's loss. Pursuit-evasion games (competition between a pursuer and an evader) \cite{nahin2012chases,7172219} and robust control  (competition between a controller and the nature) \cite{bacsar1998dynamic} are some examples of two-player zero-sum games. In this work, we consider a two-player stochastic differential game (SDG) in which the outcome of the game depends not only on the decision of both players but also on the stochastic input added by nature.\par

When the game dynamics and cost functions are known, the saddle-point equilibrium of a two-player zero-sum SDG can be characterized by the Hamilton-Jacobi-Isaacs (HJI) partial differential equation (PDE). The analytical solutions of the HJI PDEs are in general not available, and one needs to resort on numerical methods such as grid-based approaches \cite{falcone2006numerical}, \cite{huang2014automation} to solve these PDEs approximately. However, the grid-based approaches suffer from \textit{curse of dimensionality}, making them computationally intractable for systems with large dimensions \cite{mitchell2005time}. Moreover, in general, the solutions can not be computed in real-time using these methods; they need to be precomputed as lookup tables and recalled for use in an online setting \cite{huang2014automation}.

Several reinforcement learning algorithms have also been proposed to find approximate solutions to game problems. A reinforcement-learning-based adaptive dynamic programming algorithm is proposed in \cite{vrabie2011adaptive} to determine online a saddle-point solution of linear continuous-time two-player zero-sum differential games. A deep reinforcement learning algorithm based on updating players' policies simultaneously is proposed in \cite{prajapat2021competitive} to solve two-player zero-sum games.

These methods assume deterministic game dynamics and do not consider system uncertainties. In the presence of system uncertainties, the performance and safety of both players are affected unpredictably, if the uncertainties are not accommodated while designing policies. An effective uncertainty evaluation method, the multivariate probabilistic collocation, was used in \cite{liu2020adaptive} with integral reinforcement learning to solve multi-player SDGs for linear system dynamics online. A two-person zero-sum stochastic game with discrete states and actions is solved in \cite{lin2017multiagent} using Bayesian inverse reinforcement learning. Common challenges in the learning-based methods include training efficiency, and rigorous theoretical guarantees on convergence and optimality. Moreover, these approaches do not explicitly take into account the players' failure probabilities while synthesizing their policies. \par

In our work, we formulate a continuous-time, nonlinear, two-player zero-sum SDG on a state space modeled by an It\^o stochastic differential equation. Since the stochastic uncertainties in our model are unbounded, both players have nonzero probabilities of failure. Failure occurs when the state of the game enters into predefined undesirable domains, and one player's failure is the other's success. Our objective is to solve a game in which each player seeks to minimize its risk of failure (failure probability)\footnote{Throughout this section the word ``risk" simply means the probability of failure. It is not our intention to discuss various risk measures existing in the literature (e.g. \cite{artzner1999coherent, dixit2022risk}).} along with its control cost; hence, the name \textit{risk-minimizing zero-sum SDG}. We explain the risk-minimizing zero-sum SDG via the following example:
\begin{example}
Consider a pursuit-evasion game in which the pursuer catches the evader if they are less than a certain distance $\rho$ away from each other. In this setting, the evader wishes to minimize its probability of entering the ball of radius $\rho$ centered at the pursuer's location. Whereas, the pursuer wishes to minimize the probability of staying out of the ball of radius $\rho$ centered at the evader's location. The goal of each player is to balance the trade-off between the above probabilities (probabilities of failure) and the control cost (for e.g., their energy consumption). This problem can be formulated as a risk-minimizing zero-sum SDG. 
\end{example}\par

We derive a sufficient condition for this game to have a saddle-point equilibrium and show that it can be solved via an HJI PDE with \textit{Dirichlet boundary condition}. Under certain assumptions on the system dynamics and cost function, we establish the existence and uniqueness of the saddle-point equilibrium of the formulated risk-minimizing zero-sum SDG. Furthermore, explicit expressions for the saddle-point policies are derived which can be numerically evaluated using path integral control.  The use of the path integral technique to solve stochastic games was proposed in \cite{vrushabh2020robust}. In this work, we generalize their work and develop a path integral formulation to solve HJI PDEs with Dirichlet boundary conditions and find saddle-point equilibria of risk-minimizing zero-sum SDGs. The proposed framework allows us to solve the game online using Monte Carlo simulations of system trajectories, without the need for any offline training or precomputations.  \par

\section{Contributions}
The contributions of this work are as follows: 
\begin{enumerate}
    \item We formulate a continuous-time risk-minimizing zero-sum SDG in which players aim at balancing the trade-off between the failure probability and control cost. A sufficient condition for this game to have a saddle-point equilibrium is derived, and it is shown that this game can be solved via an HJI PDE with a Dirichlet boundary condition.
    \item Under certain assumptions on the system dynamics and cost function, we establish the existence and uniqueness of the saddle-point solution. We also obtain explicit expressions for the saddle-point policies which can be numerically evaluated using path integral control.
    \item The proposed control synthesis framework is validated by applying it to two classes of risk-minimizing zero-sum SDGs, namely a disturbance attenuation problem and a pursuit-evasion game. 
\end{enumerate}

\section*{Notations}
We use the same notations as described in Chapter \ref{Sec: Continous-Time Chance-Constrained Stochastic Optimal Control}. Table \ref{tab:notation game} represents the mathematical notations frequently used in this chapter. 

\begin{table}
\begin{center}
\begin{tabular}{||c | c || c | c||} 
 \hline
 \textbf{Notation} & \textbf{Description} & \textbf{Notation} & \textbf{Description} \\ [0.5ex] 
 \hline\hline
 $\mathcal{X}_s$ & safe region & $\partial\mathcal{X}_s$ & boundary of the safe region \\ 
 \hline
 $x(t)$ & controlled process & $\hat{x}(t)$ & uncontrolled process \\
 \hline
 $u(x(t),t)$ & agent control input &  $v(x(t),t)$ & adversary control input \\
 \hline
  $w(t)$ & Wiener process & $P^{\mathrm{ag}}_{\mathrm{fail}}$ & agent's probability of failure \\
 \hline
$P^{\mathrm{ad}}_{\mathrm{fail}}$ & adversary's probability of failure & $t_f$ & game terminal time \\
 \hline
 $ C\left(x_0, t_0; {u}, {v}\right)$ & cost function & $\psi\left({x}({t}_{f})\right)$ & terminal cost \\ 
 \hline
 $ V\left({x}(t), t\right)$ & running state cost & $\lambda$ & PDE linearizing constant \\ 
 \hline
 $ J(x,t)$ & value function & $\xi\left(x,t\right)$ & transformed value function \\ 
 [1ex] 
 \hline
\end{tabular}
\caption{Table of frequently used mathematical notation in Chapter \ref{Sec: Two-Player Zero-Sum Stochastic Differential Game}}
\label{tab:notation game}
\end{center}
\end{table}

\section{Problem Formulation}
We consider a two-player zero-sum stochastic differential game (SDG) on a finite time horizon $t\in[t_0, T]$, $t_0<T$. Consider a class of control-affine stochastic systems described by the following :

\begin{equation}\label{SDE_game}
\begin{aligned}
  d{x}(t)=&{f}\left({x}(t), t\right)dt+{G_u}\left({x}(t), t\right){u}({x}(t), t)dt\\
  &\!\!\!+ G_v\left({x}(t), t\right) v\left({x}(t), t\right)dt+{\Sigma}\left({x}(t), t\right)d{w}(t)
\end{aligned}
\end{equation}
where ${x}(t)\in\mathbb{R}^n$ is the state, ${u}({x}(t), t)\in\mathbb{R}^m$ is the control input of the first player (henceforth called the agent), and ${v}({x}(t), t)\in\mathbb{R}^l$ that of the second player (called the adversary). ${w}(t)\in\mathbb{R}^k$ is a $k$-dimensional standard Wiener process on a suitable probability space $\left(\Omega, \mathcal{F}, P\right)$. We assume sufficient regularity in the functions ${f}\left({x}(t), t\right)\in\mathbb{R}^n$, ${G_u}\left({x}(t), t\right)\in \mathbb{R}^{n\times m}$, ${G_v}\left({x}(t), t\right)\in \mathbb{R}^{n\times l}$ and ${\Sigma}\left({x}(t), t\right)\in\mathbb{R}^{n\times k}$ so that a unique strong solution  of (\ref{SDE_game}) exists \cite[Chapter 1]{oksendal2013stochastic}. 
Both the control inputs $u$, and $v$ are assumed to be square integrable (i.e., of finite energy).\par

Let $\mathcal{X}_{s}\subseteq\mathbb{R}^n$ be a bounded open set representing a safe region, $\partial\mathcal{X}_{s}$ be its boundary, and closure $\overline{\mathcal{X}_{s}}=\mathcal{X}_{s}\cup\partial\mathcal{X}_{s}$. Suppose that the agent tries to keep the system (\ref{SDE_game}) in the safe set $\mathcal{X}_s$ for the entire time horizon $[t_0, T]$ of the game, whereas the adversary seeks the opposite. For example, in pursuit-evasion games, the safe set $\mathcal{X}_s$ could be a region outside the ball of radius $\rho$, centered at the adversary's location. Or, in the disturbance rejection problems, if an agent wishes to navigate through obstacles in the presence of adversarial disturbances, then the region outside obstacles could be considered as a safe set. Suppose, when the game starts at $t_0$, the system is in the safe set i.e., ${x}(t_0)=x_0\in{\mathcal{X}_s}$. If the system leaves the region $\mathcal{X}_{s}$ at any time $t\in(t_0, T]$, we say that the agent fails. On the other hand, the adversary fails if the system stays in $\mathcal{X}_{s}$ for all $t\in[t_0, T]$. Therefore, we define the agent's probability of failure $P^{\mathrm{ag}}_{\mathrm{fail}}$ as   
\begin{equation}\label{pfail_game}
    P^{\mathrm{ag}}_{\mathrm{fail}}\!\coloneqq\!P_{x_0,t_0}\!\!\left(\bigvee_{t\in(t_0, T]} \!\!\!\!{x}(t)\notin \mathcal{X}_{s}\!\!\right)
\end{equation}
and the adversary's probability of failure $P^{\mathrm{ad}}_{\mathrm{fail}} \coloneqq 1-P^{\mathrm{ag}}_{\mathrm{fail}}$. We define the terminal time ${t}_{f}$ of the game as
\begin{equation*}\label{tf_game}
{t}_{f} \coloneqq 
\begin{cases}
T, & \!\!\!\!\!\!\!\!\!\!\!\!\!\!\!\!\!\!\!\!\!\!\!\!\!\!\!\!\!\!\!\!\!\!\!\!\!\!\text{if}\;\; {x}(t)\in\mathcal{X}_{s}, \forall t\in(t_0, T),\\
\text{inf}\;\{t\in(t_0, T) : {x}(t)\notin\mathcal{X}_{s}\}, & \text{otherwise}.
 \end{cases}
\end{equation*}
Alternatively, ${t}_f$ can be defined as 
\begin{equation}\label{tf with Q game}
 {t}_{f} \coloneqq \text{inf}\{t> t_0: ({x}(t), t)\notin \mathcal{Q}\} 
\end{equation}
where $\mathcal{Q}=\mathcal{X}_s\times[t_0, T)$ is a bounded set with the boundary  $\partial\mathcal{Q}=\left(\partial\mathcal{X}_s\times[t_0,T]\right)\cup\left(\mathcal{X}_s\times\{T\}\right)$, and closure $\overline{\mathcal{Q}}=\mathcal{Q}\cup\partial\mathcal{Q}=\overline{\mathcal{X}_s}\times[t_0, T]$. Note that by the above definitions, $\left({x}({t}_{f}),{t}_{f}\right)\in\partial{\mathcal{Q}}$ and the agent's failure probability $P^{\mathrm{ag}}_{\mathrm{fail}}$ in (\ref{pfail_game}) can be written in terms of ${t}_f$ as
\begin{equation*}\label{pfail2_game}
    \!P_{x_0,t_0}\!\!\left(\bigvee_{t\in(t_0, T]} \!\!\!\!{x}(t)\notin \mathcal{X}_{s}\!\!\right)\!=\!\mathbb{E}_{x_0, t_0}\left[\mathds{1} _{{x}({t}_{f})\in \partial\mathcal{X}_{s}}\right]\!.
\end{equation*}
Since the stochastic uncertainty of the system (\ref{SDE_game}) is modeled with an unbounded distribution, both the agent and the adversary have nonzero probabilities of failure. In our two-player SDG setting, we assume that both players aim to design an optimal policy against the \textit{worst} possible opponent's policy such that their own \textit{risk of failure} is minimized. Therefore, we define the following \textit{risk-minimizing} cost function:
\begin{equation}\label{C_game}
\begin{split}
 C\left(x_0, t_0; {u}, {v}\right)&\coloneqq\eta\,\mathbb{E}_{x_0, t_0}\left[\mathds{1} _{{x}({t}_{f})\in \partial\mathcal{X}_{s}}\right]\\
 &\!\!\!\!\!\!\!\!\!\!\!\!\!\!\!\!\!\!\!\!\!\!\!\!\!\!\!\!\!\!\!\!\!\!\!\!\!\!\!\!\!\!\!+\mathbb{E}_{x_0, t_0}\!\Bigg[\!\psi\!\left({x}({t}_{f}\!)\right)\!\cdot\!\mathds{1} _{{x}({t}_{f})\in \mathcal{X}_{s}}\!+\!\!\int_{t_0}^{{t}_{f}}\!\!\!\!L\!\left({x}(t), {u}(t), {v}(t), t\right)\!dt\!\Bigg]\!.\\
\end{split}
\end{equation}
The first term indicates the penalty associated with the agent's failure with the weight parameter $\eta>0$. $\psi\left({x}({t}_{f})\right)$ and $L\left({x}(t), {u}(t), {v}(t), t\right)$ denote the terminal and running costs, respectively. Note that the game ends at ${t}_f$ and the system doesn't evolve after that (this is motivated by the applications where ``collision" or ``capture" ends the game). Therefore, the running cost is integrated over the time horizon $[t_0, {t}_f]$. The agent tries to minimize $C$ by controlling $u$, whereas the adversary tries to maximize it by controlling $v$. The weight parameter $\eta$ balances the trade-off between the control cost and the failure probability. \par
Notice that if we define $\phi:\overline{\mathcal{X}_s}\to\mathbb{R}$ as: 
\begin{equation*}\label{phi(x)_game}
    \phi\left({x}\right)\coloneqq\psi\left({x}\right)\cdot \mathds{1} _{{x}\in \mathcal{X}_{s}}+\eta\cdot\mathds{1} _{{x}\in \partial\mathcal{X}_{s}},
\end{equation*}
then, the first term in (\ref{C_game}) can be absorbed in a new terminal cost function $\phi$ as follows:
\begin{equation}\label{C with phi}
\begin{aligned}
    {C}\!\left(x_0,t_0;{u}, {v}\right)\!=\!\mathbb{E}_{x_0, t_0}\!\Big[\phi\left({x}({t}_{f})\right)\!+\!\!\int_{t_0}^{{t}_{f}}\!\!\!\!L\!\left({x}, {u}, {v}, t\right)dt\Big]\!.
\end{aligned}
\end{equation}
In this work, we consider the following running cost that is quadratic in $u$ and $v$:
\begin{equation}\label{running cost}
\begin{aligned}
 L\!\left({x}, {u}, {v}, t\right)\!=\!V\!\left({x}, t\right)\!+\!\frac{1}{2}{u}^{\!T}\!{R_u}\!\left({x}, t\right){u}\!-\!\frac{1}{2}{v}^{\!T}\!{R_v}\!\left({x}, t\right){v}\!
\end{aligned}
\end{equation}
where $V\left({x}, t\right)$ denotes a state dependent cost, and $R_u\left({x}, t\right)\in\mathbb{R}^{m\times m}$ and $R_v\left({x}, t\right)\in\mathbb{R}^{l\times l}$ are given positive definite matrices (for all values of ${x}$ and $t$). Now, we formulate our risk-minimizing zero-sum SDG as follows: 

\begin{problem} [Risk-Minimizing Zero-Sum SDG]\label{Problem: Risk-minimizing SOC problem}

\begin{align}\label{risk-minimizing SOC problem (tf)}
\min_{u} \max_{v} \; & \mathbb{E}_{x_0, t_0}\!\!\left[\!\phi\left({x}({t}_{f})\right)\!+\!\!\!\int_{t_0}^{{t}_{f}}\!\!\!\left(\!\frac{1}{2}{u}^\top\!\!R_u{u}\!-\!\frac{1}{2}{v}^\top\!\!R_v{v}\!+\!V\!\!\right)\!dt\!\right] \notag\\
\textrm{s.t.} \;\; d{x} = & {f} dt +{G_u}{u}dt+ G_vvdt + {\Sigma}d{w},\\
&{x}(t_0)=x_0. \notag
\end{align}   
where the admissible policies $u$, $v$ are measurable with respect to the $\sigma$-algebra generated by ${x}(s), t_0\leq s\leq t$.
\end{problem}
Note that Problem \ref{Problem: Risk-minimizing SOC problem} is a \textit{variable-terminal-time} zero-sum SDG where the terminal time is determined by (\ref{tf with Q game}).

\section{Synthesis of Optimal Control Policies}
This section presents the main results. In Section \ref{Section: Risk-Minimizing HJI PDE}, we show that Problem \ref{Problem: Risk-minimizing SOC problem} can be solved using backward dynamic programming which results into an HJI PDE with a Dirichlet boundary condition. In Section \ref{Sec: PI}, we find a solution of a class of risk-minimizing zero-sum SDGs via path integral control. 

\subsection{Backward Dynamic Programming}\label{Section: Risk-Minimizing HJI PDE}

Notice that the cost function of the risk-minimizing zero-sum SDG (\ref{risk-minimizing SOC problem (tf)}) possesses the time-additive Bellman structure. Therefore, Problem \ref{Problem: Risk-minimizing SOC problem} can be solved by utilizing the principle of dynamic programming. For each $(x,t)\in\overline{\mathcal{Q}}$, and admissible policies $u$, $v$ over $[t, T)$, define the cost-to-go function:
\begin{equation*}\label{value function game}
\begin{aligned}
 {C}\left(x, t; {u},{v}\right)=&\mathbb{E}_{x, t}\big[\phi\left({x}({t}_{f})\right)\big]+\mathbb{E}_{x, t}\left[\int_{t}^{{t}_{f}}\left(\frac{1}{2}{u}^\top R_u{u}-\frac{1}{2}{v}^\top R_v{v}+V\right)\!dt\right].  
\end{aligned}
\end{equation*}

\begin{definition}[Saddle-point solution]\label{Def: SP}\cite[Chapter 2]{bacsar2008h}:
Given a two-player zero-sum differential game, a pair of admissible policies $(u^*, v^*)$ over $[t, T)$ constitutes a saddle-point solution, if for each $(x,t)\in\overline{\mathcal{Q}}$, and admissible policies $(u, v)$ over $[t, T)$,
\begin{equation*}
    C(x,t;u^*,v)\leq C^* \coloneqq C(x,t;u^*,v^*) \leq C(x,t;u,v^*).
\end{equation*}
The quantity $C^*$ is the \textit{value} of the game. The value of the game is defined if it satisfies the following relation
\begin{equation*}\label{Isaacs condition}
\begin{aligned}
 C^* \!\!=\! \min_{u} \max_{v} {C}\left(x, t; {u}, {v}\right) = \max_{v} \min_{u} {C}\left(x, t; {u}, {v}\right).  
\end{aligned}
\end{equation*}
\end{definition}
The following theorem provides the sufficient condition for a saddle-point solution of Problem \ref{Problem: Risk-minimizing SOC problem} to exist.

\begin{theorem}\label{theorem: solution to risk-minimizing soc game}
Suppose there exists a function $J:\overline{\mathcal{Q}}\rightarrow \mathbb{R}$ such that 
\begin{enumerate}[label=(\alph*)]
    \item $J(x,t)$ is continuously differentiable in $t$ and twice continuously differentiable in $x$ in the domain $\mathcal{Q}$;
    \item $J(x,t)$ solves the following stochastic HJI PDE:
    \begin{equation}\label{HJB PDE game}
   \!\!\!\!\!\!\!\!\!\!\begin{cases}
     \begin{aligned}
         \!\!-\partial_tJ\!=&V\! +\!f^\top\!\partial_xJ\!+\!\frac{1}{2}\text{Tr}\left(\Sigma\Sigma^\top\partial^2_xJ\right)\\ &\!\!\!\!\!\!\!\!\!\!\!\!\!\!\!\!\!\!\!+\frac{1}{2}\!\left(\partial_xJ\right)^\top\!\!\left(G_vR_v^{-1}G_v^\top\!-\!G_uR_u^{-1}G_u^\top\right)\!\partial_xJ,
          \end{aligned} &  \forall(x,t)\!\in\!\mathcal{Q}, \\
          \vspace{-4mm}&\\
    \!\!\underset{\substack{(x,t)\to(y,s) \\ (x,t)\in\mathcal{Q}}}{\lim}J(x,t)=\phi(y), & \!\!\!\forall(y,s)\in\partial\mathcal{Q}.
  \end{cases}
    \end{equation}
\end{enumerate}
Then, the following statements hold:
\begin{enumerate}[label=(\roman*)]
\item $J(x,t)$ is the value of the game formulated in Problem \ref{Problem: Risk-minimizing SOC problem}. That is,
\begin{equation*}\label{J as value function game}
   \begin{aligned}
     J\left(x, t\right)= \min_{u} \max_{v} {C}\left(x, t; {u}, {v}\right) =\max_{v} \min_{u} {C}\left(x, t; {u}, {v}\right),\;\:\forall\;(x,t)\!\in\!\overline{\mathcal{Q}}.
   \end{aligned}
\end{equation*}
\item The optimal solution to Problem \ref{Problem: Risk-minimizing SOC problem} is given by 
\begin{equation}\label{u* game}
\begin{aligned}
  u^*(x,t)&=-R_u^{-1}\!\left(x, t\right){G_u}^\top\!\!\left(x, t\right)\partial_xJ\!\left(x, t\right), 
\end{aligned}
\end{equation}
\begin{equation}\label{v*}
     v^*(x,t)=R_v^{-1}\!\left(x, t\right){G_v}^\top\!\!\left(x, t\right)\partial_xJ\!\left(x, t\right).
\end{equation}
\end{enumerate}
\end{theorem}
\begin{proof}
Let $J(x,t)$ be the function satisfying (a) and (b). By Dynkin's formula \cite{oksendal2013stochastic,durrett2019probability}, for each $(x,t)\in\overline{\mathcal{Q}}$ we have
\begin{align}\label{ito game}
    \mathbb{E}_{x,t}\left[J\left({x}({t}_{f}), {t}_{f}\right)\right]&=J(x,t)\\ &\!\!\!\!\!\!\!\!\!\!\!\!\!\!\!\!\!\!\!\!\!\!\!\!\!\!\!\!\!\!\!\!\!\!\!\!\!\!\!\!\!\!\!\!\!\!\!\!\!\!\!\!+\!\mathbb{E}_{x,t}\!\!\left[\!\int_{t}^{{t}_{\!f}}\left(\!\!\partial_tJ \!+\! (f\!+\!G_{\!u}u\!+\!G_{\!v}v)^{\!T}\!(\partial_xJ)\!+\!\frac{1}{2}\text{Tr}\!\left(\Sigma\Sigma^\top\!\partial_x^2J\right)\!\!\right)\!ds\!\right]\!\!.\notag
\end{align}
By the boundary condition of the PDE (\ref{HJB PDE game}),\newline $J\left({x}({t}_{f}), {t}_{f}\right) =\phi\left({x}({t}_{f})\right)$. Hence, from (\ref{ito game}), we obtain
\begin{equation}\label{ito after plugging BC game}
   J(x,t)=\mathbb{E}_{x,t}\left[\phi\left({x}({t}_{f})\right)\right]-\!\mathbb{E}_{x,t}\!\!\left[\!\int_{t}^{{t}_{\!f}}\!\!\left(\!\!\partial_tJ \!+\! (f\!+\!G_{\!u}u\!+\!G_{\!v}v)^{\!T}\!(\partial_{x}J)\!+\!\frac{1}{2}\text{Tr}\!\left(\Sigma\Sigma^{T}\!\partial_x^2J\right)\!\!\right)\!ds\!\right]\!\!.\notag
\end{equation}
Now, notice that the right-hand side of the PDE in (\ref{HJB PDE game}) can be expressed as the minimum and maximum value of a quadratic form in $u$ and $v$, respectively, as follows:
\begin{equation}\label{-dJ/dt as minimax}
    \begin{aligned}
         \!\!-\partial_tJ\!=\!\min_{{u}} \max_{{v}}\!\Bigg[&\frac{1}{2}u^\top\!R_uu-\frac{1}{2}v^\top\!R_vv\!+\!V+\!\left(\!f\!+\!G_uu\!+\!G_vv\right)^\top\!\partial_xJ+\!\frac{1}{2}\text{Tr}\!\left(\Sigma\Sigma^\top\partial^2_xJ\right)\!\Bigg]\!.
          \end{aligned} 
\end{equation}
Observe that the ``$\underset{u}{\min}$", ``$\underset{v}{\max}$" operations in (\ref{-dJ/dt as minimax}) can be interchanged. Hence, the game formulated in (\ref{risk-minimizing SOC problem (tf)}) satisfies the \textit{Isaacs condition} \cite{isaacs1999differential}. If $\hat{u}$ and $\hat{v}$ represent the minimum and maximum values of the right-hand side of (\ref{-dJ/dt as minimax}) respectively, then 
\begin{equation}\label{uhat, vhat}
    \hat{u} = -R_u^{-1}G_u^\top\partial_xJ, \qquad \hat{v} =R_v^{-1}G_v^\top\partial_xJ.
\end{equation}
Therefore, for an arbitrary $u$, we have 
\begin{equation}\label{pde inequality game}
    \begin{aligned}
         \!\!\!\!-\partial_tJ\!\leq\!\Bigg[&\frac{1}{2}u^\top\!R_uu-\frac{1}{2}{\hat{v}}^\top R_v\hat{v}\!+\!V+\!\left(\!f\!+\!G_uu\!+\!G_v\hat{v}\right)^\top\!\partial_xJ+\!\frac{1}{2}\text{Tr}\!\left(\Sigma\Sigma^\top\partial^2_xJ\right)\!\Bigg]\!.
          \end{aligned} 
\end{equation}
Now, notice that the equality in (\ref{ito after plugging BC game}) holds for any $v$. Replacing $v$ by $\hat{v}$ in (\ref{ito after plugging BC game}) yields
\begin{equation}\label{ito after plugging BC2}
    \begin{aligned}
       J(x,t)&=\mathbb{E}_{x,t}\left[\phi\left({x}({t}_{f})\right)\right]-\!\mathbb{E}_{x,t}\!\!\left[\!\int_{t}^{{t}_{\!f}}\!\!\!\!\left(\!\!\partial_tJ \!+\! (\!f\!+\!G_{\!u}u\!+\!G_{\!v}\hat{v})^{\!T}\!(\partial_xJ)\!+\!\frac{1}{2}\text{Tr}\!\left(\Sigma\Sigma^\top\!\partial_x^2J\right)\!\!\right)\!\!ds\!\right]\!\!.
    \end{aligned}
\end{equation}
Combining (\ref{pde inequality game}) and (\ref{ito after plugging BC2}), we obtain

\begin{equation*}\label{J leq C_hat game}
  \begin{aligned}
       \!\!\!\!\!\!J(x,t)\!&\leq\!\mathbb{E}_{x,t}\!\!\left[\!\phi\!\left({x}({t}_{f}\!)\right)\!+\!\!\!\int_{t}^{{t}_{f}}\!\!\!\!\left(\!\frac{1}{2}{u}^{T\!}\!R_u{u}\!-\!\frac{1}{2}\hat{{v}}^{\!T}\!\!R_v\hat{{v}}+\!V\!\!\right)\!ds\!\right]\\
       &= C\left(x,t;u,\hat{v}\right)
    \end{aligned}   
\end{equation*}
where the equality holds iff $ \hat{u} =-R_u^{-1}G_u^\top\partial_xJ$. Similarly, for an arbitrary $v$, we can show that 
\begin{equation}\label{J geq C_hat}
  J(x,t)\geq C\left(x,t;\hat{u},v\right)
\end{equation}
where the equality holds iff $\hat{v} =R_v^{-1}G_v^\top\partial_xJ$.
Therefore, from Definition \ref{Def: SP}, it follows that the pair of policies $(\hat{u}, \hat{v})$ defined in (\ref{uhat, vhat}) provides the optimal solution to the zero-sum game formulated in Problem \ref{Problem: Risk-minimizing SOC problem} and $J(x,t)$ is the value of the game.
\end{proof}

\begin{remark}
Theorem \ref{theorem: solution to risk-minimizing soc game} does not say anything about the existence of a function $J(x,t)$ satisfying statements (a) and (b), and it is not in the scope of this dissertation. However, in Section \ref{Sec: PI}, we focus on a special case in which (\ref{HJB PDE game}) can be linearized where the existence and uniqueness of such a function is guaranteed. 
\end{remark}\par

\subsection{Path Integral Solution}\label{Sec: PI}
In this section, we derive a path integral formulation to solve a class of risk-minimizing zero-sum SDGs that satisfy certain assumptions on the system dynamics and cost function. Let $\xi(x,t)$ be the logarithmic transformation of the value function $J(x,t)$ defined as
\begin{equation}\label{exp transformation game}
 J(x,t) = -\lambda\,\text{log}\left(\xi\left(x,t\right)\right)
\end{equation}
where $\lambda$ is a proportionality constant to be defined. Applying the transformation in (\ref{exp transformation game}) to (\ref{HJB PDE game}) yields

 \begin{equation}\label{transformed HJB PDE game}
  \!\!\!\!\!\!\begin{cases}
     \begin{aligned}
         \!\partial_t\xi\!=&\frac{V\xi}{\lambda}-\!\frac{1}{2}\text{Tr}\!\left(\Sigma\Sigma^\top\!\partial^2_x\xi\right)\!+\!\frac{1}{2\xi}\!\left(\partial_x\xi\right)^{\!T}\!\!\Sigma\Sigma^\top\!\partial_x\xi\\
         &\!\!\!\!\!\!\!\!\!\!\!\!\!\!\!\!+\!\frac{\lambda}{2\xi}\!\left(\partial_x\xi\right)^\top\!\!\left(G_vR_v^{-1}G_v^\top\!-\! G_uR_u^{-1}G_u^\top\right)\!\partial_x\xi\!-\!f^\top\partial_x\xi,
          \end{aligned} & \forall(x,t)\in\mathcal{Q}, \\
    \!\!\underset{\substack{(x,t)\to(y,s) \\ (x,t)\in\mathcal{Q}}}{\lim}\xi(x,t)\!=\!\text{exp}{\left(-\frac{\phi(y)}{\lambda}\right)}, & \!\!\!\!\!\!\!\!\!\!\!\!\!\!\!\!\!\!\!\!\!\!\!\!\!\!\!\!\!\!\!\!\!\!\forall(y,s)\in\partial\mathcal{Q}.
  \end{cases}
    \end{equation}
Now, we make the following assumption:
\begin{assumption}\label{Assumption: linearity}
For all $(x,t)\in\overline{\mathcal{Q}}$, there exists a constant $\lambda>0$ such that 
\begin{equation}\label{lambda game}
 \begin{aligned}
    \Sigma(x, t)\Sigma^\top(x, t) \!= &\lambda G_u(x, t)R_u^{-1}(x,t) G_u^\top(x, t)-\lambda G_v(x, t)R_v^{-1}(x,t) G_v^\top(x, t).\\
 \end{aligned}
\end{equation}
\end{assumption}
Assumption \ref{Assumption: linearity} is similar to the assumption required in the path integral formulation of a single agent stochastic control problem \eqref{lambda}. A possible interpretation of condition \eqref{lambda game} is that in a direction with high noise variance, the agent's control cost has to be low whereas that of the adversary has to be high. Therefore, the weights of the control cost $R_u$ and $R_v$ need to be tuned appropriately for the given diffusion coefficient $\Sigma(x, t)$ and the control gains $G_u(x, t)$ and $G_v(x, t)$ in the system dynamics \eqref{SDE_game}. Assumption \ref{Assumption: linearity} also implies that the stochastic noise has to enter the system dynamics via the control channels. 
Therefore, in what follows, we assume that system \eqref{SDE_game} can be partitioned into subsystems that are directly and non-directly driven by the noise as:
\begin{equation}\label{SDE partition}
    \begin{aligned}
  \begin{bmatrix}
   d{x}^{(1)} \\  d{x}^{(2)}
  \end{bmatrix}=& \begin{bmatrix}
    {f}^{(1)}({x}, t) \\ {f}^{(2)}({x}, t)
  \end{bmatrix}\!dt + \begin{bmatrix}
    \mathbf{0}\\{G_u}^{\!\!\!(2)}\!\left({x}, t\right)
  \end{bmatrix}\!{u}({x}, t)dt+ \begin{bmatrix}
    \mathbf{0}\\{G_v}^{\!\!\!(2)}\!\left({x}, t\right)
  \end{bmatrix}\!{v}({x}, t)dt+\begin{bmatrix}
    \mathbf{0}\\{\Sigma}^{(2)}\!\left({x}, t\right)
  \end{bmatrix}\!d{w}
\end{aligned}
\end{equation}
where $\mathbf{0}$ denotes a zero matrix of appropriate dimensions. By assuming a $\lambda$ satisfying Assumption \ref{Assumption: linearity} holds in (\ref{transformed HJB PDE game}), we obtain the linear PDE in $\xi$ with Dirichlet boundary condition:
\begin{equation}\label{linearized risk-minimizing HJB game}
 \!\!\begin{cases}
     \!\partial_t\xi\!=\!\frac{V\xi}{\lambda}\!-\!f^\top\partial_x\xi-\frac{1}{2}\text{Tr}\left(\Sigma\Sigma^\top\partial^2_x\xi\right),       & \forall(x,t)\!\in\!\mathcal{Q}, \\
    \!\!\underset{\substack{(x,t)\to(y,s) \\ (x,t)\in\mathcal{Q}}}{\lim}\xi(x,t)\!=\!\text{exp}{\left(-\frac{\phi(y)}{\lambda}\right)}, & \!\!\!\forall(y,s)\!\in\!\partial\mathcal{Q}.\\  
  \end{cases}
\end{equation}
The solution of a linear Dirichlet boundary value problem of the form (\ref{linearized risk-minimizing HJB game}) exits under a sufficiently regular boundary condition, and it is unique \cite[Chapter 6]{friedman1975stochastic}. Furthermore, the solution admits the Feynman-Kac representation \cite{patil2022chance}. Suppose $\hat{{x}}(t)\in\mathbb{R}^n$ is an uncontrolled process driven by the following SDE:
\begin{equation}\label{uncontrolled SDE game}
  d\hat{{x}}(t)\!=\!\!{f}\!\left(\hat{{x}}(t),\! t\right)\!dt\!+\!{\Sigma}\!\left(\hat{{x}}(t),\! t\right)\!d{w}(t)
\end{equation}
and let $ \hat{{t}}_{f} \coloneqq \text{inf}\{t> t_0: (\hat{{x}}(t), t)\notin \mathcal{Q}\}$. Then, the solution of the PDE (\ref{linearized risk-minimizing HJB game}) is given as
\begin{equation}\label{xi game}
\xi\left(x,t\right)=\mathbb{E}_{x, t}\left[\text{exp}{\left(-\frac{1}{\lambda}S\left(\tau\right)\right)}\right]  
\end{equation}
 where $S\left(\tau\right)$ denotes the cost-to-go of a trajectory $\tau$  of the uncontrolled system \eqref{uncontrolled SDE game} starting at $(x,t)$:
\begin{equation}\label{Stau}
   S\left(\tau\right)=\phi\left(\hat{{x}}(\hat{{t}}_{f})\right)+\int_{t}^{\hat{{t}}_{f}} V\left(\hat{{x}}(t), t\right)dt. 
\end{equation}
Equation \eqref{xi game} provides a path integral form for the exponentiated value function $\xi\left(x,t\right)$, which can be numerically evaluated using Monte Carlo sampling of trajectories generated by the uncontrolled SDE \eqref{uncontrolled SDE game}.
 
We now obtain the expressions for the saddle-point policies via the following theorem:
\begin{theorem}
Suppose Assumption \ref{Assumption: linearity} holds and the system (\ref{SDE_game}) can be partitioned as \eqref{SDE partition}. Then, a saddle-point solution of the risk-minimizing zero-sum SDG \eqref{risk-minimizing SOC problem (tf)} exists, is unique, and is given by
\begin{equation}\label{u* PI}
 \!\!\!\!u^*(x,t)dt\!=\!\mathcal{G}_u\!\left(x,t\right)\!\frac{\mathbb{E}_{x,t}\!\!\left[\text{exp}{\left(-\frac{1}{\lambda}S\left(\tau\right)\right)}\Sigma^{(2)}\!\!\left(x,t\right)d{w}\right]}{\mathbb{E}_{x,t}\left[\text{exp}{\left(-\frac{1}{\lambda}S\left(\tau\right)\right)}\right]}, 
\end{equation}
where 
\begin{equation*}
\mathcal{G}_u\!=\!R_u^{-1}{G_u^{(2)}}^\top\!\!\left(G_u^{(2)}R_u^{-1}{G_u^{(2)}}^\top\!\!\!-G_v^{(2)}R_v^{-1}{G_v^{(2)}}^\top\right)^{\!-1}
\end{equation*}
and 
\begin{equation}\label{v* PI}
 \!\!\!\!v^*(x,t)dt\!=\!\mathcal{G}_v\left(x,t\right)\!\frac{\mathbb{E}_{x,t}\!\!\left[\text{exp}{\left(-\frac{1}{\lambda}S\left(\tau\right)\right)}\Sigma^{(2)}\!\!\left(x,t\right)d{w}\right]}{\mathbb{E}_{x,t}\left[\text{exp}{\left(-\frac{1}{\lambda}S\left(\tau\right)\right)}\right]}, 
\end{equation}
where 
\begin{equation*}
\mathcal{G}_v\!=\!-R_v^{-1}{G_v^{(2)}}^\top\!\!\left(G_u^{(2)}R_u^{-1}{G_u^{(2)}}^\top\!\!\!-G_v^{(2)}R_v^{-1}{G_v^{(2)}}^\top\right)^{\!-1}.
\end{equation*}
\end{theorem}
\begin{proof}
The existence and uniqueness of the saddle-point solution follow from the existence and uniqueness of the linear Dirichlet boundary value problem \eqref{linearized risk-minimizing HJB game} \cite[Chapter 6]{friedman1975stochastic} and from Theorem \ref{theorem: solution to risk-minimizing soc game}. The saddle-point solution $u^*(x,t)$ (\ref{u* game}) and $v^*(x,t)$ (\ref{v*}) can be computed by taking the gradient of (\ref{xi game}) with respect to $x$ and using the condition (\ref{lambda game}). (The derivation of (\ref{u* game}) and (\ref{v*}) is in the same vein as the derivation of optimal controls in single-agent settings \cite{satoh2016iterative,theodorou2010generalized}; not presented here for brevity.)
\end{proof}
 Equations (\ref{u* PI}) and (\ref{v* PI}) provide the path integral forms for the saddle-point equilibrium. Similar to \eqref{xi game}, the expectations in (\ref{u* PI}) and (\ref{v* PI}) can be numerically evaluated in real-time via the Monte Carlo sampling of the trajectories generated by the uncontrolled SDE \eqref{uncontrolled SDE game}. Path integral framework allows us to solve the game online without requiring any offline training or precomputations. Even though Monte Carlo simulations must be performed in real-time in order to evaluate (\ref{u* PI}, \ref{v* PI}) for the current $(x,t)$, these simulations can be massively parallelized through the use of GPUs. 
 
\section{Simulation Results}
In this section, we apply the path integral framework to two classes of risk-minimizing zero-sum SDGs (\ref{risk-minimizing SOC problem (tf)}): a disturbance attenuation problem and a pursuit-evasion game.

\subsection{Disturbance Attenuation Problem}\label{Sec: H-inf}
Consider a special class of systems (\ref{SDE}):
\begin{equation}\label{ex1:SDE}
\begin{aligned}
  d{x}\!=\!{f}({x}, t)dt+{G_u}({x}, t)\Big(\!{u}({x}, t)dt \!+\! v\left({x}, t\right)dt\!+\!d{w}\!\Big)
\end{aligned}
\end{equation}
where $u\left({x}, t\right)\in\mathbb{R}^m$ is the control input, $v\left({x},t\right)\in\mathbb{R}^m$ is the bounded disturbance and ${w}(t)\in\mathbb{R}^m$ is a Wiener process. Here, we have two sources of noise that corrupt the system's control input $u$: the bounded noise $v$ whose statistics are unknown, and the white noise $d{w}$. 
In the disturbance attenuation problem, the objective is to design a policy $u$ in the presence of stochastic noise and bounded disturbance $v$ such that the system's control performance $ \mathbb{E}_{x_0, t_0}\!\left[\phi\left({x}({t}_{f})\right)\!+\!\int_{t_0}^{{t}_{f}}\!\left(\frac{1}{2}{u}^{\!\top}\!{u}+V\right)dt\right]$ is minimized. This problem can be solved using the following zero-sum SDG, where $u$ is considered as a control input of the first player (agent) and $v$ that of the second player (adversary):
\begin{equation}\label{J_gamma}
   \!\!\! \!\underset{u}{\min}\;\underset{v}{\max}\;\mathbb{E}_{x_0, t_0}\!\!\left[\!\phi\left({x}({t}_{f})\right)\!\!+\!\!\!\int_{t_0}^{{t}_{\!f}}\!\!\!\left(\!\frac{1}{2}{u}^\top\!{u}\!-\!\frac{\gamma^{2}}{\!2}{v}^{\!T}\!{v}\!+\!V\!\!\right)\!dt\!\right].
\end{equation}
$\gamma$ is a given positive constant that determines the level of disturbance attenuation. Theorem \ref{Propo. H_inf bound} provides an upper bound on the system's control performance (in the presence of a bounded disturbance $v$) that can be obtained by solving the game \eqref{J_gamma}.
\begin{theorem}\label{Propo. H_inf bound}
Suppose $(u^*_\gamma, v^*_\gamma)$ represent the saddle-point policies of the SDG (\ref{J_gamma}) for any $\gamma$, and let
\begin{equation}\label{delta gamma}
    \delta_\gamma \coloneqq \mathbb{E}_{x_0, t_0}^{u_\gamma^*, v_\gamma^*}\left[\int_{t_0}^{{t}_{\!f}}\!\!\!\!{v}^\top\!{v}\;dt\right]
\end{equation}
where the superscript on $\mathbb{E}$ denotes the policies under which the expectation is computed. Then, for all adversarial policies $v$ such that $\mathbb{E}_{x_0, t_0}^{u_\gamma^*, v}\left[\int_{t_0}^{{t}_{\!f}}\!{v}^\top\!{v}\;dt\right]\leq\delta$ (for any $\delta>0$), we get the following upper bound on the system's control performance in the presence of disturbance $v$:
\begin{equation}\label{UB on sys cost}
\begin{aligned}
      \mathbb{E}_{x_0, t_0}^{u_\gamma^*, v_\gamma^*}\!\!\left[\!\phi\left({x}({t}_{f})\right)\!\!+\!\!\!\int_{t_0}^{{t}_{\!f}}\!\!\!\left(\!\frac{1}{2}{u}^\top\!{u}\!+\!V\!\!\right)\!dt\!\right] + \frac{\gamma^2}{2}\left(\delta - \delta_\gamma\right) \geq\mathbb{E}_{x_0, t_0}^{u_\gamma^*, v}\!\!\left[\!\phi\left({x}({t}_{f})\right)\!+\!\!\!\int_{t_0}^{{t}_{\!f}}\!\!\!\left(\!\frac{1}{2}{u}^\top\!{u}+\!V\!\!\right)\!dt\!\right].
\end{aligned}
\end{equation}
\end{theorem}
\begin{proof}
Consider cost of the SDG (\ref{J_gamma}) under the saddle-point policies $(u^*_\gamma, v^*_\gamma)$:

\begin{subequations}\label{contra. final}
\begin{align}
\mathbb{E}_{x_0, t_0}^{u_\gamma^*, v_\gamma^*}\!\!\left[\!\phi\left({x}({t}_{f})\right)\!\!+\!\!\!\int_{t_0}^{{t}_{\!f}}\!\!\!\left(\!\frac{1}{2}{u}^\top{u}\!+\!V\!\!-\!\frac{\gamma^2}{2}{v}^\top{v}\!\!\right)\!dt\!\right] =&\;\mathbb{E}_{x_0, t_0}^{u_\gamma^*, v_\gamma^*}\!\!\left[\!\phi\left({x}({t}_{f})\right)\!\!+\!\!\!\int_{t_0}^{{t}_{\!f}}\!\!\!\left(\!\frac{1}{2}{u}^\top{u}\!+\!V\!\!\right)\!dt\!\right] -\frac{\gamma^2}{2}\delta_\gamma \label{contra. final1}\\
\geq&\;\mathbb{E}_{x_0, t_0}^{u_\gamma^*, v}\!\!\left[\!\phi\left({x}({t}_{f})\right)\!\!+\!\!\!\int_{t_0}^{{t}_{\!f}}\!\!\!\left(\!\frac{1}{2}{u}^\top{u}\!+\!V\!\!-\!\frac{\gamma^2}{2}{v}^\top{v}\!\!\right)\!dt\!\right]\label{contra. final2}\\
\geq&\;\mathbb{E}_{x_0, t_0}^{u_\gamma^*, v}\!\!\left[\!\phi\left({x}({t}_{f})\right)\!\!+\!\!\!\int_{t_0}^{{t}_{\!f}}\!\!\!\left(\!\frac{1}{2}{u}^\top{u}\!+\!V\!\!\right)\!dt\!\right] -\frac{\gamma^2}{2}\delta.\label{contra. final3}
\end{align}
\end{subequations}
The equation \eqref{contra. final1} follows from (\ref{delta gamma}). For any adversarial policy $v$, the inequality \eqref{contra. final2} follows because $v^*_\gamma$ maximizes the cost in \eqref{J_gamma}. The inequality \eqref{contra. final3} follows from the bound on $\mathbb{E}_{x_0, t_0}^{u_\gamma^*, v}\left[\int_{t_0}^{{t}_{\!f}}\!{v}^\top\!{v}\;dt\right]$. Using \eqref{contra. final1} and \eqref{contra. final3}, we get the desired inequality \eqref{UB on sys cost}. 
\end{proof}
\vspace{2mm}
In order to solve the HJI PDE associated with the game (\ref{J_gamma}) via the path integral framework described in Section \ref{Sec: PI}, it is necessary to find a constant $\lambda>0$ (by Assumption \ref{Assumption: linearity}) such that
\begin{equation*}\label{linearity on H_inf}
    \lambda\left(1-\frac{1}{\gamma^2}\right)=1.
\end{equation*}
Therefore, for all $\gamma>1$, Assumption 1 is satisfied and as a consequence, the zero-sum SDG (\ref{J_gamma}) admits a unique saddle-point solution. \par
We now present a simulation study of the disturbance attenuation problem using a unicycle navigation example. Consider the following unicycle dynamics model:
\begin{equation*} 
\begin{aligned}
    \begin{bmatrix}
    d{p}_x\\d{p}_y\\d{s}\\d{\theta}
    \end{bmatrix}\!=\!
    -k
    \begin{bmatrix}
    {p}_x\\
    {p}_y\\
    {s}\\
    {\theta}
    \end{bmatrix}dt+
    \begin{bmatrix}
    {s}\cos{{\theta}}\\{s}\sin{{\theta}}\\0\\0
    \end{bmatrix}\!dt+\begin{bmatrix}
    0 & 0\\0 & 0\\1 & 0\\0 & 1
    \end{bmatrix} \!
    \left(\begin{bmatrix}
    a\\
    \omega
    \end{bmatrix}\!dt \!+\!
    \begin{bmatrix}
    \Delta a\\
    \Delta \omega
    \end{bmatrix}\!dt \!+\!
    \begin{bmatrix}
    \sigma & 0\\
    0 & \nu 
    \end{bmatrix}\!d{w}
    \right),
\end{aligned}
\end{equation*}
where $({p}_x,\:{p}_y)$, ${s}$ and ${\theta}$ denote the position, speed, and the heading angle of the unicycle, respectively. The control input $u\coloneqq\begin{bmatrix} a & \omega \end{bmatrix}^\top$ consists of acceleration $a$ and angular speed $\omega$. $v\coloneqq\begin{bmatrix} \Delta a & \Delta \omega \end{bmatrix}^\top$ is the bounded disturbance acting on the system's control input, and $d{w}\in\mathbb{R}^2$ is the white noise with $\sigma$ and $\nu$ being the noise level parameters. As illustrated in Figure \ref{Fig. effect of eta}, the unicycle aims to navigate in a two-dimensional space from its initial position (represented by the yellow star) to the origin (represented by the magenta star), in finite time, while avoiding the red obstacles and the outer boundary. The white region that lies between the outer boundary and the obstacles is the safe region $\mathcal{X}_s$. This is a disturbance attenuation problem since the unicycle aims to design its control policy $u$ in order to minimize the control performance and risk of failure (collision with the obstacles or the outer boundary) under worst-case disturbance $v$. Therefore, we can formulate this problem as the risk-minimizing zero-sum SDG (\ref{J_gamma}). In the simulation, we set $\sigma=\nu=0.1$, $k=0.2$, $t_0=0$, $T=10$, $x_0=\begin{bmatrix}
-0.4 &-0.4 & 0 &0
\end{bmatrix}^\top$, $V({x}) = {p}_x^2 + {p}_y^2$ and $\psi\left({x}(T)\right) = {p}_x^2(T) + {p}_y^2(T)$. In order to evaluate the optimal policies (\ref{u* PI}) and (\ref{v* PI}) via Monte Carlo sampling, $10^4$ trajectories and a step size equal to $0.01$ are used. We demonstrate two experiments.
\subsubsection{Experiment 1}
In this experiment, we set $\eta = 0.67$ and plot in Figure \ref{Fig. effect of eta} $100$ sample trajectories generated using synthesized saddle-point policies $(u^*, v^*)$ for two values of $\gamma$. The trajectories are color-coded; the blue paths collide with the obstacles, while the green paths converge in the neighborhood of the origin (the target position). The figure shows that for a higher value of $\gamma$ i.e. when the adversary becomes less powerful, the failure probability of the agent $P^\mathrm{ag}_{\mathrm{fail}}$ reduces.
\subsubsection{Experiment 2}
In this experiment, we set $\gamma^2=3$, $\eta=1$, and study the effect of ignoring the adversary. First, we compute saddle-point policies $(u^*, v^*)$ for the game (\ref{J_gamma}) same as Experiment 1 and plot in Figure \ref{Fig. safe unsafe}-(a) $100$ sample trajectories generated using $(u^*, v^*)$. In this case, the agent is aware of the adversary and designs its policy $u^*$ cautiously. The probability of failure is $23\%$. In the second case, the agent is not aware of the presence of the adversary and computes its policy (say) ${\widetilde{u}}^*$ by solving a single agent optimization problem. However, in reality, the adversary is present and suppose it follows the policy $v^*$. Figure \ref{Fig. safe unsafe}-(b) shows $100$ sample trajectories generated using $({\widetilde{u}}^*, v^*)$. In this case, the agent's performance is poor, it fails $65\%$ of the time. The color-coding of the trajectories is the same as in Experiment 1.
      

\begin{figure}
     \centering
       \begin{tabular}{c c}
       \!\!\!\!\!\!\!\!\!\!\!\!\includegraphics[scale=0.5]{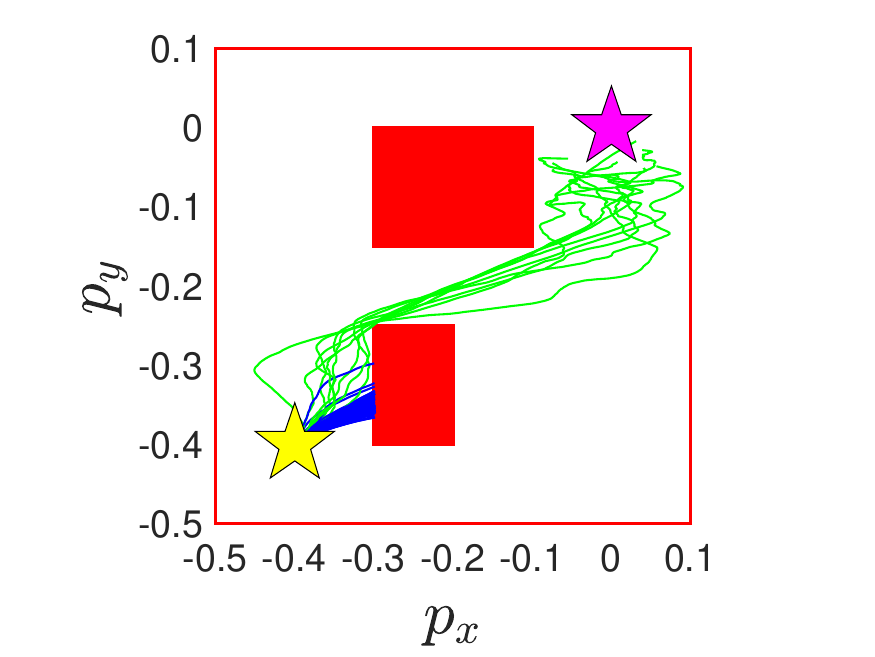} &\!\!\!\!\!\!\!\!\!\!\!\!\includegraphics[scale=0.5]{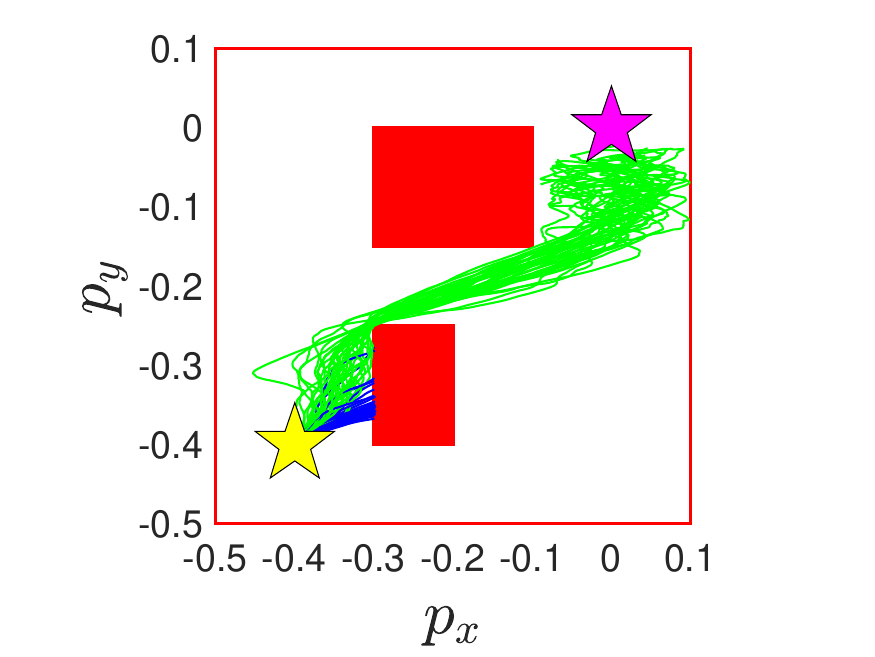} \\
       (a) $\gamma^2 = 2$, $P^\mathrm{ag}_{\mathrm{fail}} = 0.9$  & (b) $\gamma^2 = 7 $, $P^\mathrm{ag}_{\mathrm{fail}}= 0.64$ \\

       \end{tabular}
         \caption{Unicycle navigation in the presence of bounded and stochastic disturbances. The start position is shown by a yellow star and the target position (the origin) by a magenta star. $100$ sample trajectories generated using saddle-point policies ($u^*, v^*$) for two values of $\gamma$ are shown. The trajectories are color-coded; blue paths collide with the red obstacles or the outer boundary, while the green paths converge in the neighborhood of the magenta star. The failure probabilities of the agent $P^\mathrm{ag}_{\mathrm{fail}}$ are noted below each case.} 
         \label{Fig. effect of eta}
 \end{figure}
 
\begin{figure}
     \centering
       \begin{tabular}{c c}
       \!\!\!\!\!\!\!\!\!\!\!\!\includegraphics[scale=0.5]{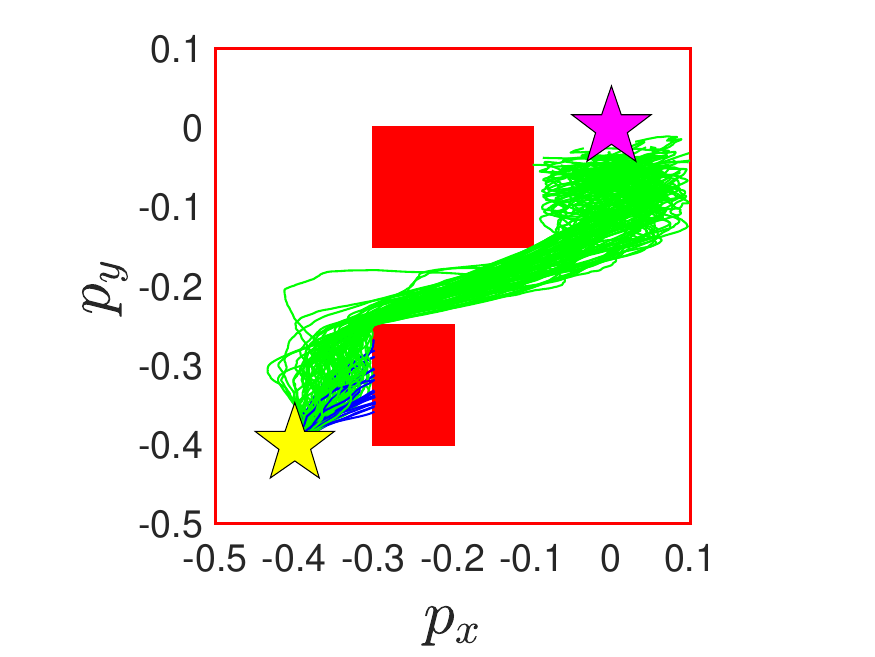} &\!\!\!\!\!\!\!\!\!\!\!\!\includegraphics[scale=0.5]{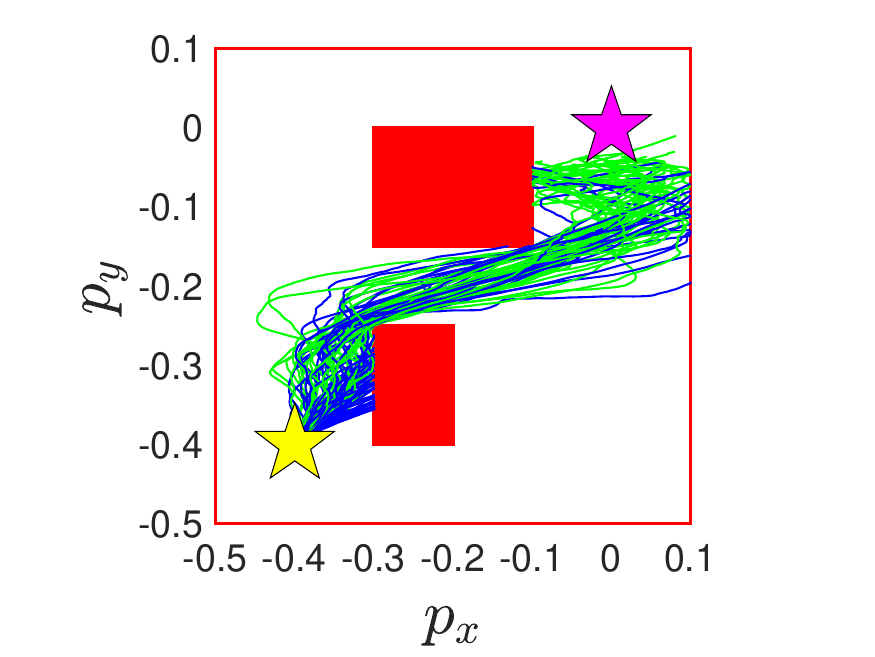} \\
       (a) Agent is aware of & (b) Agent is not aware\\ the adversary, $P^\mathrm{ag}_{\mathrm{fail}} = 0.23 $&   of the adversary, $P^\mathrm{ag}_{\mathrm{fail}}= 0.65$\\
       \end{tabular}
         \caption{Unicycle navigation in the presence of bounded and stochastic disturbances. (a) The agent is aware of the presence of an adversary. $100$ sample trajectories generated using saddle-point policies $(u^*, v^*)$. (b) The agent is not aware of the presence of an adversary. $100$ sample trajectories generated using $({\widetilde{u}}^*, v^*)$. The failure probabilities of the agent $P^\mathrm{ag}_{\mathrm{fail}}$ are noted for each case.} 
         \label{Fig. safe unsafe}
 \end{figure}
 
\subsection{Pursuit-Evasion Game}\label{Sec: PE}

Consider a two-player zero-sum SDG on a finite time horizon $[t_0, T]$, in which the adversary is chasing the agent and the agent is trying to escape from the adversary. We will call the adversary as a pursuer and the agent as an evader. Suppose the evader and the pursuer are moving in a two-dimensional plane according to
 \begin{equation}\label{sim EP model}
     \begin{split}
    d{p}^E_x={u}_xdt+\sigma^E_x d{w}^E_x,\qquad d{p}^P_x={v}_xdt+\sigma^P_x d{w}^P_x,\\
     d{p}^E_y={u}_ydt+\sigma^E_y d{w}^E_y,\qquad d{p}^P_y={v}_ydt+\sigma^P_y d{w}^P_y,
\end{split}
 \end{equation}
where ${x}_E\coloneqq\begin{bmatrix}
 {p}^E_x & {p}^E_y 
\end{bmatrix}^\top$ is the position and $u \coloneqq\begin{bmatrix}
 u_x & u_y 
\end{bmatrix}^\top$ is the control input of the evader. Similarly,  ${x}_P\coloneqq\begin{bmatrix}{p}^P_x & {p}^P_y\end{bmatrix}^\top$ and $v\coloneqq\begin{bmatrix}
 v_x &v_y 
\end{bmatrix}^\top$ are the position and control input of the pursuer. ${w}_x^E, {w}_y^E, {w}_x^P, {w}_y^P$ are independent one-dimensional standard Brownian motions. If at any time $t \in (t_0, T]$, the pursuer gets within a distance $\rho$ of the evader, then it catches the evader and the evader fails. On the other hand, if the evader avoids getting within a distance $\rho$ of the pursuer for the entire time horizon $[t_0, T]$, then that's a failure for the pursuer. The pursuer aims at designing its control policy $v$ in order to maximize the probability of catching the evader, whereas the evader seeks the opposite by designing $u$. For this two-player differential game, it is the relative position of the pursuer and evader that is important (and relevant), rather than their absolute positions. Let ${x}\coloneqq\begin{bmatrix}
  {p}_x & {p}_y
\end{bmatrix}$ be the evader's position with respect to the pursuer where 
\begin{equation*}
    {p}_x = {p}^E_x - {p}^P_x, \quad {p}_y = {p}^E_y - {p}^P_y
\end{equation*}
and the origin coincides with the pursuer's position. Thus, the coordinate system is attached to the pursuer and is not fixed in space. The system ${x}$ follows the SDE
\begin{equation}\label{rel. dynamics}
    d{x}=\begin{bmatrix}
    d{p}_x\\ d{p}_y
    \end{bmatrix} = 
    \begin{bmatrix}
    u_x\\ u_y
    \end{bmatrix}dt - 
    \begin{bmatrix}
    v_x\\ v_y
    \end{bmatrix}dt +
    \begin{bmatrix}
  \sigma_x & 0\\  0 & \sigma_y
    \end{bmatrix}d{w}
\end{equation}
where $\sigma_x =\sqrt{(\sigma_x^E)^2 + (\sigma_x^P)^2}$, $\sigma_y =\sqrt{(\sigma_y^E)^2 + (\sigma_y^P)^2}$ and ${w}$ is a two-dimensional standard Brownian motion. In this game, the safe set $\mathcal{X}_s$ can be defined as $\mathcal{X}_s \coloneqq \left\{x\in\mathbb{R}^2:\|x\|>\rho\right\}$. Suppose the control cost matrix $R_u$ of the evader is unity and that of the pursuer $R_v = {r_v}^2$, where $r_v$ is a given positive scalar constant. Therefore, the risk-minimizing zero-sum SDG takes the form:
\begin{equation}\label{SDG of PE}
    \!\!\underset{u}{\min}\;\underset{v}{\max}\;\mathbb{E}_{x_0, t_0}\!\!\left[\!\phi\left({x}({t}_{f}\!)\right)\!+\!\!\!\int_{t_0}^{{t}_{\!f}}\!\!\!\!\left(\!\frac{1}{2}{u}^{\!T}\!{u}\!-\!\frac{{r_{\!v}}^{\!2}}{2}{v}^{\!T}\!{v}\!+\!V\!\!\right)\!dt\!\right]\!.
\end{equation}
In order to solve the associated HJI equation of this game via the path integral framework, it is necessary to find a constant $\lambda>0$ (by Assumption 1) such that
\begin{equation*}
    \lambda\left(1-\frac{1}{{r_v}^2}\right)=1.
\end{equation*}
Therefore, for all $r_v>1$, Assumption 1 is satisfied and as a consequence, the zero-sum SDG (\ref{SDG of PE}) admits a unique saddle-point solution. \par 

In the simulation, we set $\sigma_x^E=\sigma_y^E=\sigma_x^P=\sigma_y^P=\sqrt{0.1}$, $\rho=0.1$, $t_0=0$, $T=2$, $x_0 = \begin{bmatrix}
  0.3 & 0.3
\end{bmatrix}^\top$, $V({x})=\psi\left({x}(T)\right)=0$,  $\eta=0.2$, ${r_v}^2 = 2$. Figure \ref{Fig. EP} shows a plot of two sample trajectories of the system (\ref{rel. dynamics}) generated using synthesized saddle-point policies $(u^*, v^*)$. The trajectories start from $x_0$ shown by the yellow star. The red disc of radius $\rho = 0.1$, centered at the origin represents that the pursuer is within a distance $\rho$ of the evader. The green trajectory never enters the red disc in the horizon $[t_0, T]$, thus, it represents a case when the evader escapes from the pursuer. The blue trajectory, on the other hand, enters the red disc and thus represents a case when the pursuer catches the evader. Figure \ref{Fig. pfail with rv2} shows a plot of failure probabilities of the agent (i.e., evader) as a function of $r_v$, when the players follow the saddle-point policies ($u^*, v^*$). These values are computed using na\"ive Monte Carlo sampling, with $400$ sample trajectories. The plot shows that as the control cost weight $r_v$ of the adversary (i.e., pursuer) increases, the chances of the evader getting caught reduces.   

 \begin{figure}
     \centering
       \begin{tabular}{c}
       \!\!\!\!\!\!\!\!\!\!\!\!\!\!\!\!\!\includegraphics[scale=0.25]{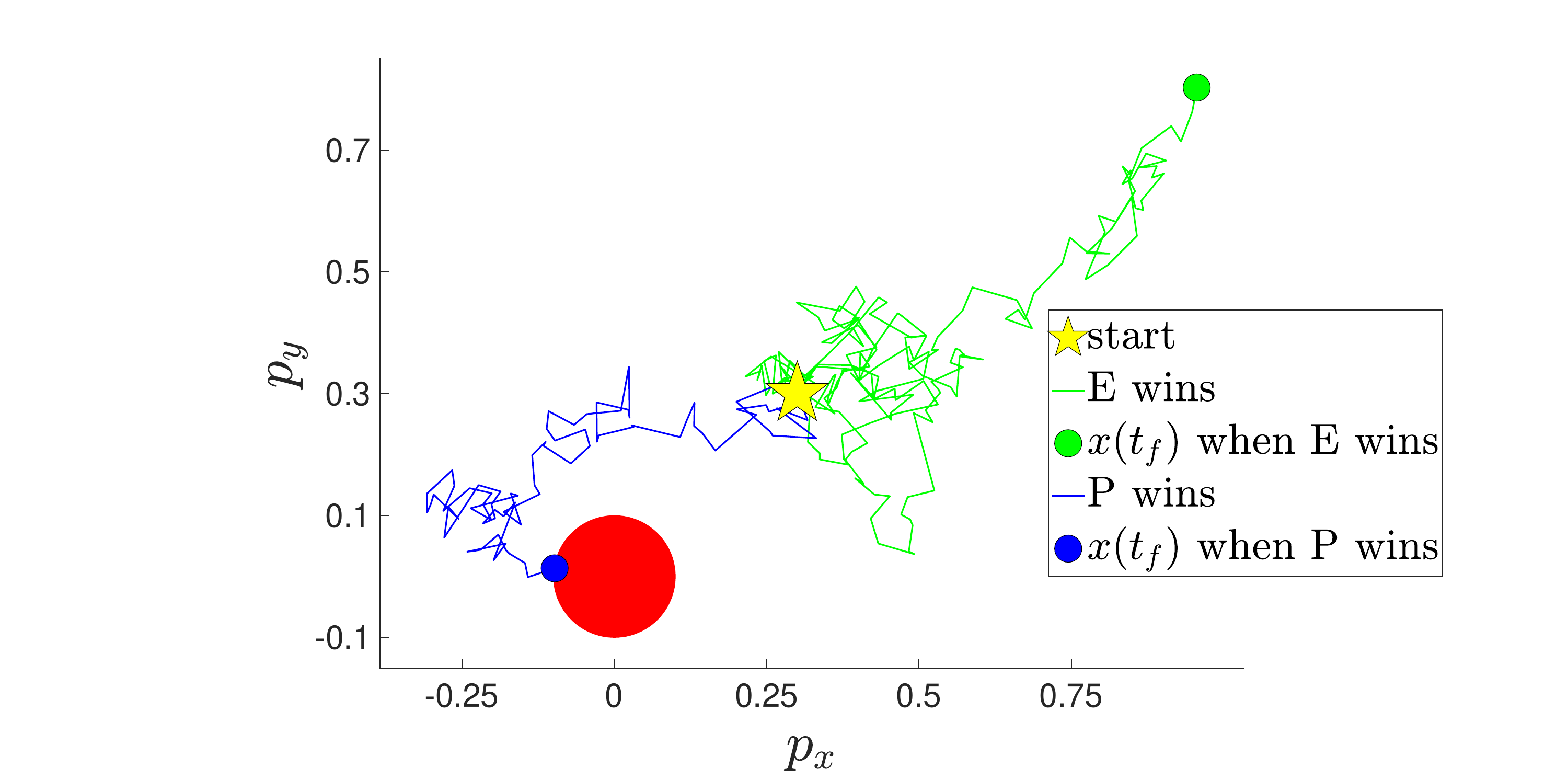} \\
       \end{tabular}
         \caption{Two sample trajectories of the relative position of the players in a pursuit-evasion game. The start position of the trajectories is shown by a yellow star. The red disc of radius $\rho = 0.1$, centered at the origin represents that the pursuer is within the distance $\rho$ of the evader. The green trajectory never enters the red disc in the horizon $[t_0, T]$, thus, it represents a case when the evader wins. The blue trajectory enters the red disc and thus represents a case when the pursuer wins.} 
         \label{Fig. EP}
 \end{figure}
 
 \begin{figure}
    \centering
\includegraphics[scale=0.25]{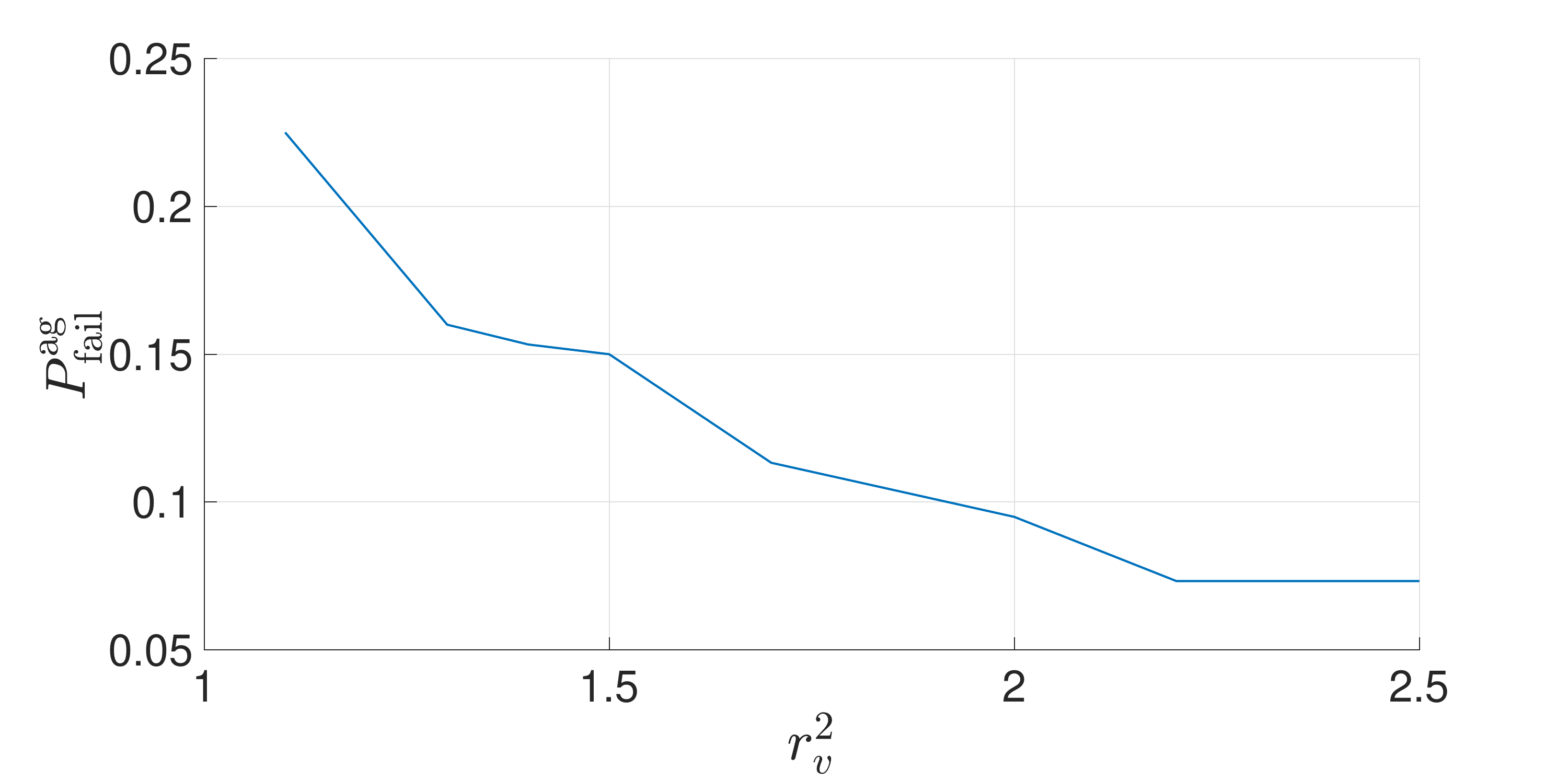} 
        \caption{Failure probabilities of the agent (i.e., evader) as a function of $r_v$, when the players follow the saddle-point policies ($u^*, v^*$).} 
        \label{Fig. pfail with rv2}
\end{figure}\par
The aim of the presented simulation studies is to validate the proposed theoretical formulation of the risk-minimizing zero-sum SDGs. Future work will emphasize scaling this framework to higher dimensional and more complex game dynamics.

\section{Publications}
\begin{itemize}
     \item \textbf{A. Patil}, Y. Zhou, D. Fridovich-Keil, T. Tanaka, ``Risk-Minimizing Two-Player Zero-Sum Stochastic Differential Game via Path Integral Control," \textit{2023 IEEE Conference on Decision and Control (CDC)}
\end{itemize}
\section{Future Work}
In the future, we plan to conduct sample complexity analysis for path integral control in order to investigate how the accuracy of Monte Carlo sampling affects the solution of SDGs. The central challenge in using the path integral framework is the particular requirement on the relationship between the cost function and the noise covariance. This requirement restricts the class of applicable system models and cost functions. In future work, we plan to find alternatives in order to get rid of this restrictive requirement (one such solution is provided in \cite{satoh2016iterative}). Another topic of future investigation could be chance-constrained stochastic games in which each player would aim to satisfy a hard bound on its failure probability.  

%% file: chapters/hierarchy.tex
\chapter[Task Hierarchical Control]{Task Hierarchical Control}
\label{Sec: task hierarchy}
\section{Motivation and Literature Review}\label{Sec: Introduction Task Hierarchical Control}
Robotic systems with a large number of degrees of freedom offer significant versatility; however, this also introduces redundancies. These systems are often used to accomplish multiple subtasks with varying levels of importance, allowing for the establishment of a task hierarchy. One of the most frequently applied methods to accomplish task hierarchical control is the null-space projection technique \cite{antonelli2008null, khatib1987unified, slotine1991general}. In the null-space projection, the top priority task is executed by employing all the capabilities of the system. The second priority task is then applied to the null space of the top priority task. In other words, the task on the second level is executed as well as possible without disturbing or interfering with the first level. The task on level three is then executed without disturbing the two higher-priority tasks, and so forth. The null-space projection technique is based on a hierarchical arrangement of the involved tasks and can be interpreted as instantaneous local optimization. \par

The null-space projection technique has been widely applied, such as in multi-robot team control \cite{sentis2009large}, whole-body behavior synthesis \cite{sentis2005synthesis}, and manipulator control \cite{dietrich2015overview}. A comprehensive overview and comparison of null-space projections is given in \cite{dietrich2015overview}. In the literature \cite{antonelli2008null, dietrich2015overview}, the individual controllers for the tasks in the hierarchy are designed using simple low-level controllers such as proportional-integral-derivative (PID) controllers. While these controllers are easy to design and are capable of generating control inputs in real time, there is no systematic way to optimize the overall performance of the hierarchy of controllers.  
Global optimization techniques such as dynamic programming on the other hand minimize some performance index across a whole trajectory. For example, to design optimal control policies for a linear system with a quadratic cost function, the linear quadratic regulator (LQR) is widely used. For nonlinear systems and cost functions, approaches such as iterative LQR (ILQR), and differential dynamic programming (DDP) can be utilized. However, even though global optimization solutions perform better than local optimization solutions, they are impractical for online feedback control, due to the heavy computational requirements. In this work, we employ the path integral control method, a stochastic optimal control framework that can be applied to nonlinear systems and enables the computation of optimal control inputs in real time through Monte Carlo simulations. \par

Based on the foundational work of \cite{kappen2005path, theodorou2010generalized}, the path integral method can be defined as a sampling-based algorithm to compute control input at each time step from a large number of Monte-Carlo simulations of the \emph{uncontrolled dynamics}. Unlike traditional optimal control methods, the path integral approach can directly deal with stochasticity and nonlinearity \cite{patil2022chance}, \cite{patil2023simulator, patil2023risk}. Moreover, unlike dynamic programming, it can evaluate control input without solving a high-dimensional Hamilton-Jacobi-Bellman partial differential equation \cite{patil2022chance}. The Monte Carlo simulations can also be highly parallelized, leveraging the GPU resources available on modern robotics platforms \cite{williams2017model}, making it particularly effective for real-time control applications.\par


This work integrates the path integral control approach with the null-space projection technique to overcome their individual limitations and enhance their respective strengths. We explain our idea via the following example:

\begin{example}

Consider a platoon of robots navigating in an obstacle-filled environment, tasked with three goals in descending order of importance:
\begin{enumerate}
    \item Avoiding collisions with obstacles
    \item Steering the platoon’s centroid toward a goal position
    \item Maintaining specific distances between the robots
\end{enumerate}
Designing an optimal controller using only the path integral method for each robot in the platoon would present scalability challenges due to the method's sampling-based nature. On the other hand, simple low-level controllers (such as PID), while computationally efficient, are difficult to tune manually for a better performance. We propose using local controllers for tasks 1 and 3 while applying the path integral controller to the more complex task 2. In this way, the path integral controller optimizes the centroid’s movement, while simpler tasks are handled by local controllers.
\end{example} 

In \cite{sentis2009large}, the authors address the issue of the task hierarchical controller falling into local minima by complementing the low-level controllers with the A* search algorithm. In graph-based methods like A*, the optimal trajectory is generated in the workspace, requiring an additional tracking controller to be designed separately, taking into account the system's dynamics. In contrast, the path integral controller computes the optimal policy directly, eliminating the need for such a decoupled approach. Furthermore, path integral controllers can adapt in real time making them particularly suitable for dynamic environments. Graph-based algorithms like A* are primarily designed for static environments and are less effective in dynamic scenarios, as they require re-planning the trajectory each time the environment changes. Although there are graph-based algorithms, such as RRTX, capable of handling dynamic environments \cite{otte2016rrtx}, these approaches tend to have higher computational costs compared to static planners and often require substantial memory resources as the number of samples increases.
\section{Contributions}
The contributions of this work are as follows:
\begin{enumerate}
    \item We introduce a new framework for hierarchical task control that combines the null-space projection technique with the path integral control method. This leverages Monte Carlo simulations for real-time computation of optimal control inputs, allowing for the seamless integration of simpler PID-like controllers with a more sophisticated optimal control technique. 
    \item Despite the wide applicability of the path integral approach, it has not been utilized for solving the task hierarchical control problem to the best of the authors' knowledge. This expands the applicability of path integral control to multi-task robotic systems, enabling a more robust handling of task prioritization.
    \item Our simulation studies demonstrate the effectiveness of the proposed approach, showing how it overcomes the limitations of the state-of-the-art methods by optimizing task performance. 
\end{enumerate}

\section{Preliminaries}\label{Sec: preliminaries}
Let $q\in\mathbb{R}^n$ be the configuration of a robot, where $n$ is the number of degrees-of-freedom (DOFs). We introduce $K$ task variables 
\begin{equation}\label{task function}
  \sigma_k = h_k(q)\in\mathbb{R}^{m_k}, \quad k\in \mathcal{K} 
\end{equation}
 where $m_k$ is the dimension of task $k$ and $\mathcal{K}=\{1, 2, ... , K\}$ denotes a set of the indices. The hierarchy is defined such that $\sigma_{1}$ is at the top priority and $\sigma_i$ is located higher in the priority order than $\sigma_j$ if $i<j $. The task variables $\sigma_k$ represent functional quantities (e.g., a cost or a potential function) as part of the desired actions, and $h_k$ is a differentiable nonlinear function. Our goal is to devise a policy for the robotic system that would accomplish the $K$ subtasks $\sigma_k, k\in \mathcal{K}$ in their descending order of importance. In the rest of the dissertation, for notational compactness, the functional dependency on $q$ is dropped whenever it is unambiguous. \par

Differentiating \eqref{task function} with respect to time, we get 
\begin{equation}\label{Jacobian}
    \dot{\sigma}_k = J_k(q)\dot{q}, \qquad J_k(q) = \frac {\partial h_k(q)}{\partial q}
\end{equation}
where $J_k(q)\in\mathbb{R}^{m_k\times n}$ is the Jacobian matrix and $\dot{q}$ represents the velocity of the robot in the configuration space. This velocity can be computed by inverting the mapping \eqref{Jacobian} \cite{siciliano1990kinematic}, \cite{antonelli2009stability}, \cite{antonelli2009prioritized}. However, in the case of a redundant system i.e., when $n>m_k$, the problem \eqref{Jacobian} admits infinite solutions. A common approach is to solve for the minimum-norm velocity, which leads to the least-squares solution:
\begin{equation*}
    \dot{q} = J^\dag_k(q)\dot{\sigma}_k
\end{equation*}
where $J_k^\dag(q) = J_k^\top(q)(J_k(q)J_k^\top(q))^{-1}$ is the right pseudo-inverse of the Jacobian matrix $J_k(q)$. In the following $J_k(q)$ is assumed to be non-singular, hence of full row rank.\par
Similar to \eqref{Jacobian}, the acceleration in the configuration space can be computed by further differentiating \eqref{Jacobian}
\begin{equation*}
    \ddot{\sigma}_k = J_k(q)\ddot{q} + \dot{J}_k(q)\dot{q}.
\end{equation*}
The minimum-norm solution for the acceleration $\ddot{q}$ is obtained as:
\begin{equation}\label{q ddot}
   \ddot{q} = J^\dag_k(q) \left(\ddot{\sigma}_k - \dot{J}_k(q)\dot{q} \right). 
\end{equation}
Equation \eqref{q ddot} provides a basic method to compute system acceleration in an open-loop style. To improve convergence, a feedback term is added to \eqref{q ddot}, as suggested by \cite{siciliano1990kinematic} and \cite{tsai1987strictly}, leading to the expanded form:
\begin{equation}\label{q ddot PD}
   \ddot{q} = J^\dag_k(q) \left(\left\{\ddot{\sigma}_{k,d} + K_{p,k}\widetilde{\sigma}_k + K_{d,k}\frac{d\widetilde{\sigma}_k}{dt}\right\} -\dot{J}_k(q)\dot{q} \right) 
\end{equation}
where $\widetilde{\sigma}_k = \sigma_{k,d} - \sigma_{k}$ denotes the error between the desired task trajectory $\sigma_{k,d}$ and the actual task trajectory $\sigma_{k}$ which can be computed from the system's current configurations using \eqref{task function}. For task $k$, the terms $K_{p,k}$ and $K_{d,k}$ are proportional and derivative gains, respectively, which shape the convergence of the error $\widetilde{\sigma}_k$. Equation \eqref{q ddot PD} is called a \textit{closed loop inverse kinematic} version of the equation \eqref{q ddot}. The controller having a similar structure is presented in \cite{siciliano1990kinematic}.\par
Note that in the above control input computation technique, the desired task trajectory $\sigma_{k,d}$ is often chosen manually. Due to the complexity of the architecture, it is often difficult to select an appropriate desired trajectory $\sigma_{k,d}$ and low-level controller gains $K_{p,k}$, $K_{d,k}$ to optimize the overall system performance. Moreover, the above method is only a local optimization technique as opposed to global optimization techniques which minimize some performance index across a whole trajectory and typically offer better solutions compared to local optimization approaches. Table \ref{tab:notation hierarchy} represents the mathematical notations frequently used in this chapter. 

\begin{table}
\begin{center}
\begin{tabular}{||c | c || c | c||} 
 \hline
 \textbf{Notation} & \textbf{Description} & \textbf{Notation} & \textbf{Description} \\ [0.5ex] 
 \hline\hline
 $q$ & configuration & $n$ & number of DOFs \\ 
 \hline
 $ \sigma_k$ & task variable $k$ & $m_k$ & dimension of task $k$ \\ 
 \hline
 $J_k(q)$ & Jacobian matrix $k$ & $K_{p,k}$ & proportional gain for task $k$ \\ 
 \hline
 $K_{d,k}$ & derivative gain for task $k$ & $x(t)$ & state of the system \\ 
 \hline
 $u(t)$ & control input & $N_k(q)$ & null-space projector for task $k$ \\ 
  \hline
 $w(t)$ & standard Brownian motion & $\hat{s}$ & diffusion coefficient \\
 \hline
 $L(x(t))$ & running state cost & $\phi(x(T))$ & a terminal cost \\
 \hline
 $I$ & number of robots & $q_g$ & goal position \\
 \hline
 $\lambda$ & PDE linearizing constant & $\Delta t$ & discretized time-step \\
 [1ex] 
 \hline
\end{tabular}
\caption{Table of frequently used mathematical notation in Chapter \ref{Sec: task hierarchy}}
\label{tab:notation hierarchy}
\end{center}
\end{table}

\section{Null-Space Projection}\label{Sec: task hierarchy-1}
Consider a robotic system with a control-affine dynamics:
 \begin{equation}\label{double integrator}
     \dot{x}(t) = f(x) + G(x)u(t)
 \end{equation}
where $x(t)$ is the state of the system\footnote{Often the case, the configuration of a system $q$ is part of its states; see examples in Section \ref{Sec: simulations}.}, $u(t)$ is the control input, $f(x)$ is a drift term, and $G(x)$ is the control coefficient. In the rest of the dissertation, for notational compactness, the functional dependencies on $x$, and $t$ are dropped whenever it is unambiguous. In the null-space projection technique, the control input $u_2$ corresponding to the second-priority task is projected onto the null space of the primary task using the formula:
\begin{equation}\label{a2_proj}
    u_2' = N_2^{}(q)u_2
\end{equation}
where $u_2'$ is the projected control input that does not interfere with the primary task. The null-space projector $N_2^{}(q)$ is obtained by evaluating
\begin{equation*}
    N_2^{}(q) = I - J_1^\dag(q) J_1(q)
\end{equation*}
where $J_1^\dag(q)$ is the right pseudo-inverse of the primary task’s Jacobian $J_1$, and $I$ is an identity matrix of suitable dimensions. Analogous to \eqref{a2_proj}, for the remaining tasks in the hierarchy ($2<k\leq K$), the control inputs are projected as $u_k' = N_k^{}(q)u_k$, with the null-space projectors recursively computed as:
\begin{align*}
    N_k^{}(q) & = N_{k-1}^{}(q)\left(I-J_{k-1}^\dag(q) J_{k-1}(q)\right),\quad 2\leq k\leq K
\end{align*}
and $N_1(q) = I$. Here $I$ is the identity matrix with suitable dimensions. Each task input is computed as if it were acting alone; then before adding its contribution to the overall system control input, a lower-priority task is projected onto the null space of the immediately higher-priority task so as to remove those control input components that would conflict with it. This technique guarantees that lower-priority objectives are constrained and therefore do not interfere with higher-priority objectives. As a result, the high-priority task is always achieved, and the lower ones are met only if they do not conflict with the task of higher priority. The final control input can be formulated by adding up the primary task control input and all the projected control inputs: 
\begin{align}\label{final acc}
    u = u_1 + \sum_{k=2}^K u_k' =  \sum_{k=1}^K N_k^{}(q)u_k. 
\end{align}
Plugging \eqref{final acc} into \eqref{double integrator}, we get
\begin{equation}\label{double integrator 2}
   \dot{x}(t) = f(x) + G(x)\sum_{k\in \mathcal{K}} N_k^{}(q)u_k.\end{equation}


A natural question arises of how many tasks can be handled simultaneously using this approach. Let us suppose that the primary task, of dimension $m_1$ is fulfilled by $n$ DOFs robotic system. The null space of its Jacobian (of full row rank) is a space of dimensions $n-m_1$. Supposing the secondary task of dimension $m_2$ does not conflict with the primary task (meaning that the secondary task acts in the null space of the primary task), the null space of their combination has dimension $n-m_1-m_2$. Choosing the tasks in a way they are not conflicting, it is useful to add tasks until $\sum m_k = n$. Thus, once all the degrees of freedom of the system are covered, it is useless to add successive tasks of lower priority since they will be projected onto an empty space (thus giving always a null contribution to the system control input). In case of conflicting tasks, it is not possible to make any generic assumption regarding the useful number of tasks but, case by case, the intersection among null spaces should be analyzed.

\subsection{Integration of Null Space Projection and Path Integral Control}
\begin{figure*}[h]
    \centering
      \begin{tabular}{c c}
     \!\!\!\!\!\!\!\!\!\!\!\!\!\!\!\!\!\!\includegraphics[scale=0.33]{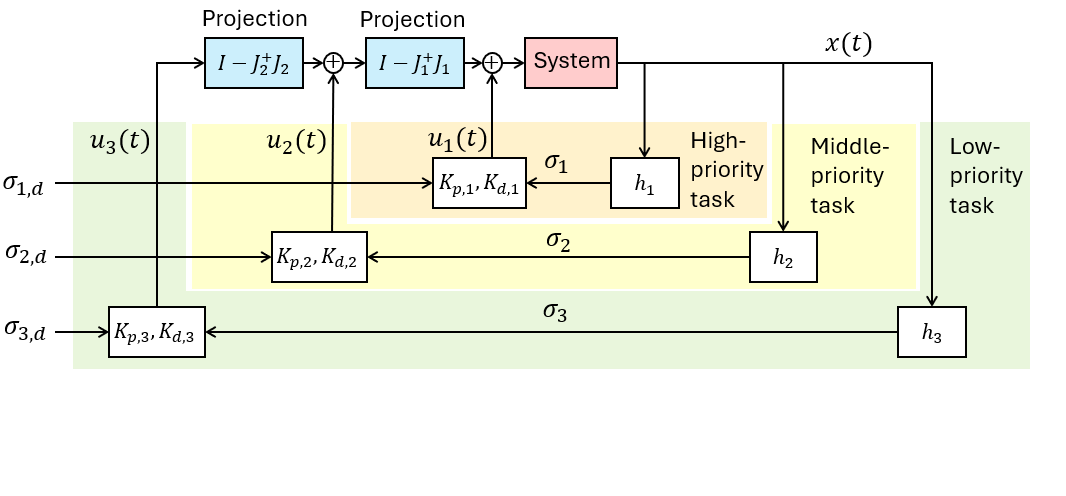} &\!\!\!\!\!\!\includegraphics[scale=0.33]{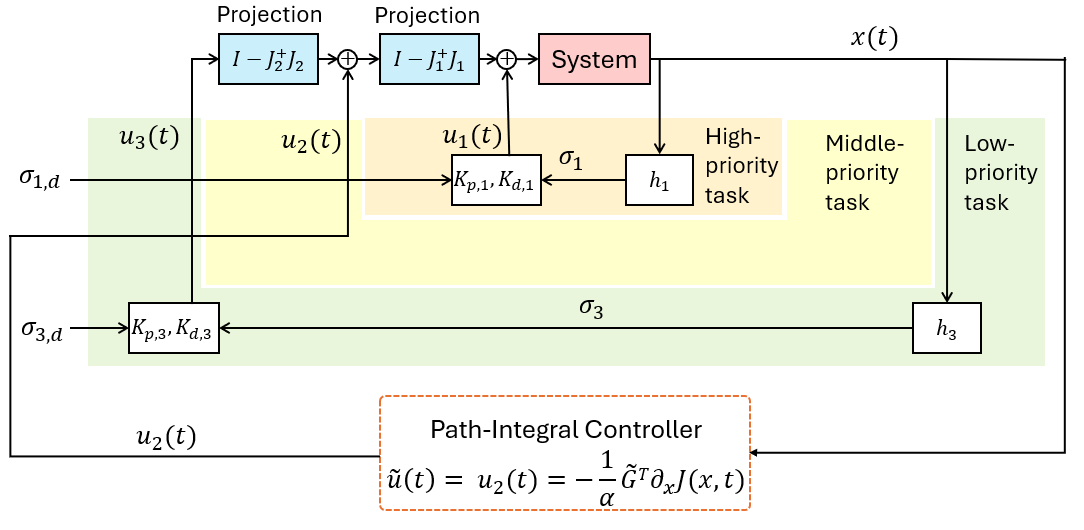} \\
     (a) Conventional task hierarchical control &(b) Integration of null-space projection and \\
     &  path integral control (proposed architecture)
      \end{tabular}
        \caption{Comparison of the conventional task hierarchical control approach and our proposed control approach, which integrates null-space projection with path integral control.} 
        \label{Fig. comparison}
\end{figure*}
In this section, we explain our proposed architecture. Consider a scenario where we need to accomplish $K$ tasks, however, only the $k$-th task is handled using a path integral controller. Here, we assume only one task is accomplished using the path integral controller. A similar formulation can be extended to cases where multiple tasks are controlled using path integral controllers. The control inputs for the other tasks $u_i, i\neq k$ are designed using PD controllers \eqref{q ddot PD}. Using \eqref{double integrator 2}, the system dynamics can be written as 
\begin{equation}\label{double integrator 3}
  \dot{x}(t) = f(x) + \!\!\!\sum_{i\in\mathcal{K}, i\neq k} \!\!\!\!G(x)N_iu_i(t) + G(x)N_ku_k(t).  
\end{equation}
Defining $\widetilde{G}(x) = G(x)N_k$, $\widetilde{u} = u_k$ and
\begin{align*}
    \widetilde{f}(x)& = f(x)+\sum_{i\in\mathcal{K}, i\neq k}\!\!\!\!G(x)N_iu_i(t),  
\end{align*}
we can rewrite \eqref{double integrator 3} as $d{x} = \widetilde{f}(x)dt + \widetilde{G}(x) \widetilde{u}\,dt.$ Note that the terms $\widetilde{f}(x)$ and $\widetilde{G}(x)$ are computed using the PD-like low-level controllers \eqref{q ddot PD} derived in Section \ref{Sec: preliminaries} and the null-space projection explained in Section \ref{Sec: task hierarchy-1}. Next, we derive a path integral controller for the control input $\widetilde{u}$. \par
First, we perturb $\widetilde{u}$ by adding a stochastic term: $\widetilde{u}\,dt \rightarrow \widetilde{u}\,dt + \hat{s}\,dw(t)$, where $w(t)$ is a standard Brownian motion and $\hat{s}\in \mathbb{R}$ is the diffusion coefficient. This perturbation allows the system to explore the policy space and discover a direction that minimizes the cost. The perturbed system dynamics are then given by 
\begin{equation}\label{SDE hierar}
     d{x} = \widetilde{f}(x)dt + \widetilde{G}(x) \left(\widetilde{u}\,dt + \hat{s}\,dw(t)\right).
\end{equation}
Now we formulate the path integral control problem with the cost function: 
\begin{equation}\label{PI cost}
    \mathbb{E}_Q\left[\phi(x(T)) + \int_0^T  \left(L(x(t))+\frac{\alpha}{2}\|\widetilde{u}\|^2\right)dt\right]
\end{equation}
for some $\alpha>0$. The expectation $\mathbb{E}_Q$ is taken with respect to the dynamics \eqref{SDE hierar}. The cost function has a quadratic control cost, an arbitrary state-dependent running cost $L(x(t))$, and a terminal cost $\phi(x(T))$. Consider the example of a platoon of robots illustrated in Section \ref{Sec: Introduction Task Hierarchical Control}. Suppose task 2 (steering the platoon’s centroid toward a goal position) is controlled by the path integral controller. Let there be $I$ robots, with the configuration of each robot be denoted by $q_i, 1\leq i\leq I$, and the goal position be $q_g$. In this case, the running cost $L(x(t))$ could be $\|\sum_{i=1}^{I}\frac{1}{I}q_i(t)-q_g\|^2$, and the terminal cost $\phi(x(T))$ could be $\|\sum_{i=1}^{I}\frac{1}{I}q_i(T)-q_g\|^2$.\par
We now formally state our problem:
\begin{problem}\label{PI problem}
   \begin{align*}
       \min_{\widetilde{u}}\; & \mathbb{E}_{Q}\left[\phi(x(T))+\int_0^T \left(L(x(t))+\frac{\alpha}{2}\|\widetilde{u}\|^2\right)dt\right]\nonumber\\
       \emph{s.t.} \;\; &  d{x} = \widetilde{f}(x)dt + \widetilde{G}(x) \left(\widetilde{u}\,dt + \hat{s}\,dw(t)\right).
   \end{align*}
\end{problem}
It is well-known that the value function $J(x,t)$ of Problem \ref{PI problem} satisfies the Hamilton-Jacobi-Bellman (HJB) partial differential equation (PDE) \cite{williams2017model}: 
 \begin{equation}\label{HJB PDE hierar}
  \begin{aligned}
         -\partial_tJ\!=\!-\frac{1}{2\alpha}\!\left(\partial_xJ\right)^\top\!\!\widetilde{G}\widetilde{G}^\top\!\partial_xJ\!+\!L
         +\!\!\widetilde{f}^\top\!\partial_xJ+\frac{\hat{s}^2}{2}\text{Tr}\left(\widetilde{G}\widetilde{G}^\top\partial^2_xJ\right),
          \end{aligned}
    \end{equation}
    with the boundary condition $J(x(T), T)=\phi(x(T))$. The optimal control is expressed in terms of the solution to this PDE as follows:
    \begin{equation*}
        \widetilde{u}(x,t) = -\frac{1}{\alpha}{\widetilde{G}}^\top\partial_xJ\left(x, t\right).
    \end{equation*}
Therefore, in order to compute the optimal controls, we need to solve this backward-in-time PDE \eqref{HJB PDE hierar}. Unfortunately, classical methods for solving partial differential equations (such as the finite difference method) of this nature suffer from the curse of dimensionality and are intractable for systems with more than a few state variables. The path integral control framework provides an alternative approach by transforming the HJB PDE into a path integral. This transformation allows us to approximate the solution using Monte Carlo simulations of the system's stochastic dynamics. We use the Feynman-Kac formula, which relates PDEs to path integrals \cite{williams2017model}. First, we define a constant $\lambda$ such that:
    \begin{equation}\label{lambda hierar}
        \hat{s}^2 = \frac{\lambda}{\alpha}.
    \end{equation}
Next, using this constant $\lambda$, we introduce the following transformed value function $\xi(x,t)$:

\begin{equation}\label{exp transformation hierar}
 J(x,t) = -\lambda\,\text{log}\left(\xi\left(x,t\right)\right).
\end{equation}
The transformation (\ref{exp transformation hierar}) allows us to write the HJB PDE \eqref{HJB PDE hierar} in terms of $\xi(x,t)$ as
\begin{equation}\label{transformed HJB PDE hierar}
    \begin{aligned}
    \!\partial_t\xi\!=\frac{L\xi}{\lambda}-\!\frac{\hat{s}^2}{2}\text{Tr}\!\left(\widetilde{G}\widetilde{G}^\top\!\partial^2_x\xi\right)\!+\!\frac{\hat{s}^2}{2\xi}\!\left(\partial_x\xi\right)^{\!T}\!\!\widetilde{G}\widetilde{G}^\top\!\partial_x\xi-\!\frac{\lambda}{2\alpha\xi}\!\left(\partial_x\xi\right)^\top\!\!\widetilde{G}\widetilde{G}^\top\!\partial_x\xi\!-\!\widetilde{f}^\top\partial_x\xi,
\end{aligned}
\end{equation}
with the boundary condition $\xi(x(T),T) =\exp\left(-\frac{\phi(x(T))}{\lambda}\right)$. Using \eqref{lambda hierar} in equation \eqref{transformed HJB PDE hierar}, we can rewrite the PDE \eqref{HJB PDE hierar} as a linear PDE in terms of $\xi(x,t)$ as:
\begin{equation}\label{linear PDE}
    \partial_t\xi\!=\!\frac{L\xi}{\lambda}\!-\!\widetilde{f}^\top\partial_x\xi-\frac{\hat{s}^2}{2}\text{Tr}\left(\widetilde{G}\widetilde{G}^\top\partial^2_x\xi\right)
\end{equation}
with the boundary condition $\xi(x(T),T) =\exp\left(-\frac{\phi(x(T))}{\lambda}\right)$.
    This particular PDE is known as the \textit{backward Chapman–Kolmogorov PDE}. Now we find the solution of the linearized PDE \eqref{linear PDE} using the Feynman-Kac lemma.
\begin{theorem}[Feynman-Kac lemma]\label{theorem: Feynman-Kac}
    The solution to the linear PDE \eqref{linear PDE} exists. Moreover, the solution is unique in the sense that $\xi$ solving \eqref{linear PDE} is given by \begin{equation}\label{xi hierar}
       \begin{aligned}
  \!\!\!\!\!\!\xi\!\left(x,t\right) & \!= \!\mathbb{E}_{P}\!\! \left[\exp\!\left(\!\!-\frac{1}{\lambda}S(x,t)\!\!\right)\!\right]\\
  \end{aligned}
    \end{equation}
    where the expectation $\mathbb{E}_P$ is taken with respect to the uncontrolled dynamics of the system \eqref{SDE hierar} (i.e., equation \eqref{SDE hierar} with $\widetilde{u}=0$) starting at $x,t$. $S(x,t)$ is the cost to go of the state-dependent cost of a trajectory given by  
    \begin{equation*}
        S(x,t) = {\phi\left({{x}}({{T}})\right)}\!+\!\!\int_{t}^{{{T}}}\!\!\!\!L\!\left({{x}}(t)\right)\!dt .
    \end{equation*}
\end{theorem} 
\begin{proof}
    The proof follows from \cite[Theorem 9.1.1]{oksendal2013stochastic}.
\end{proof}
We now obtain the expression for the optimal control policy for Problem \ref{PI problem} via the following theorem.

\begin{theorem}
The optimal solution of Problem \ref{PI problem} exists, is unique and is given by
\begin{equation}\label{path integral control hierar}
 \widetilde{u}^*(x,t)dt=\mathcal{G}\left(x\right)\frac{\mathbb{E}_{P}\left[\exp{\left(-\frac{1}{\lambda}S\right)}\hat{s}\,\widetilde{G}\left(x\right)d{w}(t)\right]}{\mathbb{E}_{P}\left[\exp{\left(-\frac{1}{\lambda}S\right)}\right]}, 
\end{equation}
where the matrix $\mathcal{G}\left(x\right)$ is defined as $\mathcal{{G}}\!\left(x\right)\!={\widetilde{G}}^{\top}\!(x)\!\left(\widetilde{G}(x){\widetilde{G}}^{\top}\!(x)\right)^{-1}. $
\end{theorem}
\begin{proof}
   The existence and uniqueness of the optimal solution follow from the existence and uniqueness of the linear PDE \eqref{linear PDE} (Theorem \ref{theorem: Feynman-Kac}). The solution \eqref{path integral control hierar} can be computed by taking the gradient of (\ref{xi hierar}) with respect to $x$ \cite{theodorou2010generalized, williams2017model}. 
\end{proof}

To evaluate expectations in (\ref{xi hierar}) and (\ref{path integral control hierar}) numerically, we discretize the uncontrolled dynamics and use Monte Carlo sampling \cite{williams2017model}. After discretizing the uncontrolled dynamics, we get $ x_{t+1} = x_t + \widetilde{f}(x_t)\Delta t + \hat{s}\,\widetilde{G}(x_t)\epsilon\sqrt{\Delta t}$, where the term $\epsilon$ is a time-varying vector of standard normal Gaussian random variables and $\Delta t$ is the step size. The term $S(x,t)$ is approximately given as $ S(x,t) \approx {\phi\left({{x_T}}\right)}\!+\!\!\sum_{i=1}^{{{\tau}}}L\!\left({{x_t}}\right)\!\Delta t $, where $\tau=\frac{T-t}{\Delta t}$. Now suppose we generate $M$ samples of trajectories. Then we can approximate \eqref{path integral control hierar} as 
\begin{equation}\label{u after discretization}
   \!\!\! \widetilde{u}^*(x,t)\approx\mathcal{G}\left(x_t\right)\frac{\sum_{j=1}^{M}\left[\exp{\left(-\frac{1}{\lambda}S(x,t)\right)}\hat{s}\,\widetilde{G}\left(x_t\right)\epsilon\right]}{\sum_{j=1}^{M}\left[\exp{\left(-\frac{1}{\lambda}S(x,t)\right)}\sqrt{\Delta t}\right]}.  
\end{equation}
Fig. \ref{Fig. comparison} shows the comparison of the conventional task hierarchical control approach and our proposed control approach, which integrates null-space projection with path integral control.
\begin{remark}
    Note that as the number of samples $M\rightarrow\infty$, \eqref{u after discretization} $\rightarrow$ \eqref{path integral control hierar} i.e., path integral controller yields a globally optimal policy as $M\rightarrow\infty$.
\end{remark}
\begin{remark}
Increasing the diffusion coefficient $\hat{s}$ allows the system to explore a wider range of possibilities, but it also introduces greater noise into the control input. Therefore, $\hat{s}$ should be carefully selected to strike an appropriate balance between exploration and control noise.  
\end{remark}
The control input for each task in the hierarchy can be computed independently of the others. This allows for parallelization of the task hierarchy algorithm using GPUs, ensuring that the computational complexity remains manageable as additional tasks are introduced, provided a sufficient number of GPUs are available. As the number of agents increases, the dimensionality of the control inputs also grows. According to \cite{patil2024discrete}, in the case of the discrete-time path integral controller, the required sample size exhibits a \textit{logarithmic} dependence on the dimension of the control input. However, a formal analysis of sample complexity for the continuous-time path integral controller remains an open area for future research.

\section{Simulation Results}\label{Sec: simulations}
In this section, we present the simulation results for our proposed control framework.
\subsection{Single Agent Example}
Suppose a unicycle is navigating in a 2D space in the presence of an obstacle. The states of the unicycle model ${x}=[{p}_x \; {p}_y \; {s}\;\; \theta ]^\top$ consist of its $x-y$ position $p \coloneqq [{p}_x \; {p}_y]^\top$, speed ${s}$ and heading angle $\theta\in[0, 2\pi]$. The configuration of the system $q\coloneqq[{p}_x \; {p}_y \; \theta ]^\top$ consists of its position $p$ and the heading angle $\theta$. The system dynamics are given by the following equation:
\begin{equation}\label{unicycle}
    \begin{bmatrix}
        \dot{p}_x(t)\\ \dot{p}_y(t)\\ \dot{s}(t)\\ \dot{\theta}(t)
    \end{bmatrix} =  \begin{bmatrix} s(t)\cos (\theta(t)) \\ s(t)\sin (\theta(t))\\ 0\\ 0 \end{bmatrix} + \begin{bmatrix}
        0 & 0\\
        0 & 0\\
        1 & 0\\
        0 & 1
    \end{bmatrix} \begin{bmatrix} a(t)\\ \omega(t)  \end{bmatrix}.
\end{equation}
The control input $u\coloneqq[ a \; \omega ]^\top$ consists of acceleration $a$ and angular speed $\omega$. The simulation is set with $T=10$ seconds and $\Delta t = 0.01$ seconds. This unicycle system has to accomplish the following two tasks in descending order of importance: 1) obstacle avoidance and 2) move-to-goal. We design a proportional-derivative (PD) controller for the obstacle avoidance task and a path integral controller for the move-to-goal task.
\subsubsection{Obstacle-avoidance} 
Suppose the obstacle is placed at $c=[
    c_x \; c_y]^\top$ having radius $r$. 
In the presence of an obstacle in the advancing direction, the robot aims to keep it at a safe distance from the obstacle. The obstacle avoidance task becomes active when the robot is within a certain threshold distance and moving toward the obstacle. The task variable for obstacle avoidance $\sigma_1$ is defined as:
\begin{equation}\label{sigma1}
    \sigma_{1} =  \|p-c\|, 
\end{equation} 
and the desired value of the task variable be $\sigma_{1,d} =  r$. Differentiating \eqref{sigma1} with respect to $t$, we get 
\begin{equation}\label{sigma1_dot}
    \dot{\sigma}_1 = J_1v
\end{equation}
\begin{align*}
    J_{1} = \begin{bmatrix}
        \frac{p_{x}-c_x}{\|p-c\|} & \frac{p_{y}-c_y}{\|p-c\|} & 0
    \end{bmatrix},\quad v = \begin{bmatrix}
        s\cos\theta \\ s\sin\theta \\ \omega
    \end{bmatrix}.
\end{align*}
In order to devise the control input for task 1 $u_1\coloneqq[ a_1 \; \omega_1 ]^\top$, we further differentiate \eqref{sigma1_dot} and get
\begin{align}\label{sigma1_ddot}
    \ddot{\sigma}_1 = \delta_1 + \Lambda_1\begin{bmatrix}
        a_1 \\ \omega_1
    \end{bmatrix},
\end{align}
 where $\delta_1$ and $\Lambda_1$ are defined as:
\begin{equation}\label{delta1}
\delta_1 = \frac{\left(p_y - c_y\right)^2 s^2 \cos^2\theta + \left(p_x - c_x\right)^2 s^2 \sin^2\theta} {\|p-c\|^3},    
\end{equation}
\begin{equation}\label{lambda1}
   \Lambda_1 = \begin{bmatrix}
        \frac{p_x-c_x}{\|p-c\|}\cos\theta + \frac{p_y-c_y}{\|p-c\|}\sin\theta \\ \frac{p_y-c_y}{\|p-c\|}s\cos\theta + \frac{p_x-c_x}{\|p-c\|}s\sin\theta
    \end{bmatrix}^\top.
\end{equation}
Adding the feedback term to \eqref{sigma1_ddot} similar to \eqref{q ddot PD}, we get
\begin{equation*}
    u_1 = \begin{bmatrix}
        a_1 \\ \omega_1
    \end{bmatrix} = \Lambda_1^\dagger\left(\ddot{\sigma}_{1,d} + K_{p,1} \widetilde{\sigma}_{1} + K_{d,1} \frac{d\widetilde{\sigma}_1}{dt} - \delta_1\right)
\end{equation*}
where $\widetilde{\sigma}_1 = \sigma_{1,d} - \sigma_{1} $, and $K_{p,1}$ and $K_{d,1}$ are the proportional and derivative gains respectively for task 1.
\subsubsection{Move-to-goal}
In this task, the robot must reach the goal position  $p_g \coloneqq [
    p_{g,x} \; p_{g, y}]^\top$. The task variable $\sigma_2$ is given as
\begin{equation}\label{sigma2}
    \sigma_2 = p = \begin{bmatrix}
        p_x \\p_y
    \end{bmatrix}
\end{equation}
and the desired value of the task variable will be $\sigma_{2,d} = p_g$. \par
First, we design the control input $u_2 = [
    a_2 \; \omega_2]^\top$ using a PD controller similar to task 1. Differentiating \eqref{sigma2} with respect to time, we get 
\begin{equation}\label{sigma2_dot}
    \dot{\sigma}_2 = J_2v,\quad  J_{2} = \begin{bmatrix}
        1 & 0 & 0\\
        0 & 1 & 0
    \end{bmatrix},\quad v = \begin{bmatrix}
        s\cos\theta \\ s\sin\theta\\ \omega
    \end{bmatrix}.
\end{equation}
Differentiating \eqref{sigma2_dot} further we get  
\begin{align}\label{sigma2_ddot}
    \ddot{\sigma}_2 = \Lambda_2\begin{bmatrix}
        a_2 \\ \omega_2
    \end{bmatrix}, \quad  \Lambda_2 = \begin{bmatrix}
        \cos\theta & -s\sin\theta\\
        \sin\theta & s\cos\theta
    \end{bmatrix}.
\end{align}
Adding the feedback term to \eqref{sigma2_ddot} similar to \eqref{q ddot PD}, we get
\begin{equation*}
    u_2 = \begin{bmatrix}
        a_2 \\ \omega_2
    \end{bmatrix} = \Lambda_2^\dagger\left(\ddot{\sigma}_{2,d} + K_{p,2} \widetilde{\sigma}_{2} + K_{d,2} \frac{d\widetilde{\sigma}_2}{dt}\right)
\end{equation*}
where $\widetilde{\sigma}_2 = \sigma_{2,d} - \sigma_{2} $, and $K_{p,2}$ and $K_{d,2}$ are the proportional and derivative gains respectively for task 2. 

\begin{figure}[h]
    \centering
      \begin{tabular}{c c}
    \includegraphics[scale=0.45]{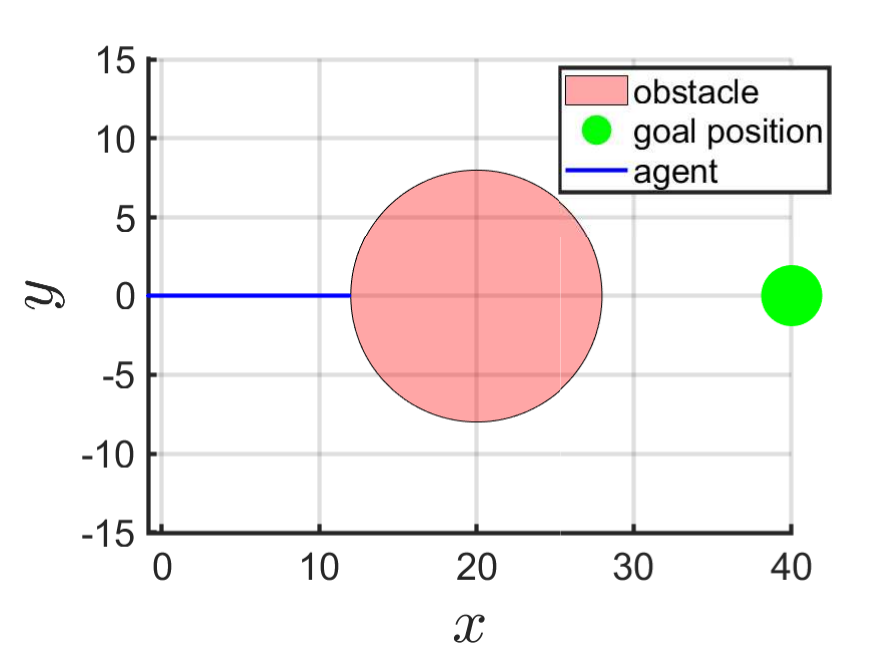} &\includegraphics[scale=0.45]{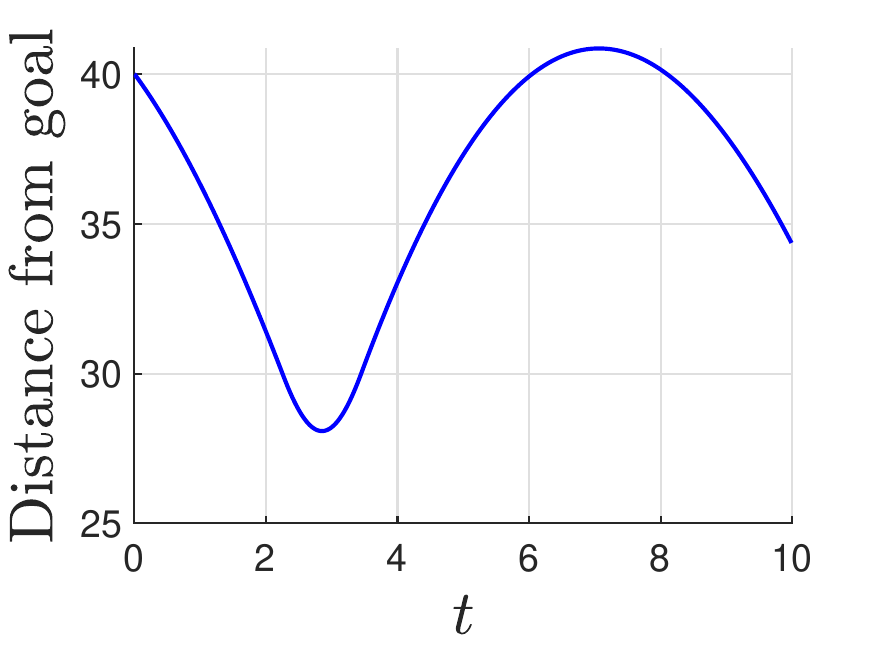} \\
     (a) Path followed without   & (b) Distance from the \\ 
     path integral controller & goal over time
      \end{tabular}
        \caption{Results of single-agent example without the path integral controller} 
        \label{Fig. single agent}
\end{figure}

Next, we design the control input $u_2$ via the path integral controller. The perturbed dynamics are given by \eqref{SDE hierar} with the diffusion coefficient $\hat{s}=0.1$. We formulate the cost function \eqref{PI cost} with 
\begin{equation*}
    L(x(t)) =  0.07\|{p(t) - p_g }\|,\;
    \phi(x(T)) = L(x(T)), \; \alpha =10.
\end{equation*}
The number of Monte Carlo samples has been set to $10^4$. \par
Lastly, the control input $u_2$ is projected onto the null space of task 1 and we get the final control input as $ u = u_1 + (I - \Lambda_1^\dagger\Lambda_1)u_2$ where $I$ is the identity matrix of suitable dimensions. 
\begin{figure}[h]
    \centering
      \begin{tabular}{c c}
     \includegraphics[scale=0.45]{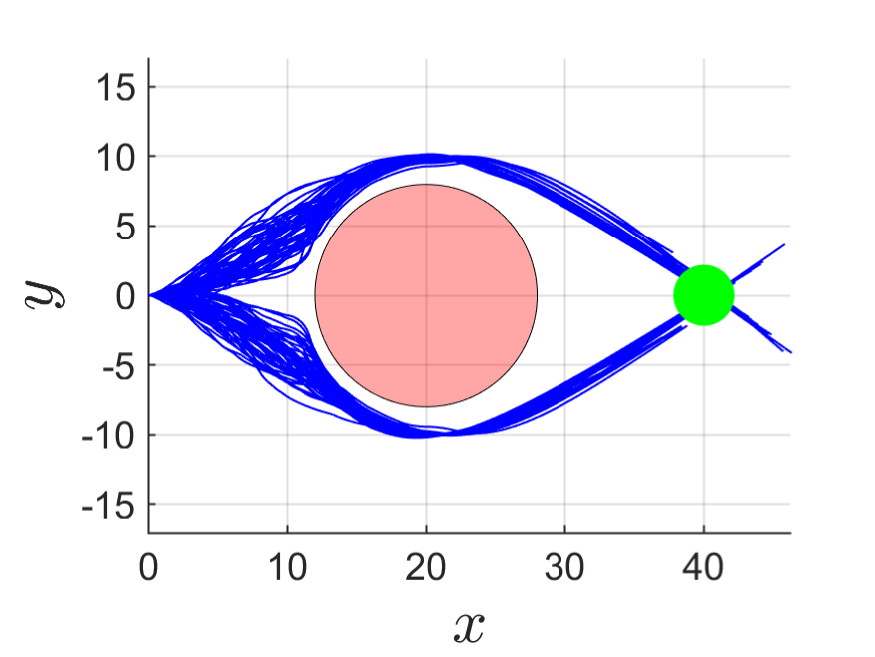} &\includegraphics[scale=0.45]{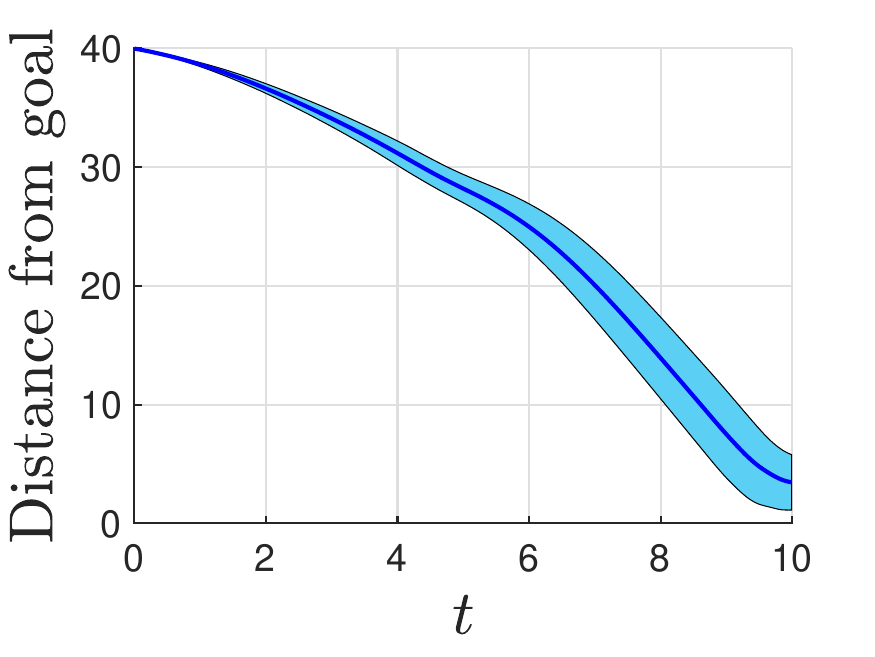} \\
     (a) Paths followed with   & (b) Mean distance from goal  \\
     path integral controller & $\pm$ standard deviation
      \end{tabular}
        \caption{Results of single-agent example using the path integral controller} 
        \label{Fig. single agent PI}
\end{figure}\par
Fig. \ref{Fig. single agent}(a) represents the results of the hierarchical controller when both tasks are controlled using PD controllers. As we can see, the robot initially moves towards the goal however, when it comes in the vicinity of the obstacle it drives itself away from the obstacle, as the obstacle avoidance task has a higher priority. The PD controllers get stuck in a local minima, the robot keeps oscillating as shown in Fig. \ref{Fig. single agent}(b) and never crosses the obstacle to reach the goal.\par
In Fig. \ref{Fig. single agent PI}, we show the results of the hierarchical controller when the first task is designed using a PD controller and the second task is designed using a path integral controller. Note that the path integral controller is designed by perturbing the control input $\widetilde{u}$ via a Brownian motion \eqref{SDE hierar}. Therefore, the outcome of the path integral controller is stochastic. In Fig. \ref{Fig. single agent PI}(a), we plot 100 trajectories obtained by the devised hierarchical controller. We can see that all the trajectories successfully avoid the obstacle by going around it and reaching the goal without getting stuck in local minimas. Fig. \ref{Fig. single agent PI}(b) shows the mean distance to the goal over time, along with the standard deviation across the 100 trajectories. The solid blue line represents the mean over the 100 trajectories and the sky-blue shaded area represents one standard deviation from the mean. 

\subsection{Two Agents Example}
Suppose two unicycles are navigating in a 2D space in the presence of an obstacle. The unicycle dynamics follow the form in \eqref{unicycle} where the state for each unicycle is defined as ${x}^{(i)}=[{p}_{x}^{(i)} \; {p}_{y}^{(i)} \; {s}^{(i)}\;\; \theta^{(i)} ]^\top$, the configurations as $q^{(i)}=[{p}_{x}^{(i)} \; {p}_{y}^{(i)} \; \theta^{(i)} ]^\top$, and the control inputs as $u^{(i)}=[ a^{(i)} \; \omega^{(i)} ]^\top$ for $i=\{1,2\}$. The simulation is set with $T=10$ seconds and $\Delta t = 0.1$ seconds. This two-unicycle system has to accomplish the following three tasks in descending order of importance: 1) obstacle avoidance, 2) steering the centroid of the robots toward a goal position and 3) maintaining a specific distance between the two unicycles. We design a proportional-derivative (PD) controller for tasks 1) and 3), and a path integral controller for task 2).
\subsubsection{Obstacle-avoidance} 
Suppose the obstacle is placed at $c=[
    c_x \; c_y]^\top$ having radius $r$.
Similar to the single agent case, the obstacle avoidance task becomes active when any of the robots is within a certain threshold distance and moving toward the obstacle. Let $\sigma_{1}^{(i)}$ represent the task variable for the robot $i$ and $\sigma_{1,d}$ represent the desired task value for both the robots. $\sigma_{1}^{(i)}$ and $\sigma_{1,d}$ can be written as 
\begin{align}\label{sigma1 2agent}
    \sigma_{1}^{(i)} =  \|p^{(i)}-c\|, \quad i=\{1,2\},\quad     \sigma_{1,d} =  r.
\end{align} 
Similar to the single agent case, we differentiate \eqref{sigma1 2agent} twice with respect to time for both robots and get 
\begin{align}\label{sigma1_ddot 2agent}
    \ddot{\sigma}_{1}^{(i)} = \delta_{1}^{(i)} + \Lambda_{1}^{(i)}\begin{bmatrix}
        a_1^{(i)} \\ \omega_1^{(i)}
    \end{bmatrix}
\end{align}
where $\delta_1^{(i)}$ and $\Lambda_1^{(i)}$ can be defined similar to \eqref{delta1}, \eqref{lambda1}:



Finally, adding the feedback term to \eqref{sigma1_ddot 2agent} similar to \eqref{q ddot PD}, we get the control input for obstacle avoidance:
\begin{equation*}
    u_1^{(i)} = \begin{bmatrix}
        a_1^{(i)} \\ \omega_1^{(i)}
    \end{bmatrix} = \Lambda_1^{(i)^\dagger}\!\!\!\left(\!\!\ddot{\sigma}_{1,d} + K_{p,1}^{(i)} \widetilde{\sigma}_{1}^{(i)} + K_{d,1}^{(i)} \frac{d\widetilde{\sigma}_1^{(i)}}{dt} - \delta_1^{(i)}\!\!\right)
\end{equation*}
where $\widetilde{\sigma}_1^{(i)} = \sigma_{1,d} - \sigma_{1}^{(i)} $, and $K_{p,1}^{(i)}$ and $K_{d,1}^{(i)}$ are the proportional and derivative gains respectively for task 1 and for $i=\{1,2\}$. Also, note that 
\begin{equation*}
    \Lambda_1 = \begin{bmatrix}
        \Lambda_1^{(1)} & 0\\
        0 &  \Lambda_1^{(2)}
    \end{bmatrix}, \quad \Lambda_1^\dagger = \begin{bmatrix}
        \Lambda_1^{(1)^\dagger} & 0\\
        0 &  \Lambda_1^{(2)^\dagger}
    \end{bmatrix}.
\end{equation*}
\subsubsection{Steering the robots' centroid toward a goal position}
This task involves steering the centroid of the robots toward a goal position $p_g$. The task variable $\sigma_2$ is given as
\begin{equation}\label{sigma2 2agents}
    \sigma_2 = \frac{p^{(1)} + p^{(2)}}{2}
\end{equation}
and the desired value of the task variable will be $\sigma_{2,d} = p_g$. \par
First, we design the control input $u_2 = [
    a_2^{(1)} \; \omega_2^{(1)} \; a_2^{(2)} \; \omega_2^{(2)}]^\top$ using a PD controller. Differentiating \eqref{sigma2 2agents} twice with respect to time we get 
\begin{align}\label{sigma2_ddot 2agents}
    \!\!\ddot{\sigma}_2 = \Lambda_2\!\!\begin{bmatrix}
        a_2^{(1)} \\ \omega_2^{(1)} \\ a_2^{(2)} \\ \omega_2^{(2)}
    \end{bmatrix}\!\!, \quad \!\!\!\!\! \Lambda_2 \!\!=\!\!\frac{1}{2} \!\!\begin{bmatrix}
        \cos\theta^{(1)} & \sin\theta^{(1)}\\
        -s^{(1)}\sin\theta^{(1)} & s^{(1)}\cos\theta^{(1)}\\
        \cos\theta^{(2)} & \sin\theta^{(2)}\\
        -s^{(2)}\sin\theta^{(2)} & s^{(2)}\cos\theta^{(2)}\\
    \end{bmatrix}^\top\!\!\!\!\!.
\end{align}
Adding the feedback term to \eqref{sigma2_ddot 2agents} similar to \eqref{q ddot PD}, we get
\begin{equation*}
    u_2 = \begin{bmatrix}
        a_2^{(1)} \\ \omega_2^{(1)}\\a_2^{(2)} \\ \omega_2^{(2)}
    \end{bmatrix} = \Lambda_2^\dagger\left(\ddot{\sigma}_{2,d} + K_{p,2} \widetilde{\sigma}_{2} + K_{d,2} \frac{d\widetilde{\sigma}_2}{dt}\right)
\end{equation*}
where $\widetilde{\sigma}_2 = \sigma_{2,d} - \sigma_{2} $, and $K_{p,2}$ and $K_{d,2}$ are the proportional and derivative gains respectively for task 2. \par

Next, we design the control input for task 2) via a path integral controller. The perturbed dynamics are given by \eqref{SDE hierar} with the diffusion coefficient $\hat{s}=0.1$. We formulate the cost function \eqref{PI cost} with $\alpha =10$,
\begin{equation*}
   L(x(t)) =  0.21\left\lVert{\frac{p^{(1)}(t) + p^{(2)}(t)}{2}- p_g }\right\rVert;
    \phi(x(T)) = L(x(T)).
\end{equation*}
The number of Monte Carlo samples has been set to $10^4$. 
\subsubsection{Maintaining a specific distance between two unicycles}
In this task, the two unicycles must maintain the distance $l$ between each other. The task variable $\sigma_3$ is given as
\begin{equation}\label{sigma3 2agents}
    \sigma_3 = \frac{1}{2}\left(p^{(1)} - p^{(2)}\right)^\top\left(p^{(1)} - p^{(2)}\right) 
\end{equation} 
and the desired value of the task variable will be $\sigma_{3,d} = \frac{l^2}{2}$.
Differentiating \eqref{sigma3 2agents} twice with respect to time we get 

\begin{equation}\label{sigma3_ddot 2agents}
  \ddot{\sigma}_{3} = \delta_{3} + \Lambda_{3}\begin{bmatrix}
        a_3^{(1)} \\ \omega_3^{(1)} \\ a_3^{(2)} \\ \omega_3^{(2)}
    \end{bmatrix}
\end{equation}
\[\delta_3 = s^{(1)^2} -2s^{(1)}s^{(2)}\left(\cos\theta^{(1)}\!\!\cos\theta^{(2)}+\sin\theta^{(1)}\sin\theta^{(2)}\right)+ s^{(2)^2},\] 
\begin{equation*}
  \!\!\! \Lambda_3 \!\! =\!\! \begin{bmatrix}
        \left(p_x^{(1)}\!\!-\!p_x^{(2)}\right)\cos\theta^{(1)}\!\! + \!\! \left(p_y^{(1)}\!\!-\!p_y^{(2)}\right)\sin\theta^{(1)}\\
        
        \left(p_x^{(2)}\!\!-\!p_x^{(1)}\right)s^{(1)}\sin\theta^{(1)} \!\!+ \!\!\left(p_y^{(1)}\!\!-\!p_y^{(2)}\right)s^{(1)}\cos\theta^{(1)}\\

        \left(p_x^{(2)}\!\!-\!p_x^{(1)}\right)\cos\theta^{(2)}\!\! + \!\!\left(p_y^{(2)}\!\!-\!p_y^{(1)}\right)\sin\theta^{(2)}\\

        \left(p_x^{(1)}\!\!-\!p_x^{(2)}\right)s^{(2)}\sin\theta^{(2)}\!\! +\!\! \left(p_y^{(2)}\!\!-\!p_y^{(1)}\right)s^{(2)}\cos\theta^{(2)}
    \end{bmatrix}^{\!\!\top}\!\!\!\!\!. 
\end{equation*}
Finally, adding the feedback term to \eqref{sigma3_ddot 2agents} similar to \eqref{q ddot PD}, we get
\begin{equation*}
    u_3 = \begin{bmatrix}
        a_3^{(1)} \\ \omega_3^{(1)}\\a_3^{(2)} \\ \omega_3^{(2)}
    \end{bmatrix} = \Lambda_3^\dagger\left(\ddot{\sigma}_{3,d} + K_{p,3} \widetilde{\sigma}_{3} + K_{d,3} \frac{d\widetilde{\sigma}_3}{dt}-\delta_3\right)
\end{equation*}
where $\widetilde{\sigma}_3 = \sigma_{3,d} - \sigma_{3} $, and $K_{p,3}$ and $K_{d,3}$ are the proportional and derivative gains respectively for task 3.\par

Lastly, the control input of the second task $u_2$ is projected onto the null space of task 1, and the control input of task 3 $u_3$ is projected onto the null spaces of both task 1 and task 2. This yields the overall control input as:
\begin{equation*}
    u = u_1 + (I - \Lambda_1^\dagger\Lambda_1)\left(u_2 + (I - \Lambda_2^\dagger\Lambda_2)u_3 \right)
\end{equation*}
where $I$ is the identity matrix of suitable dimensions. \par
In this experiment, we set the desired distance between the agents to $l = 0.5$ with the initial positions of the unicycles at $[
    p_x^{(1)} \; p_y^{(1)}\;  p_x^{(2)} \;  p_y^{(2)}]^\top = [
    -4.5 \;\; 0 \; -4 \;\; 0]^\top$. Fig. \ref{Fig. two agent}(a) represents the results of the hierarchical controller when all three tasks are controlled using PD controllers. As we can see, the centroid of the two unicycles initially moves towards the goal however, when the unicycles come in the vicinity of the obstacle they drive themselves away from the obstacle, as the obstacle avoidance task has a higher priority. Due to the limitations of the PD controllers, which get stuck in local minima, the unicycles begin to oscillate and never successfully cross the obstacle to reach the goal, as seen in Fig. \ref{Fig. two agent}(b).

\begin{figure}[h]
    \centering
      \begin{tabular}{c c}
    \includegraphics[scale=0.45]{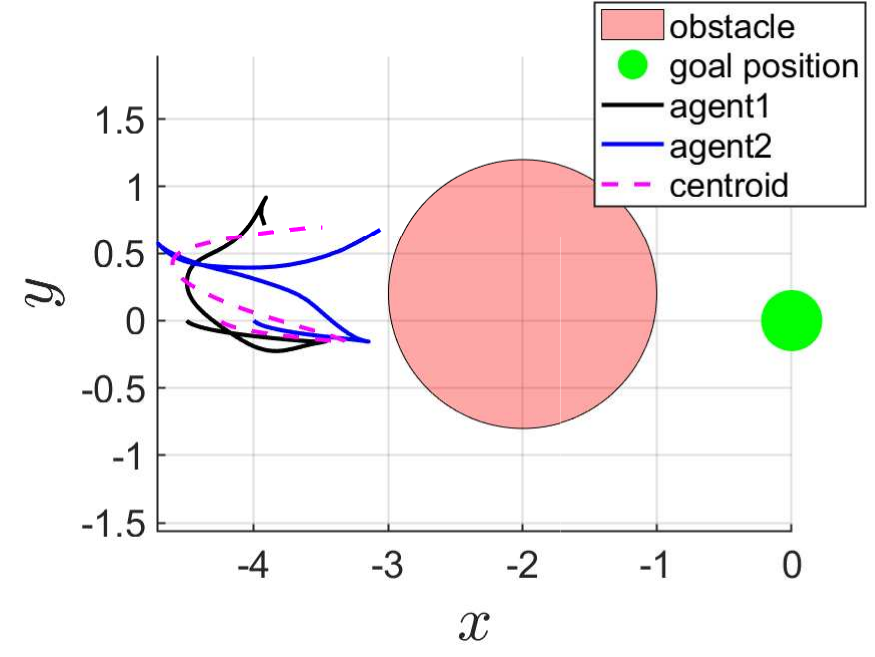} &\includegraphics[scale=0.45]{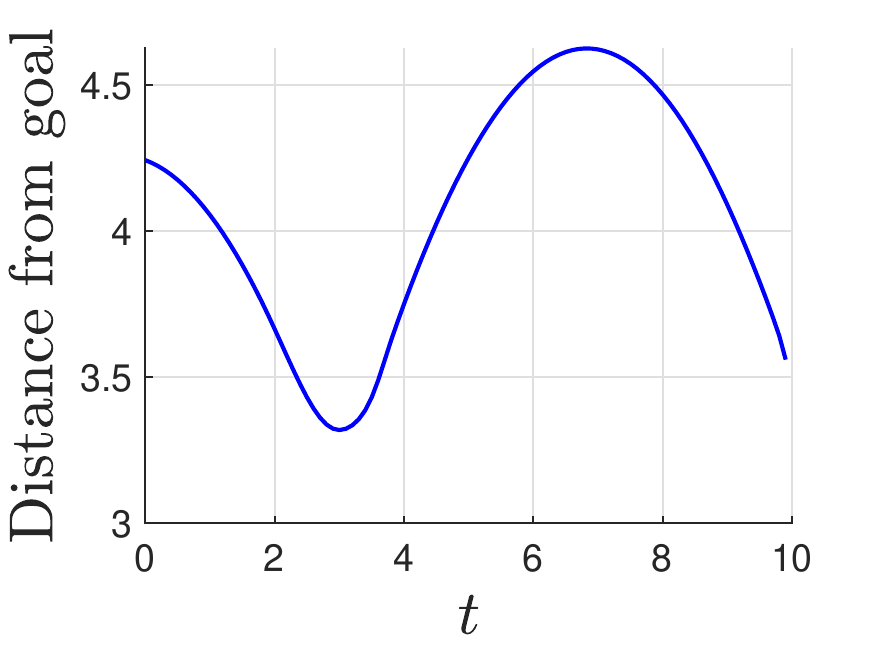} \\
     (a) Path followed without  & (b) Distance from the goal \\
     path integral controller & over time
      \end{tabular}
        \caption{Results of two-agents example without the path integral controller} 
        \label{Fig. two agent}
\end{figure}

In Fig. \ref{Fig. two agent PI}, we present the results using the hierarchical controller, where tasks 1) and 3) are managed with PD controllers, while task 2) is handled using a path integral controller. In Fig. \ref{Fig. two agent PI}(a), we plot 100 trajectories of the two agents and their centroid. The hierarchical controller successfully ensures obstacle avoidance and steers the centroid towards the goal in all trajectories, without getting stuck in local minima. Fig. \ref{Fig. two agent PI}(b) shows the mean distance of the centroid from the goal over time, along with the standard deviation across the 100 trajectories. The solid blue line represents the mean and the sky-blue shaded area represents one standard deviation from the mean. 

\begin{figure}[h]
    \centering
      \begin{tabular}{c c}
     \includegraphics[scale=0.45]{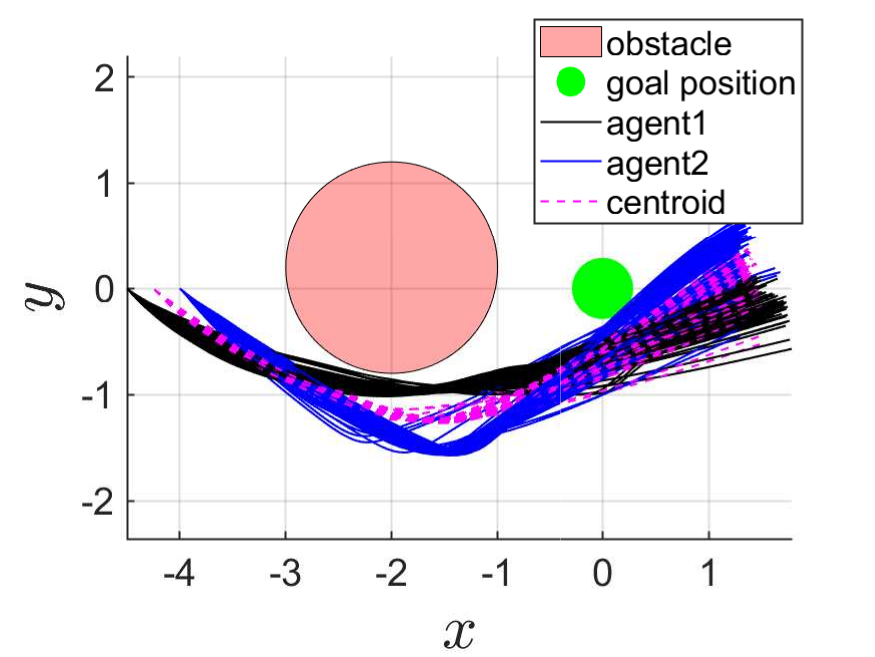} &\includegraphics[scale=0.45]{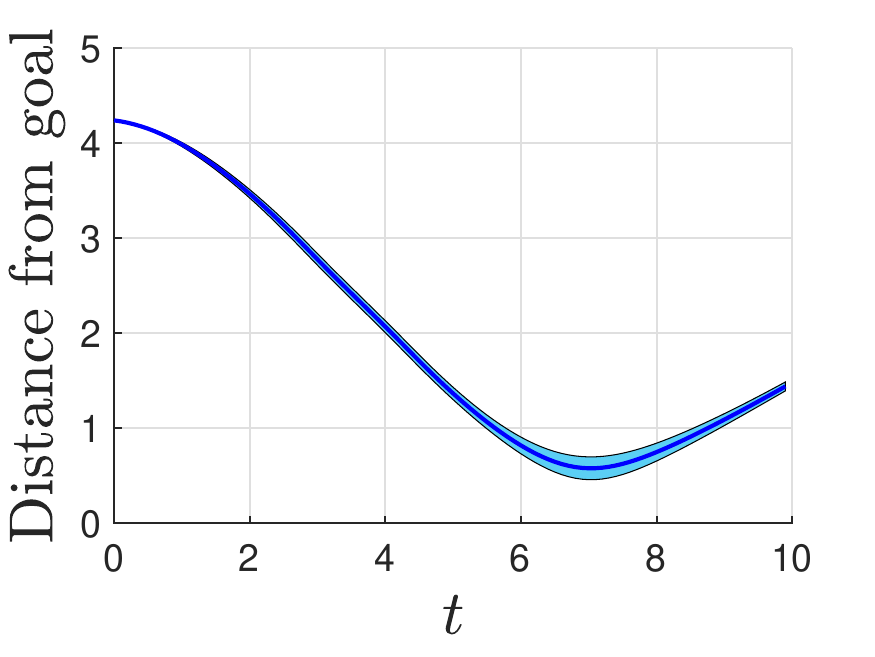} \\
     (a) Paths followed with  & (b) Mean distance from goal  \\
     path integral controller & $\pm$ standard deviation
      \end{tabular}
        \caption{Results of two-agents example using the path integral controller} 
        \label{Fig. two agent PI}
\end{figure}\par

\begin{figure}[h]
    \centering
      \begin{tabular}{c c}
     \includegraphics[scale=0.45]{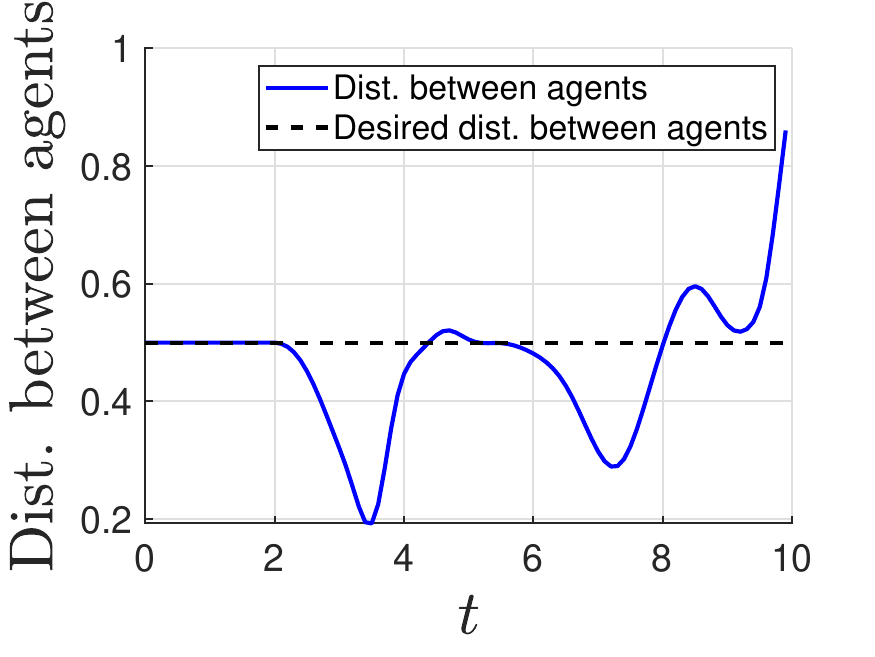} &\includegraphics[scale=0.45]{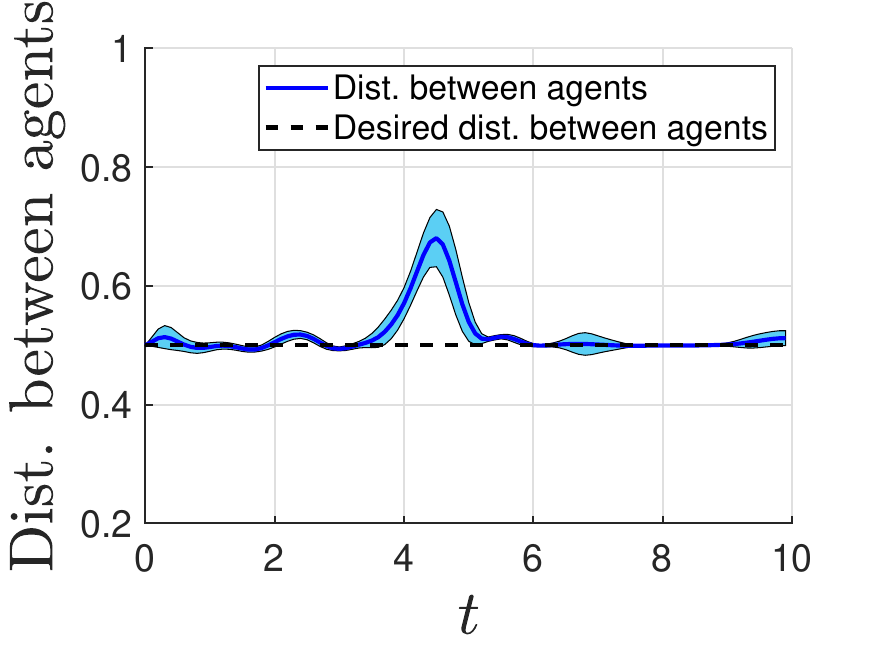} \\
     (a) Without path integral controller:  & (b) With path integral controller: Mean distance\\
     Distance between agents & between agents $\pm$ standard deviation
      \end{tabular}
        \caption{Distance between agents over time} 
        \label{Fig. dist bet agents}
\end{figure}
In Fig. \ref{Fig. dist bet agents}(a), we show the distance between the agents over time when the controller for task 2) is designed using a PD controller. Fig. \ref{Fig. dist bet agents}(b) shows the mean distance between the agents over time, using the path integral controller for task 2). The solid blue line represents the mean across 100 trajectories, while the shaded area corresponds to one standard deviation. This figure highlights that the path integral controller helps the agents better maintain the desired distance between each other compared to the PD controller.

\section{Publications}
\begin{itemize}
     \item \textbf{A. Patil}, R. Funada, T. Tanaka, L. Sentis, ``Task Hierarchical Control via Null-Space Projection and Path Integral Approach ," \textit{American Control Conference (ACC) 2025}
      \item M. Baglioni, \textbf{A. Patil}, L. Sentis, A. Jamshidnejad ``Hierarchical Optimal Control for Multi-UAV Best Viewpoint Coordination with Obstacle Avoidance," \textit{In preparation}
\end{itemize}
\section{Future Work}
For future work, we aim to incorporate the importance sampling techniques into the path integral controller, which offers improved sample efficiency compared to the standard approach. Additionally, we plan to extend the application of this framework to large-scale multi-robot systems and robotic manipulators with a high number of degrees of freedom.

%% file: chapters/deception.tex
\chapter[Deceptive Control]{Deceptive Control}
\label{Sec: deception}
\section{Motivation}

We consider a deception problem between a supervisor and an agent. The supervisor delegates an agent to perform a certain task and provides a reference policy to be followed in a stochastic environment. The agent, on the other hand, aims to achieve a different task and may deviate from the reference policy to accomplish its own task. The agent uses a deceptive policy to hide its deviations from the reference policy. In this work, we synthesize the optimal policies for such a deceptive agent.

We formulate the agent's deception problem using motivations from hypothesis testing theory. We assume that the supervisor aims to detect whether the agent deviated from the reference policy by observing the state-action paths of the agent. On the flip side, the agent's goal is to employ a deceptive policy that achieves the agent's task and minimizes the detection rate of the supervisor. We design the agent's deceptive policy that minimizes the Kullback-Leibler (KL) divergence from the reference policy while achieving the agent's task. The use of KL divergence is motivated by the log-likelihood ratio test, which is the most powerful detection test for any given significance level~\cite{cover1999elements}. Minimizing the KL divergence is equivalent to minimizing the expected log-likelihood ratio between distributions of the paths generated by the agent's deceptive policy and the reference policy. We also note that due to the Bratagnolle-Huber inequality~\cite{bretagnolle1978estimation}, for any statistical test employed by the supervisor, the sum of false positive and negative rates is lower bounded by a decreasing function of KL divergence between the agent's policy and the reference policy. Consequently, minimizing the KL divergence is a proxy for minimizing the detection rate of the supervisor. We represent the agent's task with a cost function and formulate the agent's objective function as a weighted sum of the cost function and the KL divergence.


We assume that the agent's environment follows discrete-time continuous-state dynamics. When the dynamics are linear, the supervisor's (stochastic) policies are Gaussian, and the cost functions are quadratic, minimizing a weighted sum of the cost function, and the KL divergence leads to solving a linear quadratic regulator problem. However, we consider a broader setting with potentially non-linear state dynamics, non-quadratic cost functions, and non-Gaussian reference policies. In this case, the agent's optimal deceptive policy does not necessarily admit a closed-form solution. While the agent's problem can be solved using backward dynamic programming, this approach suffers from the curse of dimensionality.

We show that, under the assumption of deterministic state dynamics, the optimal deceptive actions can be generated using path integral control without explicitly synthesizing a policy. In detail, we propose a two-step randomized algorithm for simulator-driven control for deception. At each time step, the algorithm first creates forward Monte Carlo samples of system paths under the reference policy. Then, the algorithm uses a cost-proportional weighted sampling method to generate a control input at that time step. We show that the proposed approach asymptotically converges to the optimal action distribution. Since Monte Carlo simulations can be efficiently parallelized, our approach allows the agent to generate the optimal deceptive actions online. 

\section{Literature Review}

Deception naturally occurs in settings where two parties with conflicting objectives coexist. The example domains for deception include robotics~\cite{shim2013taxonomy,dragan2015deceptive}, supervisory control settings~\cite{karabag2021deception,karabag2022exploiting}, warfare~\cite{lloyd2003art}, and cyber systems~\cite{wang2018cyber}.

We formulate a deception problem motivated by hypothesis testing. This problem has been studied for fully observable Markov decision processes~\cite{karabag2021deception}, partially observable Markov decision processes~\cite{karabag2022exploiting}, and hidden Markov models~\cite{keroglou2018probabilistic}. Different from \cite{karabag2021deception,karabag2022exploiting,keroglou2018probabilistic} that study discrete-state systems and directly solve an optimization problem for the synthesis of deceptive policies, we consider a nonlinear continuous-state system and provide a sampling-based solution for the synthesis of deceptive policies. In the security framework,  \cite{kung2016performance,bai2017data} study the detectability of an attacker in a stochastic control setting. Similar to our formulation, \cite{kung2016performance,bai2017data} provide a KL divergence-based optimization problem. While we consider an agent whose goal is to optimize a different cost function from the supervisor, \cite{kung2016performance,bai2017data} consider an attacker whose goal is to maximize the state estimation error of a controller.

KL divergence objective is also used in reinforcement learning~\cite{schulman2015trust,filippi2010optimism} to improve the learning performance and in KL control frameworks~\cite{todorov2007linearly,ito2022kullback} for the efficient computation of optimal policies. In \cite{ito2022kullback}, Ito et al. studied the KL control problem for nonlinear continuous-state space systems and characterized the optimal policies. Different from \cite{ito2022kullback}, we provide a randomized control algorithm based on path integral approach that converges to the optimal policy as the number of samples increases. Path integral control is a sampling-based algorithm employed to solve nonlinear stochastic optimal control problems numerically ~\cite{kappen2005path, theodorou2010generalized, williams2016aggressive}. It allows the policy designer to compute the optimal control inputs online using Monte Carlo samples of system paths. The Monte Carlo simulations can be massively parallelized on GPUs, and thus the path integral approach is less susceptible to the curse of dimensionality \cite{williams2017model}. 

\section{Contributions}
The contributions of this work are as follows:
\begin{enumerate}
    \item The work studies a problem of deception under supervisory control for continuous-state discrete-time stochastic systems. Given a reference policy, we formalize the synthesis of an optimal deceptive policy as a KL control problem and solve it using backward dynamic programming. 
    \item For the deterministic state dynamics, we propose a path-integral-based solution methodology for simulator-driven control. We develop an algorithm based on Monte Carlo sampling to numerically compute the optimal deceptive actions online. Furthermore, we show that the proposed approach asymptotically converges to the optimal control distribution of the deceptive agent.
    \item We present a numerical example to validate the derived simulator-driven control synthesis framework.
\end{enumerate}
\section*{Notations}
Let $(\mathcal{X}, \mathcal{B}(\mathcal{X}))$ be a measurable space where $\mathcal{X}\subseteq\mathbb{R}^n$ is a Borel set and $\mathcal{B}(\mathcal{X})$ is a Borel $\sigma$-algebra. Suppose $(\Omega, \mathcal{F}, \mathcal{P})$ is a probability space. An $( \mathcal{F}, \mathcal{B}(\mathcal{X}))$-measurable random variable $X$ is a function $X:\Omega\rightarrow\mathcal{X}$ whose probability distribution $P_X$ is defined by 
\begin{equation*}
  P_X(B) = \mathcal{P}(X\in B)\quad \forall B \in \mathcal{B}(\mathcal{X})
\end{equation*}
$P_{X_2|X_1}(\cdot|\cdot):\mathcal{B}(\mathcal{X}_2) \times \mathcal{X}_1 \rightarrow [0,1]$ represents a stochastic kernel on $\mathcal{X}_2$ given $\mathcal{X}_1$. For simplicity, we write $P_X(dx)$ and $P_{X_2|X_1}(dx_2|x_1)$ as $P(dx)$ and $P(dx_2|x_1)$. If $P_1$ and $P_2$ are probability distributions on $(\mathcal{X}, \mathcal{B}(\mathcal{X}))$ then, the Kullback-Leibler (KL) divergence from $P_1$ to $P_2$ is defined as
\begin{equation*}\label{eq: KL divergence def}
	D(P_2\|P_1)=\int_\mathcal{X} \log\frac{dP_2}{dP_1}(x)P_2(dx)
\end{equation*}
if the Radon-Nikodym derivative $\frac{dP_2}{dP_1}$ exists, and $D(P_2\|P_1)=+\infty$ otherwise. Throughout this dissertation, we use the natural logarithm. Let $\mathcal{T}=\{0, 1, ... , T\}$ be the set of discrete time indices. A set of variables $\{x_0, x_1, \hdots, x_T\}$ is denoted by $x_{0:T}$ and a Cartesian product of sets $\mathcal{X}_0\times\mathcal{X}_1\times\hdots\times\mathcal{X}_T$ is denoted by $\mathcal{X}_{0:T}$. $P_{X_{0:T}}(dx_{0:T})$ denotes the joint probability distribution of random variables $X_0, X_1, \hdots, X_T$ on $\left(\mathcal{X}_{0:T}, \mathcal{B}(\mathcal{X}_{0:T})\right)$. Table \ref{tab:notation deception} represents the mathematical notations frequently used in this chapter. 

\begin{table}
\begin{center}
\begin{tabular}{||c | c || c | c||} 
 \hline
 \textbf{Notation} & \textbf{Description} & \textbf{Notation} & \textbf{Description} \\ [0.5ex] 
 \hline\hline
 $P(dx_{t+1}|x_t, u_t)$ & agent state transition & $\mathcal{X}_t$ & state space \\ 
 \hline
 
 $\mathcal{U}_t$ & control input space & $R_{U_t|X_t}(\cdot|x_t)$ & reference policy \\ 
 \hline
 $Q_{U_t|X_t}(\cdot|x_t)$ & agent's policy & $C_{0:T}(x_{0:T}, u_{0:T-1})$ & path cost \\ 
 \hline
 $C_t(x_t, u_t)$ & stage cost & $C_T(x_T)$ & terminal cost \\ 
 \hline
 $\pi(x_{0:T}, u_{0:T-1})$ & log-likelihood ratio (LLR) & $\Pi$ & expected LLR \\ 
 \hline
 $ D(Q \|R)$ & KL divergence & $J_t(x_t)$ & value function \\ 
 \hline
  $ F_t(x_t, u_t)$ & deterministic state transition & $\delta$ & Dirac measure \\ 
 \hline
  $ Z_t(x_t)$ & exponentiated value function & &  \\ 
 [1ex] 
 \hline
\end{tabular}
\caption{Table of frequently used mathematical notation in Chapter \ref{Sec: deception}}
\label{tab:notation deception}
\end{center}
\end{table}

\section{Problem Formulation}\label{Sec: KL Control}
We consider a setting in which a supervisor contracts an agent to perform a certain task. Suppose the agent operates in a stochastic environment and follows discrete-time continuous-state dynamics. Let the state transition law of the agent be denoted by $P(dx_{t+1}|x_t, u_t):\mathcal{B}(\mathcal{X}_{t+1}) \times \mathcal{X}_{t} \times \mathcal{U}_{t} \rightarrow [0,1]$, where $\mathcal{X}_t\subseteq \mathbb{R}^n$ and $\mathcal{U}_t\subseteq \mathbb{R}^m$ be the spaces of states and control inputs at time step $t\in\mathcal{T}:=\{0, 1, \cdots, T\}$, respectively. Suppose a supervisor provides a (possibly stochastic) reference policy $\{R_{U_t|X_t}(\cdot|x_t)\}_{t=0}^{T-1}$ to the agent and expects the agent to follow the policy to accomplish a certain task. Here, $R_{U_t|X_t}:\mathcal{B}(\mathcal{U}_t) \times \mathcal{X}_t \rightarrow [0,1]$ is a stochastic kernel on $\mathcal{U}_t$ given $\mathcal{X}_t$. The agent, on the other hand, aims to achieve a different task by minimizing the following cost function, which we henceforth call as \textit{path cost}:
\begin{equation}\label{eq:path cost}
   C_{0:T}(x_{0:T}, u_{0:T-1}) \coloneqq \sum_{t=0}^{T-1} C_t(x_t, u_t)+C_T(x_T)
\end{equation}
where $C_t(\cdot, \cdot):\mathcal{X}_t\times \mathcal{U}_t\rightarrow \mathbb{R}$ for $t\in\mathcal{T}$ and $C_T(\cdot): \mathcal{X}_T\rightarrow \mathbb{R}$ represent the stage costs and the terminal cost, respectively. In order to minimize the path cost \eqref{eq:path cost}, the agent designs its policy (possibly stochastic) $\{Q_{U_t|X_t}(\cdot|x_t)\}_{t=0}^{T-1}$ that may deviate from the reference policy $\{R_{U_t|X_t}(\cdot|x_t)\}_{t=0}^{T-1}$. The agent also attempts to be stealthy to hide its deviations from the supervisor. While the agent executes its policy $Q$, suppose the supervisor observes its state-action paths $\{x_{0:T}, u_{0:T-1}\}$, and uses a likelihood ratio test to detect whether the agent deviates from the reference policy. According to the Neyman–Pearson lemma, the likelihood-ratio test is optimal among all simple hypothesis tests for a given significance level~\cite{cover1999elements}. In other words, we consider the worst-case scenario for the agent to be detected by the supervisor. Suppose the initial state $X_0 = x_0$ of the agent is known. We denote the joint probability distribution of the state-action paths induced via the reference policy by
\begin{align}\label{eq:def_ref_traj_dist}
 R_{X_{0:T}\times U_{0:T-1}}& (dx_{0:T}\times du_{0:T-1})=\prod_{t=0}^{T-1}\!\! P(dx_{t+1}|x_t, u_t)R(du_t|x_t), 
\end{align}
and the joint distribution induced via the agent's policy by 
\begin{align}\label{eq:def_dc_traj_dist}
    Q_{X_{0:T}\times U_{0:T-1}}& (dx_{0:T}\times du_{0:T-1})=\prod_{t=0}^{T-1}\!\! P(dx_{t+1}|x_t, u_t)Q(du_t|x_t). 
\end{align}
Given a path $\{x_{0:T}, u_{0:T-1}\}$ that is randomly sampled under the agent's policy, the supervisor computes the log-likelihood ratio (LLR) 
\begin{equation}
    \pi(x_{0:T}, u_{0:T-1})=\log \frac{dQ_{X_{0:T}\times U_{0:T-1}}}{dR_{X_{0:T}\times U_{0:T-1}}}(x_{0:T}, u_{0:T-1}) \label{eq:LLR}.
\end{equation} The supervisor decides that the agent uses the reference policy $R$ if $\pi(x_{0:T}, u_{0:T-1})\leq c$, and deviates from $R$ otherwise. Here $c$ is a constant chosen by the supervisor to obtain a specified significance level. An agent not wanting to be detected by the supervisor must minimize the LLR \eqref{eq:LLR}. However, since the agent's trajectories are stochastic, the agent cannot directly minimize the LLR. We consequently consider that the agent's goal is to minimize the expected LLR as follows:
\begin{equation}\label{eq:ELLR}
   \Pi = \mathbb{E}_Q\left[\log \frac{dQ_{X_{0:T}\times U_{0:T-1}}}{dR_{X_{0:T}\times U_{0:T-1}}}(x_{0:T}, u_{0:T-1})\right]
\end{equation} 
where $\mathbb{E}_Q[\cdot]$ represents the expectation with respect to the probability distribution $Q$ \eqref{eq:def_dc_traj_dist}. Note that equation \eqref{eq:ELLR} also defines the Kullback-Leibler (KL) divergence $D(Q \|R)$ between the agent's distribution $Q$  and the reference distribution $R$. Now we show that $D(Q \|R)$ can be written as the stage-additive KL divergence between $Q_{U_t|X_t}$ and $R_{U_t|X_t}$:
\begin{subequations}\label{eq:RN as product} \allowdisplaybreaks
\begin{align}
   &\int_{B} Q_{X_{0:T}\times U_{0:T-1}}(dx_{0:T}\!\times\! du_{0:T-1})\nonumber\\
   =&\int_{B}\prod_{t=0}^{T-1}\!\! P(dx_{t+1}|x_t, u_t)Q(du_t|x_t)\label{eq:RN as product1}\\
   =&\int_{B}\!\!\left(\prod_{t=0}^{T-1}\frac{dQ_{U_t|X_t}}{dR_{U_t|X_t}}\!\left(x_t,u_t\right)\!\!\right)\!\!\prod_{t=0}^{T-1}\!\! P(dx_{t+1}|x_t, u_t)R(du_t|x_t)\label{eq:RN as product2}\\
   =&\!\!\!\int_{B}\!\!\left(\prod_{t=0}^{T-1}\frac{dQ_{U_t|X_t}}{dR_{U_t|X_t}}\!\left(x_t,u_t\right)\!\!\right)\!\!R_{X_{0:T}\times U_{0:T-1}}\!(dx_{0:T}\!\!\times\! du_{0:T-1})\label{eq:RN as product3}
\end{align}
\end{subequations}
where, $B$ is a Borel set belonging to the $\sigma-$algebra $\mathcal{B}\left(\mathcal{X}_{0:T}\times\mathcal{U}_{0:T-1}\right)$. The first equality \eqref{eq:RN as product1} follows from the definition \eqref{eq:def_dc_traj_dist}, the second equality \eqref{eq:RN as product2} by the definition of the Radon-Nikodym derivative \cite{durrett2019probability} and the last one \eqref{eq:RN as product3} from the definition \eqref{eq:def_ref_traj_dist}. Using \eqref{eq:RN as product}, we can write the Radon-Nikodym derivative $\frac{dQ_{X_{0:T}\times U_{0:T-1}}}{dR_{X_{0:T}\times U_{0:T-1}}}$ as follows:
\begin{equation}\label{eq: stagewise KL}
     \!\!\!\frac{dQ_{X_{0:T}\times U_{0:T-1}}}{dR_{X_{0:T}\times U_{0:T-1}}}\left(x_{0:T}, u_{0:T-1}\right) \!=\!\!\! \prod_{t=0}^{T-1} \frac{dQ_{U_t|X_t}}{dR_{U_t|X_t}}\left(x_t,u_t\right).
\end{equation}
Using \eqref{eq: stagewise KL}, we get the following:

\begin{align}
 D(Q\|P)=&\;\mathbb{E}_Q\left[\log \frac{dQ_{X_{0:T}\times U_{0:T-1}}}{dR_{X_{0:T}\times U_{0:T-1}}}(x_{0:T}, u_{0:T-1})\right]\nonumber\\
     =&\;\mathbb{E}_Q \left[\log \prod_{t=0}^{T-1} \frac{dQ_{U_t|X_t}}{dR_{U_t|X_t}}\left(x_t, u_t\right)\right]\nonumber\\
     =&\;\mathbb{E}_Q\left[\sum_{t=0}^{T-1}\log\frac{dQ_{U_t|X_t}}{dR_{U_t|X_t}}\left(x_t, u_t\right)\right]\label{eq:llr and KL}\\
     =&\;\mathbb{E}_Q\left[\sum_{t=0}^{T-1}D(Q_{U_t|X_t}(\cdot | X_t) \|R_{U_t|X_t}(\cdot | X_t))\right].\label{eq: KL cost}
\end{align}
Since the KL divergence $ D(Q \|R)$ is equivalent to the expected LLR \eqref{eq:ELLR}, in this work, the KL divergence is used as a proxy for the measure of the agent's deviations from the reference policy. 

Minimizing the KL divergence is in fact equivalent to minimizing the detection rate of an attacker for an ergodic process as proved in \cite{bai2017data}. While we do not consider an ergodic process, the use of KL divergence is still well-motivated by the Bretagnolle–Huber inequality~\cite{bretagnolle1978estimation}. Let \(\mathcal{E}\) be an arbitrary set of events that the supervisor will identify the agent as a deceptive agent, i.e., the agent followed \(Q\). According to the Bretagnolle–Huber inequality, we have
 \begin{equation}\label{eq: BH inequlity}
     \Pr(\mathcal{E}|R) + \Pr(\neg \mathcal{E}|Q) \geq \frac{1}{2} \exp(-D(Q||R))
 \end{equation}
 where $\Pr(\mathcal{E}|R)$ and $\Pr(\neg \mathcal{E}|Q)$ denote the supervisor's false positive and negative rates, respectively. The false positive rate is the probability that the supervisor will identify the well-intentioned agent as a deceptive agent, i.e., the agent's policy is \(R\), but the supervisor thinks that the agent has followed \(Q\). Similarly, the false negative rate is the probability that the supervisor will identify the deceptive agent as a well-intentioned agent. The Bretagnolle–Huber inequality \eqref{eq: BH inequlity} states that the sum of the supervisor's false positive and negative rates is lower bounded by a decreasing function of the KL divergence between the distributions \(Q\) and \(R\). 
Therefore, an agent wanting to increase the supervisor's false classification rate should minimize the KL divergence from $R$ to $Q$.\par

The goal of the agent is to design a deceptive policy $Q$ that minimizes the expected path cost $\mathbb{E}_Q\left[C_{0:T}(X_{0:T}, U_{0:T-1})\right]$ \eqref{eq:path cost} while limiting the KL divergence $D\left(Q\|R\right)$ \eqref{eq: KL cost}. Using \eqref{eq:path cost} and \eqref{eq: KL cost}, we propose the following KL control problem for the synthesis of optimal deceptive policies for the agent: 

\begin{problem}[Synthesis of optimal deceptive policy]\label{Prob:KL control}
\begin{equation}\label{eq:prob_KL_dc1_fixed_end}
\min_{\{Q_{U_t|X_t}\}_{t=0}^{T-1}} \mathbb{E}_Q \sum_{t=0}^{T-1}  \Big\{ C_t(X_t, U_t) +\lambda D(Q_{U_t|X_t}(\cdot | X_t) \|R_{U_t|X_t}(\cdot | X_t))\Big\} + \mathbb{E}_Q C_T(X_T)
\end{equation}
where $\lambda$ is a positive weighting factor that balances the trade-off between the KL divergence and the path cost.
\end{problem}
 We explain the above KL control problem via the following example:  
\begin{example}
Consider a drone that is contracted by a supervisor to perform a surveillance task over an area. The supervisor prefers the drone to operate at high speeds (policy \(R\)) to improve the efficiency of the surveillance. The operator of the drone, the agent, on the other hand, prefers the drone to operate in a battery-saving, safe mode (policy \(Q\)) to improve the longevity of the drone. The agent does not want to get detected by the supervisor and fired. Hence, the goal of the agent is to operate in a way that would balance the energy consumption (\(\mathbb{E}_Q\left[ C_{0:T}(X_{0:T}, U_{0:T-1})\right]\)) and the deviations from the behavior desired by the supervisor (\(D(Q||R)\)).
\end{example}

\section{Synthesis of Optimal Control Policies}
In this section, we solve Problem \ref{Prob:KL control} using backward dynamic programming and propose a policy synthesis algorithm based on path integral control. 

\subsection{Backward Dynamic Programming}
Notice that the cost function of Problem \ref{Prob:KL control} possesses the time-additive Bellman structure and, therefore, can be solved by utilizing the principle of dynamic programming \cite{bertsekas2012dynamic}. Define for each $t\in\mathcal{T}$ and $x_t\in\mathcal{X}_t$, the value function:

\begin{equation}\label{eq:value function}
J_t(x_t):=\inf_{\{Q_{U_k|X_k}\}_{k=t}^{T-1}} \mathbb{E}_Q \sum_{k=t}^{T-1}  \Big\{ C_k(X_k, U_k) +\lambda D(Q_{U_k|X_k}(\cdot | X_k) \|R_{U_k|X_k}(\cdot | X_k))\Big\} + \mathbb{E}_Q C_T(X_T).\nonumber
\end{equation}
Notice that in \eqref{eq:value function}, we used ``inf" instead of ``min" since we do not know if the infimum is attained. In the following theorem, we show that the infimum is indeed attained, and therefore, ``inf" can be replaced by ``min". 
\begin{theorem}\label{thrm:Bellman recursion}
The value function $J_t(x_t)$ satisfies the following backward Bellman recursion with the terminal condition $J_T(x_T)=C_T(x_T)$:
\begin{equation}\label{eq:v_bellman} 
J_t(x_t)= -\lambda \log\Bigg\{ \int_{\mathcal{U}_t}\exp\left(-\frac{C_t(x_t, u_t)}{\lambda}\right)\times\exp\left(-\frac{1}{\lambda}\int_{\mathcal{X}_{t+1}}\!\!\!\!\!\!J_{t+1}(x_{t+1})P(dx_{t+1}|x_t, u_t)\right) R(du_t|x_t)\Bigg\}\nonumber
\end{equation}
and the minimizer of Problem \ref{Prob:KL control} is given by
\begin{equation}
\label{eq:p_star_dc}
Q_{U_t|X_t}^*(B_{U_t}|x_t)\!=\!\frac{\int_{B_{U_t}}\!\!\!\exp(-\rho_t(x_t, u_t)/\lambda)R(du_t|x_t)}{\int_{\mathcal{U}_t} \exp(-\rho_t(x_t, u_t)/\lambda)R(du_t|x_t)}
\end{equation}
where 
\begin{equation}\label{eq:rho_dc_def}
\!\!\!\rho_t(x_t, u_t)\!:=\! C_t(x_t, u_t)+\!\!\int_{\mathcal{X}_{t+1}} \!\!\!\!\!\!\!\!\! J_{t+1}(x_{t+1})P(dx_{t+1}|x_t, u_t)
\end{equation}
and $B_{U_t}$ is a Borel set belonging to the $\sigma-$algebra $\mathcal{B}(\mathcal{U}_t)$.
\end{theorem}

\begin{proof}
    By Bellman's optimality principle, the value function satisfies the following recursive relationship:
\begin{equation}\label{eq:kl_dc_bellman1}
J_t(x_t)\!=\!\!\!\!\!\inf_{Q_{U_t|X_t}}\! \int_{\mathcal{U}_t}\!\!\! \left\{\!\rho_t(x_t, u_t)\! +\!\lambda \log \frac{dQ}{dR}(u_t|x_t)\!\right\}\!Q(du_t|x_t) 
\end{equation}
where $\rho_t(x_t, u_t)$ is defined by \eqref{eq:rho_dc_def}. Invoking the Legendre duality between the KL divergence and free energy (see Appendix \ref{sec:Legendre Duality}), it can be shown that there exists a minimizer $Q^*_{U_t|X_t}$ of the right-hand side of \eqref{eq:kl_dc_bellman1}, which can be written as 
\begin{equation}\label{eq:Q_bellman}
Q_{U_t|X_t}^*(B_{U_t}|x_t)\!=\!\frac{\int_{B_{U_t}}\!\!\!\exp(-\rho_t(x_t, u_t)/\lambda)R(du_t|x_t)}{\int_{\mathcal{U}_t} \exp(-\rho_t(x_t, u_t)/\lambda)R(du_t|x_t)}
\end{equation}
where $B_{U_t}$ is a Borel set belonging to the $\sigma-$algebra $\mathcal{B}(\mathcal{U}_t)$. Using \eqref{eq:Q_bellman}, the value of \eqref{eq:kl_dc_bellman1} can be computed as
\begin{equation}
\label{eq:v_rho_bellman}
\!\!J_t(x_t)\!=\!-\lambda \log \left\{ \int_{\mathcal{U}_t}\!\! \exp\left(-\frac{\rho_t(x_t, u_t)}{\lambda}\right)\!R(du_t|x_t) \right\}.
\end{equation}
Substituting \eqref{eq:rho_dc_def} into \eqref{eq:v_rho_bellman}, we obtain the recursive expression \eqref{eq:v_bellman}.
\end{proof}
Theorem \ref{thrm:Bellman recursion} provides a recursive method to compute the value functions $J_t(x_t)$ and optimal control distributions $Q^*_{U_t|X_t}$ backward in time. As one can see, to perform the backward recursions \eqref{eq:v_bellman} and \eqref{eq:p_star_dc}, the function $J_t(x_t)$ must be evaluated everywhere in the continuous domain $\mathcal{X}_t$. Therefore, in practice, an exact implementation of backward dynamic programming is computationally costly (unless the problem has a special structure, for example, linear state dynamics and quadratic costs). The computational cost grows quickly with the dimension of the state space of the system, which is referred to as the curse of dimensionality. In the next section (Section \ref{sec:PI_control}), we show that under the assumption of the deterministic state transition law, the backward Bellman recursions can be linearized. This allows us to design a simulator-driven algorithm to compute optimal deceptive actions.

\subsection{Path Integral Solution}\label{sec:PI_control}
In this section, we focus on a special case in which the agent's dynamics are deterministic and propose a simulator-driven algorithm to compute the optimal deceptive actions via path integral control. 
\begin{assumption}\label{assump: deterministic law}
The state transition law is governed by a deterministic mapping $F_t:\mathcal{X}_t\times\mathcal{U}_t\rightarrow\mathcal{X}_{t+1}$ as
\begin{equation}\label{eq:deter state transition law}
    x_{t+1} = F_t(x_t, u_t);
\end{equation}
that is, $P(dx_{t+1}|x_t,u_t) = \delta_{F_t(x_t, u_t)}(dx_{t+1})$, where $\delta$ denotes the Dirac measure.
\end{assumption}

\begin{remark}
Note that under Assumption \ref{assump: deterministic law}, the agent can deviate from the reference policy \(R\) only if it is stochastic. Otherwise, under any deviations from the reference policy, with a positive probability, the supervisor will be sure that the agent did not follow the reference policy. Therefore, in what follows, we consider the reference policy to be stochastic. The stochasticity of the reference policy could be to account for the unmodeled elements of the dynamics, to provide robustness, or to encourage exploration. 
\end{remark}

\begin{remark}
Consider a special setting in which the state dynamics $F_t(x_t, u_t)$ is linear in $x_t$ and $u_t$, the reference policy distribution $R_{U_t|X_t}(\cdot | x_t)$ is Gaussian, and the cost functions $C_t(\cdot,\cdot)$ and $C_T(\cdot)$ are quadratic in $x_t$ and $u_t$. In such a setting, it can be shown that the optimal deceptive policy $Q^*_{U_t|X_t}(\cdot | x_t)$ is also Gaussian and can be analytically computed by backward Riccati recursions similar to the standard Linear-Quadratic-Regular (LQR) problems. In this work, we consider a broader setting with possibly non-Gaussian reference distribution, non-linear state dynamics, and non-quadratic cost functions. In this case, the optimal deceptive policy might not be efficiently computed by solving backward recursions.
\end{remark}

Now, we propose a path-integral-based solution approach for simulator-driven policy synthesis. Under assumption \ref{assump: deterministic law}, we can rewrite \eqref{eq:v_bellman} as
\begin{equation}\label{eq:v_bellman_post}
J_t(x_t)= -\lambda \log\Bigg\{ \int_{\mathcal{U}_t}\exp\left(-\frac{C_t(x_t, u_t)}{\lambda}\right)\times\exp\left(-\frac{J_{t+1}\left(F_t(x_t,u_t)\right)}{\lambda}\right)R(du_t|x_t)\Bigg\}.\nonumber
\end{equation}
We introduce the exponentiated value function as $Z_t(x_t):=\exp\left(-\frac{1}{\lambda}J_t(x_t)\right)$. Using $Z_t(x_t)$, the Bellman recursion \eqref{eq:v_bellman_post} can be linearized, and we get the following linear relationship between $Z_t$ and $Z_{t+1}$:
\begin{align}
\! Z_t(x_t)\!=\!&\int_{\mathcal{U}_t}\!\!\!\exp\left(-\frac{C_t(x_t, u_t)}{\lambda}\right) Z_{t+1}\left(F_t(x_t, u_t)\right)R(du_t|x_t)\nonumber\\
=&\int_{\mathcal{U}_t}\int_{\mathcal{X}_{t+1}}\exp\left(-\frac{C_t(x_t, u_t)}{\lambda}\right) Z_{t+1}(x_{t+1})\times P(dx_{t+1}|x_t, u_t)R(du_t|x_t).\label{eq:z_recursive}
\end{align}
Note that in \eqref{eq:z_recursive}, $P(dx_{t+1}|x_t,u_t) = \delta_{F_t(x_t, u_t)}(dx_{t+1})$ by Assumption \ref{assump: deterministic law}. 
Equation \eqref{eq:z_recursive} is a linear backward recursion in \(Z_t\). The linear solvability of the KL control problem is well-known in the literature (e.g., \cite{todorov2007linearly}). We remark that linearizability critically relies on Assumption \ref{assump: deterministic law}\footnote{We remark that in prior works where the path integral method is used to solve stochastic control problems, a certain assumption (e.g. Eq. (9) in \cite{kappen2005path}) is made to reinterpret the original problem as a problem of designing the optimal randomized policy for a deterministic transition system.}.
Now, by recursive substitution, \eqref{eq:z_recursive} can also be written as
\begin{align*}
Z_t(x_t)=
\int_{\mathcal{U}_t}\int_{\mathcal{X}_{t+1}}\cdots 
\int_{\mathcal{U}_{T-1}}\int_{\mathcal{X}_T}&
\exp\left(-\frac{C_t(x_t, u_t)}{\lambda}\right)\times \cdots \times \exp\left(-\frac{C_T(x_T)}{\lambda}\right)\\
& R(dx_{t+1:T}\times du_{t:T-1}|x_t). 
\end{align*}
Thus, by introducing the path cost function 
\begin{equation}\label{path cost function}
 C_{t:T}(x_{t:T}, u_{t:T-1}):=\sum_{k=t}^{T-1}C_k(x_k, u_k)+C_T(x_T),   
\end{equation}
we obtain
\begin{equation}
  Z_t(x_t)=\mathbb{E}_R \exp \left(-\frac{1}{\lambda}C_{t:T}(X_{t:T}, U_{t:T-1}) \right)  
 \label{eq:z_pi}
\end{equation}
where the expectation 
$\mathbb{E}_R(\cdot)$ is with respect to the probability measure $R$ \eqref{eq:def_ref_traj_dist}. Equation \eqref{eq:z_pi} expresses the exponentiated value function $Z_t(x_t)$ as the expected path cost under the reference distribution. This suggests a path-integral-based approach to numerically compute $Z_t(x_t)$. Suppose we generate a collection of $N$ independent samples of system paths $\{x_{t:T}(i), u_{t:T-1}(i)\}_{i=1}^N$ starting from $x_t$ under the reference distribution $R$. Since the reference distribution is known, a collection of such sample paths can be easily generated using a Monte Carlo simulation. If $C_{t:T}(x_{t:T}(i), u_{t:T-1}(i))$ represents the path cost of the sample path $i$, then by the strong law of large numbers \cite{durrett2019probability} as $N\rightarrow \infty$, we get 
\begin{equation}
\label{eq:mc_dc_phi_convergence}
 \!\! \!\!\frac{1}{N}\sum_{i=1}^N \exp\!\left(-\frac{1}{\lambda}C_{t:T}(x_{t:T}(i), u_{t:T-1}(i))\right) \overset{a.s.}{\rightarrow}   Z_t(x_t).
\end{equation}\par
Similarly, a collection of sample paths $\{x_{t:T}(i), u_{t:T-1}(i)\}_{i=1}^N$ starting from $x_t$ under the reference distribution $R$ can be used to sample $u_t$ from the optimal distribution $Q_{U_t|X_t}^*(\cdot|x_t)$.
Notice that the optimal deceptive policy \eqref{eq:p_star_dc} can be expressed in terms of $\{Z_t\}_{t=0}^{T}$ as 
\begin{subequations}\label{eq:p_star_dc_linear_path_cost}
\begin{align}
 Q_{U_t|X_t}^*(B_{U_t}|x_t) =&
\frac{1}{Z_t(x_t)}\int_{B_{U_t}}\exp\left(-\frac{C_t(x_t, u_t)}{\lambda}\right)\times\exp\left(-\frac{J_{t+1}\left(F_t(x_t,u_t)\right)}{\lambda}\right)R(du_t|x_t)\nonumber\\
&=\frac{1}{Z_t(x_t)}\int_{B_{U_t}}\exp\left(-\frac{C_t(x_t, u_t)}{\lambda}\right)\times Z_{t+1}\left(F_t(x_t, u_t)\right)R(du_t|x_t)\label{eq:p_star_dc_linear_path_cost1}\\
&=\!\frac{1}{Z_t(x_t)}\int_{B_{U_t}}\int_{\mathcal{X}_{t+1}}\!\!\!\!\exp\left(-\frac{C_t(x_t, u_t)}{\lambda}\right)Z_{t+1}(x_{t+1})\times P(dx_{t+1}|x_t, u_t)R(du_t|x_t) \label{eq:p_star_dc_linear_path_cost2} \\
&=\!\frac{1}{Z_t(x_t)}\!\!\int_{\!\{\mathcal{X}_{t+1:T},\,\mathcal{U}_{t:T-1}|u_t\in B_{U_t}\}} \!\!\!\!\!\!\! \exp\!\left(\!-\frac{C_{t:T}(x_{t:T}, u_{t:T-1})}{\lambda}\!\right)\times R(dx_{t+1:T}\times du_{t:T-1}|x_t).
\label{eq:p_star_dc_linear_path_cost3}
\end{align}
\end{subequations}
The step \eqref{eq:p_star_dc_linear_path_cost1} follows from the definition of $Z_t$. In \eqref{eq:p_star_dc_linear_path_cost2}, we used our assumption $P(dx_{t+1}|x_t,u_t) = \delta_{F_t(x_t, u_t)}(dx_{t+1})$. Finally, \eqref{eq:p_star_dc_linear_path_cost3} is obtained by the recursive substitution of \eqref{eq:p_star_dc_linear_path_cost2}, and $\{\mathcal{X}_{t+1:T},\,\mathcal{U}_{t:T-1}|u_t\in B_{U_t}\}$ represents a collection paths such that $u_t\in B_{U_t}$. \par
We use the above representation of \(Q^{*}_{U_t|X_t}\) to sample an action \(u_{t}\) from it. Let $r_t(i)$ be the exponentiated path cost of the sample path $i$:
\begin{equation}
\label{eq:dc_path_reward}
r_t(i):= \exp\left(-\frac{1}{\lambda}C_{t:T}(x_{t:T}(i), u_{t:T-1}(i))\right)
\end{equation}
and $r_t:=\sum_{i=1}^N r_t(i)$.
For each $t\in\mathcal{T}$, we introduce a piecewise constant, monotonically non-decreasing function $F_t: [0,N]\rightarrow [0,r_t]$ by
\[
F_t(x)=\sum_{i=1}^{\lfloor x \rfloor} r_t(i).
\]
where $\lfloor x \rfloor$ denotes $\mathrm{floor}(x)$, i.e., the greatest integer less than or equal to $x$. The function $F_t(x)$ is represented in Figure \ref{fig:f_func}.

\begin{figure}[t]
\centering
{\includegraphics[scale=0.1]{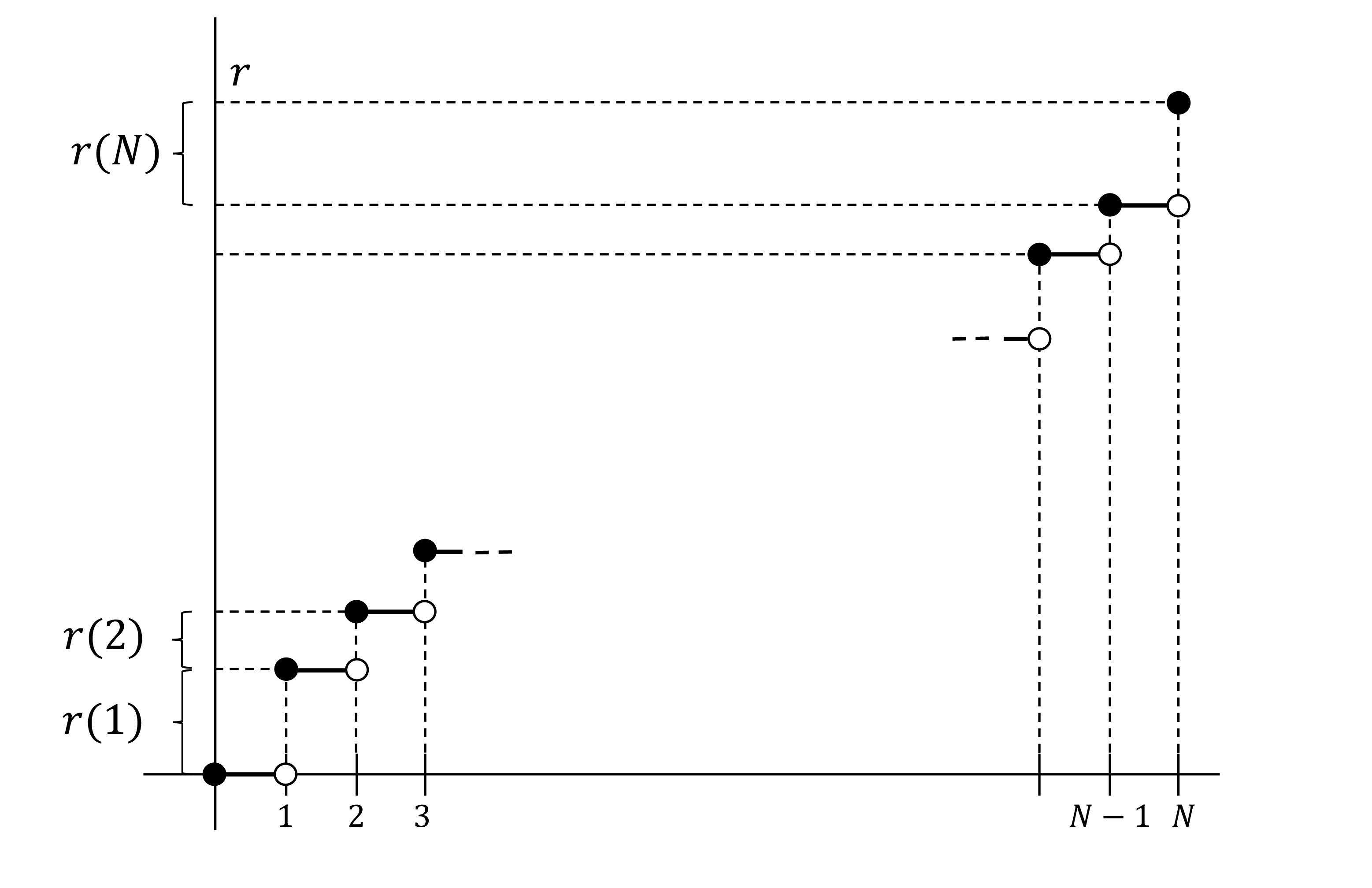}}
    \caption{Function $F_{t}(x)$.}
    \label{fig:f_func}
\end{figure}

\begin{algorithm}[t]
\caption{Sampling $u_t$ approximately from $Q_{U_t|X_t}^*(\cdot|x_t)$ by Monte Carlo simulations}\label{Algo: Q*}
\begin{algorithmic}[1]
\Require{Initial state $x_0$}
\For {$t\in\mathcal{T}$}
\State Sample $N$ paths $\{x_{t:T}(i), u_{t:T-1}(i)\}_{i=1}^N$ starting from $x_t$ under the reference distribution $R$. \label{alg:randomsamplepath}
\State For each sample path $i$, compute the exponentiated path cost $r_t(i)$ by \eqref{eq:dc_path_reward}.\label{alg:reward of a sample path}
\State Compute $r_t:=\sum_{i=1}^N r_t(i)$.\label{alg:total reward}
\State Generate a random variable $d$ according to $d \sim \mathrm{unif}[0,r_t]$. \label{alg:randomoutputpath}
\State Select a sample ID by $j_t\leftarrow F_t^{-1}(d)$.  \label{alg:sample ID} 
\State Select a control input as $u_t\leftarrow u_t(j_t)$.
\EndFor
\end{algorithmic}
\end{algorithm}
Notice that the inverse $F_t^{-1}$ of $F_t$ defines a mapping $F_t^{-1}: [0,r_t]\rightarrow \{1, 2, ... , N\}$.
To generate a sample $u_t$ approximately from the optimal distribution $Q_{U_t|X_t}^*$, we propose Algorithm \ref{Algo: Q*}.
We first, generate a random variable $d$ according to $d \sim \text{unif}[0,r_t]$. Then, we select a sample ID by $j_t\leftarrow F_t^{-1}(d)$. Finally, the control input adopted in the $j_t$-th sample path at time step $t$ is selected as $u_t$, i.e., $u_t\leftarrow u_t(j_t)$. Theorem \ref{theorem:lln} proves that as the number of Monte Carlo samples tends to infinity, Algorithm \ref{Algo: Q*} samples $u_t$ from the optimal distribution $Q_{U_t|X_t}^*(\cdot|x_t)$.

\begin{theorem}
\label{theorem:lln}
Let $B_{U_t}\in\mathcal{B}(\mathcal{U}_t)$ be a Borel set. Suppose for a given collection of sample paths $\{x_{t:T}(i), u_{t:T-1}(i)\}_{i=1}^N$, $u_t$ is computed by Algorithm  \ref{Algo: Q*} and the probability of $u_t\in B_{U_t}$ is denoted by $\Pr\{u_t \in B_{U_t} |\{x_{t:T}(i),  u_{t:T-1}(i)\}_{i=1}^N \}$. Then, as $N \rightarrow \infty$
\[
\Pr\{u_t \in B_{U_t} |\{x_{t:T}(i),  u_{t:T-1}(i)\}_{i=1}^N \} \overset{a.s.}{\rightarrow} Q^*_{U_t|X_t}(B_{U_t}|x_t). 
\]
\end{theorem}
\begin{proof}
    Let $\mathcal{I}_{B_{U_t}}$ be the set of indices of sample paths for which an action in $B_{U_t}$ is taken at time step $t$, i.e., \[\mathcal{I}_{B_{U_t}}=\{i\in \{1, 2, \ldots , N\} | u_t(i)\in B_{U_t}\}.\] Define the sum of the exponentiated path costs of the paths in $\mathcal{I}_{B_{U_t}}$ as
\[r_{B_{U_t}}=\sum_{i\in\mathcal{I}_{B_{U_t}}} r_t(i).\]
By construction of Algorithm \ref{Algo: Q*},
\begin{equation}\label{eq:rB/r}
    \Pr\{u_t \in B_{U_t} |\{x_{t:T}(i),  u_{t:T-1}(i)\}_{i=1}^N \}=\frac{r_{B_{U_t}}}{r_t}.
\end{equation}
Now, from \eqref{eq:mc_dc_phi_convergence}, as $N\rightarrow \infty$, we get
\[
\frac{r_t}{N}=\frac{1}{N}\sum_{i=1}^N \exp\left(-\frac{C_{t:T}(x_{t:T}(i), u_{t:T-1}(i))}{\lambda}\right)\overset{a.s.}{\rightarrow} Z_t(x_t). 
\]
Similarly, as $N\rightarrow \infty$,
\begin{align*}
\frac{r_{B_{U_t}}}{N}=&\frac{1}{N}\sum_{i\in\mathcal{I}_{B_{U_t}}} \exp\left(-\frac{C_{t:T}(x_{t:T}(i), u_{t:T-1}(i))}{\lambda}\right) \\
\overset{a.s.}{\rightarrow}& \int_{\{\mathcal{X}_{t+1:T},\; \mathcal{U}_{t:T-1}|u_t\in B_{U_t}\}} \exp\left(-\frac{C_{t:T}(x_{t:T},  u_{t:T-1})}{\lambda}\right)\times R(dx_{t+1:T},du_{t:T-1}|x_t)
\end{align*}
Thus, from \eqref{eq:p_star_dc_linear_path_cost3} and \eqref{eq:rB/r}
\begin{align*}
\Pr\{u_{t}\in B_{U_t} |\{x_{t:T}(i),  u_{t:T-1}(i)\}_{i=1}^N\}\overset{a.s.}{\rightarrow}& \frac{1}{Z_t(x_t)}\!\!\int_{\{\mathcal{X}_{t+1:T},\; \mathcal{U}_{t:T-1}|u_t\in B_{U_t}\}}& \!\!\!\!\!\!\!\!\!\!\! \exp\left(-\frac{C_{t:T}(x_{t:T},  u_{t:T-1})}{\lambda}\right)\\
& & \times R(dx_{t+1:T}, du_{t:T-1}|x_t)\\
=&\;Q_{U_t|X_t}^*(B_{U_t}|x_t). &
\end{align*} 
\end{proof}
\vspace{1mm}

We showed that under Assumption \ref{assump: deterministic law}, the optimal deceptive policies can be synthesized using path integral control. Algorithm \ref{Algo: Q*} allows the deceptive agent to numerically compute optimal actions via Monte Carlo simulations without explicitly synthesizing the policy. Since Monte Carlo simulations can be efficiently parallelized, the agent can generate the optimal control actions online.

\section{Simulation Results}
\begin{figure*}
     \centering
       \!\!\!\!\!\!\!\!\!\!\!\!\!\!\!\begin{tabular}{c c}
       \includegraphics[scale=0.75]{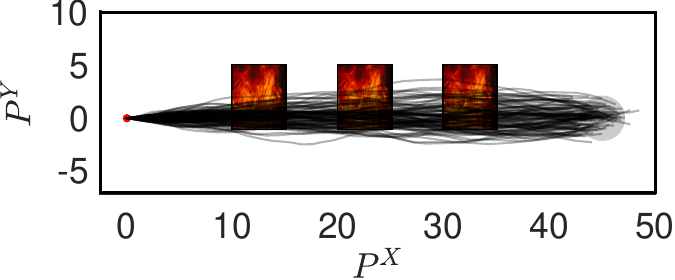} &\includegraphics[scale=0.75]{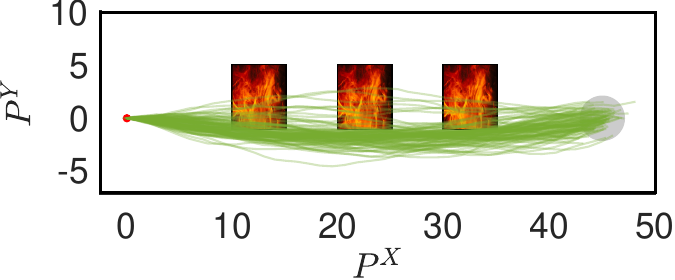} \\
       (a) Paths under $R$, $\Pr^{\mathrm{safe}} = 0.04$  & (b) Paths under $\widehat{Q}^*$ with $\lambda=3$, $\Pr^{\mathrm{safe}} = 0.48$ \\
       \vspace{1mm}\\
      \includegraphics[scale=0.75]{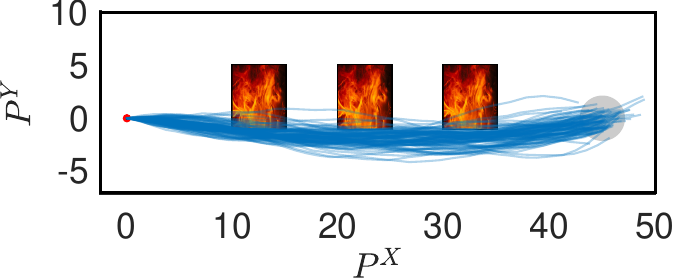} &\includegraphics[scale=0.75]{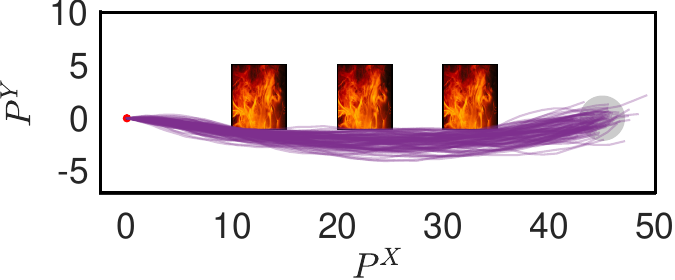} \\
       (c) Paths under $\widehat{Q}^*$ with $\lambda=2$, $\Pr^{\mathrm{safe}} = 0.62$  & (d) Paths under $\widehat{Q}^*$ with $\lambda=0.5$, $\Pr^{\mathrm{safe}} = 0.94$ \\
       \end{tabular}
         \caption{A unicycle navigation problem. The start position is shown by a red dot, and the goal region by a disk colored in gray. $100$ sample paths generated under the reference policy $R$ and the agent's policy $\widehat{Q}^*$ with three values of $\lambda$ are shown. The probability of safe paths $\Pr^{\mathrm{safe}}$ are noted below each case.} 
         \label{Fig. simulation trajs}
 \end{figure*}
 In this section, we validate the path-integral-based algorithm proposed to generate optimal deceptive control actions. The problem is illustrated in Figure \ref{Fig. simulation trajs}. A supervisor wants an agent to start from the origin and reach a disk of radius $G^R$ centered at $\begin{bmatrix} G^X & G^Y\end{bmatrix}^\top$ (shown in gray color) as fast as possible. The supervisor also expects the agent to inspect the region on the way. To encourage exploration and to provide robustness against unmodeled dynamics, the supervisor designs a randomized reference policy. The agent, on the other hand, wishes to avoid the regions on the way that are covered under fire, as shown in Figure \ref{Fig. simulation trajs}. Let these regions be represented collectively by $\mathcal{X}^{\mathrm{fire}}$. Suppose the agent's dynamics are modeled by a unicycle model as:
\begin{align*}
    P^X_{t+1} & = P^X_{t} + S_t\cos\Theta_t h, & \quad P^Y_{t+1} &= P^Y_{t} + S_t\sin\Theta_t h, &\quad\\
    S_{t+1}  &= S_{t} + A_t h, &\quad
    \Theta_{t+1} & = \Theta_{t} + \Omega_t h &
\end{align*}
where $(P^X_t, P^Y_t)$, $S_t$, and $\Theta_t$ denote the $x-y$ position, speed, and the heading angle of the agent at time step $t$, respectively. The control input $U_t\coloneqq\begin{bmatrix}A_t & \Omega_t\end{bmatrix}^T$ consists of acceleration $A_t$ and angular speed $\Omega_t$. $h$ is the time discretization parameter used for discretizing the continuous-time unicycle model. For this simulation study, we set $h=1$. Note that the agent's dynamics is deterministic as per Assumption \ref{assump: deterministic law}; however, the control input $U_t$ can be stochastic. Suppose the supervisor designs the reference policy $R$ as a Gaussian probability density with mean $\overline{u}_t$ and covariance $\Sigma_t$:
\begin{equation*}
  R_{U_t|X_t}(\cdot|x_t) = \frac{\exp\left[-\frac{1}{2}(u_t-\overline{u}_t)^\top\Sigma_t^{-1}(u_t-\overline{u}_t)\right]}{\sqrt{(2\pi)^2|\Sigma_t|}}.
\end{equation*}
The mean $\overline{u}_t\coloneqq\begin{bmatrix}\overline{a}_t & \overline{\omega}_t \end{bmatrix}^\top$ is designed using a proportional controller as 
\begin{equation*}
    \overline{A}_t = -k_{A} (S_t - S_t^{\mathrm{desired}}),\quad
    \overline{\Omega}_t = -k_{\Omega} (\Theta_t - \Theta_t^{\mathrm{desired}}) 
\end{equation*}
where $k_{A}$ and $k_{\Omega}$ are proportional gains and $S_t^{\mathrm{desired}}$, $\Theta_t^{\mathrm{desired}}$ are computed as


\begin{equation*}
  S_t^{\mathrm{desired}}  = \frac{\left\lVert \begin{bmatrix}G^X\\G^Y\end{bmatrix} - \begin{bmatrix}P^X_t\\P^Y_t\end{bmatrix}\right\rVert}{T-t} ,\quad \Theta_t^{\mathrm{desired}}\!  =\! \tan^{\!-\!1}\!\!\left(\!\frac{G^Y-P^Y_t}{G^X-P^X_t}\!\right). 
\end{equation*}

As mentioned before, the agent wishes to avoid the region $\mathcal{X}^{\mathrm{fire}}$. Suppose the cost function $C_{0:T}$ is designed as
\begin{equation*}
   C_{0:T}(X_{0:T}, U_{0:T-1}) = \sum_{t=0}^{T}\mathds{1}_{[P^X_t\;P^Y_t]^{^\top}\in\mathcal{X}^{\mathrm{fire}}}
\end{equation*}
where $\mathds{1}_{[P^X_t\;P^Y_t]^{^\top}\in\mathcal{X}^{\mathrm{fire}}}$ represents an indicator function that returns $1$ when the agent is inside the region $\mathcal{X}^{\mathrm{fire}}$ and $0$ otherwise.
For this simulation, we set
\begin{equation*}
   \begin{bmatrix}G^X\\
    G^Y\end{bmatrix}=\begin{bmatrix}45\\
    0\end{bmatrix}\!\!,\quad \Sigma_t = \begin{bmatrix}0.5 & \!\!\!\!0\\
    0 & \!\!\!\!0.5\end{bmatrix}\!\!,\quad k_A = 0.1, \quad k_\Omega = 0.2, \quad T = 50.
\end{equation*}

The agent chooses its action at each time step using Algorithm \ref{Algo: Q*} where the number of samples is $N=10^5$. \par
Suppose $\widehat{Q}^*$ denotes the deceptive agent's distribution generated by the sampling-based Algorithm  \ref{Algo: Q*}. Figure \ref{Fig. simulation trajs} shows $100$ paths under the reference distribution R (Figure \ref{Fig. simulation trajs}(a)) and the agent's distribution $\widehat{Q}^*$ for three values of $\lambda$ (Figure \ref{Fig. simulation trajs}(b) - \ref{Fig. simulation trajs}(d)). A lower value of $\lambda$ implies that the agent cares less about its deviation from the reference policy and more about avoiding the region $\mathcal{X}^{\mathrm{fire}}$. A higher value of $\lambda$ implies the opposite. We also report $\Pr^{\mathrm{safe}}$, the percentage of paths that avoid $\mathcal{X}^{\mathrm{fire}}$. Under the reference distribution $R$, only $4\%$ of the paths are safe. On the other hand, more paths are safe under the agent's distribution $\widehat{Q}^*$, and as the value of $\lambda$ reduces, $\Pr^{\mathrm{safe}}$ increases. \par

Figure \ref{Fig. LLR} shows the expected log-likelihood ratio (with one standard deviation) with respect to time $t$ for three values of $\lambda$. The expected LLR is computed as follows. Algorithm \ref{Algo: Q*} selects a control input $u_k\leftarrow u_k(j_k)$ at time step $k$, where $j_k$ is a sample ID obtained from step \ref{alg:randomoutputpath}. From the construction of the algorithm, at each time step $k$, the probability of choosing the control input $u_k\leftarrow u_k(j_k)$ under the agent's distribution $\widehat{Q}^*$ is $r_k(j_k)/r_k$, where $r_k(j_k)$ and $r_k$ are computed by steps \ref{alg:reward of a sample path} and \ref{alg:total reward} of Algorithm \ref{Algo: Q*}. Whereas the probability of choosing the control input $u_k\leftarrow u_k(j_k)$ under the reference distribution is $1/N$. Therefore, using \eqref{eq:llr and KL}, the expected LLR upto time $t\in\mathcal{T}$ can be approximately computed as
\begin{align*}
     \mathbb{E}_{Q^*}\left[\log \frac{dQ^*_{X_{0:t}\times U_{0:t-1}}}{dR_{X_{0:t}\times U_{0:t-1}}}(x_{0:t}, u_{0:t-1})\right] = \mathbb{E}_{Q^*}\!\!\! \left[\sum_{k=0}^{t-1}\log\frac{dQ^*_{U_k|X_k}}{dR_{U_k|X_k}}\left(x_k, u_k\right)\right]\!\!\approx\!\frac{1}{N_{\widehat{Q}^*}}\!\!\sum_{i=1}^{N_{\widehat{Q}^*}}\sum_{k=0}^{t-1} \frac{r_k(j_k)/r_k}{1/N}.
\end{align*}
where $N_{\widehat{Q}^*}$ is the number of paths generated by repeatedly running Algorithm \ref{Algo: Q*}. Note that since we assume the system dynamics to be deterministic (Assumption \ref{assump: deterministic law}), once the control input $u_k$ is chosen at time step $k$, the state $x_{k+1}$ is uniquely determined. Therefore, while computing the expected LLR, we only need to consider the probabilities of choosing the control input $u_k$ under policies $\widehat{Q}^*$ and $R$. Figure \ref{Fig. LLR} shows that for a lower value of $\lambda$, the expected LLR is higher, i.e., more deviation of $\widehat{Q}^*$ from $R$.
\begin{figure}
     \centering
      \includegraphics[scale=0.6]{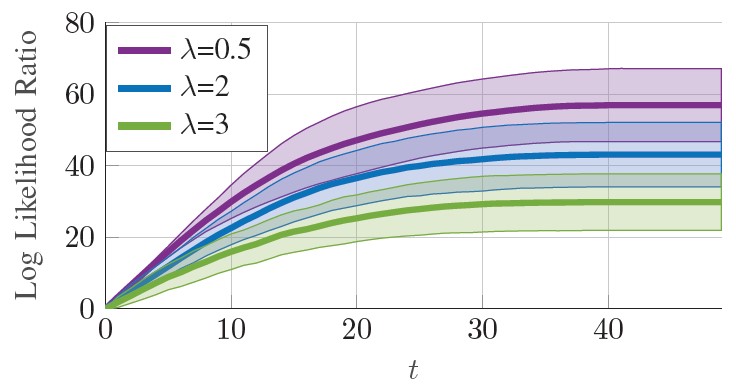}
         \caption{Expected LLR (with one standard deviation) with respect to time $t$ for three values of $\lambda$.} 
         \label{Fig. LLR}
 \end{figure}

\section{Publications}
\begin{itemize}
     \item \textbf{A. Patil}, M. Karabag, U. Topcu,  T. Tanaka, ``Simulation-Driven Deceptive Control via Path Integral  Approach," \textit{2023 Conference on Decision and Control}
\end{itemize}
\section{Future Work}
For future work, we plan to study the deception problem for continuous-time stochastic systems. We also plan to conduct the sample complexity analysis of the path integral approach to solve KL control problems. 

%% file: chapters/continuous_time_deception.tex
\chapter[Path Integral Methods for Synthesizing and Preventing Stealthy Attacks in Nonlinear Cyber-Physical Systems]{Deceptive Attack Synthesis and Its Mitigation for Nonlinear Cyber-Physical Systems: Path Integral Approach}
\label{Sec: Deceptive Attack Synthesis and Its Mitigation for Nonlinear Cyber-Physical Systems: Path Integral Approach}

\section{Motivation and Literature Review}
The work presented in this section is motivated by the need to protect cyber-physical systems (CPS) from stealthy attacks. As the next generation of engineered systems, CPS integrate computation, communication, control, and physical processes in a tightly interconnected manner \cite{kim2012cyber}. Because of this close integration across diverse technologies, CPS are susceptible to adversarial intrusions that can lead to severe ramifications for national economies, public safety, and even human lives \cite{poovendran2011special}. Notable cyber attacks—such as the StuxNet malware \cite{karnouskos2011stuxnet} and the Maroochy water bleach incident \cite{slay2007lessons}—underscore the critical importance of ensuring robust security measures for the safe operation of CPS.\par

With the increasing adoption of CPS, attack strategy and defense mechanism design have received considerable attention in the literature. In this section, we consider worst-case attack synthesis and its mitigation problems for nonlinear continuous-time CPS and propose their solutions using the path integral control method. Suppose a CPS is operated by a controller in a stochastic environment. Assume that a stealthy attacker hijacks the control authority and injects an attack signal to misguide the system. The attacker carefully designs the attack signals such that they can be disguised as the effects of natural disturbances in the CPS, such as controller noises and environmental disturbances. Suppose the controller continuously monitors the system to distinguish the attack signals from natural background noises by designing an appropriate detection test. However, knowing that the system is continuously monitored, a rational attacker will conduct a covert attack, maximizing the attack’s impact while avoiding detection. Intuitively, there is a tradeoff between the performance degradation an attacker can induce and how easy it is to detect the attack \cite{teixeira2012attack}. We solve two key problems in this section. First, we compute the worst-case stealthy attack signal assuming that the controller's policy is fixed and is known to the attacker. Next, we design the controller's policy to mitigate the risk of stealthy attacks, which will serve as the physical watermarking policy that countermeasures the potential stealthy attacks.\par

We design a stealthy attack synthesis problem motivated by hypothesis testing. 
Despite recent progress, existing applications of hypothesis testing frameworks to control systems \cite{bai2017data,kung2016performance,guo2018worst, shang2021worst} are restricted to discrete-time linear-Gaussian settings. The assumption of linear systems driven by Gaussian noises provides a significant advantage in designing and analyzing the statistical tests for detecting an attack, and the removal of such assumptions is marked as an open problem in \cite{sandberg2022secure}. Invoking Stein’s lemma, Bai et al. \cite{bai2017data} introduced the notion of $\epsilon$-stealthiness as measured by relative entropy. The worst-case degradation of linear control systems attainable by $\epsilon$-stealthy attacks was studied in \cite{guo2018worst, shang2021worst}. In this section, we generalize this analysis to the class of nonlinear continuous-time dynamical systems. We then propose the path integral approach for the synthesis of stealthy attack policies. Our preliminary result \cite{patil2023simulator} demonstrates the feasibility of such a path-integral-based attack synthesis in a discrete-time setting.\par

Next, we consider the minimax formulation in which the controller is also allowed to inject a control input into the system to mitigate the attack impact. Several mitigation strategies against stealthy attacks have been proposed in the literature \cite{sandberg2022secure}. The trade-off between control performance and system security was investigated under a stochastic game framework in \cite{miao2013stochastic}. The minimax games between the attacker and the controller have been studied by many authors in the systems and control community.  By formulating a zero-sum game, Zhang and Venkitasubramaniam [8] studied false data injection and detection problems
in infinite-horizon linear-quadratic-Gaussian systems. The works \cite{bai2017data, guo2018worst, shang2021worst} adopted
the hypothesis testing theory to characterize covert false data injection attacks against control systems. Despite recent progress, existing applications of hypothesis testing frameworks to control systems are limited to linear discrete-time settings. The goal of this section is to broaden the scope of the literature by formulating the aforementioned minimax game in continuous time for nonlinear cyber-physical systems. One effort in this direction is presented in one of our earlier works \cite{tanaka2024covert}. 


The conventional approach for attack synthesis necessitates explicit models of physical systems. However, in increasingly common scenarios, physical models are represented by various forms of ``digital twins," such as trained neural ODEs, which are easy to simulate but are not necessarily easy to express as models in the classical sense. This has elevated the importance of simulator-based CPS interaction, where the agent can directly use real-time simulation data to assist its decision-making without needing to construct an explicit model. In this section, we use a specific type of simulator-based control scheme known as the \textit{path integral control method}. Path integral control is a sampling-based algorithm employed to solve nonlinear stochastic optimal control problems numerically ~\cite{kappen2005path, theodorou2010generalized, williams2016aggressive}. It allows the policy designer to compute the optimal control inputs online using Monte Carlo samples of system paths. Such an algorithm was pioneered by Kappen \cite{kappen2005path} and has been generalized in the robotics and machine-learning literature \cite{theodorou2010generalized,williams2016aggressive}. The Monte Carlo simulations can be massively parallelized on GPUs, and thus the path integral approach is less susceptible to the curse of dimensionality \cite{williams2017model}.

\section{Contributions}
\begin{enumerate}
\item Kullback–Leibler (KL) control formulation for stealthy attack synthesis: We propose a KL control framework that models the trade-off between attack impact and detectability (Problem \ref{prob: KL}). This formulation is subsequently transformed into an equivalent quadratic-cost stochastic optimal control (SOC) problem  (Theorem \ref{thm:prob1_1}), enabling the synthesis of worst-case stealthy attacks.

\item Path-integral-based stealthy attack synthesis: We develop a novel path-integral-based method to synthesize worst-case stealthy attacks in real time for nonlinear continuous-time systems (Theorem \ref{Theorem: sol of SOC}). Notably, the proposed approach does not require an explicit system model or an explicit policy synthesis.

\item Zero-sum game formulation for attack mitigation: In order to mitigate the risk of stealthy attacks, we propose a novel zero-sum game formulation to model the competition between the attacker and the controller (Problem \ref{prob: minimax_KL}). We further show that this formulation is equivalent to both a risk-sensitive control problem and an H$_\infty$ control problem (Problems \ref{prob: risk-sensitive control} and \ref{prob: game}, respectively).  

\item Path-integral-based attack mitigation:  We provide a path-integral-based approach that allows the controller to synthesize the attack mitigating control inputs online, without relying on explicit models or policy synthesis (Theorems \ref{thm: risk-sensitive control} and \ref{thm: two-player game}).
\end{enumerate}

\section{Preliminaries}
\label{sec:formulation}

\subsection{Notation}
Given two probability measures $P$ and $Q$ on a measurable space $(\Omega, \mathcal{F})$, the Kullback-Leibler (KL) divergence from $Q$ to $P$ is defined as
$D(P\|Q)=\int_\Omega \log\frac{dP}{dQ}(\omega)P(d\omega)$
if the Radon-Nikodym derivative $\frac{dP}{dQ}$ exists, and $D(P\|Q)=+\infty$ otherwise. $\mathcal{N}(\mu, \Sigma)$ represents the Gaussian distribution with mean $\mu$ and covariance $\Sigma$.
Throughout this paper, we use the natural logarithm.

\subsection{System Setup}
{\color{black}
All the random processes considered in this paper are defined on the probability space $(\Omega, \mathcal{F}, P)$. Let $w_t$ be an $m$-dimensional standard Brownian motion with respect to the probability measure $P$, and let $\mathcal{F}_t\subset \mathcal{F}$ be an increasing family of $\sigma$-algebras such that $w_t$ is a martingale with respect to $\mathcal{F}_t$. In the sequel, $w_t$ will be used to model the natural disturbance. Consider the class of continuous-time nonlinear cyber-physical systems (shown in Figure~\ref{fig:cps_security})  with state $x_t\in\mathbb{R}^n$, control input $u_t\in\mathbb{R}^\ell$, and the input random process $v_t\in\mathbb{R}^m$:
\begin{equation}\label{continuous-time deception SDE}
dx_t=f_t(x_t) dt +g_t(x_t)u_tdt+h_t(x_t)dv_t.
\end{equation}
Under the nominal operation condition, we assume $v_t=w_t$ for all $0\leq t\leq T$, i.e.,  the natural disturbance directly enters into the system.
In this paper, we are concerned with the competition between the controller agent (showing blue in Figure~\ref{fig:cps_security}) and the attacker agent (showing red in Figure~\ref{fig:cps_security}) over the dynamical system \eqref{continuous-time deception SDE}.

\subsubsection{Controller}
The controller has a legitimate authority to operate the system \eqref{continuous-time deception SDE}. Her primary role is to compute the control input $u_t$ to minimize the expected cost of the system operation, which is denoted by $\mathbb{E}\left[\int_0^T c_t(x_t, u_t)dt\right]$ with some measurable functions $c(\cdot,\cdot)$. We assume that the controller applies a state feedback policy. Specifically, we impose a restriction $u \in \mathcal{U}$ where
\[
\mathcal{U}=\{u: \text{$u_t$ is $\mathcal{F}_t^x$-adapted It\^o process}\}
\]
and  $\mathcal{F}_t^x=\sigma(\{x_s:0\leq s\leq t\})\subset \mathcal{F}$ is a sigma algebra generated by the state random process $x_s, 0\leq s\leq t$.

The secondary role of the controller is to detect the presence of the attacker, who is able to alter the disturbance process $v_t$. We assume that the controller is able to monitor $v_t$. When an anomaly is found, the controller has the authority to trigger an alarm and halt the system's operation. We assume that the controller agent is always present.

\subsubsection{Attacker} Unlike the controller, the attacker agent may or may not be present. When the attacker is absent (or inactive), the system \eqref{continuous-time deception SDE} is under the nominal operating condition, i.e., $v_t=w_t$. When the attacker is present, she is allowed to inject a synthetic attack signal $v_t\neq w_t$. The class of admissible attack signals $\mathcal{V}$ is specified below:
\begin{definition}
\label{def:admissible_attack}
The attack signal $v_t\in\mathbb{R}^m, 0\leq t\leq T$ is in the admissible class $\mathcal{V}$ if it is an It\^o process of the form 
\begin{equation}
\label{eq:admissible_v}
dv_t=\theta_t(\omega)dt+dw_t, \; v_0=0
\end{equation}
where $\theta_t(\omega)\in\mathcal{R}^m$ satisfies the following conditions:
\begin{itemize}
\item[(i)] $\theta: [0,T]\times\Omega \rightarrow \mathbb{R}^m$ is $\mathcal{B}[0,T]\times \mathcal{F}$ measurable where $\mathcal{B}$ is the Borel $\sigma$-algebra;
\item[(ii)] $\theta_t(\omega)$ is $\mathcal{F}_t$-adapted; and
\item[(iii)] $P(\omega\in\Omega: \int_0^t |\theta_s(\omega)|ds<\infty \; \forall t \in [0,T])=1$.
\end{itemize}
\end{definition}
\begin{remark}
The admissible class $\mathcal{V}$ of attack signals allows the attacker to inject a time-varying, possibly randomized, bias term $\theta_t(\omega)\in\mathcal{R}^m$ to the disturbance input. The progressive measurability conditions (i) and (ii), and the $L_1$ integrability condition (iii) ensure that \eqref{eq:admissible_v} is a well-defined It\^o process. 
\end{remark}

The objective of the attacker is twofold.
First, she tries to maximize the expected cost $\mathbb{E}\left[\int_0^T c_t(x_t, u_t)dt\right]$ of the system operation. This can be achieved by altering the statistics of $v_t$ from those of the natural disturbance $w_t$. 
However, a significant change in the disturbance statistics will increase the chance of being detected. Therefore, the second objective of the attacker is to stay ``stealthy" by choosing $v_t$ that is statistically similar to $w_t$.
\begin{remark}
One can consider more general attack signals of the form
\begin{equation}
\label{eq:not_admissible_v}
dv_t=\theta_t(\omega)dt+\sigma_t(\omega)dw_t, \; v_0=0
\end{equation}
with a diffusion coefficient $\sigma_t(\omega)\neq 1$.
However, from the attacker's perspective, there's no benefit in such a generalization because the choice $\sigma_t(\omega)\neq 1$ only ``hurts" the stealthiness of the attack signal.
To see this, it is sufficient to consider the following two hypotheses: 
\begin{align*}
&H_0:  dv_t=dw_t \text{ (No attack)} \\
&H_1:  dv_t=\sigma dw_t, \sigma\neq 1 \text{ (Attack on the diffusion coefficient)} 
\end{align*}
and whether the detector can tell which model has generated an observed continuous sample path $v\in\mathcal{C}[0,T]$. It can be shown that there exists a hypothesis test $\phi: \mathcal{C}[0,T] \rightarrow \{H_0, H_1\}$ that returns a correct result with probability one (See Appendix \ref{sec: Attack on Diffusion Coefficient}). Therefore, we can restrict ourselves to the admissible class $\mathcal{V}$ we defined as in Definition~\ref{def:admissible_attack} without loss of generality.
\end{remark}
}

\begin{figure}[!tbp]
\centering
\includegraphics[scale=0.6]{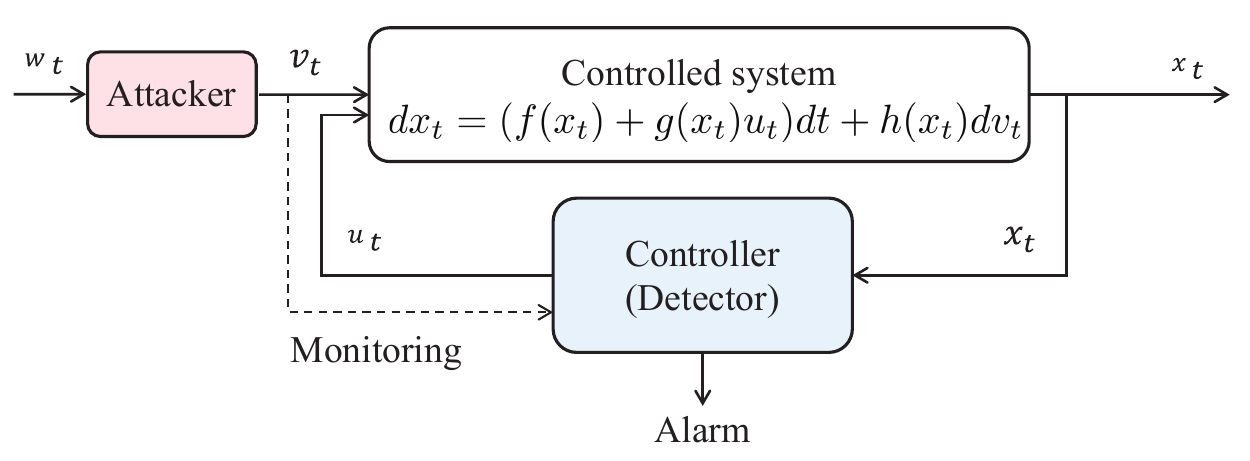}
\caption{Attacker vs controller/detector.}
\vspace{-5pt}
\label{fig:cps_security}
\end{figure}

\subsection{KL Divergence as a Stealthiness Measure}

{\color{black}
By controlling the drift term $\theta_t(\omega)$ in \eqref{eq:admissible_v}, the attacker can alter the distribution of the attack signal $v_t$.
To deal with the distributions in the space of continuous time paths, it is convenient to discuss it from the perspective of the change of measures. 
Even though $v_t$ in \eqref{eq:admissible_v} is not the standard Brownian motion under $P$ (unless $\theta_t(\omega)\equiv 0$), by Girsanov's theorem \cite{oksendal2013stochastic}, there always exists an alternative measure $Q$ in which $v_t$ is the standard Brownian motion. Girsanov's theorem also guarantees that such a measure $Q$ is equivalent to $P$ and hence the Radon-Nikodym derivative $\frac{dP}{dQ}(\omega)$ exists.

If an attack policy is fixed, then the controller's task is to determine whether an observed signal $v_t$ is generated by the natural disturbance, i.e.,
\begin{equation}
H_0: dv_t=dw_t
\end{equation}
or it is a synthetic attack signal, i.e.,
\begin{equation}
H_1: dv_t=\theta_t(\omega)dt+dw_t.
\end{equation}
This is a binary hypothesis testing problem.
Let $A\in \mathcal{F}$ be an event in which the controller triggers the alarm.
The quality of a hypothesis test (a choice of $A$) is usually evaluated in terms of the probability of type-I error (false alarm) and the probability of type-II error (failure of detection). Using notations introduced so far, the probability of type-I error can be expressed as $Q(A)$, whereas the probability of type-II error can be expressed as $P(A^c)$. 

It is well known from the Neyman-Pearson Lemma that the optimal trade-off between the type-I and type-II errors can be attained by a threshold-based likelihood ratio test. That is, for a fixed type-I error probability, the region $A$ that minimizes the type-II error probability is given by 
\begin{equation}
A=\left\{\omega\in\Omega: \log\left(\frac{dP}{dQ}(\omega)\right)\geq \tau\right\}
\end{equation}
where $\tau$ is an appropriately chosen threshold \cite{cvitanic2001generalized}.

Assuming that the controller adopts the Neyman-Pearson test (which is the most pessimistic assumption for the attacker), the attacker tries to decide on an attack policy that is difficult to detect. 
Unfortunately, for each attack policy $v_t$, it is generally difficult to compute type-I and type-II errors attainable by the Neyman-Pearson test analytically.
This poses a significant challenge in studying the attacker's policy.
Therefore, in this paper, we adopt an information-theoretic surrogate function that is more amenable to optimization -- namely, the KL divergence $D(P\|Q)=\mathbb{E}^P\log\frac{dP}{dQ}$ -- that serves as a ``stealthiness" measure.

The KL divergence has been adopted as a stealthiness measure in prior studies (e.g., \cite{bai2017data}), which can be justified by the following arguments:
The first argument is based on non-asymptotic inequalities (i.e., the inequalities that hold for any finite horizon lengths $T$). 
It follows from Pinsker's inequality and Bretagnolle-Huber inequality \cite[Theorem 14.2]{lattimore2020bandit} that
\begin{align}
Q(A)+P(A^c)&\geq 1-\sqrt{\frac{1}{2}D(P\|Q)} \label{eq:pinsker} \\
Q(A)+P(A^c)&\geq \frac{1}{2}\exp\left(-D(P\|Q)\right).
\end{align} 
These inequalities suggest that the attacker can enhance her stealthiness by choosing a policy that makes $D(P\|Q)$ small -- see Figure~\ref{fig:error_bounds}.
\begin{figure}[h]
\centering
\includegraphics[width=0.6\columnwidth]{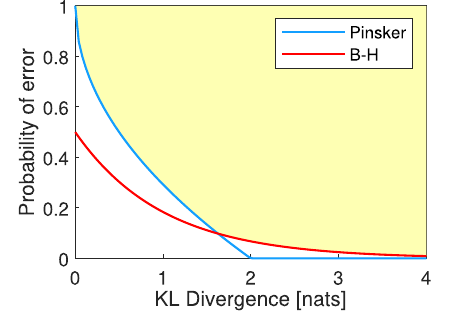}
\caption{Error bounds and achievable region (shown in yellow).}
\vspace{-5pt}
\label{fig:error_bounds}
\end{figure}\par
The second argument is asymptotic in the sense that it is concerned with how type-I and type-II errors behave as $T\rightarrow \infty$. 
The Chernoff-Stein lemma \cite{cover1999elements} states that when the type-I error probability is constrained, the exponent of the type-II error under the Neyman-Pearson test behaves as
\begin{equation}
\label{eq:error_asymptotes}
\begin{cases}
\max_\tau & -\log P(A^c) \\
\text{s.t. } & Q(A)\leq \alpha
\end{cases}=D(P\|Q) T+\mathcal{O}(\sqrt{T}).
\end{equation}
These results indicate that choosing a small $D(P\|Q)$ reduces the rate at which $P(A^c)$ decays to zero.\footnote{While \cite[Theorem 11.8.3]{cover1999elements} is specialized to i.i.d. samples, its generalization to ergodic samples \cite{polyanskiy2014lecture} has been used in \cite{bai2017data}. The continuous time result \eqref{eq:error_asymptotes} follows from \cite{tanaka2024covert}.}

\begin{remark}
Despite these justifications, the higher-order analysis of \eqref{eq:error_asymptotes} (e.g., \cite{lungu2024optimal}) suggests that the KL divergence $D(P\|Q)$ may not be an appropriate measure of stealthiness in the regime of finite $T$. Developments of stealthiness measures that outperform $D(P\|Q)$ in the finite data-length regime are critical for some applications (such as the fastest anomaly detection) and are postponed as an important future work.
\end{remark}
}

\subsection{Problem Formulation}
{\color{black}
Summarizing the discussion so far, this subsection formally states the problem studied in this paper.
\subsubsection{Stealthy Attack Synthesis}\label{sec: Worst-Case Attack Synthesis}
We first formulate a problem for optimal attack policy synthesis under a simplifying assumption that the controller's policy $u_t$ is fixed and is known to the attacker. Adopting $D(P\|Q)$ as the stealthiness measure, the attacker is incentivized to keep $D(P\|Q)$ small while maximizing the expected cost $\mathbb{E}\left[\int_0^T c_t(x_t, u_t)dt\right]$.
Introducing a constant $\lambda>0$ balancing these two requirements, the problem is formulated as follows:
\begin{problem}[KL control problem]\label{prob: KL}
\begin{equation}
\begin{aligned}
\label{eq:kld_soft}
\max_{v \in \mathcal{V}}\; &  \mathbb{E}^P\left[\int_0^T c_t(x_t,u_t)dt\right] - \lambda D(P \| Q)\\
\text{s.t.} \; & dx_t\!=\!f_t(x_t) dt +g_t(x_t)u_tdt+h_t(x_t)dv_t.
\end{aligned}
\end{equation}
\end{problem}
\vspace{2mm}
The term ``KL control problem" is inherited from \cite{theodorou2012relative} where problems with a similar structure were studied.
}

\vspace{3mm}
\subsubsection{Attack Risk Mitigation}\label{Sec: Attack Risk Mitigation}
{\color{black} Here, we consider the scenario shown in Figure~\ref{fig:cps_security} where the controller is now able to apply a control input $u_t$ to combat with the potential attack input $v_t$. As before, we assume the existence of an attack detector to detect the stealthy attack. We model the competition between the controller and the attacker as a minimax game in which the controller acts as the minimizer and the attacker acts as the maximizer. This results in the following \emph{minimax KL control} problem, which is a two-player dynamic zero-sum game between the controller (cost minimizer) and the attacker (cost maximizer):

\begin{problem}[Minimax KL control problem]\label{prob: minimax_KL}
 \begin{equation}
\label{eq:minmax_kl}
\begin{aligned}
\min_{u\in\mathcal{U}} & \max_{v\in\mathcal{V}} \mathbb{E}^{P} \left[ \int_0^T c_t(x_t, u_t)dt \right]-\lambda D(P \| Q)\\
\text{s.t.} \; & dx_t\!=\!f_t(x_t) dt\! +g_t(x_t)u_tdt+h_t(x_t)dv_t
\end{aligned}
\end{equation}   
We assume the constant $\lambda>0$ is known to both players in advance.
\end{problem}
}

\section{Stealthy Attack Synthesis via Path Integral Approach}
\label{sec:deception}
{\color{black}
This section summarizes the main technical results concerning Problem 1.
We present structural properties of the optimal attack signal $v_t$ (Theorem~\ref{thm:prob1_1}) and show how to numerically compute $v_t$ for real-time implementations.

Our first result states that Problem 1 reduces to a quadratic-cost stochastic optimal control (SOC) problem in terms of the bias input $\theta_t$:
\begin{problem}[Quadratic-cost SOC problem]\label{prob: SOC}
\begin{align}
\max_\theta\; &  \mathbb{E}^P\left[\int_0^T \left\{c_t(x_t,u_t)-\frac{\lambda}{2}\|\theta_t\|^2\right\}dt\right] \label{eq:soc_theta}\\
\text{s.t.} \; & dx_t\!=\!f_t(x_t) dt +g_t(x_t)u_tdt+h_t(x_t)(\theta_t dt + dw_t) \nonumber \\
&  \text{where $\theta_t$ satisfies conditions (i)-(iii) in Definition~\ref{def:admissible_attack}}. \nonumber
\end{align}
\end{problem}
\begin{theorem}
\label{thm:prob1_1}
Let $u_t$ be a given $\mathcal{F}_t^x$-adapted process. Then, the following statements hold:
\begin{itemize}
\item[(i)] Problem 1 is equivalent to the quadratic-cost SOC problem \eqref{eq:soc_theta}.
\item[(ii)] There exists a deterministic state feedback policy (which can be written as $\theta_t(x_t)$) for Problem \ref{prob: SOC}.
\item[(iii)] An optimal attack policy $v_t$ (governed by It\^o process  \eqref{eq:admissible_v}) is characterized by 
\begin{equation}
\label{eq:theta_opt}
\theta_t^*(x_t)=\frac{1}{\lambda}h_t^\top\partial_xV_t(x_t)
\end{equation}
where $V_t$ solves the Hamilton-Jacobi Bellman equation
\begin{equation}\label{eq:HJB in V}
\partial_t V_t =  -\frac{1}{2\lambda}\left(\partial_x{V}_t\right)^\top h_th_t^\top\partial_x{V}_t-c_t-(f_t+g_tu_t)^\top\partial_x{V}_t-\frac{1}{2}\text{Tr}\left(h_th_t^\top\partial^2_x{V}_t\right).
\end{equation}
\end{itemize}
\end{theorem}
\begin{remark}
Theorem~\ref{thm:prob1_1} implies that the optimal attack signal can be written as $dv_t=\theta_t(x_t)dt+dw_t$, i.e., an addition of a deterministic bias $\theta_t(x_t)$ computed by \eqref{eq:theta_opt} to the natural disturbance $w_t$.
This is in contrast to the result in discrete-time \cite{bai2017data}, where the optimal attack is a randomized policy.
\end{remark}
\begin{proof}
Notice that the KL divergence $D(P \|Q)$ term in Problem \ref{prob: KL} can be rewritten as

\begin{equation*}
  D(P \|Q)= \mathbb{E}^P \log \frac{dP}{dQ} 
\end{equation*}
Recall that $P$ is the measure in which $w_t$ is a standard Brownian motion and $Q$ is the measure in which the attack signal $v_t$ (governed by It\^o process  \eqref{eq:admissible_v}) is a standard Brownian motion. Using the Girsanov theorem \cite{oksendal2003stochastic}, we obtain
\begin{subequations}
    \label{eq:KL_girsanov}
    \begin{align}
D(P \|Q)= &\mathbb{E}^P \log \frac{dP}{dQ} \nonumber\\
=&\mathbb{E}^P \left[ \int_0^T \theta_t^\top(\omega) dw_t + \frac{1}{2}\int_0^T \|\theta_t(\omega)\|^2 dt\right] \label{eq:KL_girsanov-a}\\
=&\frac{1}{2}\mathbb{E}^P \left[ \int_0^T \|\theta_t(\omega)\|^2 dt \right]. \label{eq:KL_girsanov-b}
\end{align}
\end{subequations}
Notice that the term $\int_0^T \theta_t(\omega)^\top dw_t$ in \eqref{eq:KL_girsanov-a} is an It\^o integral. We obtain \eqref{eq:KL_girsanov-b} using the following property of It\^o integral \cite[Chapter 3]{oksendal2013stochastic}:
\begin{equation*}
    \mathbb{E}^P \int_0^T \theta_t^\top(\omega) dw_t=0.
\end{equation*} 
Equation \eqref{eq:KL_girsanov} proves statement (i). In order to prove statements (ii) and (iii), we use the \emph{dynamic programming principle}. We introduce the value function $V_t(x_t)$ for each time $t\in[0,T)$ and the state $x_t\in\mathbb{R}^n$ as 
\begin{equation}
\label{eq:value_f}
V_t(x_t):=\max_\theta  \mathbb{E}^P\int_t^T \left(c_t(x_s,u_s) - \frac{\lambda}{2}\|\theta_s\|^2\right)ds.
\end{equation} 
 Let us define \begin{equation}
\label{eq:value_f_min}
\overline{V_t}(x_t):=\min_\theta  \mathbb{E}^P\int_t^T \left(-c_s (x_s,u_s) + \frac{\lambda}{2}\|\theta_s\|^2\right)ds.
\end{equation} 
Note that $V_t(x_t)=-\overline{V_t}(x_t)$. The stochastic Hamilton-Jacobi-Bellman (HJB) equation \cite{fleming2006controlled, stengel1994optimal} associated with \eqref{eq:value_f_min} is expressed as follows:
\begin{equation}\label{eq:HJB}
  -\partial_t\overline{V}_t =  \min_{{\theta}_t}\Big[\frac{\lambda}{2}\|\theta_t\|^2+\left(f_t+g_tu_t + h_t\theta_t\right)^{\top}\partial_x\overline{V}_t
   -c_t+\frac{1}{2}\text{Tr}\left(h_t h_t^\top\partial^2_x\overline{V}_t\right)\Big].
\end{equation}
Solving \eqref{eq:HJB}, we get the optimal $\theta_t$ as
\begin{equation}\label{eq:theta_star}
    \theta_t^*(x_t) = -\frac{1}{\lambda}h_t^\top\partial_x\overline{V_t}(x_t).
\end{equation}
Putting \eqref{eq:theta_star} in \eqref{eq:HJB}, we get
\begin{equation}\label{eq:HJB2}
  -\partial_t\overline{V}_t =  -\frac{1}{2\lambda}\left(\partial_x\overline{V}_t\right)^\top h_th_t^\top\partial_x\overline{V}_t-c_t +(f_t+g_tu_t)^\top\partial_x\overline{V}_t+\frac{1}{2}\text{Tr}\left(h_th_t^\top\partial^2_x\overline{V}_t\right).
\end{equation}
Since $V_t(x_t)=-\overline{V_t}(x_t)$, this proves statements (ii) and (iii).
\end{proof}
Generally, it is challenging to compute \eqref{eq:theta_opt} since it requires the solution $V_t(x_t)$ to a nonlinear, possibly high-dimensional, partial differential equation (PDE) \eqref{eq:HJB in V}.
Fortunately, the structure of the optimal control problem \eqref{eq:soc_theta} allows for an application of the \emph{path integral method} \cite{kappen2005path},  offering a Monte-Carlo-based attack signal synthesis. The applicability of the path-integral method to the quadratic-cost SOC problems shown in Problem \ref{prob: SOC} was first pointed out by the work of Kappen\cite{kappen2005path}.
For each time $t\in[0,T)$ and the state $x_t\in\mathbb{R}^n$, the path-integral method allows the adversary to compute the optimal attack signal $\theta_t(x_t)$ by evaluating the path integrals along randomly generated trajectories $\{x_s, u_s\}$, $t\leq s \leq T$ starting from $x_t$. The next result provides the details.
}

\begin{theorem}\label{Theorem: sol of SOC}
  The solution of \eqref{eq:value_f} exists, is unique and is given by \begin{equation}
\label{eq:value_f2}
V_t(x_t)=\lambda \log \mathbb{E}^Q \left[\exp \left\{\frac{1}{\lambda}\int_t^T c_s(x_s, u_s)ds \right\}\right]. 
\end{equation}  
Furthermore, the optimal bias input $\theta_t^*(x_t)$ is given by
\begin{equation}\label{eq: theta_star} \!\!\!\!\!\theta_t^*dt\!=\!\mathcal{H}_t(x_t)\frac{\mathbb{E}^Q\left[\text{exp}{\left\{\frac{1}{\lambda}\int_t^T\!\! c_s(x_s, u_s)ds\right\}}h_t(x_t)d{w}_t\right]}{\mathbb{E}^Q\left[\text{exp}{\left\{\frac{1}{\lambda}\int_t^T \!\!c_s(x_s, u_s)ds\right\}}\right]} 
\end{equation}
where the matrix $\mathcal{H}_t(x_t)$ is defined as
\begin{equation*}
    \mathcal{H}_t(x_t) = h_t^\top(x_t)\left(h_t(x_t)h_t(x_t)^\top\right)^{-1}.
\end{equation*}
\end{theorem}
\begin{proof}
 In Theorem \ref{thm:prob1_1}, we proved that the value function $\overline{V}_t(x_t)$ \eqref{eq:value_f_min} satisfies the PDE \eqref{eq:HJB2}. We introduce the exponential transformation of the value function $\overline{V}_t(x_t) = -\lambda\log\Psi_t(x_t)$ (known as Cole-Hopf transformation \cite{kappen2005path}). This will reformulate \eqref{eq:HJB2} as:
\begin{equation}\label{eq:Psi}
    \!\partial_t\Psi_t\!=\!\frac{-c_t\Psi_t}{\lambda}\!-\!(f_t+g_tu_t))^\top\partial_x\Psi_t-\frac{1}{2}\text{Tr}\left(h_th_t^\top\partial^2_x\Psi\right).
\end{equation}
The PDE \eqref{eq:Psi} is linear in terms of $\Psi_t$ and is known as the backward Chapman-Kolmogorov PDE \cite{williams2017model}. According to the Feynman-Kac lemma \cite{oksendal2003stochastic}, the solution of the linear PDE \eqref{eq:Psi} exists and is unique in the sense that $\Psi_t$ solving \eqref{eq:Psi} is given by
\begin{equation}\label{eq:Psi_sol}
    \Psi_t(x_t) = \mathbb{E}^Q \left[\exp \left\{\frac{1}{\lambda}\int_t^T c_s(x_s, u_s)ds \right\}\right]
\end{equation}
where $Q$ is the probability measure in which the attack $v_t$ is a standard Brownian motion i.e., $\theta_s$, $t\leq s\leq T$ is zero. Since $V_t(x_t) = -\overline{V_t}(x_t) = \lambda\log\Psi_t(x_t)$, using \eqref{eq:Psi_sol}, we get the desired result \eqref{eq:value_f2}. Solving \eqref{eq:theta_opt}, i.e., taking the gradient of ${V_t(x_t)}$ with respect to $x$, we get the optimal bias input \eqref{eq: theta_star}.
\end{proof}
Since the right-hand side of \eqref{eq:value_f2} contains the expectation operation with respect to $Q$, using the strong law of large numbers \cite{durrett2019probability}, we can prove that as $N\rightarrow\infty$,
\[
\lambda \log \left[\frac{1}{N}\sum_{i=1}^N \exp \left\{\frac{1}{\lambda}\int_t^T c_s(x^i_s, u^i_s)ds\right\}\right] \overset{a.s.} {\rightarrow} V_t(x_t)
\]
where $\{x_s^i, u^i_s, t\leq s\leq T\}_{i=1}^N$ are randomly drawn sample paths from distribution $Q$.
Since $Q$ is the measure in which $v_t$ is the standard Brownian motion, generating such a sample ensemble is easy. It suffices to perform $N$ independent simulations of the dynamics $dx_s=f_s(x_s)ds +g_s(x_s)u_sds+ h_s(x_s)dw_s$. 
Similarly, the optimal bias input $\theta_t^*$ \eqref{eq: theta_star} can be readily computed by the same simulated ensemble $\{x_s^i, u_s^i, t\leq s\leq T\}_{i=1}^N$ and their path costs \cite{kappen2005path, williams2016aggressive}. According to the strong law of large numbers, as $N\rightarrow\infty$,

\begin{equation} \label{eq: theta_star_MC}\!\!\mathcal{H}_t(x_t)\frac{\frac{1}{N}\sum_{i=1}^N \text{exp}{\left\{\frac{1}{\lambda}\int_t^T c_s(x_s^{i}, u_s^{i})ds\right\}}h_t(x_t)\epsilon}{\frac{\sqrt{\Delta t}}{N}\sum_{i=1}^N \text{exp}{\left\{\frac{1}{\lambda}\int_t^T c_s(x_s^i, u_s^i)ds\right\}}} \overset{a.s.} {\rightarrow} \theta^*_t
\end{equation}
where $\epsilon\sim\mathcal{N}(0,1)$ and $\Delta t$ is the step size. Equation \eqref{eq: theta_star_MC} implies that if the attacker has a simulator engine that can generate a large number of sample trajectories $\{x_t^i, u_t^i\}_{i=1}^N$ from the distribution $Q$ starting from the current state-time pair $(x_t,t)$, then drift term $\theta_t^*$ \eqref{eq: theta_star_MC} of the optimal attack signal $v_t^*$  can be computed directly from the sample ensemble $\{x_t^i, u_t^i\}_{i=1}^N$. A notable feature of such a simulator-driven attack synthesis is that it computes the optimal stealthy attack signal in real time for nonlinear CPS without requiring an explicit model of the system or an explicit policy synthesis step. 

\section{Attack Risk Mitigation}\label{sec: Attack impact mitigation}
\label{sec:mitigation}

In this section, we turn our attention to the controller's problem, who is interested in mitigating the risk of stealthy attacks when the attacker is present. In Section \ref{sec: Connections with Risk-Sensitive Control and Two-Player Zero-Sum Stochastic Differential Game}, we establish a connection between the minimax KL control problem (Problem \ref{prob: minimax_KL}) with risk-sensitive control and two-player zero-sum stochastic differential game. In Sections \ref{sec: Risk-Sensitive Control via Path Integral Approach} and \ref{sec: Two-player Zero-Sum Stochastic Differential Game via Path Integral Approach}, we present the solutions of risk-sensitive control problems and two-player stochastic differential games, 
respectively, using the path integral approach. This allows the controller to synthesize the risk-mitigating control signals via Monte Carlo simulations.  

\subsection{Connections with Risk-Sensitive Control and Two-Player Zero-Sum Stochastic Differential Game}\label{sec: Connections with Risk-Sensitive Control and Two-Player Zero-Sum Stochastic Differential Game}
\begin{figure*}[h]
    \centering   \!\!\!\!\!\!\!\!\!\!\!\!\!\!\!\!\!\!\!\!\includegraphics[scale=0.37]{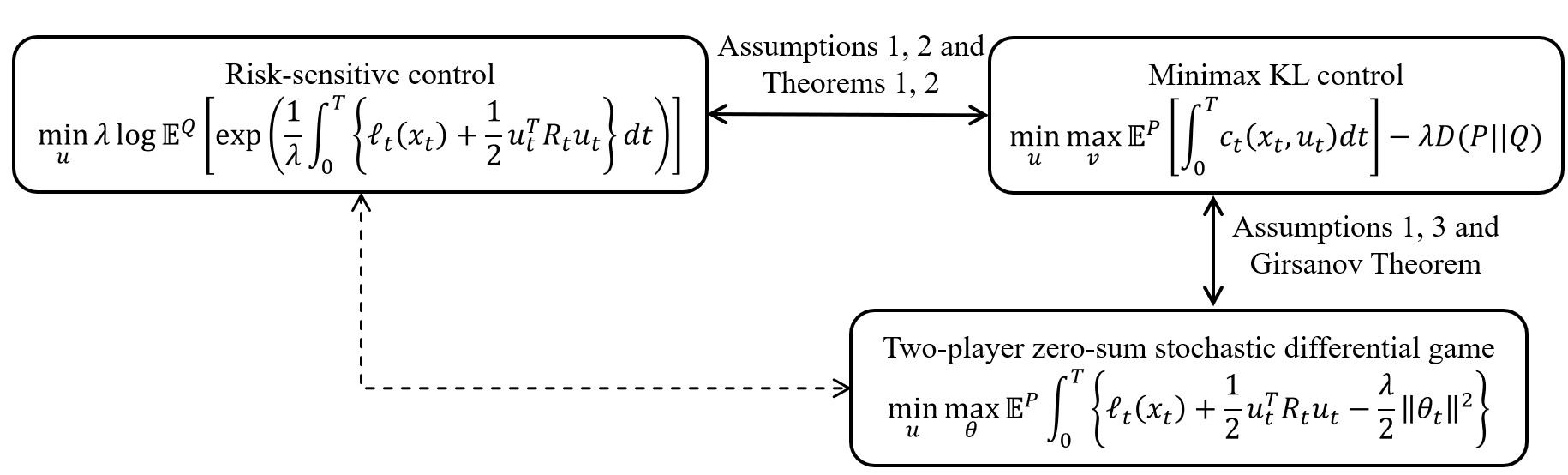}
    \vspace{-1ex}
    \caption{Connections between minimax KL control, risk-sensitive control, and two-player stochastic differential game}
    \label{fig:jacobson}
\end{figure*}
First, we establish the connection between the minimax KL control problem \eqref{eq:minmax_kl} and the \emph{risk-sensitive control problem} \cite{fleming2006controlled, whittle1981risk, jacobson1973optimal}. To this end, we make the following assumption:
\begin{assumption}\label{Assum:quadratic cost}
 the  cost function $c_t$ is quadratic in $u_t$:
\begin{equation}\label{eq: ct}
c_t(x_t, u_t)=\ell_t(x_t)+\frac{1}{2}u_t^\top R_t(x_t) u_t
\end{equation}
where $R_t(x_t)\succeq 0$ for all $t$. 
\end{assumption}
We can determine the value of the inner maximization problem in \eqref{eq:minmax_kl} using \eqref{eq:value_f2}. Using \eqref{eq:value_f2} and \eqref{eq: ct}, Problem \ref{prob: minimax_KL} can be written equivalently as follows:

\begin{problem}[Risk-sensitive control problem]\label{prob: risk-sensitive control}
 \begin{equation}
\label{eq:risk_sensitive}
\begin{aligned}
\min_{u} \; & \lambda \log \mathbb{E}^Q \!\left[\exp \!\left(\frac{1}{\lambda}\!\int_0^T\!\! \left\{\!\ell_t(x_t)\!+\!\frac{1}{2}u_t^{\!\top} R_t u_t\!\right\}dt\!\right)\!\right]\\
\text{s.t.} \; & dx_t\!=\!f_t(x_t) dt\! +g_t(x_t)u_tdt+h_t(x_t)dw_t.
\end{aligned}
\end{equation}   
\end{problem}
\vspace{3mm}
 The objective function in \eqref{eq:risk_sensitive} is also related to the risk measure called \emph{Entropic value-at-risk (EVaR)}
\cite{ahmadi2012entropic}.
This equivalence shows the intimate relationship between the minimax KL control problem and the risk-sensitive control problem. Problem \ref{prob: risk-sensitive control} is a class of risk-sensitive control problems with certain structural constraints. Specifically, the cost function is quadratic in the control input $u_t$, and the state equation is affine in both the control input $u_t$ and the noise $w_t$. This class of risk-sensitive control problems can be solved using the basic path-integral method \cite{broek2012risk}.\par

Now, we establish the connection between the minimax KL control problem \eqref{eq:minmax_kl} and the \emph{two-player zero-sum stochastic differential game}. Using the Girsanov's theorem (Theorem \ref{thm:prob1_1}-(i)) and assuming that the cost function $c_t$ is quadratic in $u_t$ \eqref{eq: ct}, Problem \ref{prob: minimax_KL} can be rewritten as

\begin{problem}[Two-player stochastic differential game]\label{prob: game}
\begin{equation}
\label{eq:h_inf}
\begin{aligned}
\min_u&\max_\theta \mathbb{E}^P\! \left[\int_0^T\!\! \left\{\ell_t(x_t)+\frac{1}{2}u_t^\top R_t u_t -\frac{\lambda}{2}\|\theta_t\|^2\right\}dt\right]\\
\text{s.t.} \; & dx_t\!=\!f_t(x_t) dt\! +g_t(x_t)u_tdt+h_t(x_t)\left(\theta_t dt \!+\!dw_t\!\right).
\end{aligned}
\end{equation}
\end{problem}
\vspace{3mm}
Problem \ref{prob: game} is a class of two-player zero-sum stochastic differential games with certain structural constraints. Specifically, the cost function is quadratic in $u_t$, $\theta_t$, and the state equation is affine in both the inputs $u_t$, $\theta_t$, and the noise input $w_t$. This class of stochastic differential games can be solved using the basic path-integral method \cite{patil2023risk}.\par

The equivalence of Problems \ref{prob: minimax_KL}, \ref{prob: risk-sensitive control}, and \ref{prob: game} under assumptions \ref{Assum:quadratic cost}, \ref{continuous-time deception Assumption: linearity_risk} and \ref{continuous-time deception Assumption: linearity} is shown in Figure~\ref{fig:jacobson}. Problem \ref{prob: game} is also known as the \emph{non-linear $H_\infty$ control} problem. We remark that the connection between the risk-sensitive control and the $H_\infty$ control (indicated by a dashed line in Figure~\ref{fig:jacobson}) for a class of linear systems is already known (e.g., \cite{jacobson1973optimal}). 
Figure~\ref{fig:jacobson} re-establishes this connection, possibly for a broader class of dynamical systems.

\subsection{Risk-Sensitive Control via Path Integral Approach}\label{sec: Risk-Sensitive Control via Path Integral Approach}

The applicability of the path-integral method to the class of risk-sensitive control problems shown in Problem \ref{prob: risk-sensitive control} was pointed out by the work of Broek et. al. in \cite{broek2012risk}. For each time $t\in[0,T)$ and the state $x_t\in\mathbb{R}^n$, the path-integral method, under a certain assumption, allows the controller to compute the optimal policy $u_t$ under the presence of worst-case attack by evaluating the path integrals along randomly generated trajectories $\{x_s\}$, $t\leq s \leq T$ starting from $x_t$. This can be derived from the fact that under a certain assumption, the value function of the risk-sensitive control problem \eqref{eq:risk_sensitive} can be computed by Monte-Carlo sampling. To see this, for each state-time pair $(x_t,t)$, introduce the value function 
\begin{equation}
\label{eq:value_f_risk}
\!\!V_t(x_t)\!=\min_u \lambda \log \mathbb{E}^Q \!\left[\exp \!\left(\frac{1}{\lambda}\!\int_t^T\!\! \left\{\!\ell_s(x_s)\!+\!\frac{1}{2}u_s^{\!\top} R_s u_s\!\right\}ds\!\right)\!\right]\!\!.
\end{equation}
First, we make the following assumption which is essential to solve \eqref{eq:value_f_risk} using the path integral method \cite{broek2012risk}.

\begin{assumption}\label{continuous-time deception Assumption: linearity_risk}
   For all $(x_t,t)$, there exists a constant $0<\xi<\lambda$ satisfying the following equation: 
   \begin{equation}\label{eq: linearizability_risk}
    \!\!h_t(x_t)  h_t^{\!\top}(x_t)\!\! =\! \xi g_t(x_t)R_t^{-1}(x_t)g_t^\top(x_t).
\end{equation}

\end{assumption}
\vspace{3mm}
Assumption \ref{continuous-time deception Assumption: linearity_risk} is similar to the assumption required in the path integral formulation of a stochastic control problem \cite{satoh2016iterative}. A possible interpretation of condition \eqref{eq: linearizability_risk} is that in a direction with high noise variance, the control cost of the risk-sensitive control problem \eqref{eq:risk_sensitive} has to be low. Therefore, the weight of the control cost $R_t$ need to be tuned appropriately for the given $\lambda$, $h_t(x_t)$ and  $g_t(x_t)$ for all $t$. See \cite{broek2012risk} for further discussion on this condition.

\begin{theorem}\label{thm: risk-sensitive control}
    Under Assumption \ref{continuous-time deception Assumption: linearity_risk}, the solution of \eqref{eq:value_f_risk} exists, is unique and is given by
    \begin{equation}
\label{eq:value_f2_risk}
V_t(x_t)=-\gamma \log \mathbb{E}^Z \left[\exp \left\{-\frac{1}{\gamma}\int_t^T \ell_s(x_s)ds \right\}\right]
\end{equation}
where $Z$ is the probability measure defined under the system dynamics \eqref{continuous-time deception SDE} when both $u_s$ and $\theta_s$, $t\leq s\leq T$ are zero and 
\begin{equation}\label{eq: gamma}
\gamma = \frac{\xi\lambda}{\lambda-\xi} .   
\end{equation}
Furthermore, the optimal controller signal $u_t^*$ is given by
\begin{equation}\label{eq: u_star_risk2} \!\!\!\!\!u_t^*dt\!=\!\mathcal{H}_t(x_t)\frac{\mathbb{E}^Z\left[\text{exp}{\left\{-\frac{1}{\gamma}\int_t^T \ell_s(x_s)ds\right\}}h_t(x_t)d{w}_t\right]}{\mathbb{E}^Z\left[\text{exp}{\left\{-\frac{1}{\gamma}\int_t^T \ell_s(x_s)ds\right\}}\right]} 
\end{equation}
where the matrix $\mathcal{H}_t(x_t)$ is defined as
\begin{equation*}
    \!\!\mathcal{H}_t(x_t)\!=\!R_t^{-1} g_t^\top(x_t)\!\left(\!g_t(x_t)R_t^{-1}g_t^\top(x_t)\! -\frac{1}{\lambda} h_t(x_t)h_t(x_t)^\top\!\right)^{\!-1}
\end{equation*}
\end{theorem}
\begin{proof}
   We use \textit{dynamic programming principle} and \textit{Feynman-Kac lemma} \cite{oksendal2003stochastic} to prove this theorem. The stochastic Hamilton-Jacobi-Bellman (HJB) equation \cite{broek2012risk} associated with \eqref{eq:value_f_risk} is expressed as follows:
\begin{equation}\label{eq:HJB_risk}
  -\partial_t{V}_t =  \min_{{u}_t}\Big[\frac{1}{2}u_t^\top\!R_tu_t+ \ell_t+ \left(\!f_t\!+\!g_tu_t\right)^\top\!\partial_xV_t
 +\frac{1}{2\lambda}\|h_t^\top\partial_xV_t\|^2 + \frac{1}{2}\text{Tr}\left(h_t h_t^\top\partial^2_x{V}_t\right)\Big]
\end{equation}
Solving \eqref{eq:HJB_risk}, we get the optimal $u_t$ as 
   \begin{equation}\label{eq:u_star_risk}
       u_t^*(x_t) = -R_t^{-1}g_t^\top\partial_xV_t(x_t).
   \end{equation}
   Putting \eqref{eq:u_star_risk} in \eqref{eq:HJB_risk}, we get 
   \begin{equation}\label{eq:HJB_risk2}
         \!\!-\partial_tV_t\!=\ell_t\! +\!f_t^\top\!\partial_xV_t\!+\!\frac{1}{2}\text{Tr}\left(h_th_t^\top\partial^2_xV_t\right)+\frac{1}{2}\!\left(\partial_xV_t\right)^\top\!\!\left(\frac{1}{\lambda}h_th_t^\top\!-\!g_tR_t^{-1}g_t^\top\right)\!\partial_xV_t.
\end{equation}
Let $\Psi_t(x_t)$ be the logarithmic transformation of the value function $V_t(x_t)$ (known as Cole-Hopf transformation) defined as $V_t(x_t) = -\gamma\log\Psi_t(x_t)$
where $\gamma>0$ is a proportionality constant to be defined. Applying this transformation of the value function to \eqref{eq:HJB_risk2} yields \begin{equation}\label{eq:Psi_risk}
 \begin{aligned}         \!\!\!\!\partial_t\Psi_t\!=&\frac{\ell_t\Psi_t}{\gamma}-\!\frac{1}{2}\text{Tr}\!\left(h_th_t^\top\!\partial^2_x\Psi_t\right)\!+\!\frac{1}{2\Psi_t}\!\left(\partial_x\Psi_t\right)^{\!T}\!\!h_th_t^\top\!\partial_x\Psi_t \\
         &+\!\frac{\gamma}{2\Psi_t}\!\left(\partial_x\Psi_t\right)^\top\!\!\left(\frac{1}{\lambda}h_th_t^\top\!-\! g_tR_t^{-1}g_t^\top\right)\!\partial_x\Psi_t\!-\!f_t^\top\partial_x\Psi_t.
          \end{aligned}               
 \end{equation}
If we assume that Assumption \ref{continuous-time deception Assumption: linearity_risk} holds and $\gamma$ satisfies \eqref{eq: gamma}, we obtain a linear PDE in $\Psi_t$ known as the backward Chapman-Kolmogorov PDE:
\begin{equation}\label{eq:Psi_risk2}
 \!\partial_t\Psi_t\!=\!\frac{\ell_t\Psi_t}{\lambda}\!-\!f_t^\top\partial_x\Psi_t-\frac{1}{2}\text{Tr}\left(h_th_t^\top\partial^2_x\Psi_t\right).
\end{equation}
According to the Feynman-Kac lemma \cite{oksendal2003stochastic}, the solution of the linear PDE \eqref{eq:Psi_risk2} exists and is unique in the sense that $\Psi_t$ solving \eqref{eq:Psi_risk2} is given by
\begin{equation}\label{eq:Psi_risk3}
    \Psi_t(x_t)= \mathbb{E}^Z \left[\exp \left\{-\frac{1}{\gamma}\int_t^T \ell_s(x_s)ds \right\}\right]
\end{equation}
   where $Z$ is the probability measure defined under the system dynamics \eqref{continuous-time deception SDE} when both $u_s$ and $\theta_s$, $t\leq s\leq T$ are zero. Since $V_t(x_t) = -\gamma\log\Psi_t(x_t)$, using \eqref{eq:Psi_risk3}, we get the desired result \eqref{eq:value_f2_risk}. Solving \eqref{eq:u_star_risk} (taking the gradient of \eqref{eq:value_f2_risk} with respect to $x$), we get the optimal controller policy \eqref{eq: u_star_risk2}.
\end{proof}

\begin{remark}
Notice that according to Assumption \ref{continuous-time deception Assumption: linearity_risk}, $0<\xi<\lambda$. Therefore, $\gamma$ is always positive.     
\end{remark}

Since the right-hand side of \eqref{eq:value_f2_risk} contains the expectation operation with respect to $Z$, using the strong law of large numbers \cite{durrett2019probability}, we can prove that as $N\rightarrow\infty$,
\[-\gamma \log \left[\frac{1}{N}\sum_{i=1}^N \exp \left\{-\frac{1}{\gamma}\int_t^T \ell_s(x^i_s)ds\right\}\right]\overset{a.s.} {\rightarrow} V_t(x_t)
\]
where $\{x_s^i, t\leq s\leq T\}_{i=1}^N$ are randomly drawn sample paths from distribution $Z$. Since $Z$ is the measure in which both $u_s$ and $\theta_s$, $t\leq s\leq T$ are zero, generating such a sample ensemble is easy.
It suffices to perform $N$ independent simulations of the dynamics $dx_s=f_s(x_s)ds + h_s(x_s)dw_s, \; x_t=x$.  

\subsection{Two-Player Zero-Sum Stochastic Differential Game via Path Integral Approach}\label{sec: Two-player Zero-Sum Stochastic Differential Game via Path Integral Approach}
The applicability of the path-integral method to the class of stochastic differential games shown in Problem \ref{prob: game} was pointed out by our work in \cite{patil2023risk}. For each time $t\in[0,T)$ and the state $x_t\in\mathbb{R}^n$, the path-integral method, under a certain assumption, allows both the controller and the adversary to compute the optimal policies $u_t$ and $\theta_t$- known as \emph{saddle-point policies} \cite[Chapter 2]{bacsar2008h} by evaluating the path integrals along randomly generated trajectories $\{x_s\}$, $t\leq s \leq T$ starting from $x_t$. This can be derived from the fact that under a certain assumption, the value of the game \eqref{eq:h_inf} can be computed by Monte-Carlo sampling. To see this, for each state-time pair $(x_t,t)$, introduce the value of the game as 
\begin{equation}
\label{eq:value_f_game}
V_t(x_t)\!=\!\min_u \max_\theta  \mathbb{E}^P\!\!\left[\int_t^T\!\!\! \left(\!\ell_s(x_s)\! + \!\frac{1}{2}u_s^{\!\top} R_s u_s\! -\! \frac{\lambda}{2}\|\theta_s\|^2\!\!\right)ds\right]\!.
\end{equation}
First, we make the following assumption, which is essential to solve \eqref{eq:value_f_game} using the path integral approach \cite{patil2023risk}.

\begin{assumption}\label{continuous-time deception Assumption: linearity}
   For all $(x,t)$, there exists a constant $\alpha>0$ satisfying the following equation: 
   \begin{equation}\label{eq: linearizability_game}
    \!\!h_t(x_t)  h_t^{\!\top}(x_t)\!\! =\! \alpha\!\left(\!g_t(x_t)R_t^{-1}\!g_t^\top(x_t)\! -\! \frac{1}{\lambda} h_t(x_t)  h_t^{\!\top}(x_t)\!\!\right)\!.
\end{equation}
\end{assumption}
\vspace{3mm}
Assumption \ref{continuous-time deception Assumption: linearity} is similar to the assumption required in the path integral formulation of a single-agent stochastic control problem \cite{satoh2016iterative}. A possible interpretation of condition \eqref{eq: linearizability_game} is that in a direction with high noise variance, the controller's control cost has to be low. Therefore, the weight of the control cost $R_t$ need to be tuned appropriately for the given $\lambda$, $h_t(x_t)$ and  $g_t(x_t)$ for all $t$. See \cite{patil2023risk} for further discussion on this condition.

\begin{theorem}\label{thm: two-player game}
  Under Assumption \ref{continuous-time deception Assumption: linearity}, the solution of \eqref{eq:value_f_game} exists, is unique and is given by
  \begin{equation}
\label{eq:value_f2_game}
V_t(x_t)=-\alpha \log \mathbb{E}^Z \left[\exp \left\{-\frac{1}{\alpha}\int_t^T \ell_s(x_s)ds \right\}\right]
\end{equation}
where $Z$ is the probability measure defined under the system dynamics \eqref{continuous-time deception SDE} when both $u_s$ and $\theta_s$, $t\leq s\leq T$ are zero.
Furthermore, the saddle-point policies of Problem \ref{prob: game} are given by 
\begin{equation}\label{eq: u_star_game} \!\!\!\!\!u_t^*dt\!=\!\mathcal{H}^u_t(x_t)\frac{\mathbb{E}^Z\left[\text{exp}{\left\{-\frac{1}{\alpha}\int_t^T \ell_s(x_s)ds\right\}}h_t(x_t)d{w}_t\right]}{\mathbb{E}^Z\left[\text{exp}{\left\{-\frac{1}{\alpha}\int_t^T \ell_s(x_s)ds\right\}}\right]} 
\end{equation}
where the matrix $\mathcal{H}^u_t(x_t)$ is defined as
\begin{equation*}
    \!\!\mathcal{H}^u_t(x_t)\!=\!R_t^{-1} g_t^\top(x_t)\!\left(\!g_t(x_t)R_t^{-1}g_t^\top(x_t)\! -\frac{1}{\lambda} h_t(x_t)h_t(x_t)^\top\!\right)^{\!-1}
\end{equation*}
and
\begin{equation}\label{eq: theta_star_game} \!\!\!\!\!\theta_t^*dt\!=\!\mathcal{H}^\theta_t(x_t)\frac{\mathbb{E}^Z\left[\text{exp}{\left\{-\frac{1}{\alpha}\int_t^T \ell_s(x_s)ds\right\}}h_t(x_t)d{w}_t\right]}{\mathbb{E}^Z\left[\text{exp}{\left\{-\frac{1}{\alpha}\int_t^T \ell_s(x_s)ds\right\}}\right]} 
\end{equation}
where the matrix $\mathcal{H}^\theta_t(x_t)$ is defined as
\begin{equation*}
    \!\!\mathcal{H}^\theta_t(x_t)\!=\!-\frac{1}{\lambda} h_t^\top(x_t)\!\left(\!g_t(x_t)R_t^{-1}g_t^\top(x_t)\! -\frac{1}{\lambda} h_t(x_t)h_t(x_t)^\top\!\right)^{\!\!-1}\!\!\!\!.
\end{equation*}
\end{theorem}
\begin{proof}
   The stochastic Hamilton-Jacobi-Isaacs (HJI) equation \cite{bacsar2008h} associated with \eqref{eq:value_f_game} is expressed as follows:
   \begin{equation}\label{eq:HJI}
         \!\!-\partial_tV_t\!=\!\min_{{u_t}} \max_{{\theta_t}}\!\Bigg[\frac{1}{2}u_t^\top\!R_tu_t-\frac{\lambda}{2}\|\theta_t\|^2+\!\ell_t\!+\!\left(\!f_t\!+\!g_tu_t\!+\!h_t\theta_t\right)^\top\!\partial_xV_t+\!\frac{1}{2}\text{Tr}\!\left(h_th_t^\top\partial^2_xV_t\right)\!\Bigg]\!.
\end{equation}
   Solving \eqref{eq:HJI}, we get 
   \begin{equation}\label{eq:saddle-point}
       u_t^*(x_t) \!= \!-R_t^{-1}g_t^\top\partial_xV_t(x_t), \quad \theta_t^*(x_t)\! =\!\frac{1}{\lambda}h_t^\top\partial_xV_t(x_t).
   \end{equation}
Putting \eqref{eq:saddle-point} in \eqref{eq:HJI}, we get  
\begin{equation}\label{eq:HJI2}
    \begin{aligned}
         \!\!-\partial_tV_t\!=\ell_t\! +\!f_t^\top\!\partial_xV_t\!+\!\frac{1}{2}\text{Tr}\left(h_th_t^\top\partial^2_xV_t\right)+\frac{1}{2}\!\left(\partial_xV_t\right)^\top\!\!\left(\frac{1}{\lambda}h_th_t^\top\!-\!g_tR_t^{-1}g_t^\top\right)\!\partial_xV_t.
\end{aligned}
\end{equation}
Let $\Psi_t(x_t)$ be the logarithmic transformation of the value function $V_t(x_t)$ defined as $V_t(x_t) = -\alpha\log\Psi_t(x_t)$
where $\alpha>0$ is a proportionality constant defined by the Assumption \ref{continuous-time deception Assumption: linearity}. Applying this transformation of the value function to \eqref{eq:HJI2} yields \begin{equation}\label{eq:Psi_game}
 \begin{aligned}         \!\!\!\!\partial_t\Psi_t\!=&\frac{\ell_t\Psi_t}{\alpha}-\!\frac{1}{2}\text{Tr}\!\left(h_th_t^\top\!\partial^2_x\Psi_t\right)\!+\!\frac{1}{2\Psi_t}\!\left(\partial_x\Psi_t\right)^{\!T}\!\!h_th_t^\top\!\partial_x\Psi_t \\
         &+\!\frac{\alpha}{2\Psi_t}\!\left(\partial_x\Psi_t\right)^\top\!\!\left(\frac{1}{\lambda}h_th_t^\top\!-\! g_tR_t^{-1}g_t^\top\right)\!\partial_x\Psi_t\!-\!f_t^\top\partial_x\Psi_t.
          \end{aligned}               
 \end{equation}
By assuming an $\alpha$ satisfying Assumption \ref{continuous-time deception Assumption: linearity} holds in \eqref{eq:Psi_game}, we obtain a linear PDE in $\Psi_t$ known as the backward Chapman-Kolmogorov PDE:
\begin{equation}\label{eq:Psi_game2}
 \!\partial_t\Psi_t\!=\!\frac{\ell_t\Psi_t}{\lambda}\!-\!f_t^\top\partial_x\Psi_t-\frac{1}{2}\text{Tr}\left(h_th_t^\top\partial^2_x\Psi_t\right).
\end{equation}
According to the Feynman-Kac lemma \cite{oksendal2003stochastic}, the solution of the linear PDE \eqref{eq:Psi_game2} exists and is unique in the sense that $\Psi_t$ solving \eqref{eq:Psi_game2} is given by
\begin{equation}\label{eq:Psi_game3}
    \Psi_t(x_t)= \mathbb{E}^Z \left[\exp \left\{-\frac{1}{\alpha}\int_t^T \ell_s(x_s)ds \right\}\right]
\end{equation}
   where $Z$ is the probability measure defined under the system dynamics \eqref{continuous-time deception SDE} when both $u_s$ and $\theta_s$, $t\leq s\leq T$ are zero. Since $V_t(x_t) = -\alpha\log\Psi_t(x_t)$, using \eqref{eq:Psi_game3}, we get the desired result \eqref{eq:value_f2_game}. Solving \eqref{eq:saddle-point} (taking the gradient of \eqref{eq:value_f2_game} with respect to $x$), we get the saddle-point policies \eqref{eq: u_star_game} and \eqref{eq: theta_star_game}.
\end{proof}

 Since the right-hand side of \eqref{eq:value_f2_game} contains the expectation operation with respect to $Z$, using the strong law of large numbers \cite{durrett2019probability}, we can prove that as $N\rightarrow\infty$,
\begin{equation}\label{eq: eq:value_f2_game_MC}
    -\alpha \log \left[\frac{1}{N}\sum_{i=1}^N \exp \left\{-\frac{1}{\alpha}\int_t^T \ell_s(x^i_s)ds\right\}\right]\overset{a.s.} {\rightarrow} V_t(x_t)
\end{equation}
where $\{x_s^i, t\leq s\leq T\}_{i=1}^N$ are randomly drawn sample paths from distribution $Z$. Since $Z$ is the measure in which both $u_s$ and $\theta_s$, $t\leq s\leq T$ are zero, generating such a sample ensemble is easy.
It suffices to perform $N$ independent simulations of the dynamics $dx_s=f_s(x_s)ds + h_s(x_s)dw_s, \; x_t=x$. 

\begin{remark}\label{rem: equivalence}
Notice that the solution obtained by solving the risk-sensitive control problem \eqref{eq:value_f2_risk} is the same as the one obtained by solving the two-player zero-sum stochastic differential game \eqref{eq:value_f2_game} if $\alpha = \gamma$.
From \eqref{eq: linearizability_risk}, \eqref{eq: gamma}, and \eqref{eq: linearizability_game} we can indeed prove that $\alpha = \gamma$ i.e., the required condition to apply the path integral method to solve the risk-sensitive control problems is the same as the one in the two-player zero-sum stochastic differential games.    
\end{remark}

Similar to \eqref{eq: eq:value_f2_game_MC}, the saddle-point policies $u_t^*, \theta_t^*$ can be readily computed by the same simulated ensemble $\{x_s^i, t\leq s\leq T\}_{i=1}^N$ under distribution $Z$ and their path costs \cite{patil2023risk}. According to the strong law of large numbers, as $N\rightarrow\infty$, 

\begin{equation}\label{eq: u_star_game_MC} \mathcal{H}^u_t(x_t)\frac{\frac{1}{N}\sum_{i=1}^N\text{exp}{\left\{-\frac{1}{\alpha}\int_t^T \ell_s(x_s^i)ds\right\}}h_t(x_t)\epsilon}{\frac{\sqrt{\Delta t}}{N}\sum_{i=1}^N\text{exp}{\left\{-\frac{1}{\alpha}\int_t^T \ell_s(x_s^i)ds\right\}}} \overset{a.s.} {\rightarrow} u_t^*
\end{equation}

\begin{equation}\label{eq: theta_star_game_MC} \mathcal{H}^\theta_t(x_t)\frac{\frac{1}{N}\sum_{i=1}^N\text{exp}{\left\{-\frac{1}{\alpha}\int_t^T \ell_s(x_s^i)ds\right\}}h_t(x_t)\epsilon}{\frac{\sqrt{\Delta t}}{N}\sum_{i=1}^N\text{exp}{\left\{-\frac{1}{\alpha}\int_t^T \ell_s(x_s^i)ds\right\}}} \overset{a.s.} {\rightarrow} \theta_t^*
\end{equation}
where $\epsilon\sim\mathcal{N}(0,1)$ and $\Delta t$ is the step size. Equation \eqref{eq: u_star_game_MC} implies that under the assumption \eqref{eq: linearizability_risk} or \eqref{eq: linearizability_game}, the minimax KL control problem (Problem \ref{prob: game}) can be solved using the path integral approach. That is, if the controller has a simulator engine that can generate a large number of sample trajectories $\{x_t^i\}_{i=1}^N$ from the distribution $Z$ starting from the current state-time pair $(x_t,t)$, then the optimal attack mitigating control signal $u^*_t$ \eqref{eq: u_star_game} against the worst-case attack can be computed directly from the sample ensemble $\{x_t^i\}_{i=1}^N$. A notable feature of such a simulator-driven attack mitigation policy synthesis is that it requires neither an explicit model of the system nor an explicit policy synthesis step.

\section{Numerical Experiments}
In this section, we present two numerical studies illustrating the proposed attack synthesis and mitigation approach.
\subsection{Collision Avoidance with the Unsafe Region}
 \begin{figure*}
     \centering
       \begin{tabular}{c c}
\!\!\!\!\!\!\!\!\!\!\!\!\!\!\!\!\!\!\!\!\includegraphics[scale=0.37]{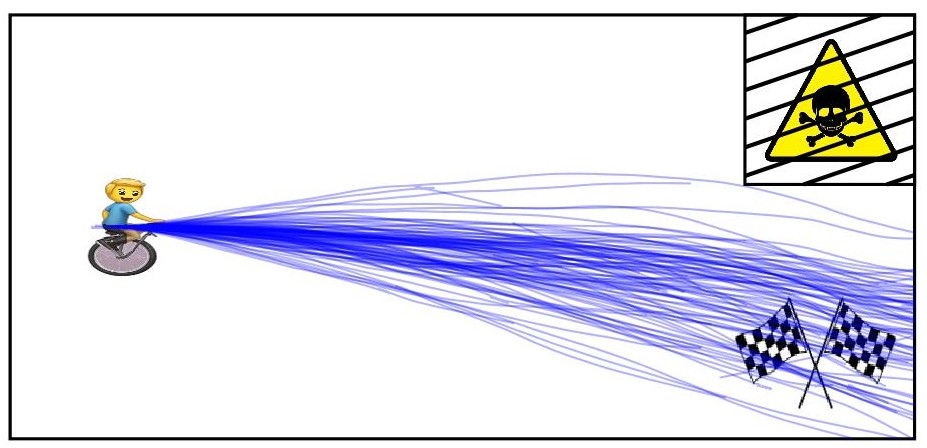} &\!\!\!\!\!\!\includegraphics[scale=0.37]{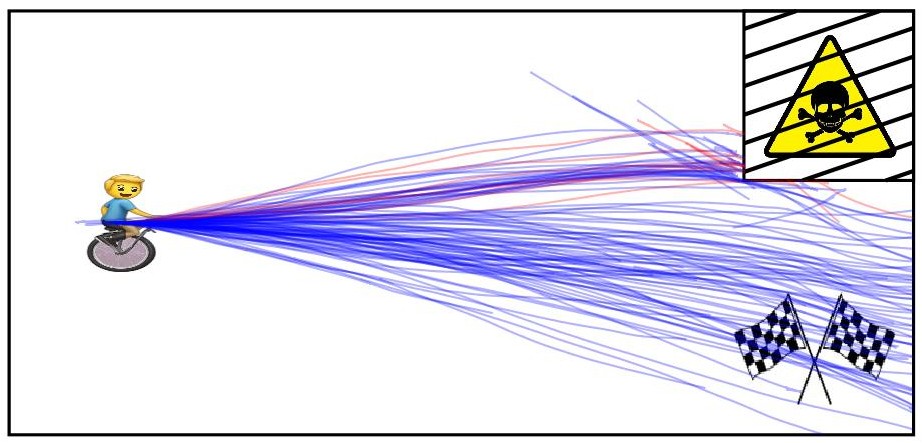} \\
\!\!\!\!\!\!\!\!\!\!\!\!\!\!\!\!\!\!\!\!(a) No attack, $P^{\text{crash}} \approx 0$  &\!\!\!\!\!\!(b) Stealthy attack, $\lambda=2$, $P^{\text{crash}} \approx 0.09$ \\

\!\!\!\!\!\!\!\!\!\!\!\!\!\!\!\!\!\!\!\!\includegraphics[scale=0.37]{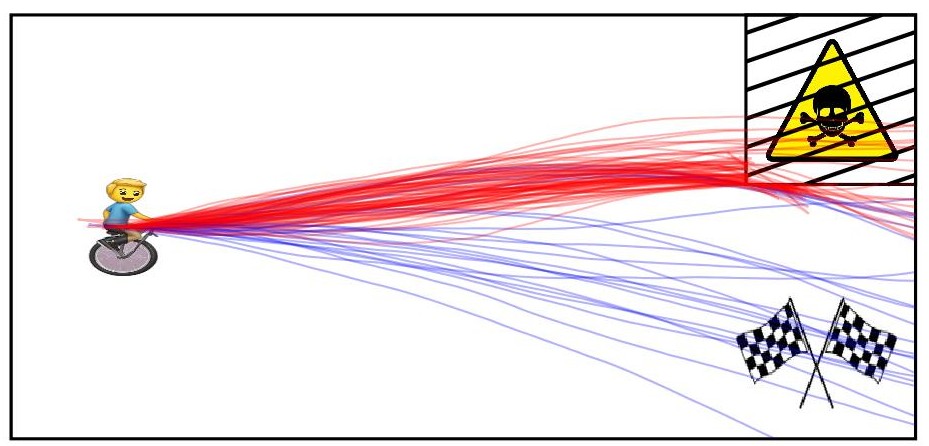} &\!\!\!\!\!\!\includegraphics[scale=0.37]{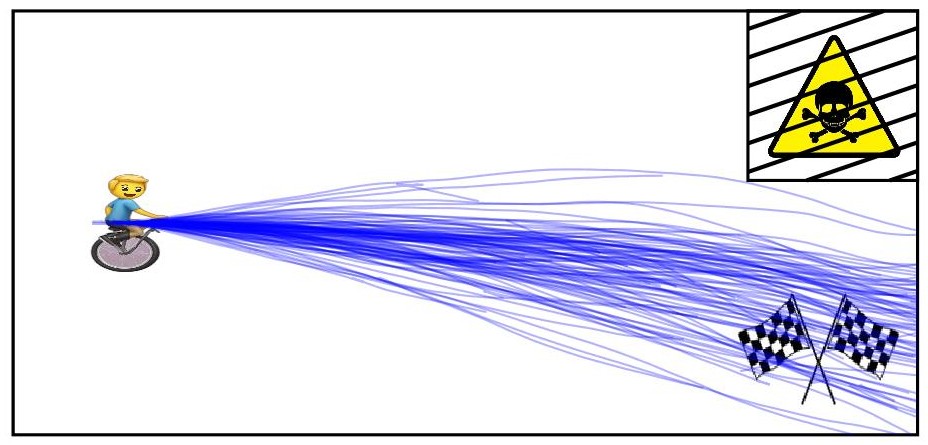} \\
\!\!\!\!\!\!\!\!\!\!\!\!\!\!\!\!\!\!\!\!(c) Stealthy attack, $\lambda=0.1$, $P^{\text{crash}} \approx 0.78$  &\!\!\!\!\!\!(d) Attack mitigation, $\lambda=0.1$, $P^{\text{crash}} \approx 0$ \\

       \end{tabular}
         \caption{A unicycle navigation problem. $100$ sample paths generated without the attacker and with the attacker for two values of $\lambda$ are shown. The probability of crashed paths $P^{\text{crash}}$ are noted below each case.} 
         \label{Fig. worst-case attack}
 \end{figure*}
 
Consider an autonomously operated unicycle that receives control commands $u_t$ to move from an initial location to a target destination. In this setting, a noise signal $w_t$ perturbs the control commands, and a stealthy attacker may hijack control by injecting an attack signal $v_t\neq w_t$, aiming to misguide the unicycle to an unsafe region. Consider the following unicycle dynamics model:
\begin{equation*}
\begin{aligned}
    \begin{bmatrix}
    d{p}^x_t\\d{p}^y_t\\d{s_t}\\d{\phi_t}
    \end{bmatrix}\!\!=\!
    \begin{bmatrix}
    {s}_t\cos{{\phi_t}}\\{s_t}\sin{{\phi_t}}\\0\\0
    \end{bmatrix}\!dt+\!\!\begin{bmatrix}
    0 & 0\\0 & 0\\1 & 0\\0 & 1
    \end{bmatrix} \!
    \begin{bmatrix}
    a_t\\
    \omega_t
    \end{bmatrix}\!dt  +\!
    \begin{bmatrix}
    0 & 0\\0 & 0\\1 & 0\\0 & 1
    \end{bmatrix} \!
    \begin{bmatrix}
    \sigma_t & 0\\
    0 & \nu_t 
    \end{bmatrix}\left(
    \begin{bmatrix}
    \Delta a_t\\
    \Delta \omega_t
    \end{bmatrix}\!dt \!+\!
    d{w}_t\right)
   ,
\end{aligned}
\end{equation*}
where $p_t:=\begin{bmatrix}
  {p}^x_t & {p}^y_t  
\end{bmatrix}^\top$, ${s}_t$ and ${\phi_t}$ denote the position, speed, and the heading angle of the unicycle, respectively, at time $t$. The control input $u_t:=\begin{bmatrix} a_t & \omega_t \end{bmatrix}^\top$ consists of acceleration $a_t$ and angular speed $\omega_t$. $\theta_t:=\begin{bmatrix} \Delta a_t & \Delta \omega_t \end{bmatrix}^\top$ is the attacker's bias input, $d{w}_t\in\mathbb{R}^2$ is the white noise and $\sigma_t$, $\nu_t$ are the noise level parameters. As illustrated in Figure \ref{Fig. worst-case attack}, the unicycle’s objective is to navigate from its initial position to the target position $\begin{bmatrix}
    \mathcal{G}^x & \mathcal{G}^y
\end{bmatrix}^\top$ (shown by two cross flags at the bottom right). Let $\mathcal{X}^\text{unsafe}$ denote the unsafe region shown by a hatched box at the top right. \par

\subsubsection{Stealthy Attack Synthesis}
First, we will compute the worst-case attack signal (i.e., solve Problem \ref{prob: KL}) assuming that the controller's policy $u_t$ is fixed and is known to the attacker. $u_t$ is designed to drive the unicycle to the target location. For the simulation, we set $\sigma_t=\nu_t=0.1$, $\forall t$, $T=5$, $c_t({x_t, u_t}) = b_t\left[\left(\mathcal{G}^x-{p}^x_t\right)^2 + \left(\mathcal{G}^y-{p}^y_t\right)^2\right] + \frac{1}{2}u_t^\top u_t + \eta_t\mathds{1} _{{p_t}\in\mathcal{X}^{\text{unsafe}}}$ where $b_t = 0.1, \eta_t = 0.1, \forall t$. $\mathds{1} _{{p_t}\in\mathcal{X}^{\text{unsafe}}}$ is an indicator function which returns $1$ when the position $p_t$ of the unicycle is inside the unsafe region $\mathcal{X}^\text{unsafe}$ and $0$ otherwise. In order to evaluate the optimal bias input \eqref{eq: theta_star} via Monte Carlo sampling, $10^4$ trajectories and a step size equal to $0.01$ are used. Figure \ref{Fig. worst-case attack}(a) shows the plot of $100$ trajectories when the system is under no attack i.e., the bias input $\theta_t=0, \forall t$. The trajectories are color-coded; the red paths crash with the unsafe region $\mathcal{X}^{\text{unsafe}}$, while the blue paths converge in the neighborhood of the target position. Figures \ref{Fig. worst-case attack}(b) and \ref{Fig. worst-case attack}(c) show the plots of $100$ trajectories with the same color-coding scheme, when the system is under the attack for two values of $\lambda$. A lower value of $\lambda$ implies that the attacker cares less about being stealthy and more about crashing the unicycle with the unsafe region $\mathcal{X}^{\text{unsafe}}$. A higher value of $\lambda$ implies the opposite. We also report $P^{\text{crash}}$, the percentage of paths that crash with the unsafe region $\mathcal{X}^{\text{unsafe}}$. Under no attack (Figure \ref{Fig. worst-case attack}(a)), none of the paths crash with the unsafe region. On the other hand, some paths crash under the adversary's attack, and for a lower value of $\lambda$, $P^{\text{crash}}$ is higher.

\subsubsection{Attack Risk Mitigation}
Now, we will solve the controller's problem who is interested in mitigating the risk of stealthy attacks i.e., solve Problem \ref{prob: minimax_KL}. In Section \ref{sec: Connections with Risk-Sensitive Control and Two-Player Zero-Sum Stochastic Differential Game}, we established the connections between the minimax KL control problem (Problem \ref{prob: minimax_KL}) with risk-sensitive control (Problem \ref{prob: risk-sensitive control}) and two-player zero-sum stochastic differential game (Problem \ref{prob: game}). We showed in Section \ref{sec: Two-player Zero-Sum Stochastic Differential Game via Path Integral Approach} that under Assumption \ref{continuous-time deception Assumption: linearity_risk} or \ref{continuous-time deception Assumption: linearity}, the solutions of both the problems—Problems \ref{prob: risk-sensitive control} and \ref{prob: game}—lead to the same optimal policy of the controller \eqref{eq: u_star_game} under the worst-case attack signal. In order to use the path integral control to solve Problem \ref{prob: risk-sensitive control} or \ref{prob: game}, it is necessary to find a constant $\alpha>0$ (by Assumption \ref{continuous-time deception Assumption: linearity}) such that 
\begin{equation*}
    \begin{bmatrix}
    \sigma_t & 0\\
    0 & \nu_t 
    \end{bmatrix}\begin{bmatrix}
    \sigma_t & 0\\
    0 & \nu_t 
    \end{bmatrix}^\top = \alpha\left(I_{2\times2} - \frac{1}{\lambda}\begin{bmatrix}
    \sigma_t & 0\\
    0 & \nu_t 
    \end{bmatrix}\begin{bmatrix}
    \sigma_t & 0\\
    0 & \nu_t 
    \end{bmatrix}^\top\right)
\end{equation*}
where $I_{2\times2}$ is an identity matrix of size $2\times2$. We solve Problem \ref{prob: game} and evaluate the saddle-point policies \eqref{eq: u_star_game}, \eqref{eq: theta_star_game} via path integral control. Note that solving Problem \ref{prob: risk-sensitive control} will also give us the same results. We use $10^4$ Monte Carlo trajectories and a step size equal to $0.01$. Figure \ref{Fig. worst-case attack}(d) shows the plots of $100$ sample trajectories generated using synthesized saddle-point policies $(u^*_t, \theta^*_t)$ for $\lambda = 0.1$. In Figures \ref{Fig. worst-case attack}(b) and \ref{Fig. worst-case attack}(c), the controller is unaware of the attacker. However, in Figure \ref{Fig. worst-case attack}(d), the controller is aware of the attacker and designs an attack mitigating policy to combat the potential attacks. As we observe, under the attack mitigating policy, the controller is able to avoid the crashes with the unsafe region $\mathcal{X}^{\text{unsafe}}$.  

\subsection{Cruise Control}
\begin{figure*}
     \centering
       \begin{tabular}{c c}
\!\!\!\!\!\!\!\!\!\!\!\!\!\!\!\!\!\!\!\!\includegraphics[scale=0.17]{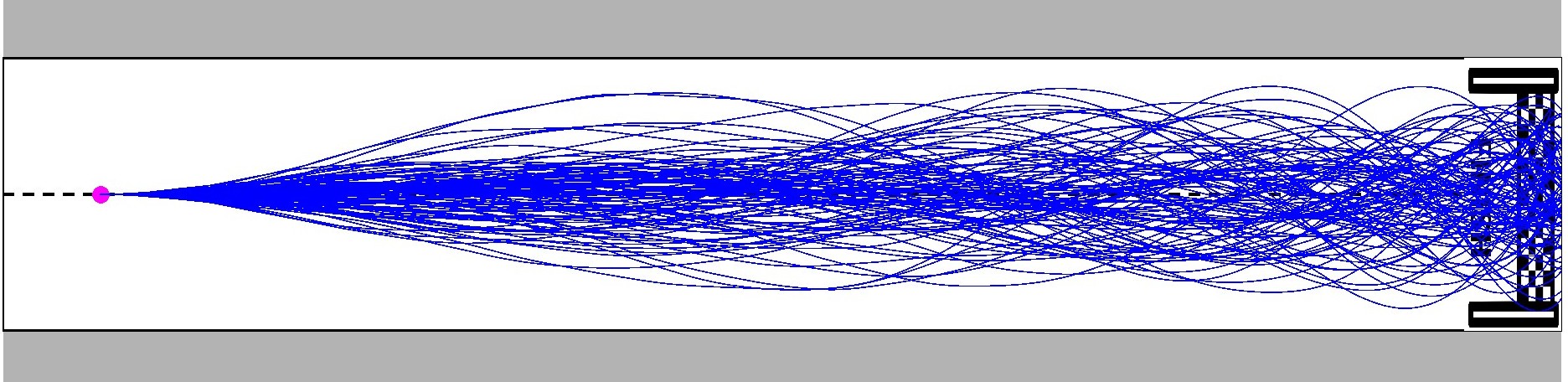} &\includegraphics[scale=0.17]{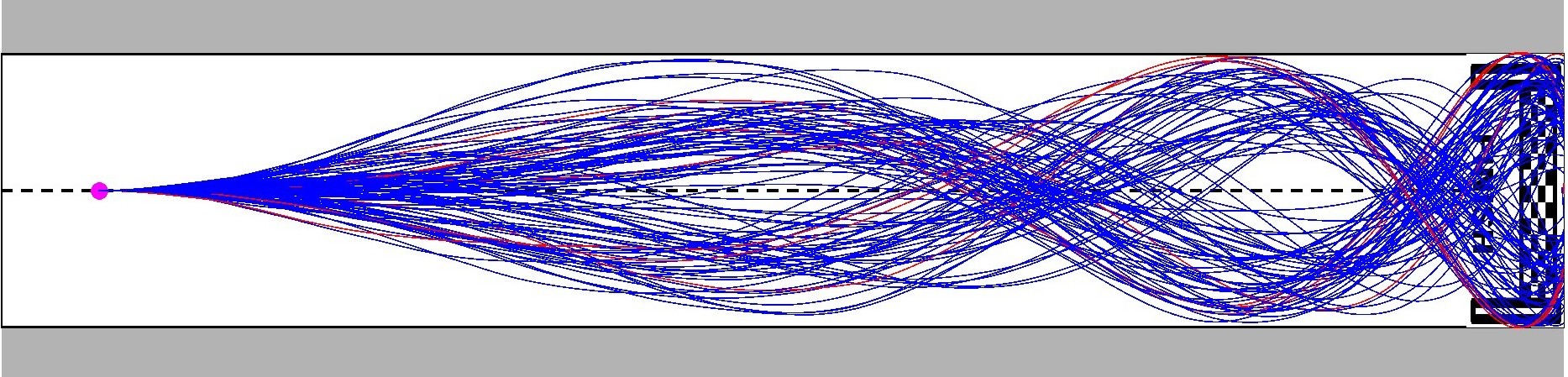} \\
\!\!\!\!\!\!\!\!\!\!\!\!\!\!\!\!\!\!\!\!(a) No attack, $P^{\text{crash}} \approx 0$  &(b) Stealthy attack, $\lambda=3$, $P^{\text{crash}} \approx 0.07$ \\\\

\!\!\!\!\!\!\!\!\!\!\!\!\!\!\!\!\!\!\!\!\includegraphics[scale=0.17]{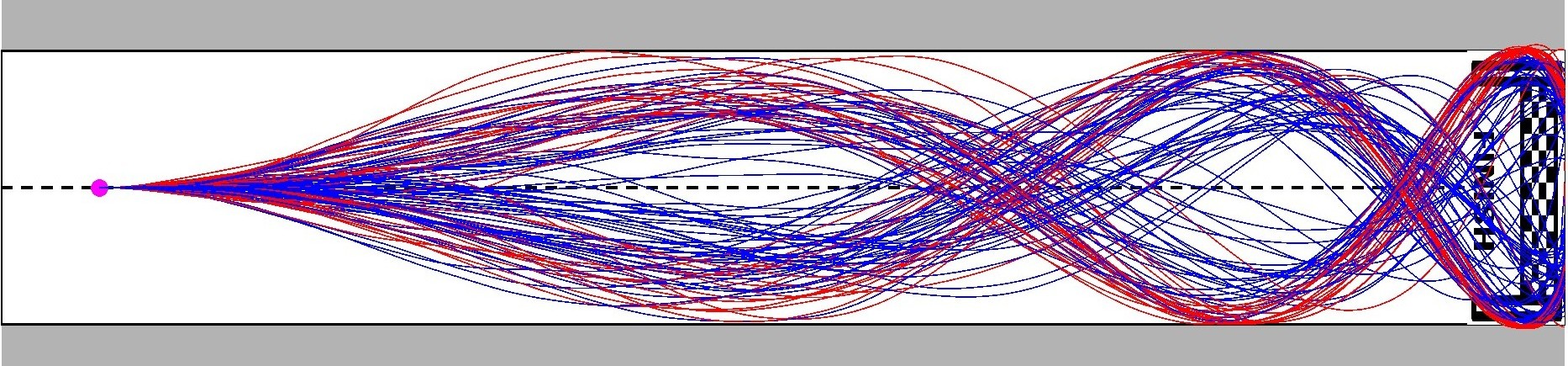} &\includegraphics[scale=0.17]{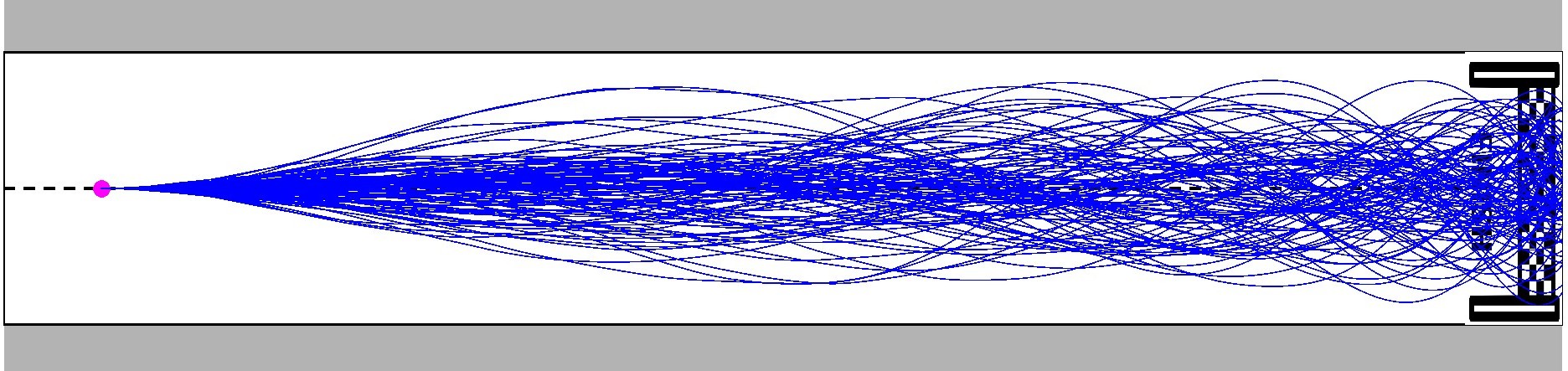} \\
\!\!\!\!\!\!\!\!\!\!\!\!\!\!\!\!\!\!\!\!(c) Stealthy attack, $\lambda=1.5$, $P^{\text{crash}} \approx 0.35$  &(d) Attack mitigation, $\lambda=1.5$, $P^{\text{crash}} \approx 0$ \\

       \end{tabular}
         \caption{Cruise control problem. The start position is shown by a magenta dot on the left, and the finish line is shown on the right. $100$ sample paths generated without the attacker and with the attacker for two values of $\lambda$ are shown. The probability of crashed paths $P^{\text{crash}}$ are noted below each case.} 
         \label{Fig. car_worst-case attack}
 \end{figure*}
 
Consider an autonomously controlled car that receives control commands $u_t$ to remain on the road without veering off track. In this setup, a noise signal $w_t$ perturbs the control commands, and a stealthy attacker may hijack control by injecting an attack signal $v_t$, aiming to drive the car off the track. Consider the following car model:
\begin{equation*}
\begin{aligned}
    \begin{bmatrix}
    d{p}^x_t\\d{p}^y_t\\d{s_t}\\d{\delta_t}\\d\phi_t
    \end{bmatrix}\!=
    \begin{bmatrix}
    {s_t}\cos{{\delta}_t}\\{s_t}\sin{{\delta}_t}\\0\\\frac{s_t\tan{\phi}_t}{L}
    \\0
    \end{bmatrix}\!dt +\begin{bmatrix}
    0 & 0\\0 & 0\\1 & 0\\0&0\\0 & 1
    \end{bmatrix} \!
    \begin{bmatrix}
    a_t\\
    \zeta_t
    \end{bmatrix}\!dt \!\!\!\!+\begin{bmatrix}
    0 & 0\\0 & 0\\1 & 0\\0 & 0\\0 & 1
    \end{bmatrix}\begin{bmatrix}
    \sigma_t & 0\\
    0 & \nu_t 
    \end{bmatrix} \!
    \left(\begin{bmatrix}
    \Delta a_t\\
    \Delta\zeta_t
    \end{bmatrix}dt + d{w}_t
    \right)
\end{aligned}
\end{equation*}
where $p_t:=\begin{bmatrix}
  {p}^x_t & {p}^y_t  
\end{bmatrix}^\top$, ${s}_t$, ${\delta_t}$ and $\phi_t$ denote the position, speed, and the heading angle and the front wheel angle of the car, respectively, at time $t$. $L$ is the inter-axle distance. The control input $u_t:=\begin{bmatrix} a_t & \zeta_t \end{bmatrix}^\top$ consists of acceleration $a_t$ and the front wheel angular rate $\zeta_t$. $\theta_t:=\begin{bmatrix} \Delta a_t & \Delta \zeta_t \end{bmatrix}^\top$ is the attacker's bias signal, $d{w}_t\in\mathbb{R}^2$ is the white noise and $\sigma_t$, $\nu_t$ are the noise level parameters. As illustrated in Figure \ref{Fig. car_worst-case attack}, the car’s objective is to drive on the road without going off track until it crosses the finish line. Let $\mathcal{G}^{\text{center}}$ denote the center of the road, $\mathcal{G}^{\text{finish}}$ denote the finish line, and $\mathcal{X}^\text{unsafe}$ denote the off-the-track region shown in gray.

\subsubsection{Stealthy Attack Synthesis}
First, we will compute the worst-case attack signal assuming that the controller's policy $u_t$ is fixed and is known to the attacker. $u_t$ is designed to drive the car to the finish line without going off track. For the simulation, we set $L=0.05, \sigma_t=\nu_t=\sqrt{0.005}$, $\forall t$, $T=10$, $c_t({x_t, u_t}) = b_t\left[\left(\mathcal{G}^{\text{finish}}-{p}^x_t\right)^2 + \left(\mathcal{G}^{\text{center}}-{p}^y_t\right)^2\right] + \frac{1}{2}u_t^\top R_t u_t + \eta_t\mathds{1} _{{p_t}\in\mathcal{X}^{\text{unsafe}}}$ where $b_t = 0.02, R_t =4, \eta_t = 0.02, \forall t$. $\mathds{1} _{{p_t}\in\mathcal{X}^{\text{unsafe}}}$ is an indicator function which returns $1$ when the car goes off track and $0$ otherwise. In order to evaluate the optimal bias input \eqref{eq: theta_star} via Monte Carlo sampling, $10^4$ trajectories and a step size equal to $0.02$ are used. Figure \ref{Fig. car_worst-case attack}(a) shows the plot of $100$ trajectories when the car is under no attack i.e., the attack signal $\theta_t=0, \forall t$. The trajectories are color-coded; the red paths go off track, while the blue paths cross the finish line without going off track. Figures \ref{Fig. car_worst-case attack}(b) and \ref{Fig. car_worst-case attack}(c) show the plots of $100$ trajectories with the same color-coding scheme, when the system is under the attack for two values of $\lambda$. A lower value of $\lambda$ implies that the attacker cares less about being stealthy and more about driving the car off track. We also report $P^{\text{crash}}$, the percentage of paths that go off track. Under no attack (Figure \ref{Fig. car_worst-case attack}(a)), none of the paths go off track. On the other hand, trajectories under the adversary's attack drift off from the center of the road, and some of them go off track. Also, notice that for a lower value of $\lambda$, $P^{\text{crash}}$ is higher. 

\subsubsection{Attack Risk Mitigation}
Now, we will solve the controller's problem who is interested in mitigating the risk of stealthy attacks i.e., solve Problem \ref{prob: minimax_KL}. In order to use the path integral control to solve Problem \ref{prob: risk-sensitive control} or \ref{prob: game}, it is necessary to find a constant $\alpha>0$ (by Assumption \ref{continuous-time deception Assumption: linearity}) such that 
\begin{equation*}
    \begin{bmatrix}
    \sigma_t & 0\\
    0 & \nu_t 
    \end{bmatrix}\begin{bmatrix}
    \sigma_t & 0\\
    0 & \nu_t 
    \end{bmatrix}^\top = \alpha\left(R_t^{-1} - \frac{1}{\lambda}\begin{bmatrix}
    \sigma_t & 0\\
    0 & \nu_t 
    \end{bmatrix}\begin{bmatrix}
    \sigma_t & 0\\
    0 & \nu_t 
    \end{bmatrix}^\top\right).
\end{equation*}
We solve Problem \ref{prob: game} and evaluate the saddle-point policies \eqref{eq: u_star_game}, \eqref{eq: theta_star_game} via path integral control. We use $10^4$ Monte Carlo trajectories and a step size equal to $0.02$. Figure \ref{Fig. car_worst-case attack}(d) shows the plots of $100$ sample trajectories generated using synthesized saddle-point policies $(u^*_t, \theta^*_t)$ for $\lambda = 1.5$. In Figures \ref{Fig. car_worst-case attack}(b) and \ref{Fig. car_worst-case attack}(c), the controller is unaware of the attacker. However, in Figure \ref{Fig. car_worst-case attack}(d), the controller is aware of the attacker and designs an attack-mitigating policy to combat the potential attacks. As we observe, under the attack mitigating policy, the controller is able to cross the finish line without going off-track.

\section{Publications}
\begin{itemize}
     \item \textbf{A. Patil}, K. Morgenstein,  L. Sentis, T. Tanaka, ``Path Integral Methods for Synthesizing and Preventing Stealthy Attacks in Nonlinear Cyber-Physical Systems," \textit{under review in Transactions on Automatic Control (TAC)}
\end{itemize}

\section{Conclusion and Future Work}\label{sec:conclusion}
We presented a framework for understanding and counteracting stealthy attacks in continuous-time, nonlinear cyber-physical systems. By representing the trade-off between remaining covert and degrading system performance through a Kullback–Leibler (KL) divergence measure of stealthiness, we formulated both (i) a KL control problem to characterize the attacker’s worst-case policy, and (ii) a minimax KL control problem to design a controller that mitigates the risk of such attacks. Key to our approach is a simulator-driven path integral control method that uses forward Monte Carlo simulations rather than closed-form models, thus remaining viable even for high-dimensional or complex systems. Numerical experiments on a unicycle navigation and cruise control illustrated that the attacker can indeed degrade the system performance while maintaining low detectability, yet the controller can act against the attack signals through risk-aware policy design.\par
Going forward, our path integral–based methodology can be extended to broader settings, such as partial observability or more general cost functions, and can accommodate advanced learning-based models that serve as ``digital twins" of physical processes. This work thus provides both a theoretical and computational foundation for tackling stealthy attack synthesis and mitigation in modern CPS. We also plan to conduct the sample complexity analysis of the path integral approach to solve both the KL control and the minimax KL control problems.

%% file: chapters/LQR.tex
\chapter[Discrete-Time Stochastic LQR via Path Integral Control and Its Sample Complexity Analysis]{Discrete-Time Stochastic LQR via Path Integral Control and Its Sample Complexity Analysis}
\label{Sec: sample complexity of LQR}

\section{Motivation and Literature Review}
We presented five application domains of path integral control approach. Although the path integral method is convenient for a broad class of stochastic optimal control problems, the outcomes of Monte-Carlo simulations are inherently probabilistic; therefore, applying path integral methods to safety-critical missions would require rigorous performance guarantees. The purpose of this section is to introduce a theoretical framework for sample complexity analysis that allows us to understand the interplay between the achievable control performance by the path integral method and the required computational cost. \par

Despite the wide applicability of path integral control, not enough work has been done on its sample complexity analysis. To the best of the authors' knowledge, the only prior work addressing this problem is the work by Yoon et al. \cite{yoon2022sampling}. 
The authors of \cite{yoon2022sampling} considered the continuous-time path integral control, and applied Chebyshev and Hoeffding inequalities to relate the instantaneous (pointwise-in-time) error bound in control input and the sample size of the Monte-Carlo simulations performed at that particular time instance. While this result sheds light on the sample complexity analysis, the work  \cite{yoon2022sampling} is limited in the following aspects.
First, the effect of time discretization is not addressed. Since numerical implementations of path integral methods necessitate time discretization, and the control performance depends on the discretization step sizes, such an analysis is not negligible. 
Second, it is not clear how the pointwise-in-time bound in \cite{yoon2022sampling} can be translated into a more explicit, end-to-end (trajectory-level) error bound. Third, \cite{yoon2022sampling} does not provide machinery to compute the required sample size to achieve an acceptable loss of control performance. Such an analysis is crucial for safety-critical applications. Another analysis of sample complexity for path integral control appears in \cite{jasra2022multilevel}, where the authors demonstrate how naive Monte Carlo simulations can exhibit exponentially increasing variance over longer time horizons. As a remedy, they propose a multilevel Monte Carlo (MLMC) approach.

\section{Contributions}
We make the following contributions to this work:
\begin{enumerate}
    \item Derivation of a path integral formulation for discrete-time stochastic Linear Quadratic Regulator (LQR) problem: Recognizing the difficulties in answering the above questions in the continuous-time setting, we take a step back and study the sample complexity of path integral in the discrete-time setting. We develop a discrete-time counterpart of the path integral control scheme for the stochastic LQR problems. Since a formulation of path integral control for discrete-time LQR has not been provided explicitly in the literature (note, however, a recent work \cite{ito2022kullbackleibler}), we decided to derive it based on a more general setup of the so-called Kullback-Leibler (KL) control problem and using its intimate connection with path integral control \cite{theodorou2012relative, arslan2014information}.
    \item Derivation of an end-to-end (trajectory-level) bound on the error in the control signals: There is often a need to assess the safety of the control policy across the entire time horizon. Notably, the end-to-end error bound in the control signals is challenging to evaluate because it couples states across the time horizon as a whole. In this work, we provide the end-to-end bound on the error between the optimal control signal (computed by the classical Riccati solution) and the one obtained by the path integral method as a function of sample sizes. Our sample complexity analysis reveals that the required number of samples for path integral control depends only logarithmically on the dimension of the control input. 
    \item Formulation of a chance-constrained optimization problem to quantify the worst-case performance of the path integral LQR control: We provide a method to quantify the worst-case control performance by formulating a chance-constrained LQR problem for an adversarial agent. This result, together with (2), will allow us to relate the sample size and the worst-case control performance of the path integral method.
\end{enumerate}

The main purpose of our study is to establish a theoretical foundation for the sample complexity of the path integral method. While the path integral method is particularly powerful in nonlinear settings, in this work, we have chosen the LQR setup since the availability of an analytical expression of the optimal control law allows for quantitative error analysis. The results presented in this work will be the necessary first step for future study on the sample complexity of the path integral method for nonlinear systems.

\section*{Notations}
We use the same notations as described in Section \ref{Sec: deception}. 

\section{KL Control via Path Integral}
In this section, we introduce a class of discrete-time continuous-state KL control problems similar to Section \ref{Sec: KL Control} and present the path-integral-based solution approach.
{Let $\mathcal{X}_t\subseteq \mathbb{R}^n$ and $\mathcal{U}_t\subseteq \mathbb{R}^m$ be the spaces of states and control inputs at time step $t$, respectively. }
Suppose the state transition is governed by the mapping $x_{t+1}=F_t(x_t, u_t)$, {$F_t: \mathcal{X}_t \times \mathcal{U}_t \rightarrow \mathcal{X}_{t+1}$ for each $t\in\mathcal{T}:=\{0, 1, \cdots, T\}$} with a given initial state $X_0 = x_0$.
Notice that we assume the state transition law is deterministic and can also be written as $P_{X_{t+1}|X_t,U_t}(dx_{t+1}|x_t, u_t)=\delta_{F_t(x_t, u_t)}(dx_{t+1})$ where $\delta$ is the Dirac measure. 
{In contrast, the control policy to be designed is allowed to be randomized and is represented by a sequence of stochastic kernels \cite{bertsekas1996stochastic} $Q_{U_t|X_t}, t\in\mathcal{T}$. 
A sequence of stochastic kernels $R_{U_t|X_t}, t\in\mathcal{T}$ represents a nominal (reference) policy which is given in advance.} 
We denote the joint probability distributions of the state-control trajectories induced by the policies $R$ and $Q$ as Equations \eqref{eq:def_ref_traj_dist} and \eqref{eq:def_dc_traj_dist} respectively. Let the functions $C_t(\cdot, \cdot):\mathcal{X}_t\times \mathcal{U}_t\rightarrow \mathbb{R}$ for $t\in\mathcal{T}$ and $C_T(\cdot): \mathcal{X}_T\rightarrow \mathbb{R}$ represent the stage-wise and the terminal cost functions, respectively and the path cost can be written as \eqref{path cost function}. Introducing {a positive weighing factor $\lambda$}, the KL control problem is formulated as: 
\begin{problem}[KL control problem]\label{Prob:KL control sample complexity}
\begin{align*}
\min_{\{Q_{U_t|X_t}\}_{t=0}^{T-1}} \quad \mathbb{E}_Q^{x_0}\; C_{0:T}(X_{0:T}, U_{0:T-1})  +\lambda D(Q_{X_{0:T}, U_{0:T-1}} \|R_{X_{0:T}, U_{0:T-1}}). 
\end{align*}
\end{problem}
The expectation $\mathbb{E}_Q^{x_0}(\cdot)$ is with respect to the probability measure \eqref{eq:def_dc_traj_dist} with a fixed initial state $x_0$. {$\lambda$ is a positive constant that balances a trade-off between the path cost $C_{0:T}(X_{0:T}, U_{0:T-1})$ and the KL divergence $D(Q_{X_{0:T}, U_{0:T-1}} \|R_{X_{0:T}, U_{0:T-1}})$. 
} Define the value function for Problem~\ref{Prob:KL control sample complexity} by
\begin{align*}
J_t(x_t):=\min_{\{Q_{U_k|X_k}\}_{k=t}^{T-1}}  \mathbb{E}_Q^{x_t}\; C_{t:T}(X_{t:T}, U_{t:T-1}) +\lambda D(Q_{X_{t:T}, U_{t:T-1}} \|R_{X_{t:T}, U_{t:T-1}}). 
\end{align*}
Here, $\mathbb{E}_Q^{x_t}$ denotes the expectation with respect to the measure $Q_{X_{t:T}, U_{t:T-1}}$ induced by the policy $\{Q_{U_k|X_k}\}_{k=t}^{T-1}$ and the state transition law $x_{k+1}=F_k(x_k,u_k)$, $k=t, \cdots T-1$ with a fixed initial state $x_t$. The expectation $\mathbb{E}_R^{x_t}$ with respect to the measure $R_{X_{t:T}, U_{t:T-1}}$ is defined similarly. {The next theorem provides an explicit representation of the value function, which plays a central role in the path integral control algorithm. This result can be thought of as a discrete-time counterpart of the Feynmann-Kac lemma used in \cite{kappen2005path} in the derivation of continuous-time path integral control.}
\begin{theorem}
\label{thm:KL}
For each $t\in\mathcal{T}$, the value function admits a representation
\begin{equation}
\label{eq:jt_pi}
J_t(x_t)=-\lambda \log \mathbb{E}_R^{x_t} \exp \left(-\frac{1}{\lambda}C_{t:T}(X_{t:T}, U_{t:T-1}) \right).
\end{equation}
The optimal policy $Q_{U_t|X_t}^*$ for Problem~\ref{Prob:KL control sample complexity} is expressed as \eqref{eq:p_star_dc}.
\end{theorem}
\begin{proof}
The proof follows similar to the proof of Theorem \ref{thrm:Bellman recursion}. 
\end{proof}
Equation \eqref{eq:jt_pi} implies that the value function $J_t(x_t)$ can be computed approximately using independent Monte-Carlo simulations of {state-control trajectories} starting from $x_t$ under the reference policy  $\{R_{U_k|X_k}\}_{k=t}^{T-1}$. Let $\{x_{t:T}(i), u_{t:T-1}(i)\}_{i=1}^n$ be an ensemble of $n$ such sample {state-control trajectories}. For each $i$, let
\begin{equation}
\label{eq:dc_path_reward sample complexity}
r(i):= \exp\left(-\frac{1}{\lambda}C_{t:T}(x_{t:T}(i), u_{t:T-1}(i))\right)
\end{equation}
be the reward of the sample path $i$. Then, $J_t(x_t)$ can be evaluated approximately by $J_t(x_t) \approx -\lambda \log \left(\frac{1}{n}\sum_{i=1}^n r(i)\right)$.
Moreover, the expectation of the control input under the optimal distribution $Q^*_{U_t|X_t}$ \eqref{eq:p_star_dc} can also be evaluated as 
\begin{align}\label{eq: controller avg}
  \mathbb{E}_{Q^*}(U_t|x_t) & =  \int_{\mathcal{U}_{t}} u_t Q^*(du_t|x_t)\nonumber\\
    & =\frac{\mathbb{E}_R \left[U_t \exp \left(-\frac{1}{\lambda}C_{t:T}(X_{t:T}, U_{t:T-1}) \right)\right]}{\mathbb{E}_R \left[\exp \left(-\frac{1}{\lambda}C_{t:T}(X_{t:T}, U_{t:T-1})\nonumber \right)\right]} \nonumber\\
    &\approx \frac{\sum_{i=1}^{n} u_t(i)r(i)}{\sum_{i=1}^{n} r(i)}.
\end{align}
\section{Stochastic LQR via Path Integral}\label{Sec: Stochastic LQR via Path Integral}
In this section, we obtain the path integral algorithm for discrete-time stochastic LQR as a special case of the discrete-time, continuous-state KL control problem presented in Section \ref{Sec: KL Control}. Suppose for each $t\in\mathcal{T}$, the state transition is governed by the following linear difference equation
\begin{equation}
\label{eq:dc_lqr_state}
X_{t+1}=A_tX_t+B_tU_t+W_t
\end{equation}
where $X_t\in\mathbb{R}^n$, $U_t\in\mathbb{R}^m$ and {$W_t\sim\mathcal{N}(0,\Omega_t), t\in\mathcal{T}$ are mutually independent Gaussian random variables.} We assume that the initial state $X_0 = x_0$ is given. Now we state the stochastic LQR problem as follows:
\begin{problem}[Stochastic LQR]\label{Prob:LQR}
{Compute the state feedback policy $u_t=k_t(x_t)$ that solves the following optimal control problem:}
\begin{align}\label{eq:dc_lqr_cost}
\!\!\!\!\!\!\min_{\{k_t(\cdot)\}_{t=0}^{T-1}} & \!\mathbb{E} \sum_{t=0}^{T-1}\!\!\left(\!\frac{1}{2}X_t^{\!\top} \! M_t X_t \!+\! \frac{1}{2}U_t^{\!\top} \! N_t U_t \!\right)\!+\!\mathbb{E}\!\left(\!\frac{1}{2}X_T^{\!\top} M_T X_T \!\!\right)\nonumber\\
\!\!\mathrm{s.t.}\;\;&X_{t+1}\!=\!A_tX_t\!+\!B_tU_t\!+\!W_t, \;\;\; {X_0 = x_0}
\end{align}
where $\{M_t\}_{t=0}^T$ and $\{N_t\}_{t=0}^{T-1}$ are {sequences of positive definite matrices}.
\end{problem}

\subsection{Classical solution}\label{Sec: classical LQR}
It is well-known that the optimal policy is given by
\begin{subequations}
\label{eq:dc:lqr_policy}
\begin{align}
u_t&=k_t(x_t)=K_tx_t, \\
K_t&=-(B_t^\top \Theta_{t+1} B_t + N_t)^{-1}B_t^\top \Theta_{t+1}A_t
\end{align}
\end{subequations}
where $\{\Theta_t\}_{t=0}^T$ is a {sequence of positive definite matrices} computed by the backward Riccati recursion with $\Theta_T = M_T$:
\begin{align}\label{eq: dc_lqr_theta}
    \Theta_t=&A_t^\top \Theta_{t+1} A_t + M_t-A_t^\top \Theta_{t+1}B_t(B_t^\top \Theta_{t+1} B_t + N_t)^{-1}B_t^\top \Theta_{t+1}A_t. 
\end{align}

\subsection{Path-integral-based solution}
We now recover the classical result in Section \ref{Sec: classical LQR} using the KL control framework described in Section~\ref{Sec: KL Control}. 
It turns out that Problem 2 can be related to Problem 1 under the following assumption:
\begin{assumption}
\label{asmp:1}
For each $t=0, 1, \cdots, T-1$, there exists a positive definite matrix $\hat{\Omega}_t$ and $\lambda >0$  such that \begin{equation}
\label{eq:omega_asmp}
N_t=\lambda \hat{\Omega}_t^{-1} \text{ and } B_t\hat{\Omega}_t B_t^\top=\Omega_t.
\end{equation}
\end{assumption}
This assumption implies that the control input in the direction with higher noise variance is cheaper than that in the direction with lower noise variance. In order to satisfy this assumption the noise $W_t$ has to enter the dynamics via the control channel. {While Assumption~\ref{asmp:1} is somewhat restrictive, it is a common assumption in the path integral control literature. See, e.g., \cite{kappen2005path} for its system theoretic interpretation. Its relaxations have been studied in the literature (e.g., \cite{levine2018reinforcement}).} For simplicity, we accept Assumption \ref{asmp:1} in what follows.

Consider the KL control problem in which the state transition law is given by $x_{t+1}=F_t(x_t, u_t)$ where
\begin{equation}\label{eq: Ft for LQR}
    F_t(x_t, u_t)=A_t x_t + B_t u_t \;\;\; \forall t=0, 1, \cdots, T-1.
\end{equation}
Suppose that the reference policy $R_{U_t|X_t}$ is given as a non-degenerate zero-mean Gaussian probability density function with covariance $\hat{\Omega}_t$, i.e.,
\begin{equation}
\label{eq:dc_lqr_ref}
R_{U_t|X_t}(u_t|x_t)=\mathcal{N}(0, \hat{\Omega}_t).
\end{equation}
Finally, suppose that the stage-wise and terminal cost functions are given by
\begin{align}
    C_t(x_t, u_t)&=\frac{1}{2}x_t^\top M_t x_t, \quad \forall t=0, 1, \cdots, T-1 \\
    C_T(x_T)&=\frac{1}{2}x_T^\top M_T x_T. \label{C_t for linear setting}
\end{align}
Under this setup, Theorem~\ref{thm:KL} leads to the following result:
\begin{theorem}\label{theorem: KL control solution}
The KL control problem defined by \eqref{eq:omega_asmp} through \eqref{C_t for linear setting} admits the optimal policy $Q_{U_t|X_t}^*$ given by
\begin{equation} 
\label{eq:lem_q_star}
Q_{U_t|X_t}^*(u_t|x_t)=\mathcal{N}(-\hat{H}_t^{-1}\hat{G}_t^\top x_t, \lambda \hat{H}_t^{-1})
\end{equation}
where $\hat{G}_t=A_t^\top \hat{\Theta}_{t+1}B_t$ and $\hat{H}_t=B_t^\top\hat{\Theta}_{t+1}B_t+\lambda \hat{\Omega}_t^{-1}$
and $\hat{\Theta}_{t}$ is the solution of the  Riccati difference equation 
\begin{align}\label{eq:dc_lqr_riccati_theta}
\hat{\Theta}_t=A_t^\top\hat{\Theta}_{t+1}A_t + M_t  - \!A_t^\top \hat{\Theta}_{t+1}B_t(B_t^\top \hat{\Theta}_{t+1}B_t\!+\!\lambda\hat{\Omega}_t^{-1})^{-1}B_t^\top \hat{\Theta}_{t+1}A_t 
\end{align}
with $\hat{\Theta}_T=M_T$.
\end{theorem}

\begin{proof}
Suppose the value function $J_t(x_t)$ takes the form ${J}_t(x_t)=\frac{1}{2}x_t^\top \hat{\Theta}_t x_t+\hat{\kappa}_t$, where $\hat{\Theta}_T = {M}_T$ and $\hat{\kappa}_T = 0$. Define $Z_t(x_t)\coloneqq \text{exp}\left(-\frac{1}{\lambda}J_t(x_t)\right)$. Now, from \eqref{eq:p_star_dc}, the optimal policy is given by 
\begin{align}
Q_{U_t|X_t}^*(u_t|x_t)=\frac{1}{Z_t(x_t)}\exp\left(-\frac{C_t(x_t, u_t)}{\lambda}\right)\times Z_{t+1}\left(F_t(x_t, u_t)\right)R(u_t|x_t).
\label{dc_lqr_q_star1}
\end{align}
Introducing $\hat{F}_t= A_t^\top\hat{\Theta}_{t+1} A_t + {M}_t$, $\hat{G}_t= A_t^\top\hat{\Theta}_{t+1} B_t $ and $\hat{H}_t= B_t^\top\hat{\Theta}_{t+1} B_t +\lambda \hat{\Omega}_t^{-1}$ and using the reference policy \eqref{eq:dc_lqr_ref}, equation \eqref{dc_lqr_q_star1} can be written as
\begin{align*}
Q_{U_t|X_t}^*(u_t|x_t)=&\frac{1}{Z_t(x_t)}\frac{1}{\sqrt{(2\pi)^m\det \hat{\Omega}_t}}\times\exp\left(\frac{1}{2\lambda}x_t^\top (\hat{G}_t\hat{H}_t^{-1}\hat{G}_t^\top-\hat{F}_t)x_t-\frac{1}{\lambda}\hat{\kappa}_{t+1}\right)
 \\
&\times \exp\left(-\frac{1}{2\lambda}(u_t+\hat{H}_t^{-1}\hat{G}_t^\top x_t)^\top \hat{H}_t(u_t+\hat{H}_t^{-1}\hat{G}_t^\top x_t)\right).
\end{align*}
or \eqref{eq:lem_q_star} for short. From the condition $\int_{\mathcal{U}_t}Q_{U_t|X_t}^*(u_t|x_t)du_t=1$, $Z_t(x_t)$ is determined as
\begin{align*}
\!\!Z_t(x_t)\!=\!\!\sqrt{\!\!\frac{\det\!\! \left(\!\lambda\hat{H}_t^{-\!1}\!\right)}{\det \hat{\Omega}_t }}\!\exp\!\left(\!\!\frac{x_t^\top (\hat{G}_t\hat{H}_t^{-1}\hat{G}_t^\top-\hat{F}_t)x_t}{2\lambda}\!-\!\frac{\hat{\kappa}_{t+1}}{\lambda}\!\!\right)\!\!.
\end{align*}
Therefore, the value function is written as
\[
{J}_t(x_t)=-\lambda \log Z_t(x_t)=\frac{1}{2}x_t^\top \hat{\Theta}_t x_t+\hat{\kappa}_t
\]
where $\hat{\Theta}_t$ is defined by \eqref{eq:dc_lqr_riccati_theta} and $\hat{\kappa}_t$ by
\begin{align*}
    &\hat{\kappa}_t\!=\!\frac{\lambda}{2}\!\log\det \hat{\Omega}_t \!+\!\frac{\lambda}{2}\!\log\det \!\!\left(\!\!\hat{\Omega}_t^{-1}\!\!\!+\!\frac{1}{\lambda}B_t^\top \hat{\Theta}_{t+1}B_t\!\!\right)\!\!+\hat{\kappa}_{t+1} .\label{eq:dc_lqr_riccati_kappa}
\end{align*}
\end{proof}

Critically, observe that Riccati equations \eqref{eq: dc_lqr_theta} and \eqref{eq:dc_lqr_riccati_theta} coincide under Assumption~\ref{asmp:1}.
Moreover, the mean of the control input under the optimal policy \eqref{eq:lem_q_star} is
\begin{equation}
\mathbb{E}_{Q^*}(u_t|x_t)=-\hat{H}_t^{-1}\hat{G}_t^\top x_t=K_t x_t,
\end{equation}
which coincides with the LQR solution \eqref{eq:dc:lqr_policy}.
Since \eqref{eq: controller avg} provides us with a Monte-Carlo-based approach to compute $\mathbb{E}_{Q^*}(u_t|x_t)$, this connection implies that the optimal LQR input \eqref{eq:dc:lqr_policy} can be computed by Monte-Carlo simulations. Specifically, at each time step $t$, we generate sample {state-control trajectories}  $\{x_{t:T}(i), u_{t:T-1}(i)\}_{i=1}^n$ under the reference policy $R_{U_k|X_k}$. Since the reference policy is currently given by \eqref{eq:dc_lqr_ref}, from Assumption \ref{asmp:1}, this amounts to performing $n$ independent simulations of the ``uncontrolled" dynamics $X_{k+1}=A_kX_k+W_k, \; W_k\sim\mathcal{N}(0, \Omega_k)$ for $k=t, \cdots, T-1$ with the initial state $X_t=x_t$. The path reward $r(i)$ is computed for each sample path using \eqref{eq:dc_path_reward sample complexity}. Finally, the optimal control input for stochastic LQR is determined by 
\begin{equation}
    \label{eq:lqr_kl_control}
    u_t\approx \frac{\sum_{i=1}^{n} r(i)u_t(i)}{\sum_{i=1}^{n} r(i)}.
\end{equation}
Equation \eqref{eq:lqr_kl_control} is the discrete-time counterpart of the path integral control scheme introduced in the original work \cite{kappen2005path}.
Notice that the path-integral-based LQR implementation using \eqref{eq:lqr_kl_control} does not require solving the backward Riccati equation \eqref{eq:dc:lqr_policy}.

\section{Sample Complexity Analysis}
Suppose $u_t=K_tx_t$ represents the optimal control input computed by solving the classical Riccati equation (\ref{eq: dc_lqr_theta}) and $\hat{u}_t$ represents the control input obtained via path integral approach using a finite number of samples $n_t\in \mathbb{N}$, i.e.,
\begin{equation}\label{eq:PI with E}
    \hat{u}_t = \sum_{i=1}^{n_t} \frac{r(i)}{r}u_t(i) \text{ where } r=\sum_{i=1}^{n_t} r(i).
\end{equation}

\subsection{Sample Complexity Bound}

By the strong law of large numbers, $\hat{u}_t\overset{a.s.}{\rightarrow} u_t$, as $n_t\rightarrow\infty$. In this section, we analyze the finite sample performance of this approximation. Define the empirical means $\hat{E}$ and true expectations $E$ as

\begin{align}
\hat{E}^{ru}_t & \coloneqq\frac{1}{n_t} \sum_{i=1}^{n_t} r(i)u_t(i), & \quad \hat{E}^{r}_t & \coloneqq \frac{1}{n_t} \sum_{i=1}^{n_t} r(i), \label{eq:Ehat}\\ 
{E}^{ru}_t & \coloneqq \!\!\lim_{n_t\to\infty} \frac{1}{n_t} \sum_{i=1}^{n_t} r(i)u_t(i), & \quad  {E}^{r}_t & \coloneqq \!\! \lim_{n_t\to\infty} \frac{1}{n_t} \sum_{i=1}^{n_t} r(i).\label{eq:E}
\end{align}
Since $r(i)\in[0,1]$, by Hoeffding's inequality \cite{hoeffding1994probability}, we get
\begin{equation}\label{eq:eps1}
    \text{Pr}\left(\left|\hat{E}^{r}_t - {E}^{r}_t\right|\leq \gamma_1\right)\geq 1-2e^{-2n_t\gamma_1^2}.
\end{equation}
 We know that $u_t(i)$ is sampled from $R_{U_t|X_t}$ in \eqref{eq:dc_lqr_ref} which is a zero-mean Gaussian distribution with covariance $\hat{\Omega}_t$. Therefore, $u_t(i)\in \mathrm{SG}(\|\hat{\Omega}_t\|)$, where $\mathrm{SG}(\|\hat{\Omega}_t\|)$ represents a class of sub-Gaussian distributions with parameter $\|\hat{\Omega}_t\|$. 
 Since $r(i)\in [0, 1]$, we also have $\left(r(i)u_t(i)\right)\in\mathrm{SG}(\|\hat{\Omega}_t\|)$. Using Hoeffding's bound for sub-Gaussians and the union bound \cite{prekopa1988boole}, we get
{\begin{equation}\label{eq:eps2}
    \text{Pr}\left(\left\|\hat{E}^{ru}_t - {E}^{ru}_t\right\|_\infty\leq \gamma_2\right)\geq 1-2me^{-n_t\gamma_2^2/(2\|\hat{\Omega}_t\|)}
\end{equation}}
where $m$ is the dimension of $u_t(i)$. The following theorem provides an end-to-end sample complexity bound on the error in the path integral control signal. 
\begin{theorem}
\label{lem:complexity}
Let $\{\epsilon_t\}_{t=0}^{T-1}$, $\{\alpha_t\}_{t=0}^{T-1}$ and $\{\beta_t\}_{t=0}^{T-1}$ be given sequences of positive numbers and \begin{equation*}
    \epsilon \coloneqq \sum\nolimits_{t=0}^{T-1} \epsilon_t^2,\quad\alpha \coloneqq \sum\nolimits_{t=0}^{T-1} \alpha_t, \quad \beta \coloneqq \sum\nolimits_{t=0}^{T-1} \beta_t.
\end{equation*}
Suppose $\alpha+\beta<1$. If $\hat{u}_t$ is computed by \eqref{eq:PI with E} with $n_t$ satisfying 
{\begin{equation}\label{eq:Nt}
   \!\!\!\!\!\! n_t\! \geq\! \frac{\Bigg(\!\!\hat{E}^r_t\!\sqrt{2\|\hat{\Omega}_t\|\log\frac{2m}{\beta_t}}\!+\!\left(\!\epsilon_t\hat{E}^r_t\!\! +\! \|\hat{E}^{ru}_t\|_\infty\!\right)\!\!\sqrt{\frac{1}{2}\log\frac{2}{\alpha_t}}\!\Bigg)^{\!\!\!2}}{\epsilon_t^2(\hat{E}^r_t)^4}
\end{equation}}
 {\color{black} 
{and
\begin{equation}\label{eq:Nt2}
\hat{E}_t^r>\sqrt{\frac{1}{2n_t}\log\frac{2}{\alpha_t}}
\end{equation}
 }}
for all $t=0,1,\hdots, T-1$, then 
{\begin{equation}\label{eq:final bound}
    \|\hat{u}-u\|_\infty^2:=\sum_{t=0}^{T-1}\|\hat{u}_t-{u}_t\|_\infty^2 \leq \epsilon
\end{equation}}
with probability greater than or equal to $1-\alpha-\beta$.
\end{theorem}


\begin{proof}
{Using the definitions in \eqref{eq:Ehat} and \eqref{eq:E}, we have
\begin{align}\label{u_t-u_hat_t}
    \|u_t-\hat{u}_t\|_\infty  =& \|{E}^{ru}_t/{E}_t^r - \hat{E}^{ru}_t/\hat{E}^{r}_t\|_\infty\nonumber\\
    = & \frac{1}{{E}^{r}_t\hat{E}^{r}_t}\left\|{E}^{ru}_t\hat{E}^{r}_t-\hat{E}^{ru}_t\hat{E}^{r}_t + \hat{E}^{ru}_t\hat{E}^{r}_t-\hat{E}^{ru}_t{{E}}_{r}\right\|_\infty\nonumber\\
     \leq & \frac{1}{{E}^{r}_t\hat{E}^{r}_t}\!\!\left(\|{E}^{ru}_t\hat{E}^{r}_t-\hat{E}^{ru}_t\hat{E}^{r}_t\|_\infty + \|\hat{E}^{ru}_t\hat{E}^{r}_t-\hat{E}^{ru}_t{{E}}_{r}\|_\infty\right).
\end{align}
The absolute value in the denominator is removed since $r(i)\in(0,1]$. Given $\alpha_t$ and $n_t$, choose $\gamma_1 = \sqrt{\frac{1}{2n_t}\log\frac{2}{\alpha_t}}$. Then \eqref{eq:eps1} implies that 
\begin{equation}\label{gamma_1}
    \left|\hat{E}^{r}_t - {E}^{r}_t\right| \leq \sqrt{\frac{1}{2n_t}\log\frac{2}{\alpha_t}}
\end{equation}
holds with probability at least $1-\alpha_t$. Combining with \eqref{eq:Nt2}, we have
\begin{equation}\label{consequence of 35}
    0 < \hat{E}^{r}_t - \sqrt{\frac{1}{2n_t}\log\frac{2}{\alpha_t}} \leq {E}^{r}_t. 
\end{equation}
Similarly, choosing $\gamma_2 = \sqrt{\frac{2\|\hat{\Omega}_t\|}{n_t}\log\frac{2m}{\beta_t}}$, \eqref{eq:eps2} implies that 
\begin{equation}\label{gamma_2}
    \left\|\hat{E}^{ru}_t - {E}^{ru}_t\right\|_\infty\leq \sqrt{\frac{2\|\hat{\Omega}_t\|}{n_t}\log\frac{2m}{\beta_t}}
\end{equation}
holds with probability at least $1-\beta_t$. Applying \eqref{gamma_1}, \eqref{consequence of 35} and \eqref{gamma_2} to \eqref{u_t-u_hat_t}, and using union bound, we obtain that
\begin{equation}\label{eq:bound2}
\!\!\|u_t-\hat{u}_t\|_\infty \leq  \frac{\hat{E}^{r}_t\sqrt{\frac{2\|\hat{\Omega}_t\|}{n_t}\log\frac{2m}{\beta_t}} + \|\hat{E}^{ru}_t\|_\infty \sqrt{\frac{1}{2n_t}\log\frac{2}{\alpha_t}}}{\left(\hat{E}^r_t-\sqrt{\frac{1}{2n_t}\log\frac{2}{\alpha_t}}\right)\hat{E}^r_t} 
\end{equation}
with probability at least $1-\alpha_t-\beta_t$. Hence we have
\begin{equation}\label{eq:bound3}
\text{Pr}\left(\|u_t-\hat{u}_t\|_\infty^2 \leq  \epsilon_t^2\right) \geq  1-\alpha_t-\beta_t
\end{equation}
if the right hand side of \eqref{eq:bound2} is less than or equal to $\epsilon_t$. Solving this condition for $n_t$, we get \eqref{eq:Nt}. Therefore, we conclude that if $n_t$ satisfies \eqref{eq:Nt} and \eqref{eq:Nt2}, we obtain \eqref{eq:bound3}. Using the union bound on \eqref{eq:bound3}, we get the desired results.}

\end{proof}
\begin{remark}
    The sample complexity in \eqref{eq:Nt} reveals that the required number of samples for path integral control depends only logarithmically on the dimension $m$.
\end{remark}

\subsection{Upper Bound on the Control Performance}
Often in practice, we are interested in an upper bound on the performance loss of the control system. 
To obtain such a bound, we now analyze the impact of $\|\hat{u}-u\|_\infty$ on the control performance. 
 
Let $\hat{U}_t$ be the ``noisy" control input determined by the path integral controller. The closed-loop dynamics is 
\begin{equation}\label{eq:cl_Uhat}
   X_{t+1}={A}_tX_t+B_t\hat{U}_t+W_t, \quad W_t\sim \mathcal{N}(0,\Omega_t)  
\end{equation}
and the accrued LQR cost is 
\begin{equation}\label{eq:PI cost}  
  \!\!\!\! \mathcal{L}\coloneqq \!\mathbb{E}\!\left[ \sum_{t=0}^{T-1}\!\!\left(\!\frac{1}{2}X_t^{\!\top} \! {M}_t X_t \!+\! \frac{1}{2}\hat{U}_t^{\!\top} \! {N}_t \hat{U}_t\!\!\right)\!\!+\!\frac{1}{2}X_T^{\!\top} M_T X_T\!\right].
\end{equation}
Introducing $V_t:=\hat{U}_t-U_t$, \eqref{eq:cl_Uhat} can be rewritten as
\begin{equation}
    \label{eq:cl_adversary}
X_{t+1}=\widetilde{A}_tX_t+B_tV_t+W_t, \quad W_t\sim \mathcal{N}(0,\Omega_t)
\end{equation}
where we introduced $\widetilde{A}_t\coloneqq A_t + B_tK_t$. Also, (\ref{eq:PI cost}) can be written in terms of $V_t$ as
\begin{align*}
    \mathcal{L}= \mathbb{E} \sum_{t=0}^{T-1}\left(\frac{1}{2}X_t^\top \! \widetilde{M}_t X_t \!+\! X_t^\top \! \widetilde{N}_t {V}_t + \frac{1}{2}V_t^\top \! {N}_t {V}_t\right)+\mathbb{E}\left[\frac{1}{2}X_T^\top M_T X_T\right]
\end{align*}
where $\widetilde{M}_t\coloneqq M_t + K_t^\top N_tK_t$ and $\widetilde{N}_t\coloneqq K_t^\top N_t$. 
From Theorem~\ref{lem:complexity}, we know that if $n_t$ satisfies \eqref{eq:Nt} and \eqref{eq:Nt2} then  $\sum_{t=0}^{T-1}\|v_t\|_\infty^2\leq \epsilon$ with probability at least $1-\alpha-\beta$. Now we formulate the problem to search for the state feedback policy $v_t= \pi_t(x_t)$  that maximizes $\mathcal{L}$ while satisfying $\sum_{t=0}^{T-1}\|v_t\|_\infty^2\leq \epsilon$ with probability at least $1-\alpha-\beta$. The problem is formulated as the following chance-constrained optimization problem.
\begin{problem}[Chance-constrained LQR]\label{prob: worst-case loss}
  \begin{align}\label{eq:worst-case loss}
\max_{\{\pi_t(\cdot)\}_{t=0}^{T-1}}& \mathbb{E}\sum_{t=0}^{T-1}\left(\frac{1}{2}X_t^\top \! \widetilde{M}_t X_t \!+\! X_t^\top \! \widetilde{N}_t {V}_t + \frac{1}{2}V_t^\top \! {N}_t {V}_t\right)+\mathbb{E}\left[\frac{1}{2}X_T^\top M_T X_T\right]\\
\mathrm{s.t.} \quad&  X_{t+1}=\widetilde{A}_tX_t+B_tV_t+W_t, \quad W_t\sim \mathcal{N}(0,\Omega_t)\nonumber\\
&\text{Pr}\left(\sum\nolimits_{t=0}^{T-1}\|v_t\|_\infty^2\leq \epsilon\right)\geq1-\alpha-\beta. \nonumber
\end{align}  
\end{problem}
Let $f^*$ be the value of the optimization problem \ref{prob: worst-case loss}. Then, if $n_t$ satisfies \eqref{eq:Nt} and \eqref{eq:Nt2}, $ \mathcal{L}\leq f^*$.
Finding a worst-case policy $\pi_t$ that solves the Problem \ref{prob: worst-case loss} involving a chance constraint on the control input of the systems is inherently challenging. Some approaches have been proposed in the literature to solve the optimization problems with a chance constraint on the state of the system \cite{blackmore2011chance,zhou2013reliable}. Finding the solution to Problem \ref{prob: worst-case loss} is left as a topic for future work.

\section{Simulation Results}
{Consider an LQR problem where $A_t=\begin{bmatrix}
   0.9 & -0.1\\ -0.1 & 0.8  
\end{bmatrix}$, $B_t= \begin{bmatrix}
    1 \\ 0
\end{bmatrix}$, $\Omega_t = \begin{bmatrix}
   4 & 0\\ 0 & 0  
\end{bmatrix}$, $M_t=0.1I$, $N_t=10$. $I$ represents an identity matrix of size $2\times2$.} In figure \ref{Figure planned trajectories}, we plot the state and control input trajectories for two values of $n$. The red dashed lines represent the analytical solution obtained by the classical Riccati equation. The blue solid lines represent the solution obtained by path integral. As evident from the figure, as $n$ is increased, the trajectories obtained by path integral become less noisy and align well with the analytical LQR solution. {We also plot in figures \ref{Figure planned trajectories}(b) and \ref{Figure planned trajectories}(d) the bounds $u_t\pm\epsilon_t$. As one can see the error $\epsilon_t$ reduces as $n$ is increased.}
\begin{figure}[tbhp]
    \centering
      \begin{tabular}{c c}
      \includegraphics[scale=0.35]{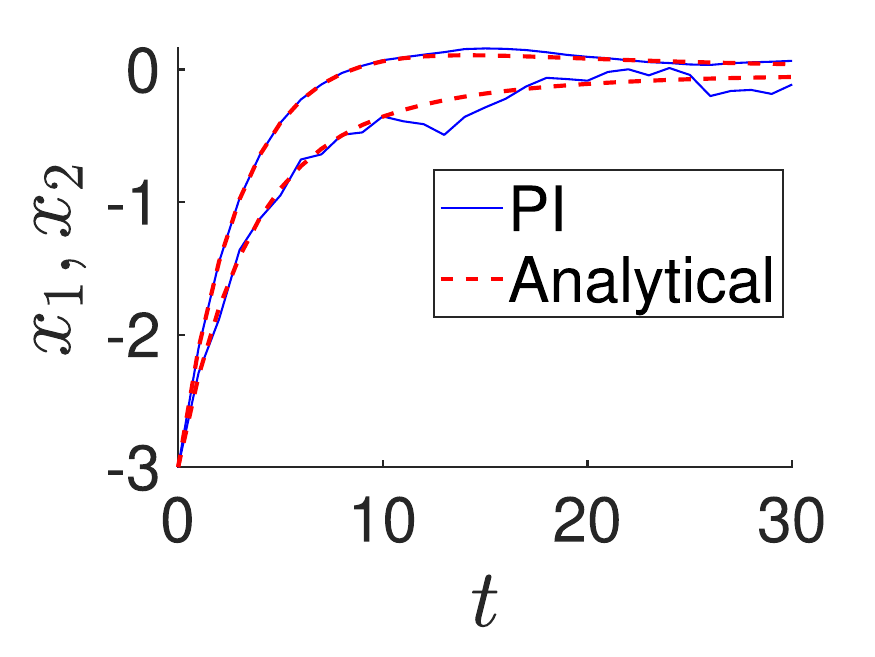} &\!\!\!\!\!\!\includegraphics[scale=0.35]{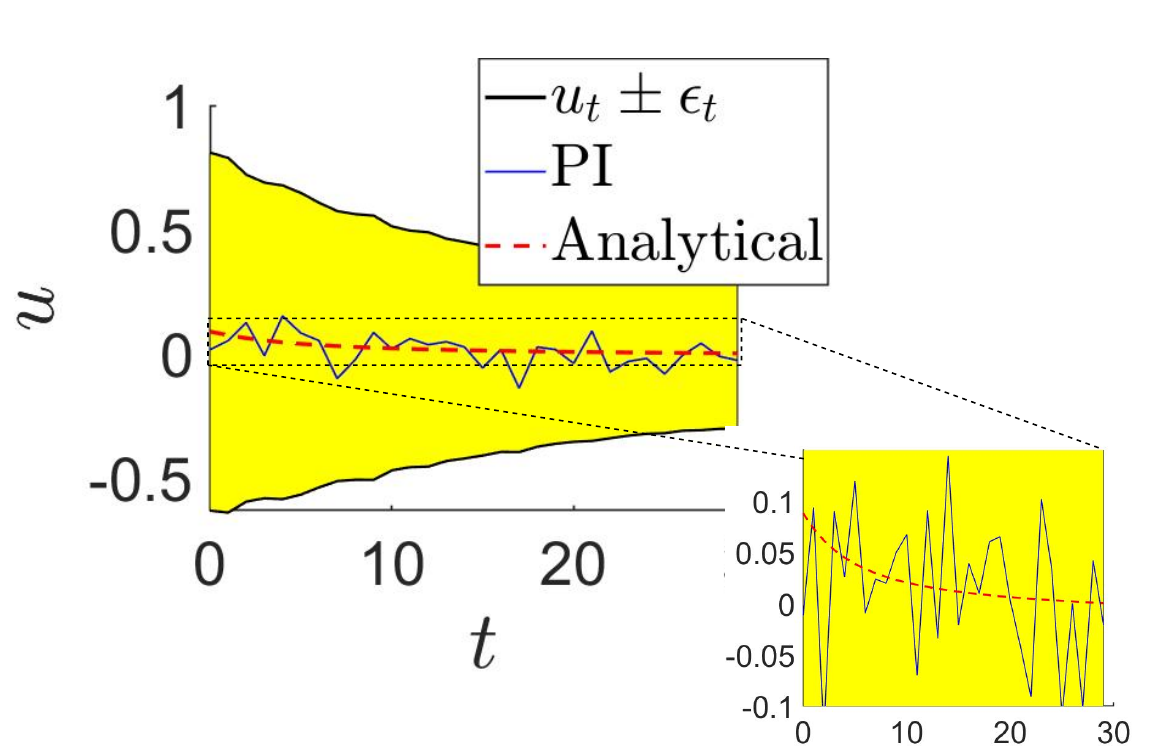} \\
      (a) State trajectory with $n=10^{3}$&(b) Control input trajectory with $n=10^{3}$\\
        \includegraphics[scale=0.35]{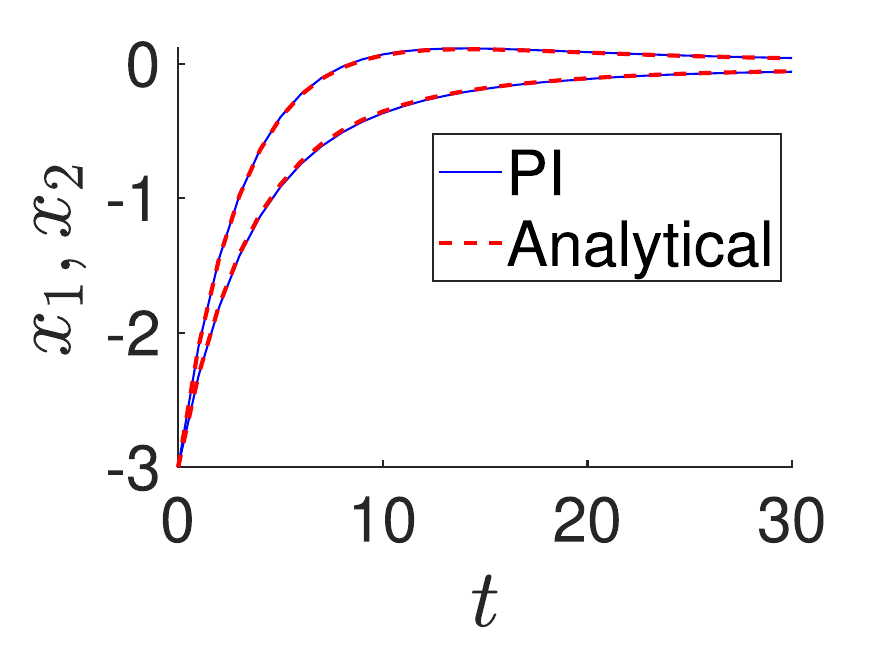} &\!\!\!\!\!\!\includegraphics[scale=0.35]{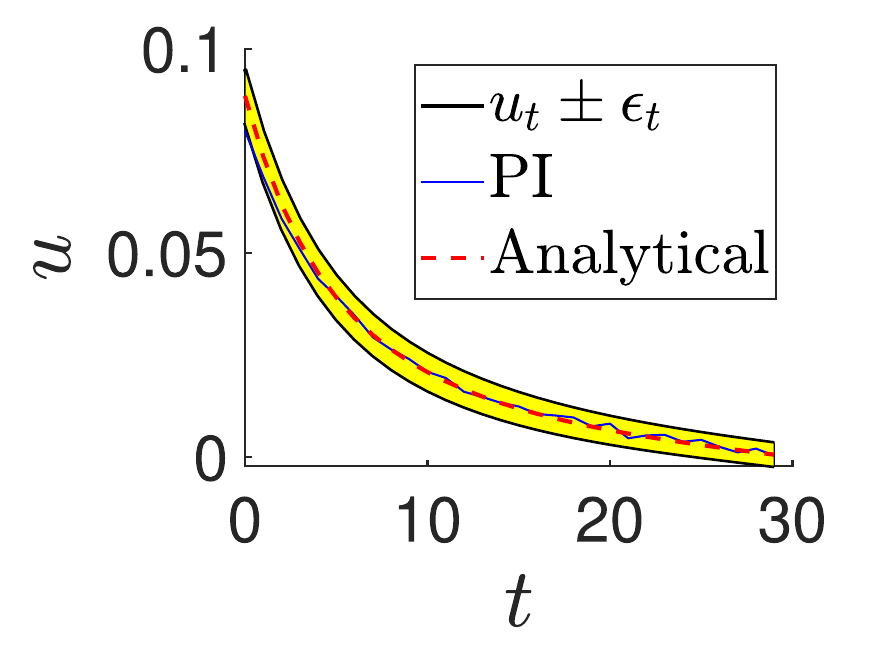} \\
      (c) State trajectory  with $n=10^{7}$ &(d) Control input trajectory with $n=10^{7}$\\
       \end{tabular}
        \caption{State and input trajectories for two values of $n$. The red dashed line represents the solution obtained by the Riccati equation, whereas the blue solid line represents the solution obtained by path integral control. {The bounds $u_t\pm\epsilon_t$ are plotted in (b) and (d).}}
        \label{Figure planned trajectories}
\end{figure}
Figure \ref{Fig: stuff vs sample size}(a) represents how $\epsilon_0$ changes with the sample size. 
Figure \ref{Fig: stuff vs sample size}(b) represents the LQR costs obtained by the classical Riccati equation (red dotted line) and by the path integral approach (blue solid line) as functions of the sample size. Notice that the path integral  solution converges to the analytical solution as the sample size is increased.   

\begin{figure}[tbhp]
    \centering
      \begin{tabular}{c c}
    \includegraphics[scale=0.4]{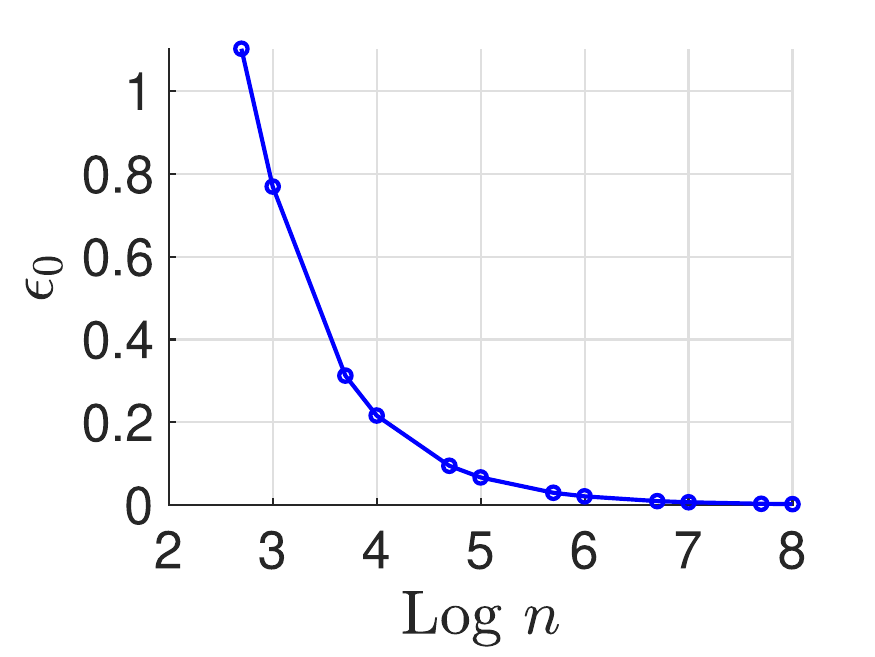} &\includegraphics[scale=0.4]{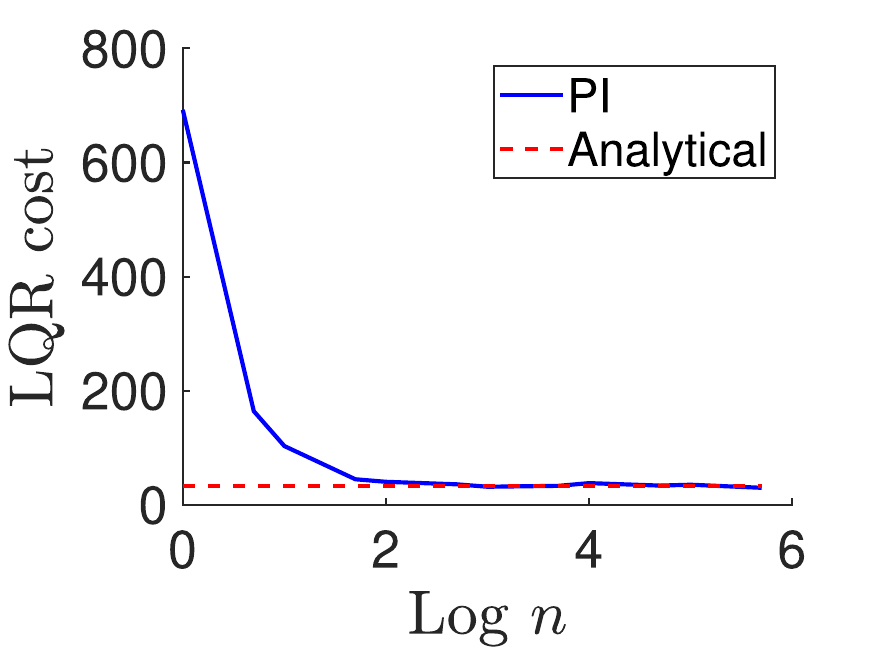} \\
      (a) $\epsilon_0$ vs sample size&(b) LQR costs vs sample size
       \end{tabular}
        \caption{$\epsilon_0$ and LQR cost vs sample size}
        \label{Fig: stuff vs sample size}
\end{figure}

\section{Publications}
\begin{itemize}
     \item \textbf{A. Patil}, G. Hanasusanto, T. Tanaka, ``Discrete-Time Stochastic LQR via Path Integral Control and Its Sample Complexity Analysis," \textit{IEEE Control Systems Letters (L-CSS) }
     \item \textbf{A. Patil}, G. Hanasusanto, T. Tanaka, ``Discrete-Time Stochastic LQR via Path Integral Control and Its Sample Complexity Analysis," \textit{2024 IEEE Conference on Decision and Control (CDC)}
\end{itemize}
\section{Future Work}
Future work will build upon the foundation presented in this dissertation to carry out sample complexity analysis of path integral for nonlinear continuous-time stochastic control problems. Another direction we would like to work on is to ``robustify" the path integral control method by exploiting techniques from $H^\infty$ control.


%% file: chapters/publications.tex
\chapter[Publications]{Publications}
\label{Sec: publications}
This section includes all the publications from my PhD studies. 
\section{Journal Publications}
\begin{itemize}
     \item \textbf{A. Patil}, G. Hanasusanto, T. Tanaka, ``Discrete-Time Stochastic LQR via Path Integral Control and Its Sample Complexity Analysis," \textit{IEEE Control Systems Letters (L-CSS) }
    \item \textbf{A. Patil}, A. Duarte, F. Bisetti, T. Tanaka, ``Strong Duality and Dual Ascent Approach to Continuous-Time Chance-Constrained Stochastic Optimal Control," \textit{under review in Transactions on Automatic Control (TAC)}
     \item \textbf{A. Patil}, K. Morgenstein,  L. Sentis, T. Tanaka, ``Deceptive Attack Synthesis and Its Mitigation for Nonlinear Cyber-Physical Systems: Path Integral Approach," \textit{under review in Transactions on Automatic Control (TAC)}
   \item M. Baglioni, \textbf{A. Patil}, L. Sentis, A. Jamshidnejad ``Hierarchical Optimal Control for Multi-UAV Best Viewpoint Coordination with Obstacle Avoidance," \textit{In preparation}
\end{itemize}

\section{Conference Publications}
\begin{itemize}
\item C. Martin, \textbf{A. Patil}, W. Li, T. Tanaka, D. Chen, ``Model Predictive Path Integral Control for Roll-to-Roll Manufacturing", \textit{submitted to  Modeling, Estimation and Control Conference, (MECC)}, 2025.
\item \textbf{A. Patil}, G. Hanasusanto, T. Tanaka, ``Discrete-Time Stochastic LQR via Path Integral Control and Its Sample Complexity Analysis," \textit{2024 IEEE Conference on Decision and Control (CDC)}
 \item \textbf{A. Patil}*, M. Karabag*, U. Topcu,  T. Tanaka, ``Simulation-Driven Deceptive Control via Path Integral  Approach," \textit{2023 IEEE Conference on Decision and Control (CDC)}
    \item \textbf{A. Patil}, Y. Zhou, D. Fridovich-Keil, T. Tanaka, ``Risk-Minimizing Two-Player Zero-Sum Stochastic Differential Game via Path Integral Control," \textit{2023 IEEE Conference on Decision and Control (CDC)}

    \item \textbf{A. Patil}, T. Tanaka, ``Upper and Lower Bounds for End-to-End Risks in Stochastic Robot Navigation," \textit{2023 IFAC World Congress}
     \item \textbf{A. Patil}, A. Duarte, A. Smith, F. Bisetti, T. Tanaka, ``Chance-Constrained Stochastic Optimal Control via Path Integral and Finite Difference Methods," \textit{2022 IEEE Conference on Decision and Control (CDC)}
      \item \textbf{A. Patil}, T. Tanaka, ``Upper Bounds for Continuous-Time End-to-End Risks in Stochastic Robot Navigation," \textit{2022 European Control Conference (ECC)}
       \item \textbf{A. Patil}, R. Funada, T. Tanaka, L. Sentis, ``Task Hierarchical Control via Null-Space Projection and Path Integral Approach," \textit{2024 American Control Conference (ACC)}
\end{itemize}

%% file: chapters/summary.tex
\chapter[Summary]{Summary}
\label{Sec: conclusion}
The summary of this dissertation is depicted in Figure \ref{Fig. Conclusion}. The dissertation focuses on advancing path integral control theory for solving Stochastic Optimal Control (SOC) problems, which are prevalent in systems influenced by uncertainties and random disturbances, such as autonomous robotics. Traditional SOC methods like dynamic programming face computational challenges, particularly in high-dimensional, nonlinear systems. The path integral approach offers a promising alternative by transforming the SOC problem into an expectation over noisy system trajectories, allowing for efficient computation using Monte Carlo sampling. The dissertation explores and develops the path integral control theory for six different classes of SOC problems as shown in Figure \ref{Fig. Conclusion}, with an emphasis on real-time applications and scalability through parallel computing on GPUs. Through these contributions, the research enhances the utility of the path integral control method for handling complex, nonlinear, and safety-critical systems. This work also presents a sample complexity analysis and outlines future directions for developing autonomous systems based on path integral control theory.
\begin{figure}[h]
    \centering
    \!\!\!\!\!\!\!\!\!\!\!\!\includegraphics[scale=0.42]{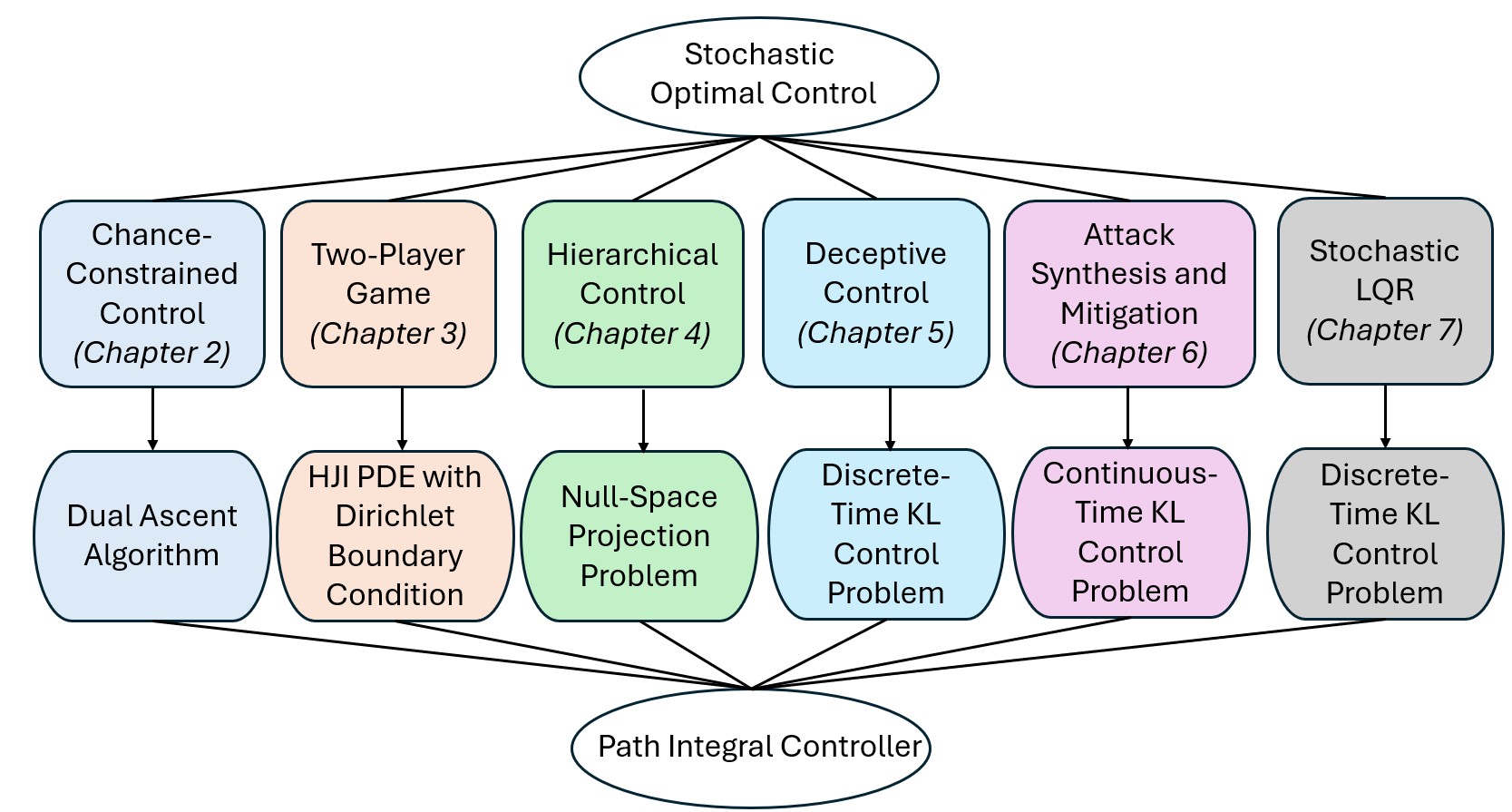} 
      
        \caption{Summary of the dissertation} 
        \label{Fig. Conclusion}
\end{figure}

%% file: chapters/appendix.tex
\chapter{Appendix} \label{Sec: appendix}
\section{Nonconvex Optimization and Strong Duality}\label{sec: Nonconvex Optimization and Strong Duality}
Let $\mathcal{X}$ be a vector space. Suppose $f_0$ and $f_1$ are not necessarily convex, real-valued functions defined on $\mathcal{X}$. 
The domains of the functions $f_0$ and $f_1$ are denoted by $\text{dom}(f_0)$ and $\text{dom}(f_1)$, respectively.
Consider the following optimization problem with a single inequality constraint:
\begin{subequations}
\label{eq:primal_prob}
\begin{align}
\min_{x\in \mathcal{X}} &\quad f_0(x) \\
&\quad f_1(x)\leq 0.
\end{align}
\end{subequations}
We assume $\text{dom}(f_0)\cap\text{dom}(f_1)$ is nonempty, and $f_0(x)>-\infty \;\; \forall x \in \text{dom}(f_0)\cap\text{dom}(f_1)$.

Let $\eta\geq 0$ be the dual variable. Define the dual function $g:[0,\infty)\rightarrow \bar{\mathbb{R}}$ by
\begin{equation}
\label{eq:def_dual_func}
g(\eta)=\inf_{x\in\mathcal{X}} \; f_0(x)+\eta f_1(x).
\end{equation}
Since \eqref{eq:def_dual_func} is a pointwise infimum of a family of affine functions, it is concave. Moreover, since affine functions are upper semicontinuous, $g(\eta)$ is also upper semicontinuous.

\begin{assumption}
\label{asmp:strict_feasible}
    Problem \eqref{eq:primal_prob} is strictly feasible. That is, there exists $x_0\in \text{dom}(f_0)$ such that $f_1(x_0)<0$.
\end{assumption}
\begin{lemma}
\label{lem:existence_dual_sol}
    Under Assumption~\ref{asmp:strict_feasible}, there exists a dual optimal solution $\eta^*$ such that $0\leq \eta^*<\infty$ such that $g(\eta^*)=\sup_{\eta \geq 0} g(\eta)$.
\end{lemma}
\begin{proof}
Since $g(\eta)$ is upper semicontinuous, by Weierstrass' theorem \cite[Proposition A.8]{bertsekas1997nonlinear}, it is sufficient to prove that the function $-g(\eta)$ is coersive. Specifically, it is sufficient to show that there exists a scalar $\gamma$ such that the set 
\[
\eta(\gamma):=\{\eta \geq 0 \mid g(\eta) \geq \gamma\}
\]
is nonempty and compact.

We will show that the choice $\gamma:=\inf_{x\in\mathcal{X}} f_0(x)$ will make $\eta(\gamma)$ nonempty and bounded. To see that $\eta(\gamma)$ is nonempty, notice that 
\[
g(0)=\inf_{x\in\mathcal{X}} f_0(x)=\gamma
\]
and thus $0\in \eta(\gamma)$. 

To prove $\eta(\gamma)$ is bounded, we need to show that $g(\eta)\searrow -\infty$ as $\eta \rightarrow +\infty$. Let $x_0$ be a strictly feasible solution, i.e., $f_1(x_0)<0$. We have
\begin{align*}
g(\eta)&=\inf_x \; f_0(x)+\eta f_1(x) \\
&\leq f_0(x_0)+\eta \underbrace{f_1(x_0)}_{<0}\rightarrow -\infty 
\end{align*}
as $\eta \rightarrow +\infty$. This completes the proof.
\end{proof}
To proceed further, we need an additional set of assumptions.
\begin{assumption}
\label{asmp:continuity}
    \begin{itemize}
    \item[(i)] For each $\eta\geq 0$, the set
    \[
    X(\eta):=\argmin_{x\in\mathcal{X}} \; f_0(x)+\eta f_1(x)
    \]
    is nonempty.
    \item[(ii)] The function $\eta \mapsto f_1(X(\eta)): [0,\infty)\rightarrow \mathbb{R}$ is well-defined. That is, if $x_1\in X(\eta)$ and $x_2\in X(\eta)$, then $f_1(x_1)=f_2(x_2)$.
    \item[(iii)] The function $f_1(X(\cdot))$ is continuous.
    \end{itemize}
\end{assumption}
\begin{lemma}
\label{lem:comp_slackness}
Suppose Assumption~\ref{asmp:strict_feasible} holds so that a dual optimal solution $0\leq \eta^*<+\infty$ exists (c.f., Lemma~\ref{lem:existence_dual_sol}). The following statements hold:
\begin{itemize}
\item[(i)] If $\eta^*=0$, then $f_1(X(\eta^*))\leq 0$.
\item[(ii)] If $\eta^*>0$, then $f_1(X(\eta^*))= 0$.
\end{itemize}
\end{lemma}
\begin{proof}
(i) Suppose $\eta^*=0$ but $f_1(X(\eta^*))>0$. Set $\eta_\epsilon=\epsilon>0$.
By continuity of $f_1(X(\cdot))$, we have $f_1(X(\eta_\epsilon))>0$ for a sufficiently small $\epsilon$. Notice that
\begin{align*}
g(\eta_\epsilon)&=f_0(X(\eta_\epsilon))+\underbrace{\eta_\epsilon f_1(X(\eta_\epsilon))}_{>0} \\
&>f_0(X(\eta_\epsilon)) \\
&\geq \min_{x\in\mathcal{X}} \; f_0(x) \\
&=\min_{x\in\mathcal{X}}  \; f_0(x)+\eta^* f_1(x) \\
&=g(\eta^*).
\end{align*}
However, this chain of inequalities contradicts to the fact that $\eta^*$ is the dual optimal solution (i.e., $\eta^*$ maximizes $g(\eta)$).

(ii) Suppose $\eta^*>0$ and $f_1(X(\eta^*))>0$.
Set $\eta_\epsilon=\eta^*+\epsilon$ where $\epsilon>0$ is sufficiently small.
By continuity of $f_1(X(\cdot))$, we have $f_1(X(\eta_\epsilon))>0$. Notice that 
\begin{align*}
g(\eta_\epsilon)&=f_0(X(\eta_\epsilon))+\eta_\epsilon f_1(X(\eta_\epsilon))\\
&=f_0(X(\eta_\epsilon))+\eta^* f_1(X(\eta_\epsilon))+\underbrace{\epsilon f_1(X(\eta_\epsilon))}_{>0} \\
&>f_0(X(\eta_\epsilon))+\eta^* f_1(X(\eta_\epsilon)) \\
&\geq \min_{x\in\mathcal{X}}  \; f_0(x)+\eta^* f_1(x) \\
&=g(\eta^*).
\end{align*}
This contradicts to the fact that $\eta^*$ maximizes $g(\eta)$.

Similarly, suppose $\eta^*>0$ and $f_1(X(\eta^*))<0$.
Set $\eta_\epsilon=\eta^*-\epsilon$ and choose $\epsilon>0$ sufficiently small so that $\eta_\epsilon>0$ and $f_1(X(\eta_\epsilon))<0$. We have
\begin{align*}
g(\eta_\epsilon)&=f_0(X(\eta_\epsilon))+\eta_\epsilon f_1(X(\eta_\epsilon))\\
&=f_0(X(\eta_\epsilon))+\eta^* f_1(X(\eta_\epsilon))-\epsilon \underbrace{f_1(X(\eta_\epsilon))}_{<0} \\
&>f_0(X(\eta_\epsilon))+\eta^* f_1(X(\eta_\epsilon)) \\
&\geq \min_{x\in\mathcal{X}}  \; f_0(x)+\eta^* f_1(x) \\
&=g(\eta^*).
\end{align*}
Again, this contradicts to the fact that $\eta^*$ maximizes $g(\eta)$. Therefore, if $\eta^*>0$ then $f_1(X(\eta^*))=0$ must hold.
\end{proof}
The following is the main result of this section.
\begin{theorem}\label{theorem: strong duality2}
    Consider problem \eqref{eq:primal_prob} and suppose Assumptions~\ref{asmp:strict_feasible} and \ref{asmp:continuity} hold. Then, there exists a dual optimal solution $0\leq \eta^*<+\infty$ that maximizes $g(\eta)$.
    Moreover, the set \[
    X(\eta^*):=\argmin_{x\in\mathcal{X}} \; f_0(x)+\eta^* f_1(x)
    \]
    is nonempty, and any element $x^*\in X(\eta^*)$ is a primal optimal solution such that $f_0(x^*)=g(\eta^*)$. i.e., the duality gap is zero. Consequently, a minimizer of $f_0(x)+\eta^*f_1(x)$ is an optimal solution of \eqref{eq:primal_prob}.
\end{theorem}
\begin{proof}
By Lemma~\ref{lem:comp_slackness}, any element $x^*\in X(\eta^*)$ satisfies $f_1(x^*)\leq 0$. Hence, $x^*$ satisfies primal feasibility.
It also follows from Lemma~\ref{lem:comp_slackness} that $\eta^* f_1(x^*)=0$, i.e., the complementary slackness condition holds. Therefore,
\begin{align*}
f_0(x^*)&=f_0(x^*)+\eta^* f_1(x^*) \quad \text{(Complementary slackness)}\\
&=\min_{x\in \mathcal{X}} f_0(x)+\eta^* f(x) \quad \text{(since $x^*\in X(\eta^*)$)} \\
&=g(\eta^*).
\end{align*}
Hence, the strong duality holds.
Therefore, $x^*$ is a primal optimal solution.
It follows from the identity above that a minimizer of $f_0(x)+\eta^*f_1(x)$ is an optimal solution for \eqref{eq:primal_prob}.
\end{proof}

\section{Legendre Duality}
\label{sec:Legendre Duality}

Let $P$ and $Q$ be probability distributions on $(\mathcal{X}, \mathcal{B}(\mathcal{X}))$, and $C:\mathcal{X}\rightarrow \mathbb{R}$ a given cost function. Define the internal energy $U(Q,C)$, free energy $F(P,C)$ and relative entropy (KL divergence) $D(Q\|P)$ as:
\begin{align*}
&U(Q,C):=\int_\mathcal{X} C(x)Q(dx)  \\
&F(P,C):= -\lambda \log \int_\mathcal{X} \exp\left(-\frac{C(x)}{\lambda}\right)P(dx) \\
&D(Q\|P):= \int_\mathcal{X} \log\frac{dQ}{dP}(x)Q(dx).
\end{align*}
Then the following duality relationship holds:
\begin{align*}
F(P,C)&=\inf_Q \{ U(Q,C)+\lambda D(Q\|P) \} \\
    -\lambda D(Q\|P)&=\inf_C \{ U(Q,C)-F(P,C) \}.
\end{align*}
Also, the optimal probability distribution $Q^*$ is given by
\begin{equation*}
    Q^*(B) = \frac{\int_{B}\exp(-C(x)/\lambda)P(dx)}{\int_{\mathcal{X}} \exp(-C(x)/\lambda)P(dx)}, \quad \forall B \in \mathcal{B}(\mathcal{X}).
\end{equation*}
See \cite{boue1998variational,theodorou2012relative} for further discussions.

\section{The Likelihood Ratio $\frac{dQ^*}{dP}(x)$}\label{sec: The Likelihood Ratio}
Let $Q(x)$ be the probability distribution of the trajectories defined by system \eqref{SDE} under the given policy $u(\cdot)$ and $Q^*(x)$ be the probability distribution of the trajectories defined by system \eqref{SDE} under the optimal policy $u^*(\cdot)$ Let $P(x)$ be the probability distribution of the trajectories defined by the uncontrolled system \eqref{uncontrolled SDE}. Suppose the likelihood ratio of observing a sample path $x$ under the distributions $Q$ and $P$ is denoted by $\frac{dQ}{dP}(x)$ and the expectation under any distribution $Q$ is denoted by $\mathbb{E}^Q[\cdot]$. Let $k(t)$ be a process defined by 
\begin{equation}\label{k(t)}
    k(t) \coloneqq \Sigma^{-1}Gu(t).
\end{equation}
Now, we state the following theorem:
\begin{theorem}[The Girsanov theorem]
    The likelihood ratio $\frac{dQ}{dP}(x)$ of observing a sample path $x$ under distributions $Q$ and $P$ is given by 
    \begin{equation}\label{dQ*/dP}
        \frac{dQ}{dP}(x) = \exp\left({\int_{t_0}^{{{t}}_{f}}}k(t)^\top d{w}(t) + \frac{1}{2}\int_{t_0}^{{{t}}_{f}}\|k(t)\|^2dt\right)
    \end{equation}
    where the process $k(t)$ is defined by \eqref{k(t)}. 
\end{theorem}
\begin{proof}
    Refer to \cite[Chapter 8.6]{oksendal2013stochastic}.
\end{proof}

\begin{theorem}\label{Thm: likelihood ratio}
    Suppose we generate an ensemble of $N$ trajectories $ \{x^{(i)}\}_{i=1}^N$ under the distribution $P$. Let $r^{(i)}$ be the path reward associated with the trajectory $x^{(i)}$ as given in \eqref{r(i)}. Define $r \coloneqq \sum_{i=1}^{N}r^{(i)}$. Then as $N\rightarrow\infty$, 
    
    \begin{equation*}
    \frac{r^{(i)}}{r/N} \overset{a.s.}{\rightarrow} \frac{dQ^*}{dP}(x)     
    \end{equation*}
\end{theorem}
\begin{proof}
    The likelihood ratio $\frac{dQ}{dP}(x)$ of observing a sample path $x$ is given by \eqref{dQ*/dP}. Using this result, we obtain
   
    \begin{equation}\label{E log likelihood}
        \mathbb{E}^{Q}\log\left(\frac{dQ}{dP}(x)\right) =\frac{1}{2}\int_{t_0}^{{t}_f} \mathbb{E}^Q\|k(t)\|^2dt.
    \end{equation}
In \eqref{E log likelihood}, we used the property of It\^o integral \cite[Chapter 3]{oksendal2013stochastic}:
    \begin{equation*}
        \mathbb{E}^Q\left[{\int_{t_0}^{{{t}}_{f}}}k(t)^\top d{w}(t)\right] = 0
    \end{equation*} 
    Now, for a given value of $\eta$, we wanted to solve the following problem
    \begin{align}\label{optimization problem}
        \min_{u(\cdot)}\mathbb{E}^Q_{x_0, t_0}\!\!\left[\phi\!\left({x}({t}_{f}); \eta\right)\!+\!\!\!\int_{t_0}^{{t}_{f}}\!\!\!\left(\!\frac{1}{2}{u}^\top\!R{u}+V\!\right)\!dt\!\right].
    \end{align}
    According to Assumption \ref{Assum: linearity}, there exists a positive constant $\lambda$ such that
    \begin{equation}\label{linearity 2}
    R = \lambda G^\top\Sigma^{-\top}\Sigma^{-1}G.
    \end{equation}
   Now, we get
\begin{subequations}
    \begin{align}
        &\min_{u(\cdot)}\mathbb{E}^Q_{x_0, t_0}\!\!\left[\phi\!\left({x}({t}_{f}); \eta\right)\!+\!\!\!\int_{t_0}^{{t}_{f}}\!\!\!\left(\!\frac{1}{2}{u}^\top\!R{u}+V\!\right)\!dt\!\right]\label{with R}\\
        = & \min_{u(\cdot)}\mathbb{E}^Q_{x_0, t_0}\!\!\left[\phi\!\left({x}({t}_{f}); \eta\right)\!+\!\!\!\int_{t_0}^{{t}_{f}}\!\!\!\left(\!\frac{1}{2}{u}^{\!\!\top}\!\lambda G^{\!\top}\Sigma^{-\!\top}\Sigma^{-\!1}G{u}\!+\!\!V\!\!\right)\!dt\!\right]\label{replacing R}\\
        = & \min_{k(\cdot)}\mathbb{E}^Q_{x_0, t_0}\!\!\left[\phi\!\left({x}({t}_{f}); \eta\right)\!+\!\!\!\int_{t_0}^{{t}_{f}}\!\!\!\left(\!\frac{\lambda}{2}\|k(t)\|^2+V\!\right)\!dt\!\right]\label{replacing with k}\\
        = & \min_{Q(x)}\int_{\mathcal{T}}\left( \phi\!\left({x}({t}_{f}); \eta\right) + \int_{t_0}^{{t}_f}V dt + \lambda\log\frac{dQ}{dP}(x)\right)Q(dx).\label{minimization under Q}
        \end{align}
\end{subequations}
Equation \eqref{replacing R} is obtained by plugging  \eqref{linearity 2} into \eqref{with R}. \eqref{replacing with k} obtained by using \eqref{k(t)} and \eqref{minimization under Q} by using \eqref{E log likelihood}. Thus, we converted the problem \eqref{optimization problem} into a KL control problem.
Invoking the Legendre duality between the KL divergence and free energy (See Appendix \ref{sec:Legendre Duality}) it can be shown that there exists a minimizer $Q^*$ of \eqref{minimization under Q} which can be written as 
\begin{align}\label{Q*}
    Q^*(dx) = \frac{\exp\left(-\frac{1}{\lambda}\left(\phi\!\left({x}({t}_{f}); \eta\right) + \int_{t_0}^{{t}_f}V dt\right)\right)P(dx)}{\mathbb{E}^P \left[\exp\left(-\frac{1}{\lambda}\left(\phi\!\left({x}({t}_{f}); \eta\right) + \int_{t_0}^{{t}_f}V dt\right)\right)\right]}
\end{align}
Using \eqref{Q*}, we can write the expression for the Radon-Nikodym derivative $\frac{dQ^*}{dP}(x)$ as 
\begin{equation}\label{dQ/dP}
    \frac{dQ^*}{dP}(x) = \frac{\exp\left(-\frac{1}{\lambda}\left(\phi\!\left({x}({t}_{f}); \eta\right) + \int_{t_0}^{{t}_f}V dt\right)\right)}{\mathbb{E}^P \left[\exp\left(-\frac{1}{\lambda}\left(\phi\!\left({x}({t}_{f}); \eta\right) + \int_{t_0}^{{t}_f}V dt\right)\right)\right]}
\end{equation}
\end{proof}
Using \eqref{dQ/dP}, we can conclude that given the ensemble of N trajectories $\{x^{(i)}\}_{i=1}^N$ sampled under distribution $P$, the likelihood ratio $\frac{dQ^*}{dP}$ of observing a sample path $x^{(i)}$ is given by $\frac{r^{(i)}}{r/N}$ where $r^{(i)}$ is defined by \eqref{r(i)} and $r \coloneqq \sum_{i=1}^{N}r^{(i)}$. Using the strong law of large numbers as $N\rightarrow\infty$, we get  
\begin{equation*}
    \frac{r^{(i)}}{r/N} \overset{a.s.}{\rightarrow} \frac{dQ^*}{dP}(x).  
\end{equation*}

\section{Attack on Diffusion Coefficient}\label{sec: Attack on Diffusion Coefficient}
Consider the following hypothesis testing problem on a sample path $v_t, 0\leq t \leq 1$:
\begin{align*}
H_0 :& \quad dv_t=dw_t, \; v_0=0 \\
H_1 :& \quad dv_t=\sigma dw_t, \; v_0=0, \; \sigma=1.1.
\end{align*}
That is, under $H_0$, $v_t$ follows the law of the standard Brownian motion, whereas under $H_1$, $v_t$ is slightly ``noisier." Fig.~\ref{fig:h0h1} shows ten independent sample paths of $v_t$ generated under each hypothesis.

\begin{figure}
\centering
\includegraphics[scale=0.5]{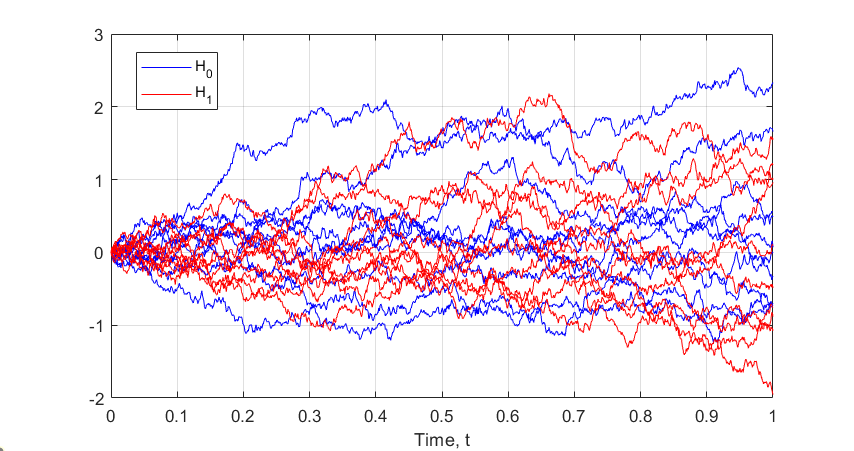}
\caption{Sample paths under $H_0$ and $H_1$.}
\label{fig:h0h1}
\end{figure}

Suppose the detector's task is to determine $H_0$ vs. $H_1$ based on discrete-time measurements $v_{t_k}, t_k=hk$, $k=1,2,\cdots, K(=1/h)$ where $h$ is the discretization step size.
If the underlying model is $dv_t=\sigma dw_t$, then the detector observes $K=(1/h)$ i.i.d. samples of increments:
\[
y_k:=v_{t_k}-v_{t_{k-1}}
\stackrel{\text{i.i.d.}}{\sim} \mathcal{N}(0, \sigma^2 h), \;\; k=1, 2, \cdots, K(=1/h).
\]
If we denote $P_0=\mathcal{N}(0,h)$ and $P_1=\mathcal{N}(0,\sigma^2 h)$, the likelihood ratio is
\begin{equation*}
    \frac{dP_1}{dP_0}(y_1, \cdots, y_K)=\prod_{k=1}^K \frac{\frac{1}{\sqrt{2\pi \sigma^2 h}}\exp\left(-\frac{y_k^2}{2\sigma^2 h}\right)}{\frac{1}{\sqrt{2\pi h}}\exp\left(-\frac{y_k^2}{2 h}\right)}
=
\frac{1}{\sigma^K}\exp\left[\frac{1}{2h}\left(1-\frac{1}{\sigma^2}\right)\sum_{k=1}^K y_k^2\right].
\end{equation*}
Thus, the acceptance region for $H_1$ under the Neyman-Pearson test is
\begin{equation}
\frac{1}{\sigma^K}\exp\left[\frac{1}{2h}\left(1-\frac{1}{\sigma^2}\right)\sum_{k=1}^K y_k^2\right] \geq \tau
 \label{eq:h1_region}
\end{equation}
where $\tau$ is a threshold.

Let us represent the probability of type-I error (false alarm) by $\alpha$ and the probability of type-II error (failure of detection) by $\beta$. In this example, both $\alpha$ and $\beta$ can be computed analytically. 
The false alarm rate $\alpha$ is the probability of the event \eqref{eq:h1_region} when $y_k\stackrel{\text{i.i.d.}}{\sim} \mathcal{N}(0, h)$.
Notice that \eqref{eq:h1_region} is equivalent to
\[
\sum_{k=1}^K \left(\frac{y_k}{\sqrt{h}}\right)^2 \geq \frac{2\sigma^2}{\sigma^2-1}(K\log \sigma + \log \tau)
\]
and $\sum_{k=1}^K \left(\frac{y_k}{\sqrt{h}}\right)^2\sim \chi^2(K)$ (the $\chi^2$-distribution with $K$ degrees of freedom).
Since the CDF of $\chi^2(K)$ is
\[
F(x)=P\left(\frac{x}{2},\frac{K}{2}\right)=\frac{1}{\Gamma\left(\frac{K}{2}\right)}\int_0^{x/2} t^{K/2-1}e^{-t}dt
\]
where $\Gamma(a)$ is the Gamma function and $P(x,a)$ is the regularized lower incomplete Gamma function, we have
\begin{equation}
\alpha = 1-P\left(\frac{\sigma^2}{\sigma^2-1}(K\log \sigma + \log \tau), \frac{K}{2}\right)=Q\left(\frac{\sigma^2}{\sigma^2-1}(K\log \sigma + \log \tau), \frac{K}{2}\right)
 \label{eq:alpha}
\end{equation}
where $Q(x,a)$ is the regularized upper incomplete Gamma function.\par

On the other hand, the detection failure rate $\beta$ is the probability of the event
\[
\sum_{k=1}^K \left(\frac{y_k}{\sqrt{\sigma^2 h}}\right)^2 < \frac{2}{\sigma^2-1}(K\log \sigma + \log \tau)
\]
when $y_k\stackrel{\text{i.i.d.}}{\sim} N(0, \sigma^2 h)$.
Since $\sum_{k=1}^K \left(\frac{y_k}{\sqrt{\sigma^2 h}}\right)^2\sim \chi^2(K)$, we have
\begin{equation}
\beta = P\left(\frac{1}{\sigma^2-1}(K\log \sigma + \log \tau), \frac{K}{2}\right). \label{eq:beta}
\end{equation}
Using \eqref{eq:alpha} and \eqref{eq:beta}, the trade-off curve between $\alpha$ and $\beta$ can be plotted by varying the parameter $\tau$.

Fig.~\ref{fig:alpha_beta} shows the $\alpha$--$\beta$ trade-off curve when the discretization step size is set to be $h=1/100, 1/200$ and $1/300$. It shows how different the trade-off curve can be if different step sizes are chosen. It can be numerically verified that both $\alpha$ and $\beta$ can be made arbitrarily small simultaneously by taking $h\searrow 0$.

\begin{figure}
\centering
\includegraphics[width = 0.8\textwidth]{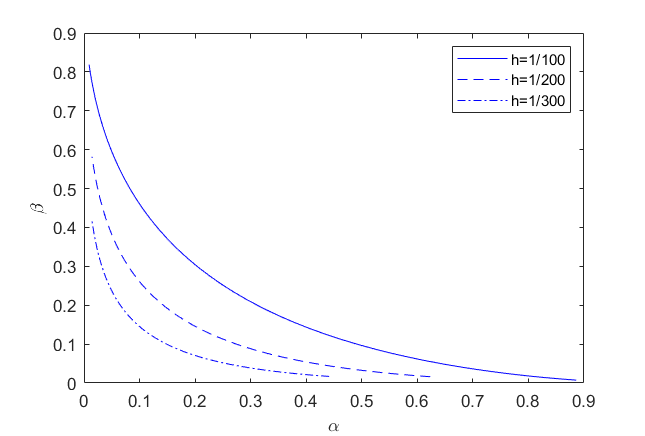}
\caption{$\alpha$--$\beta$ trade-off curves.}
\label{fig:alpha_beta}
\end{figure}
Let $\mu_0$ and $\mu_1$ be the probability measures under which $v_t$ in $H_0$ and $H_1$ are the standard Brownian motions, respectively. It can be shown that $\mu_0$ and $\mu_1$ are mutually singular, i.e., the Radon-Nikodym derivative $\frac{d\mu_1}{d\mu_0}$ is ill-defined. This implies that the hypothesis testing problem formulated above is ill-posed -- that is, given a realized sample path $v_t, 0\leq t \leq 1$, the detector can estimate which model has generated the path correctly with probability one.